\documentclass[12pt]{gatechthesis}

\usepackage{graphicx}
\usepackage{color}
\usepackage{algorithm}
\usepackage{algorithmic}
\usepackage{amsmath}
\usepackage{amssymb}
\usepackage{amsthm}
\usepackage{comment}
\usepackage{graphicx}
\usepackage{bbm}
\usepackage{url}
\usepackage{xr}
\usepackage{framed}
\usepackage{float}
\usepackage[framemethod=TikZ]{mdframed}

\usepackage{listings}
\usepackage{graphicx}
\usepackage{xspace}
\usepackage{hyperref}

\usepackage{xcolor}
\usepackage{multirow}

\newcommand{\g}{\gamma_{\K}}
\newcommand{\D}{\mathcal{D}}
\newcommand{\E}{\mathbb{E}}
\newcommand{\K}{\mathcal{K}}

\def\N{\mathcal{N}}
\newcommand{\W}[1]{ W^{(#1)}}
\def\v{\text{vec}}
\def\R{\text{ReLU}}

\def\t{t_0}
\def\FT{\mathcal{F}_{thred}}
\def\TT{\mathcal{T}_{thred}}

\def\g{\gamma_{\K}}

\def\D{\mathcal{D}}
\def\E{\mathbb{E}}
\def\K{\mathcal{K}}

\def\argmin{\mathop{\textnormal{argmin}}}
\def\argmax{\mathop{\textnormal{argmax}}}

\def\dist{\textnormal{dist}}

\newcommand{\citet}[1]{\citeauthor{#1} \cite{#1}}
\newcommand{\citep}[1]{(\citeauthor{#1} \cite{#1})}

\newcommand{\EE}[1]{\mathbb{E}_{#1}}

\newcommand{\reals}{\mathbb{R}}

\newcommand{\balpha}{\boldsymbol{\alpha}}
\newcommand{\alg}{\text{OAlg}\xspace}

\newcommand{\BTLb}{\textnormal{\textbf{FTL$^+$}}\xspace}
\newcommand{\FTLb}{\textnormal{\textbf{FTL}}\xspace}
\newcommand{\FTRLb}{\textnormal{\textbf{FTRL}}\xspace}

\newcommand{\BTRLb}{\textnormal{\textbf{FTRL$^+$}}\xspace}

\newcommand{\OFTRLb}{\textnormal{\textbf{OptimisticFTRL}}\xspace}
\newcommand{\BRb}{\textnormal{\textbf{BestResp$^+$}}\xspace}
\newcommand{\OFTLb}{\textnormal{\textbf{OptimisticFTL}}\xspace}
\newcommand{\MDb}{\textnormal{\textbf{OMD$^+$}}\xspace}

\newcommand{\BTL}{\textnormal{\textsc{FTL$^+$}}\xspace}
\newcommand{\FTL}{\textnormal{\textsc{FTL}}\xspace}
\newcommand{\FTRL}{\textnormal{\textsc{FTRL}}\xspace}
\newcommand{\FTPL}{\textnormal{\textsc{FTPL}}\xspace}
\newcommand{\BTRL}{\textnormal{\textsc{FTRL$^+$}}\xspace}

\newcommand{\OFTRL}{\textnormal{\textsc{OptimisticFTRL}}\xspace}
\newcommand{\BR}{\textnormal{\textsc{BestResp$^+$}}\xspace}
\newcommand{\OFTL}{\textnormal{\textsc{OptimisticFTL}}\xspace}
\newcommand{\MD}{\textnormal{\textsc{OMD$^+$}}\xspace}
\newcommand{\HB}{\textnormal{\textsc{HeavyBall}}\xspace}
\newcommand{\FW}{\textnormal{\textsc{Frank-Wolfe}}\xspace}

\newcommand{\xinit}{z_{\rm init}}

\newcommand{\uregret}[1]{\textsc{Reg}^{#1}_T}
\newcommand{\regret}[1]{\balpha\textsc{-Reg}^{#1}}
\newcommand{\avgregret}[1]{\overline{\balpha\textsc{-Reg}}^{#1}}
\newcommand{\XX}{\mathcal{X}}
\newcommand{\YY}{\mathcal{Y}}
\newcommand{\ZZ}{\mathcal{Z}}

\newcommand{\pr}[1]{\left(#1\right)}
\newcommand{\yof}{\widetilde{y}}
\newcommand{\yftl}{\hat{y}}
\newcommand{\xof}{\widetilde{x}}
\newcommand{\xav}{\bar{x}}
\newcommand{\yav}{\bar{y}}

\newcommand{\V}[1]{D_{#1}^{\phi}} 
\newacronym{FGNRD}{FGNRD}{Fenchel game no-regret dynamics}

\newtheorem{lemma}{Lemma}
\newtheorem{corollary}{Corollary}
\newtheorem{theorem}{Theorem}
\newtheorem{definition}{Definition}

\newtheorem{proposition}{Proposition}

\title{ Understanding Modern Techniques in Optimization: Frank-Wolfe, Nesterov's Momentum, and Polyak's Momentum}
\author{Jun-Kun Wang}
\approvaldate{}
\school{School of Computer Science}
\department{College of Computing}
\bibliography{references}

\begin{document}

\makeTitlePage{May}{2021}


\begin{approvalPage}{5}


\committeeMember{Dr. Jacob Abernethy}{School of Computer Science}{Georgia Institute of Technology}
\committeeMember{Dr. Guanghui Lan}{Industrial  \& System Engineering}{Georgia Institute of Technology}
\committeeMember{Dr. Vidya Muthukumar}{
School of Electrical and Computer Engineering \\
Industrial  \& System Engineering}{Georgia Institute of Technology}
\committeeMember{Dr. Richard Peng}{School of Computer Science}{Georgia Institute of Technology}
\committeeMember{Dr. Santosh Vempala}{School of Computer Science}{Georgia Institute of Technology}

\end{approvalPage}
\makeDedication{ To my parents
}

\begin{frontmatter}
    
\begin{acknowledgments}

I must begin by expressing my deep gratitude to my advisor, Dr. Jacob Abernethy. Jake has been an outstanding mentor, and he always cares about his students. He gave me a lot of constructive feedback in research and life, and he is very supportive. His patience and advice helped me greatly improve my writing of research papers and improve the skills of giving presentations. The five years of my Ph.D. journey were not possible without him. 
I will always remember the discussion during which Jake came up with the idea of interpreting Frank-Wolfe as playing a two-player game; I was fortunate in my first year of Ph.D. to learn how Jake identified a connection, wrote down the idea on a note, and managed to meet a conference deadline. I feel this experience helps me a lot in doing research, especially for developing the skills of seeing connections between different algorithms and connections between solving different problems. I also thank him gave me a great deal of freedom in my Ph.D. research. 

The five years of Ph.D. are some of my happiest and enjoyable experience, and this is largely due to Jake and also the professors and peers I collaborated with. I want to thank my other collaborators: Bhuvesh Kumar,
Chi-Heng Lin, Guanghui Lan, Kevin Lai, and Kfir Levy for the valuable discussions. I also want to thank my committee members Dr. Guanghui Lan, Dr. Richard Peng, Dr. Santosh Vempala, and Dr. Vidya Muthukumar for giving advice.

Finally, I cannot thank my family enough for always supporting me and encouraging me. My mom and my dad give me love and caring every step of the way. I am grateful for their love and support. I also thank my sister for her encouragement. I always feel happy and fortunate when my parents, my sister, and I have a chat online over the past five years. I love my family dearly.

\end{acknowledgments}

    \makeTOC
    \makeListOfTables
    \makeListOfFigures
    
\begin{summary}

Optimization is essential in machine learning, statistics, and data science. Among the first-order optimization algorithms, the popular ones include the Frank-Wolfe method, Nesterov's accelerated methods, and Polyak's momentum. While theoretical analysis of the Frank-Wolfe method and Nesterov's methods are available in the literature, the analysis can be quite complicated or less intuitive. Polyak's momentum, on the other hand, is widely used in training neural networks and is currently the default choice of momentum in Pytorch and Tensorflow. It is widely observed that Polyak's momentum helps to train a neural network faster, compared with the case without momentum. However, there are very few examples that exhibit a provable acceleration via Polyak's momentum, compared to vanilla gradient descent. There is an apparent gap between the theory and the practice of Polyak's momentum.

In the first part of this dissertation research, we develop a modular framework that can serve as a recipe for constructing and analyzing iterative algorithms for convex optimization. Specifically, our work casts optimization as iteratively playing a two-player zero-sum game. Many existing optimization algorithms including Frank-Wolfe and Nesterov's acceleration methods can be recovered from the game by pitting two online learners with appropriate strategies against each other. Furthermore, the sum of the weighted average regrets of the players in the game implies the convergence rate. As a result, our approach provides simple alternative proofs to these algorithms. Moreover, we demonstrate that our approach of ``optimization as iteratively playing a game'' leads to three new fast Frank-Wolfe-like algorithms for some constraint sets, which further shows that our framework is indeed generic, modular, and easy-to-use.

In the second part, we develop a modular analysis of provable acceleration via Polyak's momentum for certain problems, which include solving the classical strongly quadratic convex problems, training a wide ReLU network under the neural tangent kernel regime, and training a deep linear network with an orthogonal initialization. We develop a meta theorem and show that when applying Polyak’s momentum for these problems, the induced dynamics exhibit a form where we can directly apply our meta theorem. 

In the last part of the dissertation, we show another advantage of the use of Polyak's momentum --- it facilitates fast saddle point escape in smooth non-convex optimization. This result, together with those of the second part, sheds new light on Polyak's momentum in modern non-convex optimization and deep learning.

\end{summary}
\end{frontmatter}

\begin{thesisbody}
    
\chapter{Introduction}


\section{Bridging classical convex optimization and online learning
via Fenchel game 
}
In machine learning and data science, training a model is essentially solving an optimization problem,
\begin{equation} \label{minf}
\min_{w \in \K} f(w),
\end{equation}
where $w$ is a vector that represents a model, $\K \subseteq \reals^d$ is a constraint set, 
and $f(\cdot)$ is an objective function which is typically a loss function, e.g. prediction errors of the model over a training dataset. In other words, we are searching for the best model $w$ satisfying $\K$ that minimizes the objective value. 

For $f(\cdot)$ being convex, there are many well-established results in optimization literature
and quite a few textbooks cover the results well, 
see e.g. \citet{BT01}, \citet{BN01}, \citet{HL93}, \citet{R96}, \citet{N04}, \citet{B04}, \citet{BL06}.
Online learning (a.k.a. no-regret learning), on the other hand, is a growing and an active research area in machine learning, see e.g. \citet{LW94}, \citet{cesa2006prediction}, \citet{S07}, \citet{H14}, \citet{OO19}, \citet{RS16}.
The standard protocol in online learning is that
in each round $t$, the learner must select a point $x_t \in \K$, where $\K$ is her decision space.
Then the learner is charged a loss $\ell_t(x_t)$ and 
typically can observe the loss function $\ell_t(\cdot)$ after she takes an action $x_t$.
The objective of interest in most of the online learning literature is the learner's \emph{regret}, defined as
\begin{equation}
  \uregret{x} := \sum_{t=1}^T \ell_t(x_t) - \min_{x \in \K} \sum_{t=1}^T \ell_t(x).
\end{equation}
The goal of the learner is to minimize her regret and to compete with the comparator who foresees all the loss functions and commits to a fixed action. 
\begin{algorithm}[H]
\caption{Online-to-batch conversion (adapted from the presentation of \citet{L07})}
\label{alg:online-to-batch}
\begin{algorithmic}[1] \normalsize
\STATE \textbf{Input:} number of rounds $T$.
\STATE \textbf{Input:} Training data \{$s_1,s_2, \dots, s_T$\} and an online learning algorithm $\alg^x$.
\FOR{$t=1, 2, \ldots, T$}
\STATE Let $x_t$ be the action of the online learning algorithm $\alg^x$ at $t$.
\STATE Feed $\alg^x$ with $\ell_t(\cdot) := f(\cdot;s_t)$ as the loss function, where $f(\cdot;s_t)$ is a loss function associated with a sample $s_t$.
\ENDFOR
\STATE \textbf{Output:} $\bar{x} = \frac{ \sum_{t=1}^T x_t }{T}$
\end{algorithmic}
\end{algorithm} 

A natural question is  ``Can we apply an online learning algorithm to solving an offline problem (\ref{minf}) and also obtain some theoretical guarantees?''. The answer is yes and there is a way to achieve this goal by a technique called the ``Online-to-Batch Conversion'', which is shown on Algorithm~\ref{alg:online-to-batch} and has the following guarantee (see also e.g. \citet{CCG04} and Appendix~B in \citet{S07}).
\begin{theorem}[Adapted from the presentation of \citet{L07}] \label{thm:otob}
Assume $\E_{s \in \D}[f(\cdot;s)]$ is convex.
Then, with probability $1-\delta$, the online-to-batch conversion (Algorithm~\ref{alg:online-to-batch}) guarantees that
\[
\E_{s \sim D} f(\bar{x},s) \leq \E_{s \sim D} f(x_*;s) + \frac{\uregret{x}}{T} +2 \sqrt{ \frac{2 \log ( 2 / \delta)}{T} }, 
\]
where $x_* \in \arg\min_{x \in \K} \E_{s \sim D} f(x;s)$.
\end{theorem}
Theorem~\ref{thm:otob} says that the average regret of the player, $\frac{\uregret{x}}{T}$, gives a bound of the convergence rate for solving the offline problem (\ref{minf}).

It is well known in the literature that if the 
online loss functions $\{ \ell_t(\cdot)\}$ are convex, then the optimal regret is $O\left(\sqrt{T}\right)$; if the online loss functions are strongly convex, then $O\left(\log (T)\right)$ is achievable, see e.g. \citet{S07}, \citet{cesa2006prediction}, \citet{RS16}.
By Theorem~\ref{thm:otob}, these regret bounds imply a convergence rate $O\left(\frac{1}{\sqrt{T}}\right)$ or $O\left(\frac{\log (T)}{T}\right)$ for solving (\ref{minf}) when we \emph{convert} an online learning algorithm to an offline one.
However, in the optimization literature,
there are algorithms that achieve an accelerated rate $O\left(\frac{L}{T^2}\right)$ for solving a $L$-smooth convex optimization problem,
e.g. Nesterov's methods (\cite{N83a,N83b,N05,N04}).
Furthermore, when an underlying problem is both $\mu$-strongly convex and $L$-smooth, the optimal convergence rate is $O\left(\exp\left(-\frac{T}{\sqrt{\kappa}}\right)\right)$, where $\kappa:= \frac{L}{\mu}$ is the condition number of the underling function $f(\cdot)$, see e.g. \citet{L20}.
The gap implies that (offline) optimization and online learning have not been well-connected yet. In this thesis, we will show how to bridge optimization and online learning in a modular and unified way.
\paragraph{Our contributions:}

Our approach of connecting offline convex optimization and online learning is based on iteratively solving the following two-player game 
which we call the \emph{Fenchel Game}. We
define the payoff function of the game $g : \K \times \reals^d$ as follows:
\begin{equation} 
 g(x,y) := \langle x, y \rangle - f^*(y), 
\end{equation}
where $f(\cdot)$ is the underlying function of (\ref{minf})
and $f^*(\cdot)$ is the conjugate of $f(\cdot)$, defined as
$f^*(y) := \sup_{x \in \text{dom}(f) }  \langle x, y \rangle - f(x)$.
In this game, the $y$-player tries to maximize the payoff function $g(\cdot,\cdot)$, while the $x$-player tries to minimize it.
The equilibrium of this game is $\min_{w \in \K} f(w)$ under the assumption that $f(\cdot)$ is convex and lower
semi-continuous. Therefore, approximately solving the game is equivalent to approximately solving the offline convex problem (\ref{minf}).

This game perspective provides a modular framework for designing and analyzing offline convex optimization algorithms. 
We will show that several
algorithms together with their convergence rates can be recovered from our approach of \emph{optimization as iteratively playing a game}.
The algorithms that we will recover include Frank-Wolfe \cite{frank1956algorithm} and its several variants \cite{LP66,L13,Netal20,LF20}, Nesterov's accelerated methods \cite{N83a,N83b,N05,N04} and their variants, Heavy Ball \cite{P64}, and the accelerated proximal method \cite{BT09}.
In particular, we show that the tools and techniques in \emph{online learning} can actually be used to design accelerated algorithms in offline convex optimization. We will establish the accelerated rate $O\left(\frac{1}{T^2}\right)$ and the accelerated linear rate $O\left(\exp\left(-\frac{T}{\sqrt{\kappa}}\right)\right)$ by using the regret analysis with an appropriate weighting scheme. 

Most importantly, our insight of \emph{optimization as iteratively playing a game} leads to three new fast Frank-Wolfe-like algorithms for certain constraints sets.
Specifically, we propose a Frank-Wolfe-like algorithm that works for non-smooth convex functions \emph{without} using the techniques of smoothing \cite{DBW12} (Algorithm~\ref{alg:new_fw}), 
an accelerated $O\left(\frac{1}{T^2}\right)$ rate Frank-Wolfe-like algorithm for smooth convex problems with constraint sets satisfying a notion called strongly convex (Algorithm~\ref{alg:gaugeFW}), and a fast  parallelizable projection-free algorithm for the nuclear-norm-ball constraint (Algorithm~\ref{alg:main}).
The introduction of the new algorithms verifies that our approach is indeed very modular. 

Our results are summarized in 
Table~\ref{tab:summary_exist} and Table~\ref{tab:summary_new} in Chapter~\ref{ch2}. The materials of Chapter~\ref{ch2} are based on the following papers.
\begin{itemize}
\item \citetitle{AW17}. \newline
Jacob Abernethy and Jun-Kun~Wang. NeurIPS 2017 (Spotlight).
\item \citetitle{ALLW18}.\newline
Jacob Abernethy, Kevin Lai, Kfir Levy, and Jun-Kun~Wang. COLT 2018.
\item \citetitle{WA18}. \newline
Jun-Kun~Wang and Jacob Abernethy. NeurIPS 2018 (Spotlight).
\item \citetitle{WKAL21}. 
Jun-Kun~Wang, Bhuvesh Kumar, Jacob Abernethy, and Guanghui Lan.
\end{itemize}

\section{Acceleration via Polyak's momentum in deep learning}

Polyak's momentum (Algorithm~\ref{alg:HB1} and Algorithm~\ref{alg:HB2}) is very popular nowadays for training neural networks and it is the default choice of momentum in PyTorch and Tensorflow. The success of Polyak's momentum in deep learning is widely appreciated and almost all of the recently-developed adaptive gradient methods like Adam \cite{KB15} and AMSGrad \cite{RKK18} adopt the use of Polyak's momentum, in favor of Nesterov's momentum. 

Despite its empirical success in modern machine learning, there is limited theory showing any advantage over vanilla gradient descent. As far as we know, the strongly quadratic convex problem is perhaps  the only known example such that \emph{discrete-time} Polyak's momentum has a provable acceleration in terms of the \emph{global convergence} compared with vanilla gradient descent.  
Most of the existing results (e.g. \cite{P64,LRP16}) only establish a convergence rate in the limit, which is due to the use of Gelfand's formula \cite{G41} for approximating the spectral norm of a matrix by its spectral radius. In other words, these results fail to explain the behavior of Polyak's momentum in the non-asymptotic regime even for the classical strongly quadratic convex problems. Moreover, before our work, we are not aware of any theoretical works showing any provable acceleration of Polyak's momentum over vanilla GD in deep learning. Understanding Polyak's momentum remains elusive even though empirically Polyak's momentum appears to provide acceleration in modern machine learning problems.

\paragraph{Our contributions:}

In Chapter~\ref{ch3},
we will develop a modular analysis of Polyak's momentum when applied to  the following problems.
\begin{itemize}
\item \textbf{Strongly convex quadratic problems}

The objective is
\begin{equation} 
\textstyle
 \min_{w \in \reals^d} \frac{1}{2} w^\top \Gamma w + b^\top w,
\end{equation} 
where $\Gamma \in \reals^{d \times d}$ is a symmetric matrix such that $\lambda_{\min}(\Gamma)  > 0$.
We can define the condition number as
\[
\kappa^{\text{SC}}:= \frac{\lambda_{\max}(\Gamma)}{\lambda_{\min}(\Gamma)}.
\]

\item (\textbf{Training a wide ReLU network with the squared loss})

We will consider training the following ReLU network by Polyak's momentum, 
\begin{equation} 
\N_{W}^{\text{ReLU}}(x) := \frac{1}{\sqrt{m} } \sum_{r=1}^m a_r \sigma( \langle w^{(r)},  x \rangle ),
\end{equation}
where $\sigma(z):= z \cdot \mathbbm{1}\{ z \geq 0\}$ is the ReLU activation,
$w^{(1)}, \dots, w^{(m)}  \in \reals^d$ are the weights of $m$ neurons on the first layer, $a_1, \dots, a_m \in \reals$ are weights on the second layer, 
 and $\N_{W}^{\text{ReLU}}(x) \in \reals$ is the output predicted on input $x$. 

Giving $n$ number of training samples, following \cite{DZPS19,ADHLSW19_icml,ZY19},
we define a Gram matrix $H \in \reals^{n \times n}$ for the weights $W$,
and its expectation $\bar{H} \in \reals^{n \times n}$ over the random draws of $w^{(r)} \sim N(0,I_d) \in \reals^d$, 
as follows,
\begin{equation}
\begin{aligned}
& H(W)_{i,j}   =  \sum_{r=1}^m \frac{x_i^\top x_j}{m} \mathbbm{1}\{ \langle w^{(r)}, x_i \rangle \geq 0 \text{ } \&  \text{ }   \langle w^{(r)}, x_j \rangle \geq 0 \}
\\ & \quad
\bar{H}_{i,j}  := \underset{ w^{(r)}}{\mathbbm{E}}
[ x_i^\top x_j \mathbbm{1}\{ \langle w^{(r)}, x_i \rangle \geq 0 \text{ } \&  \text{ }   \langle w^{(r)}, x_j \rangle \geq 0 \}     ] .
\end{aligned}
\end{equation}
The matrix $\bar{H}$ is also called a neural tangent kernel (NTK) matrix in the literature (e.g. \cite{JGH18,Y19,BM19}).
We can denote the condition number of the neural tangent kernel matrix $\bar{H}$ as
\[
\kappa^{\text{ReLU}}:= \frac{\lambda_{\max}(\bar{H})}{\lambda_{\min}(\bar{H})}.
\]
\item (\textbf{Training a deep linear network with the squared loss})

We will also consider training the following deep linear network by Polyak's momentum, 
\begin{equation} 
\N_W^{L\text{-linear}}(x) := \frac{1}{\sqrt{m^{L-1} d_{y}}} \W{L} \W{L-1} \cdots \W{1} x,
\end{equation}
where $\W{l} \in \reals^{d_l \times d_{l-1}}$ is the weight matrix of the layer $l \in [L]$, and $d_0 = d$, $d_L = d_y$ and $d_l = m$ for $l \neq 1, L$.
Let 
\[
\begin{aligned}
& 
H_t \textstyle := \frac{1}{ m^{L-1} d_y } \sum_{l=1}^L [ (\W{l-1:1}_t X)^\top (\W{l-1:1}_t X ) 
\otimes
  \W{L:l+1}_t (\W{L:l+1}_t)^\top ]   \in \reals^{d_y n \times d_y n}.
\end{aligned}
\] We will denote the condition number of $H_0$ as
\[
\kappa^{\text{L-linear}}:= \frac{\lambda_{\max}(H_0)}{\lambda_{\min}(H_0)}.
\]
\end{itemize}

\begin{theorem} (Informal; see Chapter~\ref{ch3})
By setting the momentum parameter $\beta$ and $\eta$ appropriately,
Polyak's momentum (Algorithm~\ref{alg:HB1} and Algorithm~\ref{alg:HB2}) for the three problems aforementioned has
\[
\left\|
\begin{bmatrix}
\xi_t \\
\xi_{t-1} 
\end{bmatrix}
\right\| \leq \left( 1 - \frac{1}{4 \sqrt{\kappa} } \right)^{t}
\cdot 8 \sqrt{\kappa}
\left\|
 \begin{bmatrix}
\xi_0 \\
\xi_{-1}
\end{bmatrix}
\right\|,
\]
where $\xi_t$ is some residual vector and $\kappa = \{ \kappa^{\text{SC}}, \kappa^{\text{ReLU}}, \kappa^{\text{L-linear}} \}$ is the condition number of the underlying problem.
\end{theorem}
Our theorem shows the advantage of Polyak's momentum over vanilla gradient descent, as the convergence rate depends on the square root of the condition number $\sqrt{\kappa}$, while the rate of vanilla GD has a dependency on $\kappa$. Our work hence shows that Polyak's momentum does improve the neural net training at least for the two canonical models.

Chapter~\ref{ch3} of this thesis is based on the following paper.
\begin{itemize}
\item \citetitle{WLA21}  Jun-Kun Wang, Chi-Heng Lin, and Jacob Abernethy. ICML. 2021
\end{itemize}

\section{Exploiting negative curvatures via stochastic Polyak's momentum:}

In smooth non-convex optimization, when the iterate enters a region of strict saddle points, defined as
\begin{equation} 
\big\{ w \in \text{dom}(f): \| \nabla f(w) \| \leq \epsilon  \text{ and }  \nabla^2 f(w) \preceq -\epsilon I \big\},
\end{equation}
the optimization progress slows down. Therefore, it is very important to quickly \emph{escape} the saddle point region.
In the literature,
there are specialized algorithms designed to exploit the negative curvature explicitly and can escape the saddle point region faster than alternative methods (e.g. \cite{CDHS18,AABHM17,AL18,XRY18}). There are also
simple GD/SGD variants with minimal tweaks of standard GD/SGD (e.g. \cite{GHJY15,KL16,FLZCOLT19,JGNKJ17,CNJ18,JNGKJ19,DKLH18,SRRKKS19}).
However, none of these works study SGD with Polyak's momentum for escaping saddle points.

\paragraph{Our contributions:}

We will show that, under certain assumption and some minor constraints that upper-bound parameter $\beta$, if SGD with Polyak's momentum has some properties, then we demonstrates that a larger momentum parameter $\beta$ can help in escaping saddle points faster.
Some experiments are provided to support our theoretical results.
As saddle points are pervasive in the loss landscape of optimization in deep learning (\cite{dauphin14,CHMAL15}),
this result could help to explain why SGD with momentum enables training faster in optimization for deep learning.
We then provide some empirical findings showing that over-parametrization, which is another popular technique in modern machine learning, can help gradient descent exploit negative curvature in the so-called phase retrieval problem. Some discussions are provided in the end. 

Chapter~\ref{ch4} of this thesis is based on the following paper.
\begin{itemize}
\item \citetitle{WCA20}  Jun-Kun Wang, Chi-Heng Lin, and Jacob Abernethy. ICLR. 2020.
\end{itemize}

\chapter{Fenchel Game: A Modular Approach of Solving Convex Optimization via Iteratively Playing a Two-Player Game} \label{ch2}

\section{Introduction}

The main goal of this work is to develop a framework for solving convex optimization problems using iterative methods. Given a convex function $f : \reals^d \to \reals$, domain $K \subset \reals^d$, and some tolerance $\epsilon > 0$, we want to find an approximate minimizer $x \in \K$ so that $f(x) - \min_{x' \in \K} f(x') \leq \epsilon$, using a sequence of oracle calls to $f$ and its derivatives. This foundational problem has received attention for decades, and researchers have designed numerous methods for this problem under a range of oracle query models and structural assumptions on $f(\cdot)$. What we aim to show in this chapter is that a surprisingly large number of these methods---including those of Nesterov \cite{N83a,N83b,N05,N88,N04}, Frank and Wolfe \cite{frank1956algorithm}, Polyak \cite{P64}, and Beck and Teboulle \cite{BT09}---can all be described and analyzed through a single unified algorithmic framework, which we call the \emph{Fenchel game no-regret dynamics} (FGNRD). We show that several novel methods, with fast rates, emerge from FGNRD as well.

Let us give a short overview before laying out the FGNRD framework more precisely. A family of tools, largely developed by researchers in theoretical machine learning, consider the problem of sequential prediction and decision making in non-stochastic environments, often called \emph{adversarial online learning}. This online learning setting has found numerous applications in several fields beyond machine learning---finance, for example, as well as statistics---but it has also emerged as a surprisingly useful tool in game theory. What we call \emph{no-regret online learning algorithms} are particularly well-suited for computing equilibria in two-player zero-sum games, as well as solving saddle point problems more broadly. If each agent employs a no-regret online learning algorithm to choose their action at each of a sequence of rounds, it can be shown that the agents' choices will converge to a saddle point, and at a rate that depends on their choice of learning algorithm. Thus, if we are able to simulate the two agents' sequential strategies, where each aims to minimize the ``regret'' of their chosen actions, then what emerges from the resulting \emph{no-regret dynamics} (NRD) can be implemented explicitly as an algorithm for solving min-max problems. 

How does NRD help us to develop and analyze methods for minimizing a convex $f$?
What is our main focus in the present work is a particular game of interest which we call the \emph{Fenchel game}: from $f$ we can construct a two-input ``payoff'' function $g : \reals^d \times \reals^d \to \reals$ defined by
\[
  g(x,y) := \langle x, y \rangle - f^*(y).
\]
We view this as a game in the sense that if one player selects an action $x$ and a second player selects action $y$, then $g(x,y)$ is the former's ``cost'' and the latter's ``gain'' associated to their decisions. If the two players continue to update their decisions sequentially, first choosing $x_1$ and $y_1$ then $x_2$ and $y_2$, etc., and each player relies on some no-regret algorithm for this purpose, then one can show that the time-averaged iterates $\bar x, \bar y$ form an approximate equilibrium of the Fenchel game---that is, $g(\bar x, y') - \epsilon \leq g(\bar x, \bar y) \leq g(x', \bar y) + \epsilon$ for any alternative $x',y'$. But indeed, this approximate equilibrium brings us right back to where we started, since using the construction of the Fenchel game it is easy to show that $\bar x$ then satisfies $f(\bar x) - \min_{x \in \K} f(x) \leq \epsilon$. The approximation factor $\epsilon$ is important, and we will see that it depends upon the number of iterations of the dynamic and the players' strategies.

What FGNRD gives us is a recipe book for constructing and analyzing iterative algorithms for convex optimization. To simulate a dynamic we still need to make particular choices as for both players' strategies and analyze their performance. We begin in Section~\ref{sec:onlinelearning} by giving a brief overview of tools from adversarial online learning, and we introduce a handful of simple online learning algorithms, including variants of FollowTheLeader and OnlineMirrorDescent, and prove bounds on the \emph{weighted} regret---we generalize slightly the notion of regret by introducing weights $\alpha_t > 0$ for each round. We will also prove a key result that relates the error $\epsilon$ of the approximate equilibrium pair $\bar x, \bar y$, which are the weighted-average of the iterates of the two players, to the weighted regret of the players' strategies. In Section~\ref{sec:ExistingAlgs} we show how several algorithms, including the Heavy Ball method \cite{P64}, Frank-Wolfe's method \cite{frank1956algorithm}, and several variants of Nesterov Accelerated Gradient Descent \cite{N83a,N83b,N05,N88,N04,BT09}, are all special cases of the FGNRD framework, all with special choices of the learning algorithms for the $x-$ and $y-$players, and the weights $\alpha_t$; see Table~\ref{tab:summary_exist} for a summary of these recipes. In addition we provide several new algorithms using FGNRD in Section~\ref{sec:NewAlgs}, summarized in Table~\ref{tab:summary_new}.


\section{Preliminaries}


We summarize some results in convex analysis that will be used in this chapter. We also refer the readers to some excellent textbooks
(e.g. \cite{BT01,HL93,R96,N04,B04,BL06}).

\paragraph{Smoothness and strong convexity}


A function $f(\cdot)$ on $\reals^d$ is $L$-smooth with respect to a norm $\| \cdot \|$ if $f(\cdot)$ is everywhere differentiable and
it has Lipschitz continuous gradient
$\| \nabla f(x) - \nabla f(z) \|_* \leq L \| x - z\|$,  
where $\| \cdot \|_{*}$ denotes the dual norm.
A function $f(\cdot)$ is $\mu$-strongly convex  w.r.t. a norm $\| \cdot \|$ if the domain of $f(\cdot)$ is convex and that
$f( \theta x + (1-\theta) z ) \leq \theta f(x) + (1-\theta) f(z)  - \frac{\mu}{2} \theta (1-\theta) \| x- z \|^2$ for all $x,z \in \text{dom}(f)$ and $\theta \in [0,1]$.
If a function is $\mu$-strongly convex, then
$f(z) \geq f(x) + \partial f(x)^\top (z-x) + \frac{\mu}{2} \| z - x \|^2$ for 
all $x,z \in \text{dom}(f)$,
where  $\partial f(x)$ denotes a subgradient of $f$ at $x$.

\paragraph{Convex function and conjugate}
For any convex function $f(\cdot)$, its Fenchel conjugate is 
\begin{equation}
f^*(y) := \sup_{x \in \text{dom}(f) }  \langle x, y \rangle - f(x)
\end{equation}
If a function $f(\cdot)$ is convex, then its conjugate $f^*(\cdot)$ is also convex, as it is a supremum over linear functions.
Furthermore, if the function $f(\cdot)$ is closed and convex,
the following are equivalent:
(I) $y \in \partial f(x)$, (II) $x \in \partial f^*(y)$, and (III)
\begin{equation}
\langle x, y \rangle = f(x) + f^*(y),
\end{equation}
which also implies that the biconjudate is equal to the original function, i.e. $f^{**}(\cdot) = f(\cdot)$.
Moreover,
when the function $f(\cdot)$ is differentiable,
we have $\nabla f(x) = \displaystyle \sup_{y}  \langle x, y \rangle - f^*(y) $.  
We refer to the readers to \citet{BL12}, \citet{KST09}, and textbooks (e.g. \cite{R96,B04,BL06}) for more details of Fenchel conjugate.
Througout this chapter, unless specifically mentioned, we assume that the underlying convex function is proper, closed, and differentiable. 

An important property of a closed and convex function is that
$f(\cdot)$ is $L$-smoooth w.r.t. some norm $\| \cdot \|$ if and only if its conjugate $f^*(\cdot)$ is $1/L$-strongly convex w.r.t. the dual norm $\| \cdot\|_*$ (e.g. Theorem~6 in \citet{KST09}).

\paragraph{Bregman Divergence.}

We will denote the Bregman divergence $\V{z}(\cdot)$ centered at a point $z$ with respect to a $\beta$-strongly convex distance generating function $\phi(\cdot)$ as 
\begin{equation}\label{eq:BregmanDef}
\V{z}(x) := \phi(x) - \langle \nabla \phi(z), x - z \rangle  - \phi(z).
\end{equation}

\paragraph{Strongly convex sets.}
A convex set $\K \subseteq \reals^m$ is an \emph{$\lambda$-strongly convex set} w.r.t. a norm $\| \cdot \|$ 
if for any $x, z \in \K$, any $\theta \in [0,1]$, 
the $\| \cdot \|$ ball centered at $ \theta x + ( 1 - \theta) z$ with radius 
$\theta (1 - \theta) \frac{\lambda}{2} \| x - z \|^2$ is included in $\K$ \cite{D15}. Examples of strongly convex sets include 
$\ell_p$ balls: $\| x \|_p \leq r, \forall p \in (1,2]$,
Schatten $p$ balls: $\| \sigma(X) \|_p \leq r$ for $p \in (1,2]$,
and Group (s,p) balls: $\| X \|_{s,p}  = \| (\| X_1\|_s, \| X_2\|_s, \dots, \| X_m\|_s)  \|_p \leq r$ (see e.g. \cite{D15}).

\begin{table}[ht!]
\caption{Summary of recovering existing optimization algorithms from \emph{Fenchel Game}.
Here $T$ denotes the total number of iterations, 
$\alpha_t$ are the weights which set the emphasis on iteration $t$, the last two columns on the table indicate the specific strategies of the players in the FGNRD.}
\label{tab:summary_exist}       
\centering
\small
\begin{tabular}{|c|c|c|c|c|} \hline
\multicolumn{5}{c}{ \shortstack{ $L$-Smooth convex optimization: $\min_w f(w)$ \\ as a game $g(x,y):= \langle x , y \rangle - f^*(y)$.} }  \\ \hline
Algorithm   &  rate &  weight &  y-player &  x-player \\ \hline
\shortstack{
Frank-Wolfe  \\ method \cite{frank1956algorithm}}
& 
\shortstack{ Thm.~\ref{thm:equivFW} and~\ref{thm:fwconvergence}   \\
$O(\frac{L \log T}{T})$ } & $\alpha_t=1$ & \shortstack{ Sec.~\ref{section:FTL}\\ \FTL } & \shortstack{ Sec.~\ref{section:BR}\\ \BR }  \\ \hline
\shortstack{
Frank-Wolfe  \\ method \cite{frank1956algorithm}}
& 
\shortstack{ Thm.~\ref{thm:equivFW} and~\ref{thm:fwconvergence}   \\
$O(\frac{L }{T})$ } & $\alpha_t=t$ & \shortstack{ Sec.~\ref{section:FTL}\\ \FTL } & \shortstack{ Sec.~\ref{section:BR}\\ \BR } \\ \hline
\shortstack{
Linear rate \\ FW \cite{LP66} }
& 
\shortstack{ Thm.~\ref{thm:linearFW} \\
$O(\exp(- \frac{\lambda T}{L} ))$} & $\alpha_t=\frac{1}{\| \ell_t(x_t) \|^2}$ & \shortstack{ Sec.~\ref{section:FTL}\\ \FTL }  &  \shortstack{ Sec.~\ref{section:BR}\\ \BR } \\ \hline
\shortstack{
Nesterov's  \\  ($1$- memory) \\ method \cite{N88}}
& \shortstack{ Thm.~\ref{thm:metaAcc} and ~\ref{thm:Nes_constrained} \\
$O(\frac{L}{T^2})$ }& $\alpha_t=t$ & \shortstack{ Sec.~\ref{section:OFTL}\\ \OFTL } &  \shortstack{ Sec.~\ref{section:MD}\\ \MD } \\ \hline
\shortstack{
Nesterov's  \\  ($\infty$- memory) \\ method \cite{N05}}
& \shortstack{ Thm.~\ref{thm:metaAcc} and ~\ref{thm:Nes_constrained} \\
$O(\frac{L}{T^2})$ } & $\alpha_t=t$ & \shortstack{ Sec.~\ref{section:OFTL}\\ \OFTL } & \shortstack{ Sec.~\ref{section:BTRL}\\ \BTRL } \\ \hline
\shortstack{
Nesterov's  \\ first acceleration \\ method \cite{N83b}}
& 
\shortstack{ Thm.~\ref{thm:metaAcc} and ~\ref{thm:Nes1988} \\
$O(\frac{L}{T^2})$} & $\alpha_t = t$ & \shortstack{ Sec.~\ref{section:OFTL}\\ \OFTL } & \shortstack{ Sec.~\ref{section:MD}\\ \MD with \\ $\phi_t(x)= \frac{1}{2}\| x \|^2_2$} \\ \hline
\shortstack{
Heavy Ball  \\ method \cite{P64} }
& 
\shortstack{ Thm.~\ref{thm:Heavy} \\
$O(\frac{L}{T})$} & $\alpha_t=t$ & \shortstack{ Sec.~\ref{section:FTL}\\ \FTL } & \shortstack{ Sec.~\ref{section:BTRL}\\ \BTRL } \\ \hline \hline
\multicolumn{5}{c}{ \shortstack{ Non-smooth convex optimization: $\min_w f(w)$ \\ as a game $g(x,y):= \langle x , y \rangle - f^*(y)$. } }  \\ \hline \hline
Algorithm   &  rate &  weight &  y-player &  x-player \\ \hline
\shortstack{Smoothed \\ FW \cite{L13} }
& 
$O(\frac{1}{\sqrt{T}})$ & $\alpha_t=1$ & \shortstack{ Sec.~\ref{section:FTPL}\\ \FTPL }  &  \shortstack{ Sec.~\ref{section:BR}\\ \BR } \\ \hline \hline
\multicolumn{5}{c}{ \shortstack{ Composite optimization: $\min_w f(w) + \psi(w)$, where $\psi(\cdot)$ is possibly non-differentiable, \\ as a game $g(x,y):= \langle x , y \rangle - f^*(y) + \psi(x)$.} }  \\ \hline \hline
Algorithm   &  rate &  weight &  y-player &  x-player \\ \hline
\shortstack{
Accelerated  \\  proximal \\  method \cite{BT09}}
& 
\shortstack{ Thm.~\ref{thm:metaAcc} and~\ref{thm:proximal} \\
$O(\frac{L}{T^2})$ } & $\alpha_t=t$ & \shortstack{ Sec.~\ref{section:OFTL}\\ \OFTL } &  \shortstack{ Sec.~\ref{section:MD}\\ \MD } \\ \hline \hline
\multicolumn{5}{c}{\shortstack{ $L$-smooth and $\mu$ strongly convex optimization: $\min_w f(w)$ \\ as a game $g(x,y):= \langle x , y \rangle - \tilde{f}^*(y) + \frac{\mu \| x \|^2}{2}$, where $\tilde{f}(\cdot):= f(\cdot) - \frac{\mu}{2}\| \cdot \|^2 $}. }  \\ \hline \hline
Algorithm   &  rate &  $\alpha_t$ & \shortstack{ y-player} &  \shortstack{x-player} \\ \hline
\shortstack{
Nesterov's  \\ method \cite{N04} }
& 
\shortstack{ Thm.~\ref{thm:acc_linear2} \\
$O(\exp( - \sqrt{ \frac{\mu}{L} } T ))$ } & $\alpha_t \propto \exp(t)$ & \shortstack{ Sec.~\ref{section:OFTL}\\ \OFTL } &  \shortstack{ Sec.~\ref{section:BTRL}\\ \BTRL } \\ \hline
\end{tabular}
\end{table}

\begin{table}[ht!]
\caption{Summary of \emph{new} optimization algorithms from \emph{Fenchel Game}.
Here $T$ denotes the total number of iterations, 
$\alpha_t$ are the weights which set the emphasis on iteration $t$, the last two columns on the table indicate the specific strategies of the players in the FGNRD.
}
\label{tab:summary_new}       
\small
\centering
\begin{tabular}{|c|c|c|c|c|} \hline
\multicolumn{5}{c}{ \shortstack{ Non-smooth convex optimization: $\min_{w \in \K} f(w)$, where $\K$ is a $\lambda$-strongly convex set
 \\ as a game $g(x,y):= \langle x , y \rangle - f^*(y)$. 
 \\ Assume that the norm of cumulative gradient does not vanish, $\| \frac{1}{t} \sum_{s=1}^t \partial f( x_s ) \| \geq \rho$.} }  \\ \hline
Algorithm   &  rate &  weight &  y-player's alg. &  x-player's alg. \\ \hline
\shortstack{
Boundary \\ FW }
& 
\shortstack{
Thm.~\ref{thm:equiv2}\\
$O(\frac{1}{\lambda \rho T})$ }& $\alpha_t= 1$ & \shortstack{ Sec.~\ref{section:FTL}\\ \FTL } &  \shortstack{ Sec.~\ref{section:BR}\\ \BR } \\ \hline \hline
\multicolumn{5}{c}{ \shortstack{ $L$-smooth convex optimization: $\min_{w \in \K} f(w)$, 
\\ where $\K$ is a $\lambda$-strongly convex set that is centrally symmetric and contains the origin,
 \\ as a game $g(x,y):= \langle x , y \rangle - f^*(y)$ }}  \\ \hline
Algorithm   &  rate &  weight &  y-player's alg. &  x-player's alg. \\ \hline
\shortstack{
Gauge \\ FW }
& \shortstack{Thm.~\ref{thm:gaugeFW}\\
$O(\frac{L}{\lambda T^2})$} & $\alpha_t=t$ & \shortstack{ Sec.~\ref{section:OFTL}\\ \OFTL } & \shortstack{ Sec.~\ref{section:BTRL} \\ \BTRL with \\ gauge function } \\ \hline \hline
\multicolumn{5}{c}{ \shortstack{ $L$-smooth convex optimization: $\min_{w \in \K} f(w)$, 
\\ where $\K$ is a nuclear-norm ball 
$\left\{  W \in \reals^{d_1 \times d_2} : \sum_{i=1}^{ d_1 \wedge d_2 } \sigma_i(W)  \leq r  \right\}$ \\ with the spectral norm of the gradient satisfying $\| \nabla f(\cdot) \|_2 \leq G$ for all $W \in \mathcal{NB}_{d_1,d_2}(r)$,
 \\ as a game $g(x,y):= \langle x , y \rangle - f^*(y)$ }}  \\ \hline
Algorithm   &  rate &  weight &  y-player's alg. &  x-player's alg. \\ \hline
\shortstack{
Parallelizable \\ Projection-Free \\  Alg.}
& \shortstack{Cor.~\ref{cor}\\
$ \tilde{O}\big( \frac{ L r  \log ( d_1 + d_2) }{ T^2 }
$ \\ $+ \frac{G}{T} \big)$ } & $\alpha_t=t$ & \shortstack{ Sec.~\ref{section:OFTL}\\ \OFTL } & \shortstack{ \BTRL with \\ a random projection } \\ \hline \hline
\end{tabular}
\end{table}



\paragraph{Min-max problems and (approximate) Nash equilibrium}

A large number of core problems in statistics, optimization, and machine learning, can be framed as the solution of a two-player zero-sum game. Linear programs, for example, can be viewed as a competition between a feasibility player, who selects a point in $\reals^n$, and a constraint player that aims to check for feasibility violations \cite{Adler2013}. Boosting \cite{freund1999adaptive} can be viewed as the competition between an agent that selects hard distributions and a weak learning oracle that aims to overcome such challenges \cite{freund1996game}. The hugely popular technique of Generative Adversarial Networks (GANs) \cite{goodfellow2014generative}, which produce implicit generative models from unlabelled data, has been framed in terms of a repeated game, with a distribution player aiming to produce realistic samples and a discriminative player that seeks to distinguish real from fake.

Given a zero-sum game with \textit{payoff function} $g(x,y)$ which is convex in $x$ and concave in $y$, 
define $V^*=\inf_{x \in \K} \sup_{y} g(x,y)$. An $\epsilon$-\textit{equilibrium} of $g(\cdot, \cdot)$ is a pair $\hat x, \hat y$ such that
\begin{equation} \label{def:equi}
\textstyle V^* - \epsilon \leq \inf_{x \in \K} g(x, \hat y)  \leq V^* \leq   \sup_y g(\hat x, y) \leq V^* + \epsilon.
\end{equation}
\noindent


\paragraph{The Fenchel Game.}
One of the core tools of this work is as follows. In order to solve the problem 
\begin{equation}
  \min_{x \in \K} f(x)
\end{equation}
we instead construct a saddle-point problem which we call the \emph{Fenchel Game}. We
define $g : \K \times \reals^d$ as follows:
\begin{equation} \label{eq:fenchelgame}
 g(x,y) := \langle x, y \rangle - f^*(y). 
\end{equation}
This payoff function is useful for solving the original optimization problem, since an equilibrium of this game provides us with a solution to $\min_{x \in \K} f(x)$. Let $\hat x, \hat y$ be any equilibrium pair of $g$, with $\hat x \in \K$.  that is, where $V^* = \sup_y g(\hat x, y)$. Then we have 
\begin{eqnarray*}
  \inf_{x \in \K} f(x) & = & \inf_{x \in \K} \sup_{y} \{\langle x, y\rangle - f^*(y)\} = \inf_{x \in \K} \sup_{y} g(x,y)\\
    & = & \sup_y g(\hat{x}, y ) = \sup_{y}  \left\{ \langle \hat{x}, y \rangle - f^*(y) \right\} = f(\hat{x})
\end{eqnarray*}
In other words, given an equilibrium pair $\hat x, \hat y$ of $g(\cdot,\cdot)$, we immediately have a minimizer of $f(\cdot)$. This simple observation can be extended to approximate equilibria as well.
\begin{lemma} \label{lem:fenchelgame}
  If $(\hat x, \hat y)$ is an $\epsilon$-equilibrium of the Fenchel Game \eqref{eq:fenchelgame}, then $f(\hat x) - \min_{x} f(x) \leq \epsilon$.
\end{lemma}



Lemma~\ref{lem:fenchelgame} sets us up for the remainder of the chapter. The framework, which we lay out precisely in Section~\ref{sec:fgnrd_framework}, will consider two players sequentially playing the Fenchel game, where the $y$-player sequentially outputs iterates $y_1, y_2, \ldots$, while alongside the $x$-player returns iterates $x_1, x_2, \ldots$. Each player may use the previous sequence of actions of their opponent in order to choose their next point $x_t$ or $y_t$, and we will rely heavily on the use of no-regret online learning algorithms described in Section~\ref{sec:onlinelearning}. In addition, we need to select a sequence of weights $\alpha_1, \alpha_2, \ldots > 0$ which determine the ``strength'' of each round, and can affect the players' update rules. What we will be able to show is that the $\alpha$-weighted average iterate pair, defined as
\[
 (\hat x, \hat y) := \left(\frac{\alpha_1 x_1 + \ldots + \alpha_T x_T}{\alpha_1 + \cdots + \alpha_T}, \frac{\alpha_1 y_1 + \cdots + \alpha_Ty_T}{\alpha_1 + \cdots + \alpha_T}\right),
\]
is indeed an $\epsilon$-equilibrium of $g(\cdot, \cdot)$, and thus via Lemma~\ref{lem:fenchelgame} we have that $\hat x$ approximately minimizes $f$. To get a precise estimate of $\epsilon$ requires us to prove a family of regret bounds, which is the focus of the following section.  

\section{No-regret learning algorithms} \label{sec:onlinelearning}

An algorithmic framework, often referred to as \emph{no-regret learning} or \emph{online convex optimization}, has been developed mostly within the machine learning research community, has grown quite popular as it can be used in a broad class of sequential decision problems. As we will explain in Section~\ref{sec:noregret}, one imagines an algorithm making repeated decisions by selecting a vector of parameters in a convex set, and on each round is charged according to a varying convex loss function. The algorithm's goal is to minimize an objective known as regret. In Section~\ref{sec:fgnrd_framework}, we describe how online convex optimization algorithms with vanishing regret can be implemented in a two-player protocol which sequentially computes an approximate equilibria for a convex-concave payoff function. This is the core tool that allows us to describe a range of known and novel algorithms for convex optimization, by modularly combining pairs of OCO strategies. In Section~\ref{sec:oco_algs} we provide several such OCO algorithms, most of which have been proposed and analyzed over the past 10-20 years.

\renewcommand{\algorithmiccomment}[1]{\hfill \tiny//~#1\normalsize}

\begin{algorithm}[H]
\floatname{algorithm}{Protocol}
\caption{Weighted Online Convex Optimization}
\label{alg:oco_protocol}
\begin{algorithmic}[1] \normalsize
\STATE \textbf{Input:} decision set $\K \subset \reals^n$
\STATE \textbf{Input:} number of rounds $T$
\STATE \textbf{Input:} weights $\alpha_1, \alpha_2, \ldots, \alpha_T > 0$ \hfill \text{\footnotesize \# Weights determined in advance}
\STATE \textbf{Input:} algorithm $\alg$ \hfill \text{\footnotesize \#This implements the learner's update strategy}
\FOR{$t=1, 2, \ldots, $}
\STATE \textbf{Return:} $x_t \gets \alg$ \hfill \text{\footnotesize \#Alg returns a point $x_t$}
\STATE \textbf{Receive:} $\alpha_t, \ell_t(\cdot) \to \alg$ \hfill \text{\footnotesize \#Alg receives loss fn. and round weight}
\STATE \textbf{Evaluate:} $\text{Loss} \gets \text{Loss} + \alpha_t \ell_t(x_t)$ \hfill \text{\footnotesize \#Alg suffers weighted loss for choice of $x_t$}
\ENDFOR
\end{algorithmic}
\end{algorithm}

\subsection{Online Convex Optimization and Regret} \label{sec:noregret}

Here we describe the framework, given precisely in Protocol~\ref{alg:oco_protocol}, for online convex optimization. We assume we have some learning algorithm known as $\alg$ that is tasked with selecting ``actions'' from a compact and convex \emph{decision set} $\K \subset \reals^d$. On each round $t=1, \ldots, T$, \alg returns a point $x_t \in \K$, and is then presented with the pair $\alpha_t, \ell_t$, where $\alpha_t > 0$ is a weight for the current round and $\ell_t : \K \to \reals$ is a convex loss function that evaluates the choice $x_t$. 
While \alg is essentially forced to ``pay'' the cost $\alpha_t \ell_t(x_t)$, it can then update its state to provide better choices in future rounds. 

On each round $t$, the learner must select a point $x_t \in \K$, and is then ``charged'' a loss of $\alpha_t \ell_t(x_t)$ for this choice. Typically it is assumed that, when the learner selects $x_t$ on round $t$, she has observed all loss functions $\alpha_1 \ell_1(\cdot), \ldots, \alpha_{t-1} \ell_{t-1}(\cdot)$ up to, but not including, time $t$. However, we will also consider learners that are \emph{prescient}, i.e. that can choose $x_t$ with knowledge of the loss functions up to \emph{and including} time $t$. The objective of interest in most of the online learning literature is the learner's \emph{regret}, defined as
\begin{equation}
  \regret{x} := \sum_{t=1}^T \alpha_t \ell_t(x_t) - \min_{x \in \K} \sum_{t=1}^T \alpha_t \ell_t(x).
\end{equation}
Oftentimes we will want to refer to the \emph{average regret}, or the regret normalized by the time weight $A_T := \sum_{t=1}^T \alpha_t$, which we will denote $\avgregret{x} := \frac{\regret{x}}{A_T}$.
Note that in online learning literature,  
what has become a cornerstone of online learning research has been the existence of \emph{no-regret algorithms}, i.e. learning strategies that guarantee $\avgregret{x} \to 0$ as $A_T \to \infty$.

Let us consider some very simple learning strategies that will be used in this chapter,
and we note the available guarantees for each.
We also refer the readers to some tutorial of online learning
for more online learning algorithms
(see e.g. \cite{OO19,RS16,hazan2016introduction,shalev2012online}). 

\subsection{Framework: optimization as \emph{Fenchel Game}} \label{sec:fgnrd_framework}

We consider Fenchel game (\ref{eq:fenchelgame}) with weighted losses depicted in Algorithm~\ref{alg:game}. In this game, the $y$-player plays before the $x$-player plays
and the $x$-player sees what the $y$-player plays before choosing its action.
The $y$-player receives loss functions $\alpha_{t} \ell_{t}(\cdot)$ in round $t$, in which $\ell_{t}(y):= f^{*}(y)- \langle x_t, y \rangle $,
while the x-player see its loss functions $\alpha_{t} h_{t}(\cdot)$ in round $t$, in which $h_{t}(x):= \langle x, y_t \rangle - f^{*}(y_t)$.
Consequently, we can define the \textit{weighted regret} of the $x$ and $y$ players as
\begin{eqnarray} 
  \label{eq:yregret}     \regret{y} &   := & 
     \sum_{t=1}^T  \alpha_t  \ell_t(y_t) - \min_{y} \sum_{t=1}^T  \alpha_t  \ell_t(y)\\
  \label{eq:xregret}      \regret{x} &   := & 
     \sum_{t=1}^T  \alpha_t  h_t(x_t) - \sum_{t=1}^T  \alpha_t  h_t(x^*)
\end{eqnarray}
Notice that the $x$-player's regret is computed relative to $x^*$ the minimizer of $f(\cdot)$, rather than the minimizer of $\sum_{t=1}^T  \alpha_t  h_t(\cdot)$. 

\begin{algorithm}
\floatname{algorithm}{Protocol}
   \caption{Fenchel Game No-Regret Dynamics} \label{alg:game}
\begin{algorithmic}[1]
\normalsize
\STATE \textbf{Input:} number of rounds $T$
\STATE \textbf{Input:} decision sets $\XX, \YY \subset \reals^d$
\STATE \textbf{Input:} Convex-concave payoff function $g : \XX \times \YY \to \reals$
\STATE \textbf{Input:} weights $\alpha_1, \alpha_2, \ldots, \alpha_T > 0$ 
Weights determined in advance
\STATE \textbf{Input:} algorithms $\alg^Y, \alg^X$ \hfill \text{\footnotesize \#Learning algorithms for both players}
\FOR{$t=1, 2, \ldots, T$}
\STATE \textbf{Return:} $y_t \gets \alg^Y$ \hfill \text{\footnotesize \#$y$-player returns a point $y_t$}
\STATE \textbf{Update:} $\alpha_t, h_t(\cdot) \to \alg^X$ \hfill \text{\footnotesize \#$x$-player updates with $\alpha_t$ and loss $-g(\cdot,y_t)$}
\STATE \quad \quad where $h_t(\cdot) := -g(\cdot,y_t)$
\STATE \textbf{Return:} $x_t \gets \alg^X$ \hfill \text{\footnotesize \#$x$-player returns a point $x_t$}
\STATE \textbf{Update:} $\alpha_t, \ell_t(\cdot) \to \alg^Y$ \hfill \text{\footnotesize \#$y$-player updates with $\alpha_t$ and loss $g(x_t,\cdot)$}
\STATE \quad \quad where $\ell_t(\cdot) := g(x_t,\cdot)$
\ENDFOR
\STATE Output $(\xav_T,\yav_T) := \left(\frac{ \sum_{s=1}^T \alpha_s x_s  }{ A_T }, \frac{ \sum_{s=1}^T \alpha_s y_s  }{ A_T }\right)$.
\end{algorithmic}
\end{algorithm}

At times when we want to refer to the regret on another sequence $y_1', \ldots, y_T'$ we may refer to this as $\regret{}(y_1', \ldots, y_T')$.
We also denote $A_t$ as the cumulative sum of the weights $A_t:=\sum_{s=1}^t \alpha_s$ and the weighted average regret $\avgregret{} := \frac{\regret{}}{A_T}$.
Finally, for offline constrained optimization (i.e. $\min_{x \in \K} f(x)$), we let the decision space of the benchmark/comparator in the weighted regret definition to be $\XX=\K$; for offline unconstrained optimization, we let the decision space of the benchmark/comparator to be a norm ball that contains the optimum solution of the offline problem (i.e. contains $\arg\min_{x \in \reals^n} f(x)$), which means that $\XX$ of the comparator is a norm ball. We let $\YY = \reals^d$ be unconstrained.
\begin{theorem}\label{thm:meta} 
  Assume a $T$-length sequence $\balpha$ are given. Suppose in Algorithm~\ref{alg:game} the online learning algorithms $\alg^x$ and $\alg^y$ have the $\balpha$-weighted average regret $\avgregret{x}$ and $\avgregret{y}$ respectively. Then the output  $(\bar{x}_{T},\bar{y}_{T})$ is an $\epsilon$-equilibrium for $g(\cdot, \cdot)$, with
$    \epsilon = \avgregret{x} + \avgregret{y}.$
\end{theorem} 

\begin{proof}
Suppose that the loss function of the $x$-player in round $t$ is $\alpha_t h_t(\cdot) : \XX \to \reals$, where $h_t(\cdot) := g(\cdot, y_t)$. The $y$-player, on the other hand, observes her own sequence of loss functions $\alpha_t  \ell_t(\cdot) : \YY \to \reals$, where $\ell_t(\cdot) := -  g(x_t, \cdot)$.\\

\begin{eqnarray}
\frac{1}{\sum_{s=1}^T \alpha_s}   \sum_{t=1}^T \alpha_t g(x_t, y_t) 
  & = & \frac{1}{\sum_{s=1}^T \alpha_s} \sum_{t=1}^T  - \alpha_t  \ell_t(y_t)  \notag \\
  \text{} \; & \geq & 
    - \frac{1}{\sum_{s=1}^T \alpha_s} \inf_{y \in \YY} \left\{ \sum_{t=1}^T  \alpha_t  \ell_t(y) \right\} - \frac{ \regret{y} }{  \sum_{s=1}^T \alpha_s } \notag \\
  \; & = &
    \sup_{y \in \YY} \left\{ \frac{1}{\sum_{s=1}^T \alpha_s} \sum_{t=1}^T  \alpha_t g( x_t , y ) \right\} - \avgregret{y}  \notag \\
  \text{(Jensen)} \;  & \geq & 
    \sup_{y \in \YY} g\left({ \frac{1}{\sum_{s=1}^T \alpha_s} \sum_{t=1}^T  \alpha_t x_t }, y \right)  - \avgregret{y} 
     \label{eq:ylowbound}  \\
  \text{} \;  & = & 
    \sup_{y \in \YY} g\left({ \xav_T }, y \right)  - \avgregret{y} 
     \label{eq:ylowbound}  \\
  & \geq & \inf_{x \in \XX} \sup_{y \in \YY} g\left( x , y \right) - \avgregret{y}  \notag
\end{eqnarray}

Let us now apply the same argument on the right hand side, where we use the $x$-player's regret guarantee.

\begin{eqnarray}
\frac{1}{\sum_{s=1}^T \alpha_s}  \sum_{t=1}^T  \alpha_t g(x_t, y_t) & = & \frac{1}{\sum_{s=1}^T \alpha_s} \sum_{t=1}^T  \alpha_t h_t(x_t) \notag \\
  & \leq & \left\{ 
    \sum_{t=1}^T \frac{1}{\sum_{s=1}^T \alpha_s} \alpha_t h_t(x^*) \right \} + \frac{ \regret{x} }{  \sum_{s=1}^T \alpha_s }  \notag \\
  & = &  \left\{  \sum_{t=1}^T \frac{1}{\sum_{s=1}^T \alpha_s} \alpha_t g(x^*, y_t) \right \} + \avgregret{x} \notag \\
\text{(Jensen)}  & \leq &   
    g\left(x^*,{  \sum_{t=1}^T \frac{1}{\sum_{s=1}^T \alpha_s} \alpha_t y_t}\right) + \avgregret{x} 
    \label{eq:xupbound} \\
  & = & 
    g\left(x^*,{  \yav_T}\right) + \avgregret{x} 
    \label{eq:xupbound} \\
  & \leq & \sup_{y \in \YY}  g(x^*,y) + \avgregret{x}  \notag 
\end{eqnarray}
Note that $\sup_{y \in \YY}  g(x^*,y) = f(x^*)$ by Fenchel conjugacy, and hence we can conclude that $\sup_{y \in \YY} g(x^*,y)  = V^* = \sup_{y \in \YY} \inf_{x \in \XX}  g(x,y) = \inf_{x \in \XX} \sup_{y \in \YY}   g(x,y)$.
Combining (\ref{eq:ylowbound}) and (\ref{eq:xupbound}), we see that
$(\bar{x}_{T}, \bar{y}_{T})$ is an $\epsilon =\avgregret{x} + \avgregret{y} $ equilibrium.
  \end{proof}

In order to utilize minimax duality, we have to define decision sets for two players, and we must produce a convex-concave payoff function. First we will assume, for convenience, that $f(x) := \infty$ for any $x \notin \XX$. That is, it takes the value $\infty$ outside of the convex/compact set $\XX$, which ensures that $f(\cdot)$ is lower semi-continuous and convex. Now, let the $x$-player be given the set $\XX := \{\nabla f(x) : x \in \XX \}$. One can check that the closure of the set $X$ is a convex set. Section~\ref{app:show_cvx} describes the proof. 
\begin{theorem} \label{th:cvx}
The closure of (sub-)gradient space $\{ \partial f(x)| x \in \XX \}$ is a convex set.
\end{theorem}

\section{Online Convex Optimization: An Algorithmic Menu} \label{sec:oco_algs}
In this section we introduce and analyze several core  online learning algorithms.  
Later in Sections~\ref{sec:ExistingAlgs}~\&~\ref{sec:NewAlgs}, we will show how composing different online learning algorithm within the Fenchel Game No-Regret Dynamics (Protocol~\ref{alg:game}) enables to easily recover known results and methods for convex optimization (Section~\ref{sec:ExistingAlgs}), as well as to design new algorithm with novel guarantees (Section~\ref{sec:NewAlgs}).

We start by  introducing the simplest algorithmic templates, and then move towards more advanced techniques.
In Subsection ~\ref{sec:oftrl} we introduce and analyze a Meta-algorithm that generalizes many of the methods and results that we introduce in the first subsections.
For the sake of generality we provide guarantees assuming that the loss functions are strongly-convex. Setting the strong-convexity parameter to $0$ recovers the results for general convex losses.

\subsection{\FTL (Follow The Leader)} \label{section:FTL}

FTL (Follow The Leader) is perhaps the simplest strategy in online learning, which plays the best fixed action for the cumulative (weighted) loss seen so far in each round (Equation~\eqref{eq:FTL}). The corresponding analysis has been shown in many textbooks (e.g. \cite{OO19,shalev2012online}). 

\begin{lemma}{ ({\FTLb}$[\xinit]$)} \label{regret:FTL}
Let $\{ \alpha_t \ell_t(\cdot)  \}_{t=1}^T$ be a  sequence of loss functions  such that  each $\ell_t(\cdot)$ is $\mu$-strongly convex, where $\mu \geq 0$.
Given an initial point $\xinit\in\ZZ$, 
{\FTLb}$[\xinit]$ is defined as follows,
\begin{align}\label{eq:FTL}
\textstyle  z_1 &= \xinit  \nonumber\\
\textstyle z_t &= \argmin_{z \in \ZZ} \sum_{s=1}^{t-1} \alpha_s \ell_s(z) 
\end{align}
and satisfies the following regret bound,
 \begin{equation} \label{reg_FTL}
\textstyle 
\regret{z} 
\leq 
\sum_{t=1}^T
\frac{2 \alpha_t^2}{ \big( \sum_{s=1}^t \alpha_s \mu \big)  }
 \| \nabla \ell_t(z_t )\|_*^2.
\end{equation}
\end{lemma}
%
By Lemma~\ref{regret:FTL},
when we set the  weights uiniformly, i.e. $\alpha_t = 1~\forall t$, and assume a bound on the gradient norms, i.e.  $\| \nabla \ell_t(z_t) \|^2_* \leq G$, 
the uniform regret is
\begin{equation} \label{reg:FTL_unit}
\textstyle \textsc{Reg:=}\sum_{t=1}^T \ell_t(z_t) - \min_z \sum_{t=1}^T \ell_t(z) \leq \frac{G \log(T+1)}{2\mu},
\end{equation}
which is a logarithmic regret in $T$.

On the other hand, when the loss function is linear, i.e. $\ell_t(\cdot):= \langle \theta_t, \cdot \rangle$ for some loss vector $\theta_t \in \reals^d$, \FTL might suffer linear regret. That is, the uniform regret could be $\textsc{Reg:=} \sum_{t=1}^T \ell_t(z_t) - \min_{z \in \ZZ} \sum_{t=1}^T \ell_t(z) = \Theta(T)$ for some convex polytope $\ZZ$, which means that the learner fails to learn (see e.g. Example 2.2 \cite{shalev2012online}).
However, if the constraint set $\ZZ$ satisfies a notion called strongly-convexity, then obtaining a logarithmic regret is possible even when the loss function is linear.

\begin{lemma}[Theorem 3.3 in \cite{HLGS16}] \label{thm:FTL}
Let $\{ \ell_t(\cdot):= \langle \theta_t, \cdot \rangle \}_{t=1}^T$
be any sequence of linear loss functions.
Denote $G: =  \max_{t\leq T} \| \theta_t \|$ and assume that 
the support function
$\Phi( \cdot):= \max_{z \in \ZZ}(z, \cdot)$ has a unique maximizer for each cumulative loss vector $L_{t}:= \sum_{s=1}^t  \theta_s$ at round $t$.
Define $\nu_T:= \min_{1\leq t \leq T} \| L_t \|$. 
Let $\ZZ \subset \reals^d$ be an $\lambda$-strongly convex set.  
Choose $\xinit \in \text{boundary}(\ZZ)$. 
Then, after $T$ rounds, \FTLb${\bf[\xinit]}$ ensures,
\begin{equation}
\textsc{Reg:=}
\sum_{t=1}^T \ell_t(z_t) - \min_{z \in \ZZ} \sum_{t=1}^T \ell_t(z)
= \frac{2 G^2}{\lambda \nu_T} ( 1 + \log (T) ).
\end{equation}
\end{lemma}

\subsection{\BTL (Be The Leader)} \label{section:BTL}
As can be seen from Equation~\eqref{eq:BTL}, in \BTL (a.k.a. Be The Leader) the learner plays the best fixed action for the cumulative (weighted) loss seen so far \emph{including the current round}. \BTL is often used as analytic tool rather than a practical  algorithm. Nevertheless, note that in FGNRD (protocol~\ref{alg:game}) the $x$-player is allowed to view the current loss prior to playing, and can therefore apply \BTL.  
This algorithm was named by \cite{kalai2005efficient}, who also proved that it actually guarantees non-positive regret. Here we provide a tighter bound.

\begin{lemma}{{(}\BTLb )} \label{regret:BTL}
Let $\{ \alpha_t \ell_t(\cdot)  \}_{t=1}^T$ be a  sequence of loss functions  such that each $\ell_t(\cdot)$ is at least $\mu$-strongly convex,
where $\mu \geq 0$.
{\BTLb} is defined as follows,
\begin{equation}\label{eq:BTL}
\textstyle z_t = \argmin_{z \in \ZZ} \sum_{s=1}^{t} \alpha_s \ell_s(z).
\end{equation}
and satisfies the following regret bound,
 \begin{equation} \label{reg_BTL}
\textstyle 
\regret{z} 
\leq - \sum_{t=1}^T \frac{ \mu A_{t-1} }{2} \|z_{t-1} - z_t \|^2 \leq 0.
\end{equation}
\end{lemma}

%
\subsection{\OFTL} \label{section:OFTL}
In the previous subsection, we have seen that \BTL uses the knowledge of the loss function at rounds $t$ in order to ensure  negative regret. While this knowledge is oftentimes unavailable, one can often access a ``hint"  function $m_t(\cdot)$  that approximates $\ell_t(\cdot)$ prior to choosing an action $z_t$. As can be seen from Equation~\eqref{tmp:oftl} and Lemma~\ref{regret:Opt-FTL}, \OFTL makes use of the availability of such hints in order to provide better guarantees.
The next statement shows that when we have ``good" hints, in the sense that $m_t(\cdot)\approx \ell_t(\cdot)$, then \OFTL obtains improved guarantees compared to standard \FTL.

\begin{lemma}{({\OFTLb}$[\xinit]$)} \label{regret:Opt-FTL}
Let $\{ \alpha_t \ell_t(\cdot)  \}_{t=1}^T$ be a  sequence of loss functions  such that each $\ell_t(\cdot)$ is $\mu_t$-strongly convex. Given an initial point $\xinit = \arg\min_{z \in \ZZ} m_1(\cdot)$, 
{\OFTLb}$[\xinit]$ is defined as follows,
\begin{align} \label{tmp:oftl}
z_1& \leftarrow \xinit \nonumber \\
z_t &\leftarrow \arg\min_{{z \in \ZZ}} \left( \sum_{s=1}^{t-1} \alpha_{s} \ell_{s}(z) \right)
+ \alpha_t m_t(z),
\end{align}
where  $m_t(\cdot)$ is the \emph{hint (or the guess)} for the loss function $\ell_t(\cdot)$.
\\
{\OFTL} satisfies,
\begin{equation} \label{eq:OTFL}
\begin{aligned}
& \regret{z}   \leq
 \sum_{t=1}^{T}  \alpha_t \left( \ell_t(z_t) - \ell_t(w_{t+1}) \right) - \alpha_t \left(  m_t(z_t) -  m_t(w_{t+1}) \right)
\end{aligned}
\end{equation}
where 
$w_{t}:= \argmin_{z \in \ZZ} \sum_{s=1}^{t-1} \ell_s(z)$.

\end{lemma}

%

\subsection{\FTRL (Follow The Regularized Leader)}
FTRL also called dual averaging in optimization literature \cite{X10} is a classic algorithm in online learning (see e.g. \cite{OO19,hazan2016introduction}). 
Looking at Equation~\eqref{eq:FTRL} one can notice that \FTRL is similar to \FTL with an additional \emph{Regularization term} $R(\cdot)$ that is scale by a factor $1/\eta$.
The regularization term induces \emph{stability} into the decisions of the player, i.e., it enforces consecutive decisions to be close to each other; and this property is often crucial in order to ensure regret guarantees.
For example, in the case of linear loss functions, \FTRL (with appropriate choices of $\eta,R(\cdot)$) can ensure sublinear regret guarantees, while \FTL cannot.
In what follows we assume that $R(\cdot)$ is a $\beta$-strongly-convex function over $\ZZ$.


\begin{lemma}{({\FTRLb}$[R(\cdot),\eta]$)} \label{regret:FTRL}
Let $\{ \alpha_t \ell_t(\cdot)  \}_{t=1}^T$ be a  sequence of loss functions  such that each $\ell_t(\cdot)$ is $\mu$-strongly convex, where $\mu \geq 0$. Also let $\eta>0$ and $R(\cdot)$ be a $\beta$-strongly-convex function over $\ZZ$.
Then {\FTRLb}$[R(\cdot),\eta]$ is defined as follows,
\begin{equation} \label{eq:FTRL}
\textstyle z_t = \argmin_{z \in \ZZ} \sum_{s=1}^{t-1} \alpha_s \ell_s(z) + \frac{1}{\eta} R(z)
\end{equation}
and satisfies the following regret bound,
 \begin{equation} \label{reg_FTRL}
\textstyle 
\regret{z} 
\leq 
\sum_{t=1}^T
\frac{2 \alpha_t^2}{ \big( \sum_{s=1}^t \alpha_s \mu \big) + \beta }
 \| \nabla \ell_t(z_t )\|_*^2
 + 
 \frac{1}{\eta} \left( R(z^*) - R(z_1) \right).
\end{equation}
\end{lemma}

\subsection{\BTRL (Be The Regularized Leader)} \label{section:BTRL}
\BTRL is a very similar to \FTRL, with the difference that the former has an access to all past loss functions up to and \emph{including} the current round. Recall that  in our FGNRD template (protocol~\ref{alg:game}) the $x$-player is allowed to view the current loss prior to playing, and can therefore apply \BTRL.

\begin{lemma}{(\BTRLb$[R(\cdot),1/\eta]$)} \label{regret:BTRL}
Let $\{ \alpha_t \ell_t(\cdot)  \}_{t=1}^T$ be a  sequence of loss functions  such that each $\ell_t(\cdot)$ is $\mu$-strongly convex, where $\mu \geq 0$.
Also let $\eta>0$ and $R(\cdot)$ be a $\beta$-strongly-convex function over $\ZZ$.
Then {\BTRLb}$[R(\cdot),\eta]$ is defined as follows,
\begin{equation} \label{update:BTRL}
z_t \leftarrow \arg\min_{{z \in \ZZ}} \sum_{s=1}^t \alpha_{s} \ell_{s}(z) + \frac{1}{\eta} R(z),
\end{equation}
and satisfies the following regret bound,
\begin{equation} \label{regBTRL}
\begin{aligned}
\textstyle
\regret{z} 
\leq \frac{R(z^*) - R(z_0)}{\eta} - \sum_{t=1}^T \big(  \frac{ \mu A_{t-1} }{2} + \frac{\beta}{2 \eta} \big) \|z_{t-1} - z_t \|^2.
\end{aligned}
\end{equation} 
where $z_0 = \min_{z \in \ZZ} R(z)$ and $z^*$ is any point in $\ZZ$.
\end{lemma}

%
%
%

\subsection{A meta online learning algorithm: \OFTRL} \label{sec:oftrl}
Here we describe \OFTRL, a Meta-algorithm that captures all previously mentioned methods as a private cases.
As can be seen from Equation~\eqref{tmp:optftrl},  \OFTRL employs a regularization term (similarly to \FTRL and \BTRL), and makes use of a hint sequence $m_t(\cdot)$ (similarly to \OFTL).

In Lemma~\ref{regret:Opt-FTRL} we state the regret guarantees of \OFTRL, and then \ show how does the guarantees of \FTL, \BTL, \OFTL, \FTRL, and \BTRL follow as corollaries of this Lemma.
The proof of Lemma~\ref{regret:Opt-FTRL} is provided in Subsection~\ref{sec:Proof_regret:Opt-FTRL}.

\begin{lemma}{({\OFTRLb}$[R(\cdot),\eta]$)} \label{regret:Opt-FTRL}
Let $\{ \alpha_t \ell_t(\cdot)  \}_{t=1}^T$ be a  sequence of loss functions  such that each $\ell_t(\cdot)$ is $\mu_t$-strongly convex, $\mu \geq 0~\forall t$.
Also let $\eta>0$ and $R(\cdot)$ be a $\beta$-strongly-convex function over $\ZZ$.
Then {\OFTRLb}$[R(\cdot),\eta]$ is defined as follows,
\begin{equation} \label{tmp:optftrl}
z_t \leftarrow \arg\min_{{z \in \ZZ}} \left( \sum_{s=1}^{t-1} \alpha_{s} \ell_{s}(z) \right)
+ \alpha_t m_t(z)  + \frac{1}{\eta} R(z),
\end{equation}
where $m_t(\cdot)$ is the \emph{hint (or the guess)} for the loss function $\ell_t(\cdot)$; and we assume that  each $m_t(\cdot)$ is a $\hat{\mu}_t$-strongly convex function over $\ZZ$.
 {\OFTRL} satisfies,
\begin{align*}
& \regret{z} \leq \stepcounter{equation}\tag{\theequation}\label{eq:uni}
\\ & \qquad \quad \sum_{t=1}^{T}  \alpha_t \left( \ell_t(z_t) - \ell_t(w_{t+1})  -    m_t(z_t) +  m_t(w_{t+1}) \right) \tag{\text{term (A)}}
\\
&  \qquad
+ \frac{1}{\eta} \left( R(z^*) - R(w_1) \right) \tag{\text{term (B)}}
\\
& \qquad -
   \frac{1}{2} \sum_{t=1}^T \left( \frac{\beta}{\eta} + \sum_{s=1}^{t-1} \alpha_s \mu_s  \right) \| z_t - w_t \|^2  
\tag{\text{term (C)}}  
\\
& \qquad -
\frac{1}{2}  \sum_{t=1}^T \left( \frac{\beta}{\eta} + \alpha_t \hat{\mu}_t + \sum_{s=1}^{t-1} \alpha_s \mu_s  \right) \| z_t - w_{t+1} \|^2  
\tag{\text{term (D)}}
\end{align*}
where 
$w_{t}:= \argmin_{z \in \ZZ} \sum_{s=1}^{t-1} \alpha_s \ell_s(z)  + \frac{1}{\eta} R(z)$, and $z^*$ is the comparator in the regret definition  
$\regret{z}:= \sum_{t=1}^T  \alpha_t  \ell_t(z_t) - \sum_{t=1}^T  \alpha_t  \ell_t(z^*)$ (similarly to the way we define in  Equation~\eqref{eq:xregret}).

\end{lemma}

\noindent
\textbf{Remark:}
Note that the regret bound actually holds for any comparator $z^* \in \ZZ$. In our  Fenchel game fomulation, we will take $z^*$ be a minimizer of the optimization problem $\min_{x \in \K} f(x)$ and $\ZZ \leftarrow \K$.
The proof of Lemma~\ref{regret:Opt-FTRL} is deferred to Section~\ref{sec:Proof_regret:Opt-FTRL}.
Next we show how  the aforementioned guarantees for \FTLb, \BTLb, \OFTLb, \FTRLb, and \BTRLb follow from the above Lemma.

\begin{proof}[Proof of Lemma~\ref{regret:Opt-FTL} on  \OFTLb]
Observe that \OFTL is actually \OFTRL when $R(\cdot)$ is a zero function. 
Therefore, let $R(\cdot)=0$ in Lemma~\ref{regret:Opt-FTRL}
and drop term (C) and (D) in (\ref{eq:uni}) as they are non-positive, we obtain the result.
 
\end{proof}

\begin{proof}[Proof of Lemma~\ref{regret:BTRL} on  \BTRLb]
Observer that the \BTRL update is exactly equivalent to
\OFTRL with $m_t(\cdot) = \ell_t(\cdot)$.
Furthermore, 
$w_{t+1}$ in Lemma~\ref{regret:Opt-FTRL} is actually $z_t$ of \BTRL
shown on (\ref{update:BTRL}).
So term (A) and term (D) on (\ref{eq:uni})
in Lemma~\ref{regret:Opt-FTRL} is $0$,
Therefore, \BTRL regret satisfies
\begin{equation} \label{reg_x_acc0}
\begin{aligned}
\textstyle
\regret{z} 
\leq \frac{ R(z^*) - R(z_0) }{\eta} - \sum_{t=1}^T \big(  \frac{ \mu A_{t-1} }{2} + \frac{\beta}{2 \eta} \big) \|z_{t-1} - z_t \|^2.
\end{aligned}
\end{equation} 
 
\end{proof}

\begin{proof}[Proof of Lemma~\ref{regret:BTL} on  \BTLb]
Observe that $\BTL$ is actually $\BTRL$ with $R(\cdot)=0$. 
Therefore, let $R(\cdot)=0$ and $\beta=0$ in Equation~\eqref{reg_x_acc0}, we obtain the regret of $\BTL$
\begin{equation} 
\begin{aligned}
\textstyle
\regret{z} 
\leq - \sum_{t=1}^T \frac{ \mu A_{t-1} }{2} \|z_{t-1} - z_t \|^2 \leq 0.
\end{aligned}
\end{equation} 
 
\end{proof}

\begin{proof}[Proof of Lemma~\ref{regret:FTRL} on  \FTRLb]
Observe that \FTRL is actually \OFTRL where 
$m_t(\cdot)=0~\forall t$.
Therefore, let $m_t(\cdot)=0$ in Lemma~\ref{regret:Opt-FTRL}, we obtain the regret of \FTRL,
\begin{equation} \label{tmpeq:FTRL0}
\begin{aligned}
\\ &
\regret{z} 
 \leq
 \sum_{t=1}^{T}  \alpha_t \left( \ell_t(z_t) - \ell_t(z_{t+1}) \right)  + \frac{1}{\eta} \left( R(z^*) - R(z_1) \right), 
\end{aligned}
\end{equation}
where we have dropped term (C) and term (D) on (\ref{eq:uni}) since they are non-positive, and we also note that $w_t$ in Lemma~\ref{regret:Opt-FTRL} is the same as $z_t$ here.
To continue, we use Lemma~\ref{lem:diff}.
Specifically, in Lemma~\ref{lem:diff}, we let $\psi_1(\cdot) \leftarrow \sum_{s=1}^{t-1} \alpha_s \ell_s(\cdot) + \frac{1}{\eta} R(\cdot)$ and $\psi_2(\cdot) \leftarrow \sum_{s=1}^t \alpha_s \ell_s(\cdot) + \frac{1}{\eta} R(\cdot)$.
Then, we have that $\phi(\cdot) = \alpha_t \ell_t(\cdot)$,
$u_1 = z_t$, $u_2 = z_{t+1}$ and that $\sigma = \sum_{s=1}^t \alpha_s \mu + \beta$. So by Lemma~\ref{lem:diff} below, we have that
\begin{equation} \label{tmpeq:FTRL1}
 \alpha_t \left( \ell_t(z_t) - \ell_t(z_{t+1}) \right)  
 \leq \frac{2 \alpha_t^2}{ \big( \sum_{s=1}^t \alpha_s \mu \big) + \beta }
 \| \nabla \ell_t(z_t )\|_*^2.
\end{equation}
Combining (\ref{tmpeq:FTRL0}) and (\ref{tmpeq:FTRL1}) leads to the result.

\begin{lemma} [Lemma 5 in \cite{koren2015fast}] \label{lem:diff}
Let $\psi_1(\cdot), \psi_2(\cdot): \ZZ \rightarrow \reals$ be two convex functions defined over a closed and convex domain. 
Denote $u_1 := \arg\min_{z \in \ZZ} \psi_1(z)$ and 
$u_2 := \arg\min_{z \in \ZZ} \psi_2(z)$.
Assume that $\psi_2$ is $\sigma$-strongly convex with respect to a norm $\| \cdot \|$.
Define $\phi(\cdot):= \psi_2(\cdot) - \psi_1(\cdot)$.
Then,
\begin{equation}
\| u_1 - u_2 \| \leq \frac{2}{\sigma} \| \nabla \phi(u_1) \|_*.
\end{equation}
Furthermore, if $\phi(\cdot)$ is convex, then,
\begin{equation}
0 \leq \phi(u_1) - \phi(u_2) \leq \frac{2}{\sigma} \| \nabla \phi(u_1) \|_*^2.
\end{equation}

\end{lemma}

\end{proof}

\begin{proof}[Proof of Lemma~\ref{regret:FTL} on  \FTLb]
 Observe that FTL is actually FTRL with $R(\cdot)=0$ . 
Therefore, let $R(\cdot)=0$ and $\beta=0$ in Lemma~\ref{regret:FTRL}, we obtain the result.
 
\end{proof}

Next, in Subsections~\ref{section:FTPL},~\ref{section:MD}, and~\ref{section:BR}, we go on by presenting three additional online learning algorithms that cannot be captured by the  \OFTRL Meta-algorithm.

\subsection{\FTPL (Follow the Perturbed Leader)}  \label{section:FTPL}
One of the most powerful techniques that grew our of online learning is the use of \emph{perturbations} as a type of regularization to obtain vanishing regret guarantees. This idea was first suggested and analyzed by Hannan~\cite{hannan1957approximation}, and later simplified and generalized  by Kalai and Vempala~\cite{kalai2005efficient},  who coined the name  \emph{Follow the Perturbed Leader} (FTPL). The main idea is to solve the \FTL optimization problem with an additional  random linear function added to the input, and to select\footnote{Technically speaking, the results of \cite{kalai2005efficient} only considered linear loss functions and hence their analysis did not require taking averages over the input perturbation. While we will not address computational issues here due to space, actually computing the average $\argmin$ is indeed non-trivial.} $z_t$ as  the expectation of the $\argmin$ under this perturbation. More precisely,
\[
  \textstyle z_t := \E_{\xi} \left[ 
    \arg\min_{z \in \ZZ} 
      \left\{ \xi^\top z + \sum_{s=1}^{t-1} \alpha_s \ell_s(z) \right\}
    \right].
\]
Here $\xi$ is some random vector drawn according to an appropriately-chosen distribution and $\ell_s(z)$ is the loss function of the player on round $s$.
Curiously, it was shown in \cite{ALST14} that there is a strong connection between \FTPL and \FTRL.

\subsection{\BR (Best Response)} \label{section:BR}
Perhaps the most trivial strategy for a prescient learner is to ignore the history of the $\ell_s$'s, and simply play the best choice of $z_t$ on the current round. We call this algorithm \BRb, defined as
\begin{equation}
 z_t = \argmin_{z \in \ZZ} \ell_t(z)~;\qquad \textbf{(\BRb)}~.
\end{equation} 

\begin{lemma}{(\BRb)}\label{lem:BRregret}
For any sequence of loss functions $\{ \alpha_t \ell_t(\cdot)\}_{t=1}^T$, \BRb 
ensures,
 \begin{equation} \label{reg:BR}
\textstyle 
\avgregret{z}:= \sum_{t=1}^T \alpha_t \ell_t(z_t) - \min_z \sum_{t=1}^T \alpha_t \ell_t(z)  \leq 0
\end{equation}
\end{lemma}

\begin{proof}
Since $ z_t = \argmin_{z \in \ZZ} \ell_t(z)$, we have that 
$\ell_t(z_t) \leq \ell_t(z)$ for any $z \in \ZZ$.
The result follows by summing the inequalities from $t=1,\dots, T$, and recalling that the $\alpha_t$'s are non-negative.
  \end{proof}

\subsection{\MD (Prescient Mirror Descent)} \label{section:MD}
For any sequence of proper lower semi-continuous convex functions $\{ \alpha_t \ell_t(\cdot) \}_{t=1}^T$,
consider that the player uses 
\MDb for updating its action, which is defined as follows. 
\begin{equation} \label{eq:MDupdate}
\textstyle  z_t  :=  \textstyle \argmin_{z \in \ZZ}   \alpha_t \ell_t(z) +  \frac{1}{\gamma} \V{z_{t-1}}(z)~;\qquad \textbf{(\MDb$[\phi(\cdot),\gamma]$)}
\end{equation}
where we recall that the Bregman divergence $\V{z}(\cdot)$ is with respect to a $\beta$-strongly convex distance generating function $\phi(\cdot)$(see Equation~\eqref{eq:BregmanDef}). 
Note that in the above definition of  \MD, we assume that the online player is \emph{prescient}, i.e., it knows the loss functions $\ell_t$ prior to choosing $z_t$. Recall that in FGNRD (protocol~\ref{alg:game}) the $x$-player is allowed to view the current loss prior to playing, and can therefore apply \MD.  


\begin{lemma}{(\MDb$[\phi(\cdot),\gamma]$)}\label{lem:MD}
Assume that the Bregman Divergence is uniformly bounded on $\ZZ$, so that $D := \V{z_0}(z^*)$, where $z_0,z^{*}$ are any points in $\ZZ$.
For any sequence of proper lower semi-continuous convex 
loss functions $\{ \alpha_t \ell_t(\cdot) \}_{t=1}^T$,
the weighted regret of \MDb$[\phi(\cdot),\gamma]$ (Equation~\eqref{eq:MDupdate})
is bounded as follows,
$$
   \regret{z} \leq  \frac {D }{\gamma} - \sum_{t=1}^{T} \frac{\beta}{2 \gamma} \| z_{t-1} - z_t \|^2.
 $$
\end{lemma}

\begin{proof}
The key inequality we need is Lemma~\ref{lem:newMD};
using the lemma with $\theta(z) = \gamma ( \alpha_t \ell_t(z) )$, $z^+= z_t$ and $c=z_{t-1}$  we have that
\begin{equation} \label{ttb1}
\textstyle \gamma ( \alpha_t \ell_t(z_t) ) - \gamma ( \alpha_t \ell_t(z^*) ) =  \theta(z_t) - \theta(z^*) \leq \V{z_{t-1}}(z^*) - \V{z_t}(z^*) - \V{z_{t-1}}(z_t).
\end{equation}
Therefore, we have that
\begin{equation}
\begin{aligned}
& \textstyle  \regret{z} := \sum_{t=1}^T  \alpha_t  \ell_t(z_t) - \min_{z \in \XX} \sum_{t=1}^T  \alpha_t  \ell_t(z)
\\& \textstyle  \overset{(\ref{ttb1})}{\leq}  \sum_{t=1}^{T} \frac{1}{\gamma} \big( \V{z_{t-1}}(z^*) - \V{z_t}(z^*) - \V{z_{t-1}}(z_t) \big) 
\\ & \textstyle
=  \frac{1}{\gamma} \V{z_{0}}(z^*) - \frac{1}{\gamma} \V{z_T}(z^*)  + \sum_{t=1}^{T-1} ( \frac{1}{\gamma} - \frac{1}{\gamma} ) \V{z_t}(z^*) - \frac{1}{\gamma} \V{z_{t-1}}(z_t) 
\\ & \textstyle = \frac{D}{\gamma} - \sum_{t=1}^{T} \frac{\beta}{2 \gamma} \| z_{t-1} - z_t \|^2,
\end{aligned}
\end{equation}
where the last inequality uses the definition of $D$ and the strong convexity of $\phi$, which grants $\V{z_{t-1}}(z_t) \geq \frac \beta 2 \| z_t - z_{t-1}\|^2$. 
\end{proof}

\begin{lemma}[Property 1 in \cite{T08}]  \label{lem:newMD}
For any proper lower semi-continuous convex function $\theta(z)$,
let $z^+ =  \argmin_{z \in \ZZ} \theta(z) + \V{c}(z)$. 
Then, it satisfies that for any $z^* \in \ZZ$,
\begin{equation} 
\textstyle \theta(z^+) - \theta(z^*) \leq \V{c}(z^*) - \V{z^+}(z^*) - \V{c}(z^+).
\end{equation}
\end{lemma}

\begin{proof}
The result is quite well-known and also appeared in (e.g.
\cite{CT93}).
For completeness, we replicate the proof here.
Recall that the Bregman divergence with respect to the distance generating function $\phi(\cdot)$ at 
a point $c$ is: $\V{c}(z):= \phi(z) - \langle \nabla \phi(c), z - c  \rangle  - \phi(c).$ 

Denote $F(z):= \theta(z) + \V{c}(z)$.
Since $z^+$ is the optimal point of 
$\min_{z \in \ZZ} F(z)\textstyle,$
by optimality,
\begin{equation} \label{m1}
\textstyle \langle z^* - z^+ ,  \nabla F(z^+) \rangle =  
\langle z^* - z^+ , \partial \theta(z^+) + \nabla \phi (z^+) - \nabla \phi(c) \rangle  \geq 0,
\end{equation}
for any $z^{*} \in \ZZ$.
Now using the definition of subgradient, we have that
\begin{equation} \label{m2}
\textstyle \theta(z^*) \geq \theta(z^+) + \langle \partial \theta(z^+) , z^* - z^+ \rangle.
\end{equation}
By combining (\ref{m1}) and (\ref{m2}), we have that
\begin{equation} 
\begin{aligned}
 \textstyle   \theta(z^*) & \textstyle \geq \theta(z^+) + \langle \partial \theta(z^+) , z^* - z^+ \rangle.
\\ &  \textstyle \geq \theta(z^+) + \langle z^* - z^+ , \nabla \phi (c) - \nabla \phi(z^+) \rangle.
\\  & \textstyle = \theta(z^+) - \{ \phi(z^*) - \langle \nabla \phi(c), z^* - c  \rangle  - \phi(c) \}
\\&
+ \{ \phi(z^*) - \langle \nabla \phi(z^+), z^* - z^+  \rangle  - \phi(z^+) \}
\\  & \textstyle + \{ \phi(z^+) - \langle \nabla \phi(c), z^+ - c  \rangle  - \phi(c) \}
\\  & \textstyle  = \theta(z^+) - \V{c}(z^*) + \V{z^+}(z^*) +  \V{c}(z^+).
\end{aligned}
\end{equation}
\end{proof}

\subsection{Proof of Lemma~\ref{regret:Opt-FTRL}}\label{sec:Proof_regret:Opt-FTRL}
\begin{proof}
We can re-write the regret as
\begin{equation} \label{eq:term}
\begin{aligned}
\textstyle  \regret{z} & := \sum_{t=1}^T  \alpha_t  \ell_t(z_t) - \sum_{t=1}^T  \alpha_t  \ell_t(z^*)
\\ & 
=
\underbrace{
\sum_{t=1}^{T} \alpha_t \left( \ell_t(z_t) - \ell_t(w_{t+1}) \right) - \alpha_t \left(  m_t(z_t) -  m_t(w_{t+1}) \right)
}_{ \text{first term} }
\\& \quad + 
\underbrace{
\sum_{t=1}^T \alpha_t \left( m_t(z_t) - m_t(w_{t+1}) \right)
}_{ \text{second term} }
\quad + 
\underbrace{
\sum_{t=1}^T \alpha_t \left(  \ell_t(w_{t+1})  - \ell_t(z^*) \right)
}_{ \text{third term} }.
\end{aligned}
\end{equation}
In the following, we will denote 
\[
D_T:=\frac{1}{2} \left( \sum_{t=1}^T \big( \frac{\beta}{\eta} + \sum_{s=1}^t \alpha_s \mu_s  \big) \| z_t - w_t \|^2 
+ 
\sum_{t=1}^T \big( \frac{\beta}{\eta} + \sum_{s=1}^t \alpha_s \mu_s  + \alpha_t \hat{\mu}_t \big) \| z_t - w_{t+1} \|^2 \right) 
\]
 for brevity.
Let us first deal with the second term and the third term.
We will use induction to show that
\begin{equation} \label{inq:0}
\begin{aligned}
& \underbrace{
\sum_{t=1}^T \alpha_t \left( m_t(z_t) - m_t(w_{t+1}) \right)
}_{ \text{second term} }
\quad + 
\underbrace{
\sum_{t=1}^T \alpha_t \left(  \ell_t(w_{t+1})  - \ell_t(z^*) \right)
}_{ \text{third term} }
\\ & \leq \frac{1}{\eta} \left( R(z^*) - R(w_1) \right) 
- \left( \sum_{t=1}^T \big( \frac{\beta}{\eta} + \sum_{s=1}^t \alpha_s \mu_s  \big) \big( \| z_t - w_t \|^2 + \| z_t - w_{t+1} \|^2 \big) \right),
\end{aligned}
\end{equation}
for any point $z^* \in \ZZ$.
For the base case $T=0$, we have that
\begin{equation}
\begin{aligned}
& \sum_{t=1}^0 \alpha_t \left( m_t(z_t) - m_t(w_{t+1}) \right)
+
\sum_{t=1}^0 \alpha_t \left(  \ell_t(w_{t+1})  - \ell_t(z^*) \right)
\\ & \quad = 0 \leq \frac{1}{\eta} \left( R(z^*) - R(w_1) \right) - 0,
\end{aligned}
\end{equation}
as $w_1 := \arg\min_{z \in \ZZ} R(z)$. So the base case trivially holds.

Let us assume that the inequality (\ref{inq:0}) holds for $t=0,1, \dots, T-1$.
Now consider round $T$. We have that
\begin{equation}
\begin{aligned}
& \sum_{t=1}^T \alpha_t \left( m_t(z_t) - m_t(w_{t+1}) + \ell_t(w_{t+1}) \right)
\\ &
\overset{(a)}{\leq} 
\alpha_T \left( m_T(z_T) - m_T(w_{T+1}) + \ell_T(w_{T+1}) \right)
+ \frac{1}{\eta} \left( R(w_T) - R(w_1) \right)
-  D_{T-1}
\\ & \quad +  \sum_{t=1}^{T-1} \alpha_t \ell_t(w_T).
\\ &
\overset{(b)}{\leq} 
\alpha_T \left( m_T(z_T) - m_T(w_{T+1}) + \ell_T(w_{T+1}) \right)
+ \frac{1}{\eta} \left( R(z_T) - R(w_1) \right)
-  D_{T-1}
\\ & \quad 
- \frac{1}{2} \left( \frac{\beta}{\eta} + \sum_{s=1}^{T-1} \alpha_s \mu_s \right) \| z_T - w_T \|^2
+  \sum_{t=1}^{T-1} \alpha_t \ell_t(z_T)
\\ & =
\alpha_T \left( \ell_T(w_{T+1}) - m_T(w_{T+1}) \right)
+ \frac{1}{\eta} \left( R(z_T) - R(w_1) \right)
-  D_{T-1}
\\ & \quad 
- \frac{1}{2} \left( \frac{\beta}{\eta} + \sum_{s=1}^{T-1} \alpha_s \mu_s \right) \| z_T - w_T \|^2
+  \sum_{t=1}^{T-1} \alpha_t \ell_t(z_T) + \alpha_T m_T(z_T)
\\ & \overset{(c)}{\leq}
\alpha_T \left( \ell_T(w_{T+1}) - m_T(w_{T+1}) \right)
+ \frac{1}{\eta} \left( R(w_{T+1}) - R(w_1) \right)
-  D_{T-1}
\\ & \quad 
- \frac{1}{2} \left( \frac{\beta}{\eta} + \sum_{s=1}^{T-1} \alpha_s \mu_s \right) \| z_T - w_T \|^2
- \frac{1}{2} \left( \frac{\beta}{\eta} + \sum_{s=1}^{T-1} \alpha_s \mu_s \right) \| z_T - w_{T+1} \|^2
\\ & \quad
+  \sum_{t=1}^{T-1} \alpha_t \ell_t(w_{T+1}) + \alpha_T m_T(w_{T+1})
\\ & =
\sum_{t=1}^T \alpha_t \ell_t (w_{T+1}) 
+ \frac{1}{\eta} \left( R(w_{T+1}) - R(w_1) \right)
- D_T
\\ & \overset{(d)}{ \leq } 
\sum_{t=1}^T \alpha_t \ell_t (z^*) 
+ \frac{1}{\eta} \left( R(z^*) - R(w_1) \right)
- D_T,
\end{aligned}
\end{equation}
where (a) we use the induction such that the inequality (\ref{inq:0}) holds for any $z^* \in \ZZ$ including $z^* = w_T$, and (b) is because
\begin{equation}
\sum_{t=1}^{T-1} \alpha_t \ell_t(w_T) + \frac{1}{\eta} R(w_T)
\leq 
\sum_{t=1}^{T-1} \alpha_t \ell_t(z_T) + \frac{1}{\eta} R(z_T)
- \frac{1}{2} \left( \frac{\beta}{\eta} + \sum_{s=1}^{T-1} \alpha_s \mu_s \right) \| z_T - w_T \|^2,
\end{equation}
as $w_T$ is the minimizer of a $\frac{\beta}{\eta} + \sum_{t=1}^{T-1} \alpha_t \mu_t$ strongly convex function since
\[w_{T}:= \argmin_{z \in \ZZ}  \sum_{s=1}^{T-1} \alpha_s \ell_s(z) + \frac{1}{\eta} R(z),
\]
and (c) is because 
\begin{equation}
\begin{aligned}
& \sum_{t=1}^{T-1} \alpha_t \ell_t(z_T) + \alpha_T m_T(z_T)  + \frac{1}{\eta} R(z_T)
\\ & \qquad \leq 
\sum_{t=1}^{T-1} \alpha_t \ell_t(w_{T+1}) + \alpha_T m_T(w_{T+1}) + \frac{1}{\eta} R(w_{T+1})
\\ & \qquad \qquad
- \frac{1}{2} \left( \frac{\beta}{\eta} + \sum_{s=1}^{T-1} \alpha_s \mu_s 
+ \alpha_T \hat{\mu}_T
\right) \| z_T - w_{T+1} \|^2,
\end{aligned}
\end{equation}
as $z_T$ is the minimizer of a $\frac{\beta}{\eta} + \sum_{t=1}^{T-1} \alpha_t \mu_t + \alpha_T \hat{\mu}_T$ strongly convex function 
since
\[z_T := \arg\min_{{z \in \ZZ}} \left( \sum_{s=1}^{T-1} \alpha_{s} \ell_{s}(z) \right)
+ \alpha_T m_T(z)  + \frac{1}{\eta} R(z),\]
and (d) is due to \[w_{T+1}:=\arg\min_{z \in \ZZ} \sum_{t=1}^{T} \alpha_t \ell_t(z) +\frac{1}{\eta} R(z).\]
\end{proof}

\section{Recovery of existing algorithms}
\label{sec:ExistingAlgs}

What we are now able to establish, using the tools developed above, is that several iterative first order methods to minimize a convex function can be cast as simple instantiations of the Fenchel game no-regret dynamics. But more importantly, using this framework and the various regret bounds stated above, we able to establish a convergence rate for each via a unified analysis.

For everyone one of the optimization methods we explore below we provide the following:
\begin{enumerate}
  \item We state the update method described in its standard iterative form, alongside an equivalent formulation given as a no-regret dynamic. To provide the FGNRD form, we must specify the payoff function $g(\cdot,\cdot)$--typically the Fenchel game, with some variants---as well as the sequence of weights $\alpha_t$, and the no-regret algorithms $\alg^Y,\alg^X$ the two players.
  \item We provide a proof of this equivalence, showing that the FGNRD formulation does indeed produce the same sequence of iterates as the iterative form; this is often deferred to the appendix.
  \item Leaning on Theorem~\ref{thm:meta}, we prove a convergence rate for the method.
\end{enumerate}

\subsection{Frank-Wolfe method and its variants}

The \emph{Frank-Wolfe method} (FW) \cite{frank1956algorithm}, also known as \emph{conditional gradient}, is known for solving constrained optimization problems. FW is entirely first-order, while requiring access to a linear optimization oracle.
Specifically,
given a compact and convex constraint set $\K \subset \reals^d$, FW relies on the ability to (quickly) answer queries of the form $\argmin_{x \in \K} \langle x, v\rangle$, for any vector $v \in \reals^d$.
In many cases this linear optimization problem is much faster for well-behaved constraint sets; e.g. simple convex polytopes, the PSD cone, and various balls defined by vector and matrix norms \cite{D16a,D16b,BPZ19}.
When the constraint set is the nuclear norm ball, which arises in matrix completion problems,
then the linear optimization oracle corresponds to computing a top singular vector, which requires time roughly linear in the size of the matrix 
\cite{YTFUC19,hazan2016introduction}.

\begin{algorithm}[H] 
   \caption{ Frank-Wolfe \cite{frank1956algorithm} } \label{alg:fw}
Given: $L$-smooth $f(\cdot)$, convex domain $\K$, arbitrary $w_0$, iterations $T$
\begin{center} 
\begin{tabular}{c c}
    $
      \boxed{
      \begin{array}{rl}
        \gamma_t & \gets \frac{2}{t+1}\\
        v_t  & \gets \displaystyle \argmin_{v \in \K} \langle v, \nabla f(w_{t-1})  \rangle \\
        w_{t} & \gets (1 - \gamma_t) w_{t-1} + \gamma_t v_t
      \end{array}
      }
    $
    & \small
    $\boxed{
    \begin{array}{rl}
      g(x,y) & := \langle x, y \rangle - f^*(y)\\
      \alpha_t & \gets t\\
      \alg^Y & := \FTL[\nabla f(w_0)] \\
      \alg^X & := \BR \\
    \end{array}
    }$
  \\
    \small Iterative Description & 
    \small FGNRD Equivalence
\end{tabular}
\end{center}
Output: $w_T = \bar x_T$
\end{algorithm}

We describe the Frank-Wolfe method precisely in Algorithm~\ref{alg:fw}, in both its iterative form and its FGNRD interpretation. We begin by showing that these two representations are equivalent.



\begin{theorem} \label{thm:equivFW}
  The two interpretations of Frank-Wolfe, as described in Algorithm~\ref{alg:fw}, are equivalent. That is, for every $t$, the iterate $w_t$ computed iteratively on the left hand side is identically the weighted-average point $\bar x_t$ produced by the dynamic on the right hand side.
\end{theorem}

\begin{proof}
We show, via induction, that the following three equalities are maintained for every $t$. Note that three objects on the left correspond to the iterative description given in Algorithm~\ref{alg:fw} whereas the three on the right correspond to the FGNRD description.
  \begin{eqnarray}
     \nabla f(w_{t-1})  & = & y_t \label{eq:ygradfw}\\
     v_t  & = & x_t \label{eq:yv}\\
     w_t & = &\bar{x}_t  \label{eq:xfw}.
  \end{eqnarray}
  To start, we observe that since the $\alg^Y$ is set as $\FTL[\nabla f(w_0)]$, we have that the base case for \eqref{eq:ygradfw}, $y_1 = \nabla f(w_0)$, holds by definition. Furthermore, we observe that for any $t$ we have \eqref{eq:ygradfw} $\implies$ \eqref{eq:yv}. This is because, if $y_t = \nabla f(w_{t-1})$, the definition of \BR implies that 
  \begin{equation}\label{get_x}
    x_t =  \argmin_{x \in \XX} \alpha_t \left( \langle x, y_t \rangle - f^*(y_t) \right)
     = \argmin_{x \in \XX} \langle x, \nabla f(w_{t-1}) \rangle = v_t
  \end{equation}
  Next, we can show that \eqref{eq:yv} $\implies$ \eqref{eq:xfw} for any $t$ as well using induction. Assuming that $w_{t-1} = \frac{\sum_{s=1}^{t-1} \alpha_s x_s}{\sum_{s=1}^{t-1} \alpha_s} = \frac{\sum_{s=1}^{t-1} s v_s}{\sum_{s=1}^{t-1} s}$, a bit of algebra verifies
  \begin{eqnarray*}
    w_t & := & (1 - \gamma_t) w_{t-1} + \gamma_t v_t  = \left(\frac{t-1}{t+1}\right)\frac{\sum_{s=1}^{t-1} s v_s}{\sum_{s=1}^{t-1} s} + \left(\frac{2}{t+1}\right) v_t \\
    & = & \frac{\sum_{s=1}^{t} s v_s}{\sum_{s=1}^{t} s} = \frac{\sum_{s=1}^{t} \alpha_s x_s}{\sum_{s=1}^{t} \alpha_s} =: \bar x_t
  \end{eqnarray*}
  Finally, we show that \eqref{eq:ygradfw} holds for $t > 1$ via induction.
  Recall that $y_t$ is selected via \FTL against the sequence of loss functions $\alpha_t \ell_t(\cdot) := - \alpha_t g(x_t, \cdot)$
   Precisely this means that, for $t > 1$, 
  \begin{eqnarray*} 
    y_t & := & \textstyle \arg\min_{y \in \YY} \left\{ \frac{ 1}{ A_{t-1} } \sum_{s=1}^{t-1} \alpha_s \ell_s(y) \right\} \label{h}\\
   &  = & \arg\min_{y \in \YY} \left\{ \frac {1}{ A_{t-1} } \sum_{s=1}^{t-1} \alpha_s (-x_s^\top y + f^*(y) ) \right\} \label{g}\\ 
    & = & \arg\max_{y \in \YY} \left\{ {\bar x_{t-1}}^\top y - f^*(y)  \right\} 
     =  \nabla f ({\bar x_{t-1}}), \label{get_y}
  \end{eqnarray*}  
The final line follows as a result of the Legendre transform \cite{B04}. Finally, by induction, we have that ${\bar x_{t-1}} = w_{t-1}$, and hence we have established \eqref{eq:ygradfw}.
This completes the proof.
 
\end{proof}

Now that we have established Frank-Wolfe as an instance of Protocol~\ref{alg:game}, we can now prove a bound on convergence using the tools established in Section~\ref{sec:onlinelearning}.

\begin{theorem}\label{thm:fwconvergence}
  Let $w_T$ be the output of Algorithm~\ref{alg:fw}. Let $f$ be $L$-smooth and let $\K$ have squared $\ell_2$ diameter no more than $D$. Then we have
  \[
    f(w_T) - \min_{w \in \K} f(w) \leq \frac{8LD}{T+1}.
  \]
\end{theorem}
\begin{proof}
  Now that we have established that Algorithm~\ref{alg:fw} is an instance of Protocol~\ref{alg:game}, we can appeal directly to Theorem~\ref{thm:meta} to see that
  \[
    f(w_T) - \min_{w \in \K} f(w) \leq \avgregret{x}[\BR] + \avgregret{y}[\FTL].
  \]
  Recall that, by Lemma~\ref{lem:BRregret}, we have that $\avgregret{x}[\BR] \leq 0$. Let us then turn our attention to the regret of $\alg^{Y}$.

  First note that, since $f(\cdot)$ is $L$-smooth, its conjugate $f^*(\cdot)$ is $\frac 1 L$-strongly convex, and thus the function $-g(x,\cdot)$ is also $\frac 1 L$-strongly convex in its second argument. Next, if we define $\ell_t(\cdot) := -g(x_t, \cdot)$, then we can bound the norm of the gradient as
  \[
    \| \nabla \ell_t(y_t) \|^2 = \| x_t - \nabla f^*(y_t) \|^2 = \| x_t - \bar{x}_{t-1} \|^2 \leq D.
  \]
  Combining with Lemma~\ref{regret:FTL} we see that
  \begin{align*} \label{eq:FWy2}
    \avgregret{y}[\FTL] 
      & \leq \frac{1}{A_T}  \sum_{t=1}^T \frac{2 \alpha_t^2 \| \nabla \ell_t(y_t) \|^2}{\sum_{s=1}^{t} \alpha_s (1/L)} 
  & = \frac{8L}{T(T+1)}
  \sum_{t=1}^T \frac{t^2 D}{t(t+1)} \leq \frac{8LD}{T+1}.
  \end{align*}
This completes the proof.

\end{proof}

\subsubsection{Variant 1: a linear rate Frank-Wolfe over strongly convex set}

\citet{LP66}, \citet{DR70}, \citet{D79} show that under certain conditions, Frank-Wolfe for smooth convex function (\emph{not necessarily a strongly convex function}) for strongly convex sets has linear rate under certain conditions.
We show that a similar result can be derived from the game framework.

\begin{theorem}\label{thm:linearFW} 
Suppose that $\min_{x \in \K} f(x)$ is $L$-smooth convex.
and that $\K$ is a $\lambda$-strongly convex set. Also assume that the gradients of the $f(\cdot)$ in $\K$ are bounded away from $0$, i.e., $\max_{x \in \K}\|\nabla f(x)\|\geq B$.
Then, there exists a FW-like algorithm that has $O(\exp(-\frac{\lambda B }{L} T))$ rate which is an instance of Algorithm~\ref{alg:game}
with the weighting scheme $\alpha_t:= \frac{1}{ \| \nabla \ell_t(y_t) \|^2 }$ if Alg.~\ref{alg:game} sets  $\alg^Y := \FTL$ and $\alg^X := \BR$.
\end{theorem}
Note that the weights $\alpha_t$ are not predefined but rather depend on the queries of the algorithm. The proof of Theorem~\ref{thm:linearFW} is described in full detail in Section~\ref{app:linearFW}.

\subsubsection{Variant 2: a smoothing Frank-Wolfe for non-smooth functions}

Looking carefully at the proof of Theorem~\ref{thm:fwconvergence}, the fact that \FTL was suitable for the vanilla FW analysis relies heavily on the strong convexity of the functions $\ell_t(\cdot) := - g(x_t, y)$, which in turn results from the smoothness of $f(\cdot)$. But what about when $f(\cdot)$ is not smooth, is there an alternative algorithm available?

We observe that one of the nice techniques to grow out of the online learning community is the use of \emph{perturbations} as a type of regularization to obtain vanishing regret guarantees \cite{kalai2005efficient} -- their method is known as \emph{Follow the Perturbed Leader} (FTPL). The main idea is to solve an optimization problem that has a random linear function added to the input, and to select\footnote{Technically speaking, the results of \citet{kalai2005efficient} only considered linear loss functions and hence their analysis did not require taking averages over the input perturbation. While we will not address computational issues here due to space, actually computing the average $\argmin$ is indeed non-trivial.} as $x_t$ the expectation of the $\argmin$ under this perturbation. More precisely,
\[
  \textstyle y_t := \E_{Z} \left[ 
    \arg\min_{y \in Y} 
      \left\{ Z^\top y + \sum_{s=1}^{t-1} \ell_s(y) \right\}
    \right].
\]
Here $Z$ is some random vector drawn according to an appropriately-chosen distribution and $\ell_s(x)$ is the loss function of the x-player on round $s$; with the definition of payoff function $g$, i.e. $\ell_s(y):= - x_s^\top y + f^*(y)$.

One can show that, as long as $Z$ is chosen from the right distribution, then this algorithm guarantees average regret on the order of $O\left(\frac{1}{\sqrt{T}}\right)$, although obtaining the correct dimension dependence relies on careful probabilistic analysis. Recent work of \citet{ALST14} shows that the analysis of perturbation-style algorithm reduces to curvature properties of a stochastically-smoothed Fenchel conjugate. 

What is intriguing about this perturbation approach is that it ends up being equivalent to an existing method proposed by \citet{L13} (Section 3.3), who also uses a stochastically smoothed objective function. We note that
\begin{equation}
\begin{aligned}
 & \textstyle \E_{Z} \left[ 
    \arg\min_{x \in X} 
      \left\{ Z^\top x + \sum_{s=1}^{t-1} \ell_s(x) \right\}
    \right] \\ & = 
  \E_{Z} \left[ 
    \arg\max_{x \in X} 
      \left\{ ({\bar y_{t-1}} + Z/(t-1))^\top x - f^*(x)  \right\}
    \right] \\ & 
   =  \E_{Z}[\nabla f({\bar y_{t-1}} + Z/(t-1)) ]   =  \nabla \tilde f_{t-1}({\bar y_{t-1}}) 
\end{aligned}
\end{equation}
where $\tilde f_{\alpha}(x) := \E[f(x + Z/\alpha)]$. \citet{L13} suggests using precisely this modified $\tilde f$, and they prove a rate on the order of $O\left(\frac{1}{\sqrt{T}}\right)$. As discussed, the same would follow from vanishing regret of FTPL.
In other words, by plugging in FTPL as the alternative algorithm, what we're actually doing is using a ``stochastically smoothed'' version of $f$.

\subsubsection{Variant 3: an incremental Frank-Wolfe}

Recently, \citet{Netal20} and \citet{LF20} propose stochastic Frank-Wolfe algorithms
for optimizing smooth convex finite-sum functions, i.e. $\min_{x \in \K} f(x) := \frac{1}{n}\sum_{i=1}^n f_i(x)$,
where each $f_i(x) := \phi( x^\top z_i)$ represents a loss function $\phi(\cdot)$ associated with sample $z_i$. 
In each iteration the algorithms only require a gradient computation of a single component, see option (A) of Algorithm~\ref{alg:newStoFW}.
\citet{Netal20} show that the algorithm has $O(\frac{ c_{\kappa}}{T})$ expected convergence rate, where $c_{\kappa}$ is a number that depends on the underlying data matrix $z$ and in worst case is bounded by the number of components $n$. 
We show that a similar algorithm, option (B) of Algorithm~\ref{alg:newStoFW}, can be generated from Algorithm~\ref{alg:game} that has $\tilde{O}(\frac{n}{T})$ \emph{deterministic} convergence rate, which picks a sample in each iteration by cycling through the data points.
We have the following theorem and its proof is in Section~\ref{app:equivStoFW}. 
\begin{algorithm} 
   \caption{Stochastic Frank-Wolfe algorithm: \footnotesize(option (A) is the algorithm of \cite{Netal20}, while option (B) is the algorithm analyzed in this work.)} \label{alg:newStoFW}
\begin{algorithmic}[1]
\STATE \textbf{Init:} $w_0 \in \K$.
\STATE For each sample $i$, compute $g_{i,0} := \frac{1}{n} \nabla f_{i}( w_0 ) \in \reals^d$.
\FOR{$t= 1, 2, \dots, T$}
\STATE Option (A): Sample a $i_t \in [n]$ uniformly at random.
\STATE Option (B): Select a sample $i_t \in [n]$ by cycling through the samples.
\STATE Compute $\nabla f_{{i_t}}( w_t ) $ 
and set $g_{i_t,t} := \frac{1}{n} \nabla f_{{i_t}}( w_t )  $.
For other $j \neq i \in [n]$,  $g_{j,t} =  g_{j,t-1}$.
\STATE $g_t = \sum_{i=1}^n g_{i,t}$.
\STATE $v_t = \arg\min_{x \in \K} \langle x ,  g_t \rangle.$
\STATE Option (A): $w_t = (1- \frac{2}{t+1}) w_{t-1} + \frac{2}{t+1} v_t$.
\STATE Option (B): $w_t = (1- \frac{1}{t}) w_{t-1} + \frac{1}{t} v_t $.
\ENDFOR
\STATE Output $w_t$.
\end{algorithmic}
\end{algorithm}

\begin{theorem} \label{thm:equivStoFW}
 When both are run for exactly $T$ rounds, the output $\bar x_T$ of Algorithm~\ref{alg:game} with the weighting scheme $\{\alpha_t=t\}$ is identically the output $w_T$ of Algorithm~\ref{alg:newStoFW} with learning rate $\gamma_t = \frac{1}{t}$
 as long as: 
 \textbf{(I)} Alg.~\ref{alg:game} sets $\alg^Y := g_t$ (line 7 of Algorithm~\ref{alg:newStoFW}); \textbf{(II)}  Alg.~\ref{alg:game} sets $\alg^X := \BR$.
 Furthermore, 
   assume that $f(\cdot)$ is $L$-smooth convex
and that its conjugate is $L_0$-Lipschitz.
    Then option (B) of Algorithm~\ref{alg:newStoFW} outputs $w_T$ with approximation error $O\left( \frac{ \max\{ L D, L(L_0+r) n r \}  \log T }{T} \right)$,
  where $r$ is a bound of the length of any point $x$ in the constraint set $\K$, i.e. $\max_{x \in \K} \| x \| \leq r$, and $D$ is the squared of the diameter of $\K$.
  \end{theorem}
\subsubsection{Related works}

Bach \cite{B15} shows that for certain types of objectives, subgradient descent applied to the primal domain is equivalent to FW applied to the dual domain. 
\citet{D15} shows that for strongly convex and smooth objective functions, FW can achieve $O(1/T^2)$ convergence rate over strongly convex set .
\citet{D13}, \citet{D16b} show that exponential convergence for strongly convex and smooth objectives over some polytopes can be achieved by a projection-free algorithm. Their algorithms require a stronger oracle by using the standard one, but can be efficiently implemented for certain polytopes like simplex. Other linear rate of FW-like algrorithms for certain convex polytopes includes \cite{D16a,W70,GJL16,S15,FG16}.
There are also many works of Frank-Wolfe on different aspects, e.g.
online learning setting \cite{HK12}, 
minimizing some structural norms \cite{H13,YZS14},
reducing the number of gradient evaluations \cite{LZ16},
block-wise update for structural SVM \cite{SJM13,O16,W16}.
Finally, we note that Frank-Wolfe has a nice property that it tends to produce sparse solution (see e.g. \cite{J13,K08}), as it adds one component at a time.

\subsection{Accelerated methods for smooth convex optimization}

In this subsection, we are going to introduce several accelerated algorithms. To achieve acceleration, we will consider that the y-player in the game plays \OFTL
\begin{equation} 
\yof_t \leftarrow \arg\min_{{y \in \YY}} \alpha_t m_t(y) +  \sum_{s=1}^{t-1} \alpha_{s} \ell_{s}(y) 
 , \quad \text{and let } \quad m_t(\cdot):= \ell_{t-1}(\cdot),
\end{equation}
where 
the learner uses the loss function of the previous round $\ell_{t-1}(\cdot)$ as the guess $m_t(\cdot)$ of the loss function at $t$ before observing the loss function.

For the time being, let us assume that the sequence of $x_t$'s is arbitrary. We define
\begin{eqnarray} \label{def:tilde_x}
  \xav_t := \textstyle \frac{1}{A_t}\sum_{s=1}^t \alpha_s x_s \quad \quad \text{ and } \quad \quad \xof_t := \textstyle \frac{1}{A_t}(\alpha_t x_{t-1} + \sum_{s=1}^{t-1} \alpha_s x_s ).
\end{eqnarray}
It is critical that we have two parallel sequences of iterate averages for the $x$-player. Our final algorithm will output $\xav_T$, whereas the Fenchel game dynamics will involve computing $\nabla f$ at the \emph{reweighted averages} $\xof_t$ for each $t=1, \ldots, T$.

To prove the key regret bound for the $y$-player, we first need to state some simple technical facts.
\begin{eqnarray}
  \label{eq:haty}
  \yftl_{t+1} & = & \argmin_y \sum_{s=1}^t \alpha_s \left( f^*(y) - \langle x_s, y \rangle\right) = \argmax_y  \left \langle \xav_t, y \right \rangle - f^*(y) \\ & = & \nabla f(\xav_{t}) \\
  \label{eq:tildey} \yof_t & = &
\argmin_y
\alpha_t  \left( f^*(y) - \langle x_t, y \rangle\right) 
+
 \sum_{s=1}^{t-1} \alpha_s \left( f^*(y) - \langle x_s, y \rangle\right) \\ &  = &
   \nabla f(\xof_t) , \\
  \label{eq:xdiffs} \xof_t - \xav_t & = & \frac{\alpha_t}{A_t}(x_{t-1} - x_{t}).
\end{eqnarray}
Equations~\eqref{eq:haty} and~\eqref{eq:tildey} follow from elementary properties of Fenchel conjugation and the Legendre transform \cite{R96}. Equation~\eqref{eq:xdiffs} follows from a simple algebraic calculation.

\begin{lemma} \label{lem:yregretbound}
Suppose $f(\cdot)$ is a convex function that is $L$-smooth with respect to the the norm $\| \cdot \|$ with dual norm $\| \cdot \|_*$. Let $x_1, \ldots, x_T$ be an arbitrary sequence of points. Then, we have
\begin{equation} \label{wregret_y}
\textstyle
\regret{y}(\yof_1, \ldots, \yof_T) \leq L \sum_{t=1}^T \frac{\alpha_t^2}{A_t} \|x_{t-1} - x_t \|^2.
\end{equation}
\end{lemma}

\begin{proof}
Using Lemma~\ref{regret:Opt-FTL}
with $m_t(\cdot) \leftarrow \ell_{t-1}(\cdot)$, $w_t \leftarrow \hat{y}_{t}$, and $z_t \leftarrow \tilde{y}_t$,
and that $
\alpha_t \left( \ell_t(y) - \ell_{t-1}(y) \right) $ $= \alpha_t \langle x_{t-1} - x_{t}, y \rangle$ in Fenchel Game, we have
    \begin{eqnarray*}
         \textstyle \sum_{t=1}^{T} \alpha_t \ell_t(\yof_t) - \alpha_t \ell_t(y^*) 
        & \leq &  \textstyle \sum_{t=1}^T \alpha_t
     \left(   \ell_t( \yof_t ) - \ell_{t-1}( \yof_t )
      - \left( \ell_t(\yftl_{t+1}) - \ell_{t-1}(\yftl_{t+1})  \right) \right)\\
        \text{(Eqns. \ref{eq:haty}, \ref{eq:tildey})} \quad \quad
        & = & \textstyle  \sum_{t=1}^T \alpha_t \langle x_{t-1} - x_{t}, \nabla f(\xof_t) - \nabla f(\xav_t) \rangle \\
        \text{(H\"older's Ineq.)} \quad \quad
        & \leq & \textstyle  \sum_{t=1}^T \alpha_t \| x_{t-1} - x_{t}\| \| \nabla f(\xof_t) - \nabla f(\xav_t) \|_* \\
        \text{($L$-smoothness of $f$)} \quad \quad
        & \leq & \textstyle   L \sum_{t=1}^T \alpha_t \| x_{t-1} - x_{t}\| \|\xof_t - \xav_t \| \\
        \text{(Eqn. \ref{eq:xdiffs})} \quad \quad
        & = & \textstyle  L \sum_{t=1}^T \frac{\alpha_t^2}{A_t} \| x_{t-1} - x_{t}\| \|x_{t-1} - x_{t} \|
    \end{eqnarray*}
    as desired, where the first inequality is because that $m_t(\cdot)= \ell_{t-1}(\cdot)$.
  \end{proof}

\begin{theorem}\label{thm:metaAcc}
Let us consider the output $(\xav_T, \yav_T)$ of Algorithm~\ref{alg:game} under the following conditions: (a) the sequence $\{\alpha_t\}$ is positive but otherwise arbitrary (b) $\alg^y$ is chosen \OFTL, (c) $\alg^x$ is \MD with a parameter $\gamma$, and (d) we have a bound $V_{x_0}(x^*) \leq D$. Then the point $\xav_T$ satisfies
\begin{equation} \label{eq:genericbound}
\displaystyle  f(\xav_T) - \min_{x \in \XX} f(x)
  \leq \frac 1 {A_T} \left( \frac D {\gamma }  
    + \sum_{t=1}^{T} \left(\frac{\alpha_t^2}{A_t} L - \frac{\beta}{2 \gamma} \right)
         \| x_{t-1} - x_t \|^2    \right).
\end{equation}
On the other hand, following the same setting, if $\alg^x$ is chosen as \BTRL with a $\beta$-strongly convex regularizer $R(\cdot)$ and a parameter $\eta$. Then the point $\xav_T$ satisfies
\begin{equation} \label{eq:genericbound2}
\displaystyle  f(\xav_T) - \min_{x \in \XX} f(x)
  \leq \frac 1 {A_T} \left( \frac{ R(x^*) - R(\hat{x})}{\eta }  
    + \sum_{t=1}^{T} \left(\frac{\alpha_t^2}{A_t} L - \frac{\beta}{2 \eta} \right)
         \| x_{t-1} - x_t \|^2    \right),
\end{equation}
where $R(\hat{x}) = \min_{x \in \XX} R(x)$.
\end{theorem}

\begin{proof}
  We have already done the hard work to prove this theorem. Lemma~\ref{lem:fenchelgame} tells us we can bound the error of $\xav_T$ by the $\epsilon$ error of the approximate equilibrium $(\xav_T, \yav_T)$. Theorem~\ref{thm:meta} tells us that the pair $(\xav_T, \yav_T)$ derived from Algorithm~\ref{alg:game} is controlled by the sum of averaged regrets of both players, $\frac{1}{A_T}(\regret{x}[\MD] + \regret{y}[\OFTL])$. But we now have control over both of these two regret quantities, from Lemmas~\ref{lem:yregretbound} of \OFTL and~\ref{lem:MD} of \MD, 
\begin{equation}
\displaystyle  f(\xav_T) - \min_{x \in \XX} f(x)
  \leq \frac 1 {A_T} \left( \frac D {\gamma} 
    + \sum_{t=1}^{T} \left(\frac{\alpha_t^2}{A_t} L - \frac{\beta}{2 \gamma} \right)
         \| x_{t-1} - x_t \|^2    \right).
\end{equation}

  On the other hand, if the y-player is $\OFTL$ and the x-player is $\BTRL$, then, by Lemma~\ref{lem:yregretbound} of \OFTL and Lemma~\ref{regret:BTRL} of \BTRL with $\mu=0$ (as the x-player sees linear loss functions), we have 
\begin{equation} 
\displaystyle  f(\xav_T) - \min_{x \in \XX} f(x)
  \leq \frac 1 {A_T} \left( \frac{ R(x^*) - R(\hat{x})}{\eta }  
    + \sum_{t=1}^{T} \left(\frac{\alpha_t^2}{A_t} L - \frac{\beta}{2 \eta} \right)
         \| x_{t-1} - x_t \|^2    \right),
\end{equation}
where $R(\hat{x}) = \min_{x \in \XX} R(x)$.
 
\end{proof}

Theorem~\ref{thm:meta} is somewhat opaque without a specifying the sequence $\{\alpha_t\}$. But what we now show is that the summation term \emph{vanishes} when we can guarantee that $\frac{\alpha_t^2}{A_t}$ remains constant! This is where we obtain the following fast rate.

\begin{corollary} \label{cor:meta}
Following the setting as Theorem~\ref{thm:metaAcc},
if the x-player is $MD$ with a
$1$-strongly convex distance generating function $\phi(\cdot)$ and
the parameter $\gamma$ that
satisfies
$\frac{1}{CL} \leq \gamma \leq \frac{1}{4L}$ for some constant $C \geq 4$, then 
\[
   \displaystyle f(\xav_T) - \min_{x \in \XX} f(x) 
  \leq \frac{2 C L D} {T^2},
\] where $\V{x_0}(x^*) \leq D$.
Similarly,
if the x-player is $\BTRL$ 
with a $1$-strongly convex regularizer $R(\cdot)$,
and the parameter $\gamma$
satisfies
$\frac{1}{CL} \leq \eta \leq \frac{1}{4L}$ for some constant $C \geq 4$
then
\[
   \displaystyle f(\xav_T) - \min_{x \in \XX} f(x) 
  \leq \frac{2 C L \big( R(x^*) - R(\hat{x}) \big) } {T^2},
\]
where $R(\hat{x}) = \min_{x \in \XX} R(x)$.
\end{corollary}

\begin{proof}
As we use $\alpha_t = t$, we have that $A_t := \frac{t(t+1)}{2}$. The choice of $\{\alpha_t,\gamma\}$ implies $\frac{D}{\gamma} \leq CLD$ and $\frac{L\alpha_t^2}{A_t} = \frac{2Lt^2}{t(t+1)} \leq 2L \leq \frac 1 {2 \gamma}$, which ensures that the summation term in \eqref{eq:genericbound} is negative. The rest is simple algebra.

Similar calculations can be done for the bound \eqref{eq:genericbound2}, and hence omitted.
%
\end{proof}

It is worth dwelling on exactly how we obtained the above result. A less refined analysis of the \MD algorithm would have simply ignored the negative summation term in Lemma~\ref{lem:MD}, and simply upper bounded this by 0. But the negative terms $\| x_t - x_{t-1} \|^2$ in this sum happen to correspond \textit{exactly} to the positive terms one obtains in the regret bound for the $y$-player, but this is true \textit{only as a result of} using the \OFTL algorithm. To obtain a cancellation of these terms, we need a $\gamma_t$ which is roughly constant, and hence we need to ensure that $\frac{\alpha_t^2}{A_t} = O(1)$. The final bound, of course, is determined by the inverse quantity $\frac 1 {A_T}$, and a quick inspection reveals that the best choice of $\alpha_t = \theta(t)$. This is not the only choice that could work, and we conjecture that there are scenarios in which better bounds are achievable for different $\alpha_t$ tuning. We show in Subsection~\ref{sub:acclinear} that a \emph{linear rate} is achievable when $f(\cdot)$ is also strongly convex, and there we tune $\alpha_t$ to grow exponentially in $t$ rather than linearly.

\subsubsection{Nesterov's methods}

\begin{algorithm}[h] 
   \caption{ Nesterov's 1-memory method \cite{N88,T08} } \label{alg:Nes-1mem}
Given: $L$-smooth $f(\cdot)$, convex domain $\K$, arbitrary $v_0 \in \K$, 
1-strongly convex distance generating function $\phi(\cdot)$, iterations $T$
\begin{center} 
\begin{tabular}{c c}
    $
      \boxed{
      \begin{array}{rl}
        \beta_t & \gets \frac{2}{t+1}, \gamma_{t} \gets \frac{t}{4L}\\
        z_{t} & \gets (1 - \beta_t) w_{t-1} + \beta_t v_{t-1}\\
        v_{t} & \gets \underset{x \in \K}{ \argmin} 
          \gamma_t \langle  \nabla f(z_t), x \rangle  + \V{v_{t-1}}(x)
        \\
        w_{t} & \gets (1 - \beta_t) w_{t-1} + \beta_t v_{t}
      \end{array}
      }
    $
    & \small
    $\boxed{
    \begin{array}{rl}
      g(x,y) & := \langle x, y \rangle - f^*(y)\\
      \alpha_t & := t\text{ for } t=1, \ldots, T\\
      \alg^Y & := \OFTL[\nabla f(v_0)] \\
      \alg^X & := \MD[\phi(\cdot), \frac{1}{4L}] \\
    \end{array}
    }$
  \\
    \small Iterative Description & 
    \small FGNRD Equivalence
\end{tabular}
\end{center}
Output: $w_T = \bar x_T$,
\end{algorithm}

\begin{algorithm}[h] 
   \caption{ Nesterov's $\infty$-memory method \cite{N05,T08} } \label{alg:Nes-infmem}
Given: $L$-smooth $f(\cdot)$, convex domain $\K$, arbitrary $v_0 \in \K$, 
1-strongly convex regularizer $R(\cdot)$, iterations $T$
\begin{center} 
\begin{tabular}{c c}
    $
      \boxed{
      \begin{array}{rl}
        \beta_t & \gets \frac{2}{t+1}, \gamma_{t} \gets \frac{t}{4L} \\\
        z_{t} & \gets (1 - \beta_t) w_{t-1} + \beta_t v_{t-1}\\
        v_{t} & \displaystyle \gets \argmin_{x \in \K} 
          \sum_{s=1}^{t} \gamma_s \langle \nabla f(z_s), x \rangle + R(x)
        \\
        w_{t} & \gets (1 - \beta_t) w_{t-1} + \beta_t v_{t}
      \end{array}
      }
    $
    & \small
    $\boxed{
    \begin{array}{rl}
      g(x,y) & := \langle x, y \rangle - f^*(y)\\
      \alpha_t & := t\text{ for } t=1, \ldots, T\\
      \alg^Y & := \OFTL[\nabla f(v_0)] \\
      \alg^X & := \BTRL[R(\cdot), \frac{1}{4L}] \\
    \end{array}
    }$
  \\
    \small Iterative Description & 
    \small FGNRD Equivalence
\end{tabular}
\end{center}
Output: $w_T = \bar x_T$,
\end{algorithm}

Starting from 1983, Nesterov has proposed three accelerated methods for smooth convex problems
(i.e. \cite{N83a,N83b,N88,N05}). In this section, we show that our accelerated algorithm to the \textit{Fenchel game} can generate all the methods with some simple tweaks.

We first consider recovering Nesterov's (1988) 1-memory method \cite{N88}
and Nesterov's (2005) $\infty$-memory method \cite{N05}.
To be precise, we adopt the presentation of Nesterov's algorithm given in Algorithm~1 
and Algorithm~3 of \citet{T08} respectively.

\begin{theorem}\label{thm:Nes_constrained} 
 The two interpretations of Nesterov's $1$-memory method (Nesterov's $\infty$-memory method, as described in Algorithm~\ref{alg:Nes-1mem} (Algorithm~\ref{alg:Nes-infmem}, respectively), are equivalent. That is, for every $t$, the iterate $w_t$ computed iteratively on the left hand side is identically the weighted-average point $\bar x_t$ produced by the dynamic on the right hand side.
\end{theorem} 

\begin{proof}
Let us recall the notations (\ref{def:tilde_x}),
$\xav_t := \textstyle \frac{1}{A_t}\sum_{s=1}^t \alpha_s x_s \text{ and } \xof_t := \textstyle \frac{1}{A_t}(\alpha_t x_{t-1} + \sum_{s=1}^{t-1} \alpha_s x_s ).$
We show, via induction, that the following three equalities are maintained for every $t$. Note that three objects on the left correspond to the iterative description given in Algorithm~\ref{alg:Nes-1mem} whereas the three on the right correspond to the FGNRD description.

  \begin{eqnarray}
    \nabla f(z_t) & = & y_t \label{eq:ygradAcc}\\
    v_t & = & x_t \label{eq:zAcc}\\
    w_t & = & \bar{x}_t \label{eq:xAcc}.
  \end{eqnarray}
  We first note that %
  the initialization ensures that \eqref{eq:ygradAcc} holds for $t=1$. Second, the choices of learning rate $\gamma_t$ and the weighting scheme $\{\alpha_t\}$ leads to
  \begin{eqnarray} \label{eq:w_mixAcc}
w_t  =  \frac{1}{ \sum_{s=1}^t s } \sum_{s=1}^t s v_s  &=& \frac{1}{A_t} \sum_{s=1}^t \alpha_s v_s,  \quad \text{ if } (\beta_t=\frac{2}{t+1}, \alpha_t = t).
  \end{eqnarray}
From \eqref{eq:w_mixAcc}, we see that \eqref{eq:zAcc} implies \eqref{eq:xAcc}, as $w_t$ is always an average of the updates $v_t$. It remains to establish \eqref{eq:ygradAcc} and \eqref{eq:zAcc} via induction.

Let us first show \eqref{eq:ygradAcc}.
We have already shown in \eqref{eq:tildey} that $y_t = \nabla f(\tilde{x}_t)$. So it suffices to show that $\tilde{x}_t = z_t$.
We have that 
$ \textstyle z_{t}  = (1-\beta_t) w_{t-1} + \beta_t x_{t-1}
= (1-\beta_t) ( \sum_{{s=1}}^{t-1} \frac{\alpha_s}{A_{t-1}} x_{s}   ) + \beta_t x_{t-1}
 \textstyle =  (1 - \frac{2}{t+1}) ( \sum_{{t=1}}^{t-1} \frac{\alpha_t}{ \frac{t(t-1)}{2} } x_{t}   ) + \beta_t x_{t-1} = \sum_{{s=1}}^{t-1} \frac{\alpha_s}{ \frac{t(t+1)}{2} } x_{s} +  \beta_t x_{t-1}
 = \sum_{{s=1}}^{t-1} \frac{\alpha_s}{ A_{t} } x_{s} +  \frac{\alpha_t}{A_t} x_{t-1}
 = \xof_t.
 $
 To show (\ref{eq:zAcc}), observe that the update on line 5 of Algorithm~\ref{alg:game} is exactly equivalent to \MD
 shown on (\ref{eq:MDupdate})
 for which $\gamma \leftarrow \frac{1}{4L}$ and $\theta_t \leftarrow y_t = \nabla f(z_t) $. Since by induction, $x_{t-1}=v_{t-1}$, we have that $x_t = v_t$. We thus have completed the first part of proof.

Similar analysis can be conducted for the equivalency between Nesterov's $\infty$-memory method.
Specifically, the method corresponds to \BTRL is used as the x-player's strategy.
 
\end{proof}

\subsubsection{Acceleration for unconstrained smooth convex problems}

\begin{algorithm}[h] 
   \caption{ Nesterov's first acceleration method  \cite{N83b,N83a}
    } \label{alg:NesV0}
Given: $L$-smooth $f(\cdot)$, arbitrary $z_0 \in \reals^d$, iterations $T$
\begin{center} 
\begin{tabular}{c c}
    $
      \boxed{
      \begin{array}{rl}
        \theta & \gets \frac{t}{2(t+1)L}, \beta_t \gets \frac{t-1}{t+2}\\
        w_t & \gets z_{t-1} - \theta \nabla f(z_{t-1})\\
        z_t & \gets w_t + \beta_t (w_t - w_{t-1})
      \end{array}
      }
    $
    & \small
    $\boxed{
    \begin{array}{rl}
      g(x,y) & := \langle x, y \rangle - f^*(y)\\
      \alpha_t & := t\text{ for } t=1, \ldots, T\\
      \alg^Y & := \OFTL[\nabla f(z_0)] \\
      \alg^X & := \MD[\frac 1 2 \| \cdot \|^2_2, \frac{1}{4L}] 
    \end{array}
    }$
  \\
    \small Iterative Description & 
    \small FGNRD Equivalence
\end{tabular}
\end{center}
Output: $w_T = \bar x_T$
\end{algorithm}

Now let us consider that the x-player's action space is unconstrained.
That is, $\K= \reals^n$. 
We are going to show that 
our framework can recover
Nesterov's first acceleration method \cite{N83a,N83b} (see also \cite{SBC14}).

\begin{theorem}\label{thm:Nes1988} 
 The interpretations of Nesterov's first acceleration method \cite{N83a,N83b} as described in Algorithm~\ref{alg:NesV0} are equivalent. That is, for every $t$, the iterate $w_t$ computed iteratively on the left hand side is identically the weighted-average point $\bar x_t$ produced by the dynamic on the right hand side.
\end{theorem} 

\begin{proof}
First of all, in the \MD strategy of the x-player,
we can let the distance generating function of the Bregman divergence to be the squared of L2 norm, i.e.
$\phi(x):=\frac{1}{2}\| x \|^{2}_2$.
Then, the update becomes $x_{t} = \argmin_{x}   \gamma_t \langle x, \alpha_t y_t \rangle  + V_{x_{t-1}}(x) = \argmin_{x}  \gamma_t \langle x, \alpha_t y_t \rangle  + \frac{1}{2} \| x \|^{2}_{2} - \langle x_{t-1}, x - x_{t-1} \rangle - \frac{1}{2} \| x_{t-1} \|^{2}_{2} $. 
Differentiating the objective w.r.t $x$ and setting it to zero, one will get
$x_{t} = x_{{t-1}} - \gamma_t \alpha_{t} y_{t}$.

To see the equivalence, let us re-write $\bar{x}_t := \frac{1}{A_t} \sum_{s=1}^t \alpha_{s} x_{s}$ as follows,
\begin{align}
&  \bar{x}_t = \frac{A_{t-1} \bar{x}_{t-1} + \alpha_t x_t    }{A_t}
            = \frac{A_{t-1} \bar{x}_{t-1} + \alpha_t ( x_{t-1} - \gamma_t \alpha_t \nabla f( \xof_t) ) }{A_t} \nonumber
\\ &              = \frac{A_{t-1} \bar{x}_{t-1} + \alpha_t ( \frac{ A_{t-1}\bar{x}_{t-1} - A_{t-2}\bar{x}_{t-2}  }{ \alpha_{t-1} } - \gamma_t \alpha_t \nabla f( \xof_t) )}{A_t} \nonumber
\\ &             = \bar{x}_{t-1} ( \frac{A_{t-1}}{A_t} + \frac{\alpha_t (\alpha_{t-1} + A_{t-2} ) }{ A_t \alpha_{t-1} }   ) - \bar{x}_{t-2} ( \frac{\alpha_t A_{t-2} }{ A_t \alpha_{t-1} }   ) - \frac{\gamma_t \alpha_t^2}{A_t} \nabla f( \xof_t ) \nonumber
\\ &             = \bar{x}_{t-1} - \frac{\gamma_t \alpha_t^2}{A_t} \nabla f(\xof_t) + ( \frac{\alpha_t A_{t-2} }{ A_t \alpha_{t-1} }   ) (\bar{x}_{t-1} -  \bar{x}_{t-2} ) \nonumber
\\&   = \bar{x}_{t-1} - \frac{t}{2(t+1)L} \nabla f(\xof_t) + ( \frac{t-2 }{ t+1 }   ) (\bar{x}_{t-1} -  \bar{x}_{t-2} ). \label{acc_expand}
\end{align}
 
\end{proof}

Let us switch to comparing the update of (\ref{acc_expand}) of Nesterov's method with the update of the \HB algorithm.
We see that (\ref{acc_expand}) has the so called momentum term (i.e. has a $(\bar{x}_{t-1} -  \bar{x}_{t-2}$) term). But, the difference is that the gradient is evaluated at $\xof_{t} = \frac{1}{A_t}(\alpha_t x_{t-1} + \sum_{s=1}^{t-1} \alpha_s x_s )$, not 
$\bar{x}_{t-1} = \frac{1}{A_{t-1}} \sum_{s=1}^{t-1} \alpha_s x_s$, which is the consequence that the y-player plays \OFTL.
To elaborate, let us consider a scenario (shown in Algorithm~\ref{alg:heavyball}) such that the $y$-player plays \FTL instead of \OFTL.
\begin{algorithm}[h] 
   \caption{ Heavy Ball  } \label{alg:heavyball}
Given: $L$-smooth $f(\cdot)$, arbitrary $z_0 \in \K$, iterations $T$
\begin{center} 
\begin{tabular}{c c}
    $
      \boxed{
      \begin{array}{rl} 
        \eta_t & \gets \frac{t}{2(t+1)L}, \quad \beta_t \gets \frac{t-1}{t+2}\\
        v_t  & \gets w_{t-1} - w_{t-2} \\
        w_{t} & \gets w_{t-1} - \eta_t \nabla f(w_{t-1}) + \beta_t v_t
      \end{array}
      }
    $
    & \small
    $\boxed{
    \begin{array}{rl}
      g(x,y) & := \langle x, y \rangle - f^*(y)\\
      \alpha_t & := t\text{ for } t=1, \ldots, T\\
      \alg^Y & := \FTL[\nabla f(w_0)] \\
      \alg^X & := \MD[\frac 1 2 \| \cdot \|^2_2, \frac{1}{4L}] 
    \end{array}
    }$
  \\
    \small Iterative Description & 
    \small FGNRD Equivalence
\end{tabular}
\end{center}
Output: $w_T = \bar x_T$
\end{algorithm}

Following what we did in (\ref{acc_expand}), we can rewrite $\bar{x}_{t}$ of Algorithm~\ref{alg:heavyball} as
\begin{equation} \label{eq:hvy}
 \bar{x}_t =\bar{x}_{t-1} - \frac{\gamma_t \alpha_t^2}{A_t} \nabla f(\bar{x}_{t-1}) + (\bar{x}_{t-1} -  \bar{x}_{t-2} ) ( \frac{\alpha_t A_{t-2} }{ A_t \alpha_{t-1} }   ),
\end{equation}
by observing that (\ref{acc_expand}) still holds except that $\nabla f( \xof_t )$ is changed to $\nabla f(\bar{x}_{t-1})$ as the y-player uses \FTL now,
which give us the update of the Heavy Ball algorithm as (\ref{eq:hvy}). Moreover, by the regret analysis,
we have the following theorem. The proof is in Section~\ref{app:thm:Heavy}.
\begin{theorem}\label{thm:Heavy} 
Let $\alpha_{t}=t$. Assume $\K = \reals^{n}$. Also, let $\gamma_t =O(\frac{1}{L})$.
The output $\bar{x}_{T}$ of Algorithm~\ref{alg:heavyball} is an $O(\frac{1}{T})$-approximate optimal solution of $\min_{x} f(x)$.
\end{theorem}

To conclude, by comparing Algorithm~\ref{alg:NesV0} and Algorithm~\ref{alg:heavyball},
we see that Nesterov's (1983) method enjoys $O(1/T^2)$ rate since its adopts $\OFTL$, while the \HB algorithm which adopts $\FTL$
may not enjoy the fast rate, as the distance terms may not cancel out. The result also conforms to empirical studies that the \HB does not exhibit acceleration on general smooth convex problems.

\subsubsection{Accelerated proximal method}

\begin{algorithm}[h] 
   \caption{ Accelerated proximal method
    } \label{alg:AccProx}
Given: $L$-smooth $f(\cdot)$, arbitrary $w_0 \in \reals^d$, iterations $T$.
\begin{center} 
\begin{tabular}{c c}
    $
      \boxed{
      \begin{array}{rl}
        \beta_t & \gets \frac{2}{t+1}, \gamma_{t} \gets \frac{t}{4L}\\
        z_{t} & \gets (1 - \beta_t) w_{t-1} + \beta_t v_{t-1}\\
        v_{t} & \gets
\textbf{prox}_{t \gamma \psi} ( x_{t-1}- t \gamma \nabla f(z_t) ) 
        \\
        w_{t} & \gets (1 - \beta_t) w_{t-1} + \beta_t v_{t}
      \end{array}
      }
    $
    & \small
    $\boxed{
    \begin{array}{rl}
      g(x,y) & := \langle x, y \rangle - f^*(y)+ \psi(x)\\
      \alpha_t & := t\text{ for } t=1, \ldots, T\\
      \alg^Y & := \OFTL[\nabla f(v_0)] \\
      \alg^X & := \MD[\phi(\cdot), \frac{1}{4L}] \\
    \end{array}
    }$
  \\
    \small Iterative Description & 
    \small FGNRD Equivalence
\end{tabular}
\end{center}
Output: $w_T = \bar x_T$
\end{algorithm}

In this section, we consider solving composite optimization problems
\begin{equation}
\min_{x \in \reals^d} f(x) + \psi(x),
\end{equation}
where $f(\cdot)$ is smooth convex but $\psi(\cdot)$ is possibly non-differentiable convex (e.g. $\|\cdot\|_1$).
We want to show that the game analysis still applies to this problem.
We just need to change the payoff function $g$ to account for $\psi(x)$.
Specifically, we consider the following two-players zero-sum game,
\begin{equation} \label{eq:gnew}
 \min_{x} \max_{y} g(x,y):=\{ \langle x, y \rangle - f^*(y) + \psi(x)  \}.
\end{equation} 
Notice that the minimax value of the game is $\min_{x} f(x) + \psi(x)$, which is exactly the  optimum value of the composite optimization problem.
Let us denote the proximal operator as 
$\textbf{prox}_{\lambda \psi} (v) = \argmin_x \big( \psi(x)+ \frac{1}{2\lambda} \| x - v \|^2_2 \big).$
\footnote{It is known that for some $\psi(\cdot)$, their corresponding proximal operations have closed-form solutions
(see e.g. \cite{PB14} for details).}
We have Algorithm~\ref{alg:AccProx}.
We remark that
Algorithm~\ref{alg:AccProx} is essentially Algorithm~\ref{alg:Nes-1mem}, as the learners use the same stategies and the weighting scheme $\alpha_t = t$ is the same.
The only difference is the new payoff function $g(x,y)$ (\ref{eq:gnew}).

In this new game, the $x$-player plays $\MD$ with the distance generating function $\phi_x= \frac{1}{2} \| x \|^2_2$, which leads to the following update,
\begin{equation}
\begin{split}
x_{t} & = \argmin_{x}   \gamma ( \alpha_t h_t(x) ) + V_{x_{t-1}}(x)
= \argmin_{x}   \gamma ( \alpha_t \{  \langle x , y_t \rangle + \psi(x)   \} ) + \V{x_{t-1}}(x)
\\ & = \argmin_{x} \phi(x) + \frac{1}{2 \alpha_t \gamma} ( \| x \|^2_2 + 2 \langle \alpha_t \gamma y_t - x_{t-1} , x \rangle )  = \textbf{prox}_{\alpha_t \gamma \psi} ( x_{t-1}- \alpha_t \gamma \nabla f( \xof_{t}) ).
\end{split}
\end{equation}
One can view Algorithm~\ref{alg:AccProx} as a variant of the so called ``Accelerated Proximal Gradient''in \citet{BT09}.
Yet, the design and analysis of our algorithm is simpler than that of \citet{BT09}.

\begin{theorem} \label{thm:proximal}
Denote $D:= V_{x_0}(x^*)$.
The weighted average of $\bar{x}_T$ in Algorithm~\ref{alg:AccProx}
satisfies 
\[
f(\bar{x}_T) - \min_x f(x) \leq O(\frac{L D}{T^2}).
\] 
\end{theorem}

\begin{proof}
Even though the payoff function $g(\cdot,\cdot)$ is a bit different,
the proof still essentially follows the same line as
Theorem~\ref{thm:metaAcc} and Collorary~\ref{cor:meta}, as $y$-player plays $\OFTL$ and the $x$-player plays $\MD$. 
 
\end{proof}

\subsubsection{Related works}
In recent years, there are growing interest in giving new interpretations of Nesterov's accelerated algorithms or proposing new varaints. For example, \citet{T08} gives a unified analysis for some Nesterov's accelerated algorithms
\citet{N88}, \citet{N04}, \citet{N05}, using the standard techniques and analysis in optimization literature.
\citet{LRP16}, \citet{HL17} connects the design of accelerated algorithms with dynamical systems and control theory. \citet{BLS15} gives a geometric interpretation of the Nesterov's method for unconstrained optimization, inspired by the ellipsoid method.
\citet{FB15} studies the Nesterov's methods and the \HB method for quadratic non-strongly convex problems by analyzing the eigen-values of some linear dynamical systems.
\citet{AO17} proposes a variant of accelerated algorithms by mixing the updates of gradient descent and mirror descent and showing the updates are complementary.
\citet{DO18},\citet{DO19} propose a primal-dual view that recovers several first-oder algorithms with careful discretizations of a continuous-time dynamic, which also leads to a new accelerated extra-gradient descent method. 
\citet{CST20} show a simple acceleration proof of mirror prox \citet{Nemi04} 
and dual extrapolation \citet{N07} based on solving the Fenchel game.
\citet{SBC14}, \citet{wibisono2016variational}, \citet{SDJS18} connect the acceleration algorithms with differential equations.
Finally, we note an independent work \citet{LZ18}, \citet{L20} provide a game interpretation of Nesterov's accelerated method. In our work, we show a deeper connection with regret analysis in online learning and propose a modular framework that is not limited to Nesterov's method. 
We also note that
in recent years there has emerged a lot of work where learning problems are treated as repeated games, and many researchers have been studying the relationship between game dynamics and provable convergence rates (see e.g.  \cite{abernethy2008optimal,balduzzi2018mechanics,gidel2018negative,daskalakis2017training,NeurIPS2013_5148}).

\subsection{Accelerated linear-rate method for strongly convex smooth problems} \label{sub:acclinear}

\begin{algorithm}[h] 
   \caption{ Accelerated Gradient with Linear Convergence } \label{alg:blah}
Given: $L$-smooth $\mu$-strongly convex $f(\cdot)$, convex domain $\K$,
arbitrary $w_0 \in \K$, iterations $T$,
and a distance generating function $\phi(\cdot)$ that is
$1$-strongly convex, $L_{\phi}$-smooth, and differentiable,

\begin{center} 
\begin{tabular}{c c}
    $
      \boxed{
      \begin{array}{rl}
        \beta & \gets \frac{1}{2} \sqrt{ \frac{\mu}{L(1+L_{\phi})} } , \gamma_{t} \gets \alpha_t \\
        z_{t} & \gets (1 - \beta) w_{t-1} + \beta v_{t-1}\\
        v_{t} & \displaystyle \gets \argmin_{x \in \K} 
          \sum_{s=1}^{t} \gamma_s \langle \nabla \tilde{f}(z_s), x \rangle + \mu \phi(x) \\
        w_{t} & \gets (1 - \beta) w_{t-1} + \beta v_{t}
      \end{array}
      }
    $
    & \small
    $\boxed{
    \begin{array}{rl}
      g(x,y) & := \langle x, y \rangle - \tilde f^*(y) + \mu \phi(x) \\
             & \text{where } \tilde f(x) := f(x) - \mu \phi(x)  \\
      \alpha_1 & := \frac{\mu}{2L(1+L_{\phi})}\\
      \alpha_t & := \frac{ \theta}{1-\theta} A_{t-1} 
      \text{ for } t=2, \ldots, T,   \\ \text{ where  } \theta  & := \frac{1}{2} \sqrt{ \frac{\mu}{L(1+L_{\phi})} }\\
      \alg^Y & := \OFTL[\nabla f(v_0)] \\
      \alg^X & := \BTRL[\phi] \\
    \end{array}
    }$
  \\
    \small Iterative Description & 
    \small FGNRD Equivalence
\end{tabular}
\end{center}
Output: $w_T = \bar x_T$
\end{algorithm}

Nesterov observed that, when $f(\cdot)$ is both $\mu$-strongly convex and $L$-smooth, one can achieve a rate that is exponentially decaying in $T$ (e.g. page 71-81 of \cite{N04}). It is natural to ask if the zero-sum game and regret analysis in the present work
also recovers this faster rate in the same fashion. We answer this in the affirmative.
Denote $\kappa := \frac{L}{\mu}$.
In the following,
we assume
that the function $f(\cdot)$ is $L$-smooth 
with respect to some norm $\| \cdot \|$
and there exists a differentiable function $r(\cdot)$ that is $L_{\phi}$-smooth and $1$-strongly convex with respect to the same norm $\| \cdot \|$.
Furthermore, assume $f(\cdot)$ is $\mu$-strongly convex in the following sense (see also Section 3.3 of \cite{L20}),
\begin{equation} \label{new_sc}
f(z) \geq f(x) + \langle \nabla f(x), z - x \rangle + \mu \V{x}(z), 
\end{equation}
for all $z, x \in \K$, where $\V{x}(z)$ is the Bregman divegence.
In the case that the norm is the $l_2$ norm, i.e. $\| \cdot \|_2$, we can define $\phi(x):= \frac{1}{2} \| x \|^2_2$ (and hence $L_{\phi}=1$), and
the strong convexity condition (\ref{new_sc}) becomes
\begin{equation}
f(z) \geq f(x) + \langle \nabla f(z), z - x \rangle + \frac{1}{2} \| z - x \|^2_2.
\end{equation}

The function $\tilde{f}(x):= f(x) - \mu \phi(x)$ is a convex function for all $x \in \K$ (see e.g. \cite{LFN18}). Based on this property, we consider a new game
\begin{equation} \label{split2}
\textstyle \tilde g(x,y):= \langle x, y \rangle - \tilde{f}^*(y) + \mu \phi(x),
\end{equation}
where the minimax vale of the game is
$V^* := \min_x \max_y \tilde g(x, y) = \min_x \tilde{f}(x) + \mu \phi(x) = \min_x f(x)$.
In this game, the loss of the y-player in round $t$ is $\alpha_t \ell_t(y): =   \alpha_t( \tilde{f}^*(y)  -\langle x_t, y \rangle )$,
while the loss of the x-player in round $t$ is a strongly convex function 
$\alpha_t h_t(y):= \alpha_t ( \langle x, y_t \rangle + \mu \phi(x) ) $.
We have the following theorem
\begin{theorem} \label{thm:acc_linear2}
Suppose that the function $f(\cdot)$ is $L$-smooth with respect to some norm
$\| \cdot \|$
and $\phi(\cdot)$ is differentiable, $L_{\phi}$-smooth, and $1$-strongly convex with respect to the same norm.
Assume that the function $f(\cdot)$ is $\mu$-strongly convex in the sense of
(\ref{new_sc}).
Define the game
$\tilde g(x,y):= \langle x, y \rangle - \tilde{f}^*(y) + \mu  \phi(x)$.
If the y-player plays \OFTL: $y_t \leftarrow \nabla \tilde{f}(\tilde{x}_t)$
and the x-player plays \BTRL:
$
x_t \leftarrow \arg\min_{{x \in \XX}} \sum_{s=1}^t \alpha_{s} h_{s}(x) + R(x),
$
where $R(x):= \alpha_{0} \mu  \phi(x)$, then 
the weighted average points $(\xav_T, \yav_T)$ would be an $O(\exp(-\frac{T}{2 \sqrt{1+L_{\phi}} \sqrt{\kappa}}))$-approximate equilibrium
of the game, where the weights $\alpha_0, \alpha_1, \dots, \alpha_T$  satisfy
$\frac{\alpha_1}{\alpha_0} \leq \frac{\mu}{ 2 L (1+L_{\phi})}$ and that for $t \geq 2$, $\frac{\alpha_t}{A_t} = \frac{1}{2 \sqrt{1+L_{\phi}}} \sqrt{ \frac{\mu}{L} }$.
This implies that $f(\xav_T) - \min_{x \in \XX} f(x) = O(\exp(-\frac{T }{2 \sqrt{1+L_{\phi}} \sqrt{\kappa}})).$
\end{theorem}

\begin{proof}

As the proof of Lemma~\ref{lem:yregretbound},
we first bound the regret of the y-player as follows. 
    \begin{eqnarray*}
         \textstyle \sum_{t=1}^{T} \alpha_t \ell_t(\yof_t) - \alpha_t \ell_t(y^*) 
        & \leq & \textstyle  \sum_{t=1}^T \alpha_t \langle x_{t-1} - x_{t}, \yof_t - \yftl_{t+1} \rangle \\
        \text{(Eqns. \ref{eq:haty}, \ref{eq:tildey})} \quad \quad
        & = & \textstyle  \sum_{t=1}^T \alpha_t \langle x_{t-1} - x_{t}, \nabla \tilde{f}(\xof_t) - \nabla \tilde{f}(\xav_t) \rangle \\
        \text{(H\"older's Ineq.)} \quad \quad
        & \leq & \textstyle  \sum_{t=1}^T \alpha_t \| x_{t-1} - x_{t}\| \| \nabla \tilde{f}(\xof_t) - \nabla \tilde{f}(\xav_t) \|_* \\
        & = & \textstyle  \sum_{t=1}^T \alpha_t \| x_{t-1} - x_{t}\| \\ & & \quad \times \| \nabla f(\xof_t) -\mu \nabla \phi( \xof_t ) - \nabla f(\xav_t) + \mu \nabla \phi( \xav_t) \|_* \\
                        \text{(triangle inequality)} \quad \quad
        & \leq & \textstyle  \sum_{t=1}^T \alpha_t \| x_{t-1} - x_{t}\| \\ & & \qquad \times ( \| \nabla f(\xof_t) - \nabla f(\xav_t) \|_* +  \mu L_{\phi} \| \xav_t - \xof_t \| ) \\
                \text{($L$-smoothness and $L\geq \mu$)} \quad \quad
        & \leq & \textstyle  L (1+L_{\phi}) \sum_{t=1}^T \alpha_t \| x_{t-1} - x_{t}\| \|\xof_t - \xav_t \| \\
        \text{(Eqn. \ref{eq:xdiffs})} \quad \quad
        & = & \textstyle L (1+L_{\phi}) \sum_{t=1}^T \frac{\alpha_t^2}{A_t} \| x_{t-1} - x_{t}\| \|x_{t-1} - x_{t} \|
    \end{eqnarray*}
Therefore, the regret satisfies
\begin{equation} \label{reg_y_accH-2}
\textstyle
\regret{y} \leq L(1+L_{\phi}) \sum_{t=1}^T \frac{\alpha_t^2}{A_t} \|x_{t-1} - x_t \|^2.
\end{equation}
\end{proof}

For the x-player, denote
$\tilde{A}_{t}:= \sum_{s=0}^{t} \alpha_{s}.$
Notice that this is different from $A_{t}:= \sum_{s=1}^{t} \alpha_{s}$. 
Then, according to Lemma~\ref{regret:BTRL}, its regret is 
\begin{equation} \label{reg_x_acc0-2}
\begin{aligned}{}
\textstyle
\regret{x} 
\leq R(x^*) - R(x_0) - \sum_{t=1}^T \frac{\mu \tilde{A}_{t-1}}{2} \|x_{t-1} - x_t \|^2_2,
\end{aligned}
\end{equation} 
where $x_{0} = \arg\min_{x} R(x)$.
Summing (\ref{reg_y_accH-2}) and (\ref{reg_x_acc0-2}), we have
\begin{equation}
\begin{aligned}
& \regret{y} + \regret{x} \leq R(x^*) - R(x_0) + \sum_{t=1}^T ( \frac{L(1+L_{\phi}) \alpha_t^2}{A_t} - \frac{\mu \tilde{A}_{t-1}}{2} ) \|x_{t-1} - x_t \|^2_2.
\end{aligned}
\end{equation}
By choosing the weight $\{\alpha_t\}$ to satisfy $\frac{\alpha_1}{\alpha_0} \leq \frac{\mu}{ 2 L (1+L_{\phi})}$ and that for $t \geq 2$, $\frac{\alpha_t}{A_t} = \frac{1}{2} \sqrt{ \frac{\mu}{L(1+L_{\phi})} }$, the coefficient distance terms will be non-positive, i.e.
$(\frac{ L(1+L_{\phi}) \alpha_t^2}{A_t} - \frac{\mu \tilde{A}_{t-1}}{2} ) \leq 0 $, which means that the distance terms will cancel out.
To see this, let $\frac{\alpha_t}{A_t} = \theta$ for some constant $\theta >0$, we have that
\begin{equation}
\begin{split}
\frac{L(1+L_{\phi}) \alpha_t^2}{A_t} - \frac{\mu \tilde{A}_{t-1}}{2}
& = L(1+L_{\phi}) \theta^2 A_t - \frac{\mu}{2} ( A_t(1-\theta) + \alpha_0 )
\\ & \leq A_t \left( L(1+L_{\phi}) \theta^2 - \frac{\mu}{2}(1-\theta) \right).
\end{split}
\end{equation}
So it suffices to have that $L(1+L_{\phi}) \theta^2 - \frac{\mu}{2}(1-\theta) \leq 0$, which can be guaranteed by choosing $\theta = \frac{1}{2} \sqrt{ \frac{\mu}{L(1+L_{\phi})} }$.

Therefore, the optimization error $\epsilon$ after $T$ iterations satisfies that
\begin{equation}
\begin{aligned}
& \epsilon \leq
\frac{ \regret{y} + \regret{x} }{A_T} 
\leq \frac{1}{A_1} \frac{A_1}{A_2} \cdots \frac{A_{T-1}}{A_T}
( R(x^*) - R(x_0)  ) 
\\ & =  \frac{1}{A_1} (1 - \frac{\alpha_2}{A_2} ) \cdots (1 - \frac{\alpha_T}{A_T} )
( R(x^*) - R(x_0)  ) 
\\ & \leq \frac{1}{A_1} (1 - \frac{\alpha_2}{\tilde{A}_2} ) \cdots (1 - \frac{\alpha_T}{\tilde{A}_T} )
( R(x^*) - R(x_0)  ) 
\\ & \leq (1- \frac{1}{2 \sqrt{1+L_{\phi}} \sqrt{\kappa}} )^{T-1} \frac{R(x^*) - R(x_0) }{A_1},
\end{aligned}
\end{equation}
which is $O\left( \big(1 - \frac{1}{2 \sqrt{1+L_{\phi}} \sqrt{\kappa}} \big)^{T}\right) = O\left( \exp\big( - \frac{1}{2 \sqrt{1+L_{\phi}} \sqrt{\kappa}} T \big) \right)$.

\section{New algorithms}
\label{sec:NewAlgs}

\subsection{Boundary Frank-Wolfe}

We observe that the meta-algorithm previously discussed assumed that the $y$-player (i.e. the player who plays gradients) was first to act, followed by the $x$-player who was allowed to be prescient. Here we reverse their roles, and we instead allow the $y$-player to be prescient.
The new meta-algorithm is described in Algorithm~\ref{alg:example2}.
We are going to show that this framework lead to a new projection-free algorithm that works for non-smooth objective functions.
Specifically, if the constraint set is strongly convex, then this exhibits a novel projection free algorithm that grants a $O( \log T / T)$ convergence even for non-smooth objective functions.
The result relies on very recent work showing that \FTL for strongly convex sets \citet{HLGS16} grants a $O(\log T)$ regret rate.
Prior work has considered strongly convex decision sets \cite{D15}, yet with the additional assumption that the objective is smooth and strongly convex, leading to $O( 1 / T^2 )$ convergence.
\emph{Boundary Frank-Wolfe} requires neither smoothness nor strongly convexity of the objective.
What we have shown, essentially, is that a strongly convex boundary of the constraint set can be used in place of smoothness of $f(\cdot)$ in order to achieve $O( 1 / T)$ convergence.

\begin{algorithm}[H] 
   \begin{algorithmic}[1]
   \caption{Boundary Frank-Wolfe}\label{alg:new_fw}
\STATE \textbf{Input:} Init. $x_{1} \in \K$.
\FOR{$t=2, 3 \dots , T$}
\STATE $ x_t \leftarrow \argmin_{x \in \K} \frac{1}{t-1} \sum_{s=1}^{t-1} \langle x , \partial f(x_s)  \rangle$
\ENDFOR
\STATE Output: $\bar{x}_T  =  \frac{1}{T} \sum_{t=1}^T x_t$
\end{algorithmic}
\end{algorithm}

\begin{algorithm}[h] 
   \caption{Modified meta-algorithm, swapped roles}
   \label{alg:example2}
\begin{algorithmic}[1]
\FOR{$t= 1, 2, \dots, T$}
\STATE \quad   $x_t := \alg^X( g(\cdot,y_1), \ldots, g(\cdot,y_{t-1}) )$
\STATE \quad   $y_t :=  \alg^Y( g(x_1, \cdot), \ldots, g(x_{t-1}, \cdot), g(x_t, \cdot))$
\ENDFOR
\STATE Output: $\bar x_T = \frac 1 T \sum_{t=1}^T  x_t$ and $\bar y_T := \frac 1 T \sum_{t=1}^T y_t$ 
\end{algorithmic}
\end{algorithm}

\begin{theorem} \label{thm:equiv2}
  Algorithm~\ref{alg:new_fw} is a instance of Algorithm~\ref{alg:example2} if \textbf{(I)} Init. $x_1$ in Alg~\ref{alg:new_fw} equals $x_1$ in Alg.~\ref{alg:example2}; \textbf{(II)} Alg.~\ref{alg:example2} sets $\alg^X := \FTL$; and \textbf{(III)} Alg.~\ref{alg:example2} sets $\alg^Y := \BR$.
  Furthermore, when the constraint set $\XX \leftarrow \K$ is $\lambda$-strongly convex, and $\sum_{s=1}^t \partial f(x_s) $ has non-zero norm, then
  \[
  f(\bar x_T) - \min_{x \in \K} f(x) = O(\frac{ M \log T}{ \lambda L_T T})
  \]
  where $M:= \sup_{x \in \K} \| \partial f(x) \|$, 
$\Theta_t := \sum_{s=1}^t  \frac{1}{t} \partial f( x_s) $,
and  $L_T:= \min_{1\leq t \leq T} \| \Theta_t \|$.
\end{theorem}

\begin{proof}
Note that we have chosen the weighting scheme be $\alpha_t = 1$ for all $t$.
Since y-player plays \BR, its regret is $0$.
For the x-player, we use Lemma~\ref{thm:FTL},
its has regret which satisfies $\regret{x} \leq O(\frac{ M \log T}{ \lambda L_T})
$.
So, by summing the average regrets of both players, we obtain the result, i.e.
\begin{equation}
\textstyle
 f(\bar x_T) - \min_{x \in \K} f(x) \leq \frac{1}{T} ( \regret{x}[\FTL] + \regret{y}[\BR] ) \leq O(\frac{ M \log T}{ \lambda L_T T}).
\end{equation}
 
\end{proof}

Note that the rate depends crucially on $L_T$, which is the smallest averaged-gradient norm computed during the optimization.
Now let us discuss when the boundary FW works; namely, the condition that causes the cumulative gradient being nonzero. If a linear combination of gradients is $\mathbf 0$ then clearly $\mathbf 0$ is in the convex hull of subgradients $\partial f(x)$ for boundary points $x$. Since
the closure of $\{ \nabla f(x)| x \in \K \}$ is convex, according to Theorem~\ref{th:cvx}, this implies that $\mathbf 0$ is in $\{ \nabla f(x)| x \in \K \}$. If we know in advance that $\mathbf 0 \notin \text{closure}(\{ \nabla f(x)| x \in \K \})$ we are assured that the cumulative gradient will not be $\mathbf 0$. Hence, the proposed algorithm may only be useful when it is known, a priori, that the solution $x^*$ will occur not in the interior but on the boundary of $\K$. It is indeed an odd condition, but it does hold in many typical scenarios.
One may add a perturbed vector to the gradient and show that with high probability, $L_T$ is a non-zero number. The downside of this approach is that it would generally grant a slower convergence rate; it cannot achieve $\log(T)/T$ as the inclusion of the perturbation requires managing an additional trade-off.

\subsection{Gauge Frank-Wolfe}

We propose a new \FW like algorithm that not only requires a linear oracle but also enjoys $O(1/T^2)$ rate on all the strongly convex constraint sets that contain the origin,
like $l_p$ ball and Schatten $p$ ball with $p \in (1,2]$.
To describe our algorithm, denote $\K$ be any closed convex set that contains the origin.
Define ``gauge function'' of $\K$ \cite{F87,FMP14} as 
\begin{equation}
\g(x) := \inf \{ c \geq 0: \frac{x}{c} \in \K \}.
\end{equation}
Notice that, for a closed convex $\K$ that contains the origin, 
$\K = \{ x \in \reals^d: \g(x)\leq 1 \}$. Furthermore, the boundary points on $\K$ satisfy $\g(x)=1$.

Next we provide a  characterization of sets based on their gauge function.
\begin{definition}[$\lambda$-Gauge set]
Let $\K$ be a closed convex set which contains the origin.
We say that $\K$  is $\lambda$-Gauge if its squared gauge function, $\g^2(\cdot)$,  is $\lambda$-strongly-convex.
\end{definition}
This property captures a wide class of constraints. Among these are $l_p$ balls, 
Schatten $p$ balls, and the Group $(s,p)$ ball \cite{D15}.
In fact, Theorem~4 in \cite{P96} and Theorem 2 in \cite{M20} show that 
for any centrally symmetric strongly convex set $\K$ that contains the origin, the gauge function $\g^2(\cdot)$ is strongly convex w.r.t. the induced gauge norm $\g(\cdot)$ on $\K$.

We introduce a family of \BTRL algorithms that rely solely on a linear oracle, and we believe this is a novel approach to online linear optimization problems. 
The restriction we require is that the regularizer $R(\cdot)$ is chosen as the \emph{squared gauge function} $\g^2(\cdot)$ for the decision set $\K$ of the learner. Here we will assume\footnote{One can reduce any arbitrary convex loss to the linear loss case
  by convexity $\ell_t(x) - \ell(x^*) \leq \langle \partial f_t(x), x - x^* \rangle$
  (\cite{shalev2012online,RS16}).} for every $t$ that $\ell_t(\cdot) = \langle l_t, \cdot \rangle$ for some vector $l_t$, hence \BTRL (\ref{update:BTRL}) reduces to
\begin{equation}\label{eq:gauge_FTRL}
    x_t = \argmin_{x \in \K} \eta \langle L_{t} , x \rangle + \g^2(x),
\end{equation}
where $L_{t} = l_1 + \ldots + l_{t}$. 
Denote $\text{bndry}(\K)$ as the boundary of the constraint set $\K$.
We can reparameterize the above optimization, by observing that any point $x \in \K$ can be written as $\rho z$ where $z \in \text{bndry}(\K)$, and $\rho \in [0,1]$. Hence we have 
\begin{equation}\label{eq:gauge_FTRL}
     \min_{\rho \in [0,1]} \min_{z \in \text{bndry}(\K)}  \eta \langle L_{t} , \rho z \rangle + \g^2(\rho z)
     ~=~
     \min_{\rho \in [0,1]} \left( \min_{z \in \text{bndry}(\K)} \eta \langle L_{t} , z \rangle \right) \rho  + \rho^2.
\end{equation}
We are able to remove the dependence on the gauge function since it is homogeneous, $\g(\rho x) = |\rho| \g(x)$, and is identically 1 on the boundary of $\K$. The inner minimization reduces to the linear optimization $z^* := \argmin_{z \in \K} \langle L_{t}, z \rangle$, and the optimal $\rho$ is 
\begin{equation}
\rho = \max(0, \min(1, -(\eta/2) \langle L_{t}, z^* \rangle) ).
\end{equation}

\begin{algorithm}[t] 
  \caption{Gauge Frank-Wolfe (smooth convex $f(\cdot)$)}
  \label{alg:gaugeFW}
  \begin{algorithmic}[1]
    \STATE Let $\{\alpha_t = t \}$ be a $T$-length weight sequence.
    \FOR{$t= 1, 2, \dots, T$}
    \STATE  The y-player plays \OFTL: $y_t = \nabla f(\tilde{x}_t)$.
    \STATE  The x-player plays BTRL:
    \STATE  Compute $(\hat{x}_t, \rho_t) = \underset{x \in \K, \rho \in[0,1] }{\arg\min} \sum_{s=1}^t \rho \langle  x, \alpha_s y_s \rangle  +  \frac{1}{\eta} \rho^2$\quad \text{ and } set $x_t = \rho_t \hat{x}_t$.                 
    \ENDFOR
    \STATE Output $\bar{x}_{T} := \frac{ \sum_{s=1}^T \alpha_s x_s  }{ \sum_{s=1}^T \alpha_{t} }$.
  \end{algorithmic}
\end{algorithm}

\begin{theorem} \label{thm:gaugeFW}
Suppose the constraint set $\K$ is a $\lambda$-Gauge set.
Assume that the function $f(\cdot)$ is $L$-smooth convex with respect to the induced gauge norm $\g(\hat{x})$. 
Suppose that the step size $\eta$ satisfies
$\frac{1}{CL} \leq \frac{\eta}{\lambda} \leq \frac{1}{4L}$ for some constant $C \geq 4$.
Then, the output $\xav_T$ of Algorithm~\ref{alg:gaugeFW} satisfies
\[
   \displaystyle f(\xav_T) - \min_{x \in \K} f(x) 
  \leq \frac{2 C L \g^2(x^*)  } {\lambda T^2},
\]
\end{theorem}

\begin{proof}
We have just shown that line 4-5 is due to that the x-player plays
\BTRL
with the squared of the guage function as the regularizer.
So 
Algorithm~\ref{alg:gaugeFW} is an instance of the meta-algorithm,
and we can invoke Corollary~\ref{cor:meta} to obtain the convergence rate.
 
\end{proof}

We want to emphasize again that our analysis does not need the function $f(\cdot)$ to be strongly convex 
to show $O(1/T^2)$ rate. On the other hand, 
\citet{D15} shows the $O(1/T^2)$ rate under the additional assumption that the function is strongly convex.

\subsection{A Fast Parallelizable Projection-Free Algorithm for the Nuclear-Norm-Ball Constraint}

In this section,
we consider smooth convex optimization with a bounded nuclear norm constraint,
\begin{equation} \label{opt:main}
 \min_{W \in \mathcal{NB}_{d_1,d_2}(r) }  f(W),
\end{equation}
where $\mathcal{NB}_{d_1,d_2}(r)$ denotes the nuclear norm ball
in $\reals^{d_1 \times d_2}$ with radius $r$, defined as
\begin{equation}
 \mathcal{NB}_{d_1,d_2}(r) :=
\left\{  W \in \reals^{d_1 \times d_2} : \sum_{i=1}^{ d_1 \wedge d_2 } \sigma_i(W)  \leq r  \right\}
\end{equation}
where $\wedge$ is the $\min$ operator and $\sigma_i(W)$ denotes the $i^{th}$ singular value of $W$.
This optimization problem is an important task for many applications in machine learning and signal processing, including matrix completion and collaborative filtering (e.g. \cite{CR12,SS11,SS05,HJN13}), phase retrieval (e.g. \cite{CESV12}), affine rank minimization problems (e.g. \cite{JMD10,RFP10}), robust PCA (e.g. \cite{CLMW11}  ), multi-task learning (e.g \cite{JY09}), multi-class classification (e.g. \cite{DHM12,ZSY12}), distance metric learning (e.g. \cite{PNJR03,YP12}), kernel matrix learning (e.g. \cite{GA11}), learning polynomial networks (e.g. \cite{LSS14}), and more.
The typically large dimensions $d_1, d_2$ that arise in common ML tasks has led to great interest in the development of a highly efficient method to solve \eqref{opt:main}.

A natural approach to solve \eqref{opt:main} is projected gradient descent (PGD), where one performs alternating gradient updates followed by nuclear norm projections. But this last step, the projection, requires an expensive singular value decomposition (SVD) on each iteration whose complexity scales as $O(d_1 d_2 (d_1 \wedge d_2))$, cubic in the dimension (see e.g. \cite{H14}). 
On the other hand, the Frank-Wolfe method \citep{frank1956algorithm} has a key benefit for dealing with the ball constraint: 
each iteration involves solving a linear optimization oracle (LMO) of the form $\arg\max_{x \in \XX} \langle x, v\rangle$.
When the constraint is the nuclear norm ball, i.e. $\XX = \mathcal{NB}_{d_1,d_2}(r)$, the linear optimization problem reduces to computing the leading singular vector of (the negative of) the gradient matrix $-\nabla f(W)$ (see e.g. \cite{H14,J13}). In practice, one can approximate the top singular vector
efficiently via standard approaches like power iteration
or the Lanzcos algorithm, 
where the complexity is roughly proportional to the number of non-zeros in the input matrix (see e.g. \cite{YTFUC19,CP21}). 
The cost of approximately computing the singular vector is $O(d_1 d_2)$ in the worst case, up to log factors (see e.g. \cite{GL96}).
Therefore, it is observed that the Frank-Wolfe method significantly improves performance compared to PGD or Accelerated PGD \cite{N88,N05} due to its cheap iteration cost (see e.g. \cite{G16,H14}). 

\begin{table*}[t]
\centering
\begin{tabular}{|r | c | c |} \hline
Algorithm                  &  Convergence rate   & \# Computations in iteration $t$ \\ \hline \hline
Projected Gradient Descent     &  $O\left(\frac{L r^2}{T}\right)$  &    $\text{c-SVD}$       \\
Accelerated PGD  & $O\left( \frac{L r^2}{T^2}\right)$ & $\text{c-SVD}$         \\
Frank-Wolfe  &  $O\left(\frac{Lr^2}{T}\right)$   &      $\text{c-LMO}$                \\
this work       &   $\tilde{O}\left(\max\left\{ \frac{L r \log (d_1 + d_2)}{ T^2}, \frac{G}{T} \right\}\right)$        &  $\text{c-MEV} \times t$       \\ \hline 
\end{tabular}
\caption{Comparison of first-order methods for solving Problem~\ref{opt:main}.
The second column (convergence rate)
is the optimization error $f(W_T)- \min_{W \in \mathcal{NB}_{d_1,d_2}(r) } f(W)$ in iteration $t$. Here $r$ is the radius of the ball, $L$ is the smoothness constant of $f(\cdot)$,
and $G$ is the constant that bounds the spectral norm of the gradient (i.e. $\| \nabla f(\cdot) \|_2 \leq G$ for all $W \in \mathcal{NB}_{d_1,d_2}(r)$).
The last column is the cost per iteration, where $\text{c-SVD}$ represents the cost of approximating a singular value decomposition (SVD), which is $O(d_1 d_2 (d_1 \wedge d_2))$ in practice; 
$\text{c-LMO}$ is the cost of approximately solving the linear optimization oracle of Frank-Wolfe,
which is $O(d_1 d_2)$ in practice;
$\text{c-MEV}$ is the cost of computing a matrix exponential-vector product,
which is $O(d_1 d_2)$ in practice. 
The convergence rate of our method is better than the baselines under some reasonable conditions, e.g. when the largest gradient norm satisfies $G \leq  L r^2$ or when the radius $r$ is large.
Please see the main text (Section~\ref{sec:analysis}) for the discussion.
 } \label{table:complexity}
\end{table*}

A major concern of the Frank-Wolfe method is that it has a suboptimal $O(Lr^2/T)$ convergence rate for the smooth convex problems, compared to $O(L r^2/T^2)$ of the accelerated gradient methods (see e.g. \cite{N13}), where $L$ is the smoothness constant.
To deal with this issue, some works have developed Frank-Wolfe-like algorithms that enjoy a better convergence rate. However, these results only apply to the case when the constraint set is a certain convex polytope \cite{BS17,BPZ17,D16b,D16a,S15,BPTW19}, or a strongly convex set \cite{ALLW18,D15} and do not apply to the nuclear norm constraint studied in this work. 
On the other hand, 
\citet{L13} and \citet{L20} developed an $O(1/T)$ lower bound for any algorithm that relies on the linear optimization oracle to generate the iterates for 
a smooth convex problem
on the simplex, which might 
imply the hardness to get a rate beyond $O(1/T)$ for the nuclear norm ball constraint without additional assumptions.
In this work, we propose a projection-free algorithm that enjoys a provably better convergence rate than the $O(Lr^2/T)$ rate of Frank-Wolfe and Projected Gradient Descent (PGD) under some reasonable conditions. 
As can be seen from Table~\ref{table:complexity},
our algorithm has an advantage when the radius $r$ is large.
This improvement helps especially for some applications like matrix completion (e.g. \citet{J10}). 
We will return to this point in details in the later sections.

Our algorithm has an additional advantage: it naturally lends itself to a simple parallelization scheme. 
In each iteration $t$, our algorithm must compute an average $\frac{1}{m_t} \sum_{i=1}^{m_t} \Psi_{u_i}(\cdot)$, 
where $m_t \approx t$ and where $\Psi_{u_{\cdot}}(\cdot)$ is a special oracle that
maps a symmetric matrix to the spectrahedron in a randomized fashion that requires a matrix exponential vector product.
This can be efficiently approximated by the Lanczos method, with complexity in the same ballpark as the Frank-Wolfe LMO which requires matrix-vector products and costs $O(d_1 d_2)$. 
Furthermore, 
by exploiting multiple processing units that are pervasively available in modern machines, the average $\frac{1}{m_t} \sum_{i=1}^{m_t} \Psi_{u_i}(\cdot)$ is embarrassingly parallelizable as each term can be independently and simultaneously computed.

\subsubsection{Preliminaries} \label{sec:pre}


\paragraph{Smooth convex function}
We assume that the optimization problem is $L$-smooth convex
w.r.t. the nuclear norm $\| \cdot \|$.
This means that
$f(\cdot)$ is differentiable everywhere \citep{V19} and that
it has Lipschitz continuous gradient
$\| \nabla f(W) - \nabla f(Z) \|_* \leq L \| W - Z\|$,
where $\| \cdot \|_{*}$ denotes the dual norm which is the spectral norm $\| \cdot \|_2$ \citep{T15}.

\paragraph{Spectrahedron and an associated operator $\Psi_u(\cdot)$} \label{sub:oracle}
We denote $\mathcal{S}_d$ the set of symmetric $d \times d$ matrices 
and denote the \textit{spectrahedron} as
\begin{equation}
\textstyle 
\Delta_{d} := \{ X \in \mathcal{S}_d \text{ }: \text{ }  X \succeq 0, \text{Tr}(X)=1\},
\end{equation}
which is the space of $d \times d$ real positive semi-definite symmetric matrices whose trace equals $1$. For a symmetric matrix $X \in \mathcal{S}_{d}$ with $d := d_1 + d_2$, we will use the following factorization of $X$,
\begin{equation} \label{eq:factor}
\textstyle 
X = 
\begin{bmatrix}  X^{(1)}   &   X^{(2)}  \\ X^{(2)}\text{}^\top & X^{(3)}  \end{bmatrix}, 
\end{equation}
where the dimension of the sub-matrix $X^{(1)}$ is $d_1 \times d_1$, $X^{(2)}$ is $d_1 \times d_2$, and $X^{(3)}$ is $d_2 \times d_2$.

We will need the operator $\Psi_u(\cdot): \mathcal{S}_d \rightarrow \Delta_d$
, defined as
\begin{equation} \label{eq:psi}
\Psi_u(D) := \frac{ \exp(D/2) u u^\top \exp(D/2) }{ u^\top \exp(D) u  } = \frac{v_u v_u^\top}{ \| v_u \|^2 } ,
\end{equation}
where $D$ is a symmetric matrix in $\mathcal{S}_d$, 
$\exp(D)= \sum_{k=0}^{\infty} \frac{1}{k!} D^k$ is its matrix exponential,
and
$v_u := \exp(D/2) u$ with $u \sim \text{Uni}(\mathbb{S}_{d})$, i.e. uniformly sampled from the unit sphere $\mathbb{S}_{d}$. 
From (\ref{eq:psi}), we see that the computation of $\Psi_u(D)$ needs a matrix exponential-vector product, which can be done by Lanczos method
, see e.g. the discussion in Section 3 of \citet{CDST19}.
We also denote  $\bar{\Psi}(\cdot):=\E_u [ \Psi_u(\cdot)]: \mathbb{S}_d \rightarrow \Delta_d$ as follows,
\begin{equation} \label{eq:psibar}
\bar{\Psi}(D)=
\E_u \left[ \frac{ \exp(D/2) u u^\top \exp(D/2) }{ u^\top \exp(D) u  } \right] = \E_u\left[ \frac{v_u v_u^\top}{ \| v_u \|^2 } \right],
\end{equation}
where the expectation is over the random draw from the unit sphere $u \sim \text{Uni}(\mathbb{S}_{d})$.
The function $\bar{\Psi}(D)$ is a gradient of the function $\bar{\psi}(D):= \E_u [ \log ( u^\top \exp(D) u) ]$, 
which is a
continuously twice differentiable, convex spectral function \citep{LS01}.

The spectrahedron operator $\bar{\Psi}(\cdot)$ is the key to a recent breakthrough by \citet{CDST19} in \emph{online learning} literature 
for a randomized sketch of the celebrated Matrix Multiplicative Weight (MMW) algorithm  (\citet{TRK05}, \citet{WK08}, \citet{arora2012multiplicative}).
In the online setting, in each round the online learner plays an action in the spectrahedron $X_t \in \Delta_d$, and the adversary supplies a symmetric matrix $L_t \in \mathcal{S}_{d}$, and the player suffers a loss $\langle X_t, L_t \rangle := \text{tr}(L_t X_t)$. The goal of the online player is to minimize the regret,
$\text{Regret}_T:= \sum_{t=1}^T \langle L_t, X_t \rangle -
\inf_{X \in \Delta_d} \sum_{t=1}^T \langle L_t, X \rangle,$
where the second term is the loss of the best single action in hindsight.
MMW is a celebrated algorithm for this setting, and it has wide applications in machine learning and theoretical computer science.
The update is $X_t = \frac{ \exp( - \eta L_{1:t-1} ) }{ \text{tr}\big( \exp( -\eta L_{1:t-1}) \big)}$, where $L_{1:t-1} := \sum_{s=1}^{t-1}  L_s $.
The update in general is expensive because computing matrix exponential needs an eigen-decomposition whose cost is proportional to the cubic size of the matrix in practice, i.e. $O(d^3)$.
On the other hand,
\citet{CDST19} propose a simple randomized algorithm 
with a cheaper computational cost
that enjoys 
an $O(\sqrt{T})$ \emph{expected} regret guarantee as MMW. The algorithm efficiently updates the action in each round according to
$X_t = \Psi_{u_t}( - \eta  L_{1:t-1})$, which only needs a matrix exponential-vector product and can be efficiently approximated by
Lanczos method that costs only $O(d^2)$ in practice.
Our algorithm adopts this random projection oracle $\Phi_u(\cdot)$.
While it is possible to obtain an $O(\frac{1}{\sqrt{T}})$ convergence rate in expectation by applying the standard online-to-batch conversion (e.g. \cite{S07}) to the algorithm of \citet{CDST19}, we develop a new algorithm 
that has a faster convergence rate in this work.

\subsubsection{Equivalent optimization problem}
We will consider an equivalent optimization problem of (\ref{opt:main}), which is optimizing over the spectrahedron,
\begin{equation} \label{obj:eqv}
\textstyle \underset{X \in \Delta_{d_1+d_2} }{\min}  F_r(X):= f( 2 r X^{(2)}).
\end{equation}
\begin{lemma} \label{lem:eq} (Lemma 1 in \citet{G16} and Lemma~1 in \citet{J10})
Consider $\underset{X \in \Delta_{d_1+d_2} }{\min} F_r(X):= f( 2 r X^{(2)})$. Suppose that
$X \in \Delta_{d_1+d_2}$ satisfies $F_r(X) - F_r(X^*) \leq \epsilon$, where $X^* \in \Delta_{d_1+d_2}$ is a minimizer of $F_r(\cdot)$ over $\Delta_{d_1+d_2}$. Then, 
\begin{equation}
f(2 r X^{(2)}) - \min_{X \in \mathcal{NB}_{d_1,d_2}(r) }  f(X) \leq \epsilon
\end{equation}
 and that $2 r X^{(2)} \in \mathcal{NB}_{d_1,d_2}(r) $.
\end{lemma}

\noindent
\noindent
\textbf{Remark 1:} Lemma~\ref{lem:eq} shows the equivalency between problem (\ref{opt:main}) and (\ref{obj:eqv}). Solving (\ref{obj:eqv}) over the spectrahedron $\Delta_{d_1+d_2}$ is equivalent to solving problem (\ref{opt:main}).
Note that for $X \in \Delta_{d_1+d_2}$ the gradient $\nabla F_r(X)\in \mathcal{S}_{d_1+d_2}$ is given by
\begin{equation} \label{eq:grad}
\nabla F_r(X) = \begin{bmatrix} 0_{d_1 \times d_1} & \nabla f ( 2r X^{(2)}) \\
\nabla f ( 2r X^{(2)})^\top & 0_{d_2 \times d_2}
   \end{bmatrix}.
\end{equation}
In the following, we will denote the dimension $d := d_1 + d_2$.

\noindent
\textbf{Remark 2:}
We will assume that $F_r(\cdot)$ is $\hat{L}$-smooth w.r.t. the nuclear norm $\| \cdot \|$ over $\Delta_d$. 
Suppose that the original function $f(\cdot)$ is $L$-smooth over the nuclear norm ball $\mathcal{NB}_{d_1,d_2}(r)$.
Let us discuss the relation between the smoothness constant $\hat{L}$ of $F_r(\cdot)$ and $L$ of $f(\cdot)$. 
Denote $M_1:= \begin{bmatrix} I_{d_1} , 0_{d_1 \times d_2} \end{bmatrix}$ and $M_2 = \begin{bmatrix} 0_{d_1 \times d_2}\\ I_{d_2} \end{bmatrix}$.
We have that
\begin{equation} \label{eq:Lr}
\begin{aligned}
& \textstyle \| \nabla F_r(X) - \nabla F_r(Z) \|_* 
= \| \nabla f(2 r X^{(2)}) - \nabla f(2 r Z^{(2)}) \|_*
\\ & \textstyle 
\leq 2 r L \|  X^{(2)} -  Z^{(2)} \|
= 2 r L \| M_1 X M_2 - M_1 Z M_2 \| 
\\ & \textstyle 
\leq 2 r L \| X - Z \|,
\end{aligned}
\end{equation}
where we use the fact that the nuclear norm of a matrix product  $\| AB \|$ satisfies $\| A B \| \leq \sigma_{\max}(A) \| B \|$, see e.g. \citet{HCH20}, and that the largest singular values $\sigma_{\max}(M_1)=\sigma_{\max}(M_2) = 1$.
So we see that $\hat{L} \leq 2 r L$.

\subsubsection{Main result} \label{sec:analysis}

Algorithm~\ref{alg:main} shows the proposed algorithm. 
Similar to those accelerated gradient methods (e.g. \citet{N05,LZ18}),
it maintains two interleaving sequences $\{ W_t \}$ and $\{ Z_t \}$ such that the gradient is computed at an auxiliary variable $Z_t$ instead of the primary variable $W_t$ (line 5). Furthermore, like the Frank-Wolfe method, it has the steps of the ``convex averaging'' so that the iterate $W_t$ and $Z_t$ are always in the constraint set (line 4 and~7), as the outputs from the oracle (line 6) and the initial points (line 2) are all in the constraint set. We remark that line 6 is where the parallelization can be exploited for the oracle calls.


\begin{algorithm}[t] 
   \caption{ Proposed algorithm for solving (\ref{opt:main})  } \label{alg:main}
\begin{algorithmic}[1]
\STATE Set parameter $\delta>0$, $\beta_t = \frac{2}{t+1}$, $\eta \leq \frac{1}{36\hat{L}}$, and $m_t = \max \{ \lceil \log( 4 d  / \delta) \rceil , t\}$. 
\STATE Init: $W_{0}=X_{0} \in \Delta_{d} $ and $G_0 = 0_{d \times d}$ with $d = d_1 + d_2$.
\FOR{$t= 1, 2, \dots, T$}
\STATE $Z_{t} = (1 - \beta_t) W_{t-1} + \beta_t X_{t-1}$.
\STATE $G_t = G_{t-1} - \eta t \nabla F_r(Z_t)$ 
$ = G_{t-1} - \eta t \times \begin{bmatrix} 0_{d_1 \times d_1} & \nabla f( 2 r Z_t^{(2)} ) \\ \nabla f( 2 r Z_t^{(2)} )^\top & 0_{d_2 \times d_2}    \end{bmatrix}$.
\STATE $X_t = \frac{1}{m_t} \sum_{j_t=1}^{m_t} \Psi_{u_{j_t}}(G_t)$, \text{ where  each }  $u_{j_t} \sim \text{Uni}( \mathbb{S}^{d} )$. \hfill \textbf{*easily done in parallel*}
\STATE $W_{t} = (1 - \beta_t) W_{t-1} + \beta_t X_{t}$.
\ENDFOR
\STATE Output $2 r W_{T}^{(2)} \in \mathcal{NB}_{d_1,d_2}(r)$.
\end{algorithmic}
\end{algorithm}

\subsubsection{Convergence rate} \label{sub:thm}


\begin{theorem} \label{thm:conv}
Suppose that the function $F_r(\cdot)$ on (\ref{obj:eqv}) is $\hat{L}$-smooth with respect to the nuclear norm over the spectrahedron $\Delta_d$
and that the gradient norm of $f(\cdot)$ satisfies $\| \nabla f(\cdot) \|_2 \leq G$ over the ball
$\mathcal{NB}_{d_1,d_2}(r)$.
If $\eta \leq \frac{1}{36 \hat{L}}$ and $\forall t, m_t \geq \log( 4 d  / \delta)$, then
with a constant probability $1-\delta$, the output $2 r W_{T}^{(2)} \in \mathcal{NB}_{d_1,d_2}(r)$ of Algorithm~\ref{alg:main} satisfies
\begin{equation} \label{eq:thm}
\begin{aligned}
 f( 2 r W_{T}^{(2)} ) - \min_{X \in \mathcal{NB}_{d_1,d_2}(r) }  f(X) 
 \leq
O\left( \frac{ \hat{L} \log (d) }{T^2} \right) 
 + 
\frac{c_1}{T^2}
\sum_{t=1}^{T}  \frac{1}{m_t} +  \frac{c_2}{T^2}\sqrt{ \sum_{t=1}^T \frac{t^2}{m_t}  }, 
\end{aligned}
\end{equation}
where $c_1 :=  192 \hat{L} \log\frac{4d}{\delta} $ and $c_2 := \sqrt{ 448 G^2 \log\frac{2}{\delta}}$.
\end{theorem}
In Section~\ref{sub:design} and Section~\ref{app:thm} we will provide the analysis and the proof.
We note that the convergence rate in Theorem~\ref{thm:conv} has two components ---
one is a fast rate term $O( \frac{\hat{L}\log d}{T^2}) = O( \frac{Lr\log d}{T^2})$, while the other is controlled by the number of oracle calls $m_t$. 
\begin{corollary} \label{cor}
Under the same setup as Theorem~\ref{thm:conv},
if the number of oracle calls $m_t$ to construct $x_t$ is $m_t = \max \{ \lceil \log( 4 d  / \delta) \rceil , t\}$, 
then
with a constant probability $1-\delta$, 
the output $2 r W_{T}^{(2)} \in \mathcal{NB}_{d_1,d_2}(r)$ of Algorithm~\ref{alg:main} satisfies
\begin{equation} \label{rate:ours}
\begin{split}
f( 2 r W_{T}^{(2)} ) - \min_{X \in \mathcal{NB}_{d_1,d_2}(r) }  f(X)
\leq
 O\left( \frac{ \hat{L}  \log ( d  T / \delta ) }{ T^2 }
+ \frac{G \sqrt{\log(1/\delta) } }{T} \right) .
\end{split}
\end{equation}
\end{corollary}

\subsubsection{Comparison of the convergence rates}
Let us compare the convergence rates in the literature.
The convergence rate of Frank-Wolfe
for general smooth convex problem with a convex constraint set $\XX$ is
\begin{equation} \label{rate:FW}
f(X_T) - \min_{X \in \XX} f(X) \leq \frac{2 C_f}{T}, 
\end{equation}
where $C_f$ is defined as
$\textstyle
C_f := \underset{ 
\Omega
}{\sup}
\frac{1}{\theta^2} \big( f(Z') - f(Z) + \langle Z' - Z, \nabla f(Z) \rangle \big)$ with the constraint $\Omega:=\{ Z, V \in \XX, \theta \in R , Z' = Z + \theta (V - Z) \} $ (see e.g. \citet{K08,J10}).
The constant $C_f$ can be upper-bounded by 
\begin{equation} \label{cf:1}
C_f 
\leq
\sup_{ Z, V \in \XX } L \| Z - V \|^2,
\end{equation}
since smoothness implies that 
$f(Z') - f(Z) + \langle Z' - Z, \nabla f(Z) \rangle \leq L \| Z' - Z \|$ \citep{V19}.
On the other hand, projected gradient descent (PGD) is known 
to have: 
\begin{equation} \label{rate:proj}
f(X_T) - \min_{X \in \XX} f(X) \leq \frac{2 L \| X_0 - X_* \|^2}{T},
\end{equation}
where $X_*$ is one of the minimizers of $f(X)$ (see e.g. Section 3.2 of \cite{BB15}). 
Furthermore, Accelerated PGD (e.g. Nesterov's method \citet{N13,T08}) has convergence rate, 
\begin{equation} \label{rate:projacc}
f(X_T) - \min_{X \in \XX} f(X) \leq \frac{2 L \| X_0 - X_* \|^2}{T^2}.
\end{equation}

Now let us identify conditions such that
the rate of Algorithm~\ref{alg:main} stated  in Corollary~\ref{cor} is better than the baselines.
We consider two cases, which are \textit{(A):} $O( \frac{\hat{L}\log(dT/ \delta)}{T^2})$ being the dominant term (i.e. slower one) of the convergence rate (\ref{rate:ours}) and \textit{(B):} $O( \frac{G \sqrt{\log(1/\delta)}}{T})$ being the dominant term of the rate (\ref{rate:ours}).
The first case happens when $\hat{L}$ is large, i.e. $\hat{L}\gg G$, and 
$T$ is not too large.
In this case, our algorithm has a fast rate of $O(1/T^2)$ which matches that of Accelerated PGD, while Frank-Wolfe and PGD have the slow rate $O(1/T)$. 

Moreover, the convergence rate of Algorithm~\ref{alg:main} 
is better when $r$ is large. 
Specifically, for a fixed $T$ and $\delta$, the constant factor
of the convergence rate of Algorithm~\ref{alg:main}, i.e.
$\hat{L} \log ( d T/ \delta ) \leq 2 L r \log ( d T/ \delta )$,
 can be smaller than $4 L r^2$ of
Frank-Wolfe (\ref{cf:1}), of PGD (\ref{rate:proj}), and of Accelerated PGD (\ref{rate:projacc}).
This happens when the radius $r$ is large.
For example, in the experiments of \citet{J10}, to obtain a good testing performance for solving the matrix completion problem, the authors set the radius of the nuclear norm ball to be $r=4988$ for the MovieLens 100k dataset (a $943$ by $ 1682$ user-rating matrix) and $r=18080$ for the MovieLens 1m dataset (a $6040$ by $3706$ user-rating matrix). On the other hand, the logarithm of dimension $d=d_1 +d_2$ is $\log d = \log (943 + 1682) = 7.87$
for MovieLens 100k and $\log d = 9.18$ for MovieLens 1m. Therefore, 
the constants $2 Lr \log (d T / \delta)$ and $4 L r^2 $ are in significantly different scales, which suggests that our algorithm can have a better performance over
 the baselines.

Now let us switch to the case when $O( \frac{G \sqrt{\log(1/\delta)}}{T})$ is the dominant term.
In this case, Algorithm~\ref{alg:main}, Frank-Wolfe, and PGD all have the same $O(1/T)$ rate. However, they depend on different constants. 
Consider a point $\hat{W}$ in the ball $\mathcal{NB}_{d_1,d_2}(r)$ whose gradient norm is the \emph{smallest} one among the points in the ball.
Then,  
$G:= \max_{W \in \mathcal{NB}_{d_1,d_2}(r) } \| \nabla f(W) \|_2 = \max_{W \in  \mathcal{NB}_{d_1,d_2}(r) } \| \nabla f(W) - \nabla f(\hat{W}) + \nabla f(\hat{W}) \|_2 \leq L \| W - \hat{W} \|_2 + \| \nabla f(\hat{W}) \|_2\leq 2 Lr + \| \nabla f(\hat{W}) \|_2$. Hence,
if the smallest gradient norm satisfies $\| \nabla f(\hat{W}) \|_2 \leq 2 r L$,
then $G \leq 4 L r$ is smaller than $4 L r^2$ of Frank-Wolfe and PGD when $r > 1$.


\subsubsection{Analysis of the computational cost} \label{sub:comp}

In this subsection, we analyze the computational cost 
required for Algorithm~\ref{alg:main} to reach at a point whose function value is $\epsilon$-close to the optimal value of $(\ref{opt:main})$ and compare it with Frank-Wolfe. Corollary 1 states that Algorithm~\ref{alg:main} needs $T= \tilde{O}(\max\{ \sqrt{\frac{\hat{L} \log d}{\epsilon}}, \frac{G}{\epsilon}\} )$ iterations to reach
$f(2 r W_T^{(2)}) - \min_{X \in \mathcal{NB}_{d_1,d_2}(r) }  f(X)  \leq \epsilon$.
As Algorithm~\ref{alg:main} describes, it needs computing $\frac{1}{m_t} \sum_{j_t=1}^{m_t} \Psi_{u_{j_t}}(G_t)$ in each iteration $t$. 
Each call to the oracle $\Psi_{u_{\cdot}}(\cdot)$ requires a matrix exponential-vector product (recall the definition in Section~\ref{sec:pre}).
The matrix exponential-vector product can be efficiently approximated by Lanczos method in $O(d_1 d_2)$ time, see e.g. the discussion in Section 3 of \citet{CDST19} or \citet{MMS18}.
In our algorithm, the number of calls to the oracle $\Psi_{u_{\cdot}}(\cdot)$ in each iteration grows linearly with iteration $t$. The total number of oracle calls, and hence the total number of matrix exponential-vector products during the execution of the algorithm $\sum_{t=1}^T m_t$ is
\begin{equation} \label{eq:tn}
\begin{aligned}
& 
\sum_{t=1}^T \max \{ \lceil \log( \frac{4 d}{ \delta}) \rceil , t\}
 =  \lceil \log( \frac{4 d}{ \delta} ) \rceil T  + \frac{T(T+1)}{2} 
\\ & =  O\left(\max\left\{\frac{\hat{L} \log (d) }{\epsilon}, \frac{G^2}{\epsilon^2} \right\}\right),
\end{aligned}
\end{equation}
if $\log( 4 d  / \delta) \leq T$.
Now let us compare this number with that of the Frank-Wolfe method.
The Frank-Wolfe method needs $T=O(\frac{L r^2}{\epsilon})$ iterations to achieve an $\epsilon$ error. In each iteration, it makes one linear optimization oracle call for computing the top singular vector of a gradient matrix \citep{H14}.
Therefore, the total number of oracle calls and hence the total number of top singular vector computations by Frank-Wolfe is $O(\frac{L r^2}{\epsilon})$.
The top singular vector can be efficiently approximated by power iteration or by the Lanczos method, and the cost is in the order of $O(d_1d_2)$. 
So the cost of a single call to our oracle and the cost of a single call of that of Frank-Wolfe is similar.
Algorithm~\ref{alg:main} makes fewer number of oracle calls than Frank-Wolfe if 
$O(\max\{\frac{\hat{L} \log (d) }{\epsilon}, \frac{G^2}{\epsilon^2} \}) < O(\frac{L r^2}{\epsilon})$, which holds when $r$ is large as discussed. 
On the other hand, even if Algorithm~\ref{alg:main} needs more number of oracle calls, its actual running time can be better than that of Frank-Wolfe due to the fact that calls to its oracle are embarrassingly easy to be parallelized.

Modern computational resources have multi-cores or multiple processing units, which enables conducting a task in a parallel fashion. Our algorithm can immediately benefit from parallel computing.
Observe that parallelizing Step 6 of Algorithm~\ref{alg:main}, $X_t = \frac{1}{m} \sum_{j_t=1}^m \Psi_{u_{j_t}}(G_t)$, is embarrassingly easy --- simply let each worker of the machine independently and simultaneously compute some $\Psi_{u_{j_t}}(G_t)$ in parallel.
As the result, the actual time spent in computing $X_t$ can be significantly reduced by the parallel computing.
Since the actual running time is the number of iterations times the cost (computational time) per iteration,
mathematically speaking, the actual running time is, 
\begin{equation} \label{eq:time}
\begin{split}
O( \max\{ \frac{\hat{L} \log (d_1 + d_2)}{\epsilon}, \frac{G^2}{\epsilon^2} \}) \times \frac{ O(d_1 d_2)}{M},
\end{split}
\end{equation}
where $M\geq 1$ represents a factor of reduction due to the parallelization of the calls and could be viewed as the ``effective'' number of workers in a machine. 
Hence, 
the effective computational time (\ref{eq:time}) of Algorithm~\ref{alg:main} can be better than 
Frank-Wolfe, which is $O(\frac{L r^2}{\epsilon}) \times \text{c-LMO} =O(\frac{L r^2}{\epsilon}) \times O( d_1 d_2)$.
On the other hand, it is not clear if parallelizing the calls to the linear optimization oracle of Frank-Wolfe is feasible and we are not aware of any works in this direction.

\subsubsection{Algorithm design} \label{sub:design}

Let us first consider an instance of Algorithm~\ref{alg:game} 
by setting the weight $\alpha_t=t$ and let the function $F(\cdot)$ in the definition of the payoff function (\ref{eq:fenchelgame}) of the game be  $F(\cdot)=  F_r(\cdot)$ (\ref{obj:eqv}). Furthermore, let the online learning algorithms
$\alg^y$ and $\alg^x$ in the game respectively as:
\begin{eqnarray}
y_t  & \leftarrow & \argmin_{y} \left\{ \alpha_t \ell_{t-1}(y) +  \sum_{s=1}^{t-1} \alpha_s \ell_s(y)\right\} 
\\
& = & \argmax_{y}  \left \langle \xof_t, y \right \rangle - F_r^*(y) = \nabla F_r(\xof_{t}), \label{eq:yalg}
\end{eqnarray}
where we denote 
$\textstyle \xof_t := \textstyle \frac{1}{A_t}(\alpha_t x_{t-1} + \sum_{s=1}^{t-1} \alpha_s x_s ),$
and
\begin{eqnarray}
\hat{x}_t 
& \leftarrow & \bar{\Psi}( - \eta \sum_{s=1}^t \alpha_s y_s ) 
= \bar{\Psi}( - \eta \sum_{s=1}^t \alpha_s \nabla F_r(\xof_s) ). \label{eq:xalg-hat}
\end{eqnarray}
As we have seen earlier in this chapter, the strategy (\ref{eq:yalg}) is \OFTL, while the strategy (\ref{eq:xalg-hat})
can be viewed as a variant of the dual averaging strategy (see e.g. \citet{X10}) but with the learner being prescient, i.e. knows the loss function of the current round before playing an action.
By choosing the parameter $\eta$ and the weighting scheme $\{ \alpha_t\}$ appropriately, we can show that the weighted regret is
$O( \frac{\hat{L} \log d}{T^2})$. 
Hence, it will lead to an
$ O( \frac{\hat{L} \log d}{T^2})$ algorithm for solving the underlying problem
$\textstyle \underset{X \in \Delta_{d_1+d_2} }{\min}  F_r(X):= f( 2 r X^{(2)}).
$

However, the strategy (\ref{eq:xalg-hat}) involves computing the expectation, 
$\bar{\Psi}(\cdot)$, which is hard to achieve in practice. 
So we propose an unbiased version of it, 
\begin{equation} \label{eq:xalgS}
\begin{aligned}
x_t &  \leftarrow   \frac{1}{m_t} \sum_{j_t=1}^{m_t} x_{t,j_t} := \frac{1}{m_t} \sum_{j_t=1}^{m_t}  \Psi_{u_{j_t}} (-\eta \sum_{s=1}^t \alpha_s y_s ), 
\end{aligned}
\end{equation}
where each $u_{j_t} \sim \text{Uni}(\mathbb{S}_d)$.

\begin{theorem}  \label{thm:equv}
Algorithm~\ref{alg:main} is exactly equivalent to Algorithm~\ref{alg:game}
if $\alpha_t=t$, $F(\cdot) \leftarrow F_r(\cdot)$, and the y-player plays according to (\ref{eq:yalg}),
while the x-player plays according to (\ref{eq:xalgS}).
Specifically, there is a following correspondence:
$W_t = \bar{x}_t$, $X_t = x_t$, and $Z_t = \tilde{x}_t$,
given the same initialization $X_0=x_0 = \tilde{x}_1$.
\end{theorem}

We defer the proof of Theorem~\ref{thm:equv} to Section~\ref{app:thm:equv}.
By Theorem~\ref{thm:meta} and~\ref{thm:equv},
to prove Theorem~\ref{thm:conv},
it suffices to upper-bound the sum of weighted regrets of both players in the game when the x-player
plays according to
(\ref{eq:xalgS}) and the y-player outputs (\ref{eq:yalg}).
The proof of Theorem~\ref{thm:conv} is available in Section~\ref{app:thmFW}, which follows the above discussion but has to deal with the case that the x-player plays according to
(\ref{eq:xalgS}) instead of the expected one (\ref{eq:xalg-hat}).

\subsubsection{Related works}

There have been growing research works for projection-free algorithms in recent years (e.g. \cite{LZ16,GPL16,TJNO20,CP20,B15,GPL16,YFLC19,DOSSS20,VAHC20,CHHK18,Netal20}).
When the underlying function is strongly convex or satisfies a notion called quadratic growth in addition to being smooth, there are Frank-Wolfe-like algorithms for the nuclear norm constraint that achieve a better rate than the original Frank-Wolfe method, with a less expensive cost than that of a full-rank SVD (e.g. \cite{ZHHL17,G16,DFXY20}). However, it is unclear if the algorithms 
still have the benefit when the function is only smooth but not strongly convex. For the same problem in this work, \citet{G19} show that 
with a warm-start initialization, each iteration of PGD or Accelerated PGD does not need the full-rank SVD computations but a low-rank SVD instead 
under certain conditions. 
\citet{DK20} propose an efficient implementation of Matrix Multiplicative Weight Algorithm \cite{TRW05} that avoids a full-rank eigen-decomposition under certain conditions and enjoys a $O(1/t)$ local convergence rate from a warm-start initialization for 
the spectrahedron constraint.
In this work, we aim at developing an algorithm that 
avoids SVD computations 
while achieves a better convergence rate 
over the nuclear norm ball without the assumption of the strong convexity nor the need of a warm-start initialization.

We notice that in the literature, there are some works suggesting some efforts to \emph{parallelize} the Frank-Wolfe method (e.g. \cite{WSDNSX16,ZZZHZ17,ZBG18,WTZ20}). We want to emphasize that these works are fundamentally different from ours. \citet{WSDNSX16} consider a setting wherein the linear optimization problem of Frank-Wolfe can be decomposed into several smaller ones due to a property called ``block-separable'' of the variables, which is present in the dual form of structural SVM \citep{SJM13}, and propose solving them in parallel. The block-separable property does not hold in our problem. 
\citet{ZZZHZ17,WTZ20} consider a setting that there is a network of learners and each learner
commits an action in each round according to the Online Frank-Wolfe method \citep{HK12}. The goal is to minimize the sum of regrets of all the learners. So the goal and the notion of parallelization is different from ours. 
\citet{ZBG18} consider exploiting parallel computing to parallelize the computations of matrix-vector products inside the power iteration, i.e. linear optimization oracle, of Frank-Wolfe for solving problem (\ref{opt:main}), while our work deals with parallelizing the calls to the proposed oracle so that the calls can be made simultaneously. The parallelization is used on a different level.
In particular, one can parallelize the internal computations (e.g. matrix-vector products, summations) of computing a single $\Psi_{u_{\cdot}}(\cdot)$ of ours as well. But it is tricky to parallelize the calls to the linear optimization oracle of Frank-Wolfe.
The notion of parallelization is different and complementary.

%
%


\section{Detailed proofs}

\subsection{Proof of Theorem~\ref{th:cvx} } \label{app:show_cvx}

\begin{proof}
This is a result of the following lemmas.\\

\noindent
\textbf{Definition:} \textit{[Definition 12.1 in \citet{R98}] A mapping $T$ : $\mathcal{R}^n \rightarrow \mathcal{R}^n$ is called monotone if it has the property that }
\[
\langle v_1 - v_0, x_1 - x_0 \rangle \geq 0 \text{ whenever } v_0 \in T(x_0) , v_1 \in T(x_1).
\]
\textit{Moreover, $T$ is maximal monotone if there is no monotone operator that properly contains it.}\\
\noindent
\textbf{Lemma 2:} \textit{[Theorem 12.17 in \citet{R98}] For a proper, lsc, convex function $f$, $\partial f$ is a maximum monotone operator.}\\
\noindent
\textbf{Lemma 3:} \textit{[Theorem 12.41 in \citet{R98}] For any maximal monotone mapping $T$, the set ``domain of $T$`` is nearly convex, in the sense that there is a convex set $C$ such that $ C \subset \text{domain of } $T$ \subset cl(C) $. The same applies to the range of $T$.}

Therefore, the closure of $\{ \partial f(x)| x \in \XX ) \}$ is also convex, because we can define another proper, lsc, convex function $\hat{f}(x)$ such that it is $\hat{f}(x) = f(x)$ if $x \in \XX$; otherwise, $\hat{f}(x) = \infty$. Then, the sub-differential of $\hat{f}(x)$ is equal to $\{ \partial f(x)| x \in \XX \}$. So, we can apply the the lemmas to get the result.

\end{proof}

\subsection{Proof of Theorem~\ref{thm:linearFW} } \label{app:linearFW}

Note that in the algorithm that we describe below the weights $\alpha_t$ are not predefined but rather depend on the queries of the algorithm. These adaptive weights are explicitly defined in Algorithm~\ref{alg:SC-AFTL} which is used by the $y$-player.
Note that Algorithm~\ref{alg:SC-AFTL} is equivalent to performing FTL updates over the following loss sequence:
 $\left\{\tilde{\ell}_t(y) :=\alpha_t \ell_t(y) \right\}_{t=1}^T.$ 
The $x$-player plays best response, which only involves the linear optimization oracle.

\begin{algorithm}[h] 
  \caption{\small Strongly-Convex Adaptive Follow-the-Leader (SC-AFTL)}
  \label{alg:SC-AFTL}
  \begin{algorithmic}[1]
  \FOR{$t= 1, 2, \dots, T$}
    \STATE Play $y_t \in \YY$
    \STATE Receive a strongly convex loss function $\alpha_t \ell_{t}(\cdot)$ with $\alpha_t = \frac{1}{\| \nabla \ell_t(y_t) \|^2} $.
    \STATE Update $y_{t+1} = \min_{y\in \YY} \sum_{s=1}^t \alpha_y \ell_s(y) $  
   \ENDFOR 
  \end{algorithmic}
\end{algorithm}

\begin{proof}
Since the $x$-player plays best response, $\regret{x}=0$, we only need to show that the y-player's regret satisfies
$\regret{y} \leq O(\exp(-\frac{\lambda B}{L} T))$, which we do next.

We start by defining a function $s(y) := \max_{x \in \XX} - x^\top y + f^*(y)$ is a strongly convex function.
We are going to show that $s(\cdot)$ is also smooth.
We have that 
\begin{equation}
\begin{aligned}
& \| \nabla_w s(\cdot) - \nabla_z s(\cdot) \| = 
\| \arg\max_{x \in \XX} ( - w^\top w + f^*(w) ) -  \arg\max_{x \in \XX} ( - z^\top x + f^*(z) ) \|
\\ &= \| \arg\max_{x \in \XX} ( - w^\top x) - (\arg\max_{x' \in \XX}  - z^\top x') \|
\leq  \frac{ 2 \| w - z \|}{ \lambda ( \| w\| + \| z\| ) }
\leq  \frac{ \| w - z \|}{ \lambda B},
\end{aligned}
\end{equation}
where the second to last inequality uses Lemma~\ref{lm:lip} regarding $\lambda$-strongly convex sets,
and the last inequality is by assuming the gradient of $\| \nabla f(\cdot)\| \geq B$
and the fact that $w,z \in \YY$ are gradients of $f$.
This shows that $s(\cdot)$ is a smooth function with smoothness constant $L':=\frac{1}{\lambda B}$.
\begin{equation}
\begin{aligned}
T &= \sum_{t=1}^T \frac{\| \nabla \ell_t(y_t)^2 \|}{\| \nabla \ell_t(y_t)\|^2}
\overset{Proposition~\ref{sameGrad}}{ = }\sum_{t=1}^T \frac{\| \nabla s(y_t)\|^2}{\| \nabla \ell_t(y_t) \|^2}
\overset{Lemma~\ref{lem:GSmooth}}{ \leq} \sum_{t=1}^T \frac{L'}{\| \nabla \ell_t(y_t)\| ^2} (s(y_t) - s(y^*))\\
&\le \sum_{t=1}^T \frac{L'}{\| \nabla \ell_t(x_t)\|^2}  (\ell_t(y_t) - \ell_t(y^*))\label{eq:normal},
\end{aligned}
\end{equation}
where we denote $y^* := \arg\min_y s(y) $ and the last inequality follows from the fact that $s(y_t) := \ell_t(y_t)$ and $\ell_t(y) = - g(x_t,y) \leq  - g(x_y , y) = s(y)$ for any $y$. 

In the following, we will denote $c$ a constant such that $\| \nabla \ell_t(y_t) \| = \| x_t - \nabla f^*(y_t) \| = \| x_t - \bar{x}_{t-1} \|\leq c$.
We have
\begin{equation} \label{eq:ExpRate}
\begin{aligned}
T & \leq
\sum_{t=1}^T \frac{{ L' }}{\| \ell_t( y_t) \|^2} ( \ell_t(y_t)-\ell_t(y^*) )  \nonumber \\
&\overset{(a)}{=}
\sum_{t=1}^T L' ( \tilde{\ell}_t(y_t)-\tilde{\ell}_t(y^*) )  \nonumber \\
&\overset{(b)}{\le}
\frac{L\cdot L'}{2} \sum_{t=1}^T \frac{\| \nabla \ell_t( y_t) \|^{-2}}{\sum_{s=1}^t \|\nabla \ell_s(y_t) \|^{-2}}\nonumber \\
&\overset{(c)}{\le}
\frac{L\cdot L'}{2}\left( 1+\log(c^2 \sum_{t=1}^T\|\nabla \ell_t( y_t)\|^{-2}) \right)~,
\end{aligned}
\end{equation}
where (a) is by the definition of $\tilde{\ell}_t(\cdot)$,
and (b) is shown using Lemma~\ref{regret:FTL} with strong convexity parameter of $\ell_t(\cdot)$ being $\frac{1}{L}$, and (c) is by Lemma~\ref{lem:Log_sum} so that
\begin{align*}
\sum_{t=1}^T \frac{\| \ell_t( y_t) \|^{-2}}{\sum_{s=1}^t \| \ell_s( y_s) \|^{-2}}
&=
\sum_{t=1}^T \frac{c^2\| \ell_t( y_t) \|^{-2}}{\sum_{s=1}^t c^2\| \ell_s( y_s) \|^{-2}}
\leq{}
{1+\log(c^2 \sum_{t=1}^T\| \ell_t( y_t)\|^{-2})}.
\end{align*}
Thus, we get
\begin{align} \label{eq:exp}
c^2 \sum_{t=1}^T \| \nabla \ell_t(y_t) \|^{-2} = O(  e^{\frac{1}{L\cdot L'}T}) =  O(  e^{\frac{\lambda B}{L}T}).
\end{align}

\begin{equation}
\begin{aligned}
& \frac{ \regret{y} }{  A_T } :=
\frac{\sum_{t=1}^T \alpha_t(\ell_t (y_t) - \ell_t(y^*))}{A_T} 
\leq \frac{L}{2 A_T} \sum_{t=1}^T \frac{\|\nabla \ell_t(y_t)\|^{-2}}{\sum_{\tau = 1}^t \| \nabla \ell_t(y_\tau)\|^{-2}}\\
&\overset{(a)}{\le} \frac{L c^2\pr{1 + \log \pr{ c^2 \sum_{t=1}^T \| \nabla \ell_t(y_t)\|^{-2}}}}{2 c^2 \sum_{t=1}^{T} \|\nabla \ell_t(y_t)\|^{-2}}
\overset{(b)}{\le} O( \frac{ L c^2\pr{1 + \pr{\frac{\lambda B T}{L}}}}{ e^{\frac{\lambda B}{L}T}} ) = O\pr{ L c^2 e^{-\frac{\lambda B}{L}T}}
\end{aligned}
\end{equation}
where $(a)$ is by Lemma~\ref{lem:Log_sum}, $(b)$ is by (\ref{eq:exp}) and the fact that $\frac{1+\log z}{z}$ is monotonically decreasing for $z \ge 1$. This completes the proof.

\end{proof}

\begin{proposition} \label{sameGrad}
For arbitrary $y$, let $\ell(\cdot) := - g(x_y, \cdot )$. Then $- \nabla_{y} \ell(\cdot) \in \partial_{y} s(\cdot)$, where $x_y$ means that the x-player plays $x$ by \BR after observing the y-player plays $y$.
\end{proposition}
\begin{proof}
Consider any point $w \in \YY$,
\begin{equation}
\begin{aligned}
 s(w) - s(y) & = g( x_y, y) - g(x_w,w)   
\\ & = g(x_y, y) - g(x_y, w) + g(x_y, w) -  g(x_w, w)  
\geq   g(x_y, y) - g(x_y, w) + 0
\\ & \geq  \langle \partial_{y} g( x_y, y) , w - y\rangle = \langle - \nabla_{y} \ell(y) , w - y \rangle
\end{aligned}
\end{equation}
where the first inequality is because that $x_w$ is the best response to $w$, the second inequality is due to the concavity of $g(x_y, \cdot)$.
The overall statement implies that $-\nabla_{y} \ell(y)$ is a subgradient of $s$ at $y$.
\end{proof}

\begin{lemma} \label{lem:GSmooth}
For any $L$-smooth convex function $\ell(\cdot):\reals^d \mapsto \reals$, if $x^* =\argmin_{x\in \reals^d} \ell(x)$, then 
$$ \| \nabla \ell(x)\|^2 \le 2 L \left( \ell(x) - \ell(x^*)\right), \quad \forall x\in \reals^d~.$$
\end{lemma} 

\begin{lemma}  \label{lm:lip}
\footnote{\citet{P96} discuss the smoothness of the support function on strongly convex sets.
Here, we state a more general result.}
Denote  
$x_p = \argmax_{x \in \K} \langle p, x\rangle $ and $x_q = \argmax_{x \in \K} \langle q, x\rangle $, where $p,q \in \reals^d$ are any nonzero vectors.  
If a compact set $\K$ is a $\lambda$-strongly convex set,
then 
\begin{equation}
    \| x_p - x_q \| \leq \frac{2 \|p - q\|}{\lambda ( \| p \| + \|q \| )}.
\end{equation}
\end{lemma}

\begin{proof}
\citet{P96}
show that a strongly convex set $\K$ can be written as intersection of some Euclidean balls. 
Namely,
    \[ \K = \underset{u: \| u\|_2 = 1}{\cap} B_{\frac{1}{\lambda}} \left( x_u - \frac{u}{\lambda} \right) ,\]
where $x_u$ is defined as $x_u = \argmax_{x \in \K} \langle \frac{u}{\|u\|}, x\rangle$.

    Let  ${x_p = \argmax_{x \in \K} \langle \frac{p}{\|p\|}, x\rangle }$ and ${x_q = \argmax_{x \in \K} \langle \frac{q}{\|q\|}, x \rangle}$.
    Based on the definition of strongly convex sets, we can see that
    $x_q \in B_{ \frac{1}{\lambda}  } ( x_p - \frac{p}{\lambda \| p \|})$ and $x_p \in B_{\frac{1}{\lambda}  } ( x_q - \frac{q}{\lambda \| q \|} )$.
    Therefore, 
    \[
        \| x_q - x_p -  \frac{p}{\lambda \| p\|} \|^2 \leq \frac{1}{\lambda^2},
    \]
    which leads to
    \begin{equation}
        \label{eqn:ineqSumP1}
     \|p \| \cdot  \| x_p - x_q \|^2 \leq \frac{2}{\lambda} \langle x_p - x_q,  p \rangle.
    \end{equation}
    Similarly,
    \[
        \| x_p - x_q - \frac{q}{\lambda \|q\|} \|^2 \leq \frac{1}{\lambda^2}, 
    \]
    which results in
    \begin{equation}
        \label{eqn:ineqSumP2}
      \| q \| \cdot  \| x_p - x_q \|^2 \leq \frac{2}{\lambda} \langle x_q - x_p,  q \rangle.
    \end{equation}
    Summing (\ref{eqn:ineqSumP1}) and (\ref{eqn:ineqSumP2}), one gets
    $(\|p\| + \| q \|) \| x_p - x_q \|^2 \leq \frac{2}{\lambda} \langle x_p - x_q, p-q \rangle$.
    Applying the Cauchy-Schwarz inequality completes the proof.
\end{proof}

\begin{lemma} (\citet{L17}) \label{lem:Log_sum} 
For any non-negative real numbers $a_1,\ldots, a_n\geq 1$,
\begin{align*}
\sum_{i=1}^n \frac{a_i}{\sum_{j=1}^i a_j} 
\le 
1+\log\left( \sum_{i=1}^n a_i\right) ~.
\end{align*}
\end{lemma}

\subsection{Proof of Theorem~\ref{thm:equivStoFW} } \label{app:equivStoFW}

\begin{proof}
The equivalency of the update follows the proof of Theorem~\ref{thm:equivFW}. Specifically, 
 we have that the objects on the left in the following equalities  correspond to Alg.~\ref{alg:game} and those on the right to Alg.~\ref{alg:newStoFW}.
  \begin{eqnarray}
    x_t  & = & v_t    \\
    \bar{x}_t  & = & w_t .
  \end{eqnarray}
To analyze the regret of the y-player, we define $\{ \hat{y}_t \}$ as the points if the y-player would have played \FTL.
\begin{equation}
\begin{split}
\hat{y}_t &:= \arg\min_{y} \frac{1}{t-1} \sum_{s=1}^{t-1} \ell_t(y)
= \arg\max_{y} \frac{1}{t-1} \sum_{s=1}^{t-1} \langle x_s, y \rangle - f^*(y)
\\ & = \nabla f(\bar{x}_t)
 = \frac{1}{n} \sum_{i=1}^n \nabla f_i(\bar{x}_t) 
\end{split}
\end{equation}
Then, we have that
\begin{equation}
\begin{split}
\avgregret{y} & = \frac{1}{T} \big( \sum_{t=1}^T \ell_t(\hat{y}_t) - \ell_t(y_*) \big) + \frac{1}{T} \big( \sum_{t=1}^T \ell_t(y_t) - \ell_t(\hat{y}_t) \big)
\\ & \overset{(a)}{\leq} 
\frac{4 L D \log T}{T} + \frac{1}{T} \big( \sum_{t=1}^T \ell_t(y_t) - \ell_t(\hat{y}_t) \big)
\\ & = \frac{4 L D \log T}{T} + \frac{1}{T} 
\sum_{t=1}^T  \big( f^*(y_t) - f^*(\hat{y}_t) +  \langle x_t, \hat{y}_t - y_t \rangle \big) 
\\ & \overset{(b)}{\leq} \frac{4 L D \log T}{T} +
\sum_{t=1}^T \frac{1}{T} (L_0 + r) \| y_t - \hat{y}_t \|
\\ & = \frac{4 L D \log T}{T} + \frac{1}{T} (L_0 + r)
\sum_{t=1}^T \| \frac{1}{n} \sum_{i=1}^n g_{i,t} - \frac{1}{n} \sum_{i=1}^n \nabla f_i(\bar{x}_t) \|
\\ & = \frac{4 L D \log T}{T} + \frac{1}{T} (L_0 + r)
\sum_{t=1}^T \| \frac{1}{n} \sum_{i \neq i_t}^n \big( g_{i,t}  - \nabla f_i(\bar{x}_t) \big)\|
\\ & \overset{(c)}{\leq} \frac{4 L D \log T}{T} + \frac{L(L_0 + r)}{Tn} 
\sum_{t=1}^T \sum_{i \neq i_t}^n \| \bar{x}_{\tau_t(i)}  - \bar{x}_t \|
\\ & \overset{(d)}{\leq} \frac{4 L D \log T}{T} + \frac{L(L_0 + r)}{Tn} 
\sum_{t=1}^T \sum_{i \neq i_t}^n \frac{2nr}{t}
\\ & = O( \frac{ \max\{ LD , L (L_0+r) n r \} \log T}{T} ), 
\end{split}
\end{equation}
where (a) is by the regret of \FTL (Lemma~\ref{regret:FTL}),
\[
\begin{aligned}
\frac{1}{T} \big( \sum_{t=1}^T \ell_t(\hat{y}_t) - \ell_t(y_*) \big)
      & \leq \frac{1}{T}  \sum_{t=1}^T \frac{2  \| \nabla \ell_t(\hat{y}_t) \|^2}{\sum_{s=1}^{t}  (1/L)} 
  & = \frac{4LD \log T}{T},
\end{aligned}
\]
where we used that $\| \nabla \ell_t(\hat{y}_t) \|^2 = \| x_t - \nabla f^*(\hat{y}_t) \|^2 = \| x_t - \bar{x}_{t-1} \|^2 \leq D$,
(b) we assume that the conjugate is $L_0$-Lipschitz and that $\max_{x \in \K} \| x \| \leq r$,
(c) we denote $\tau_t(i) \in [T]$ as the last iteration that $i_{th}$ sample's gradient is computed at $t$, and (d) is because that
\begin{equation}
\begin{split}
\| \bar{x}_{\tau_t(i)} - \bar{x}_t \| & =
\| \frac{1}{\tau_t(i)}  \sum_{s=1}^{\tau_t(i)} x_s - \frac{1}{t}  \sum_{s=1}^t x_s  \| \leq \| \sum_{s=1}^{\tau_t(i)} x_s ( \frac{1}{\tau_t(i)} - \frac{1}{t} ) \|
+ \| \frac{1}{t} \sum_{s=\tau_t(i) + 1}^{t} x_s \|
\\ & 
= \frac{t - \tau_t(i) }{t} \| \bar{x}_{\tau_t(i)} \|
+ \| \frac{1}{t} \sum_{s=\tau_t(i) + 1}^{t} x_s \|
\\ & 
\leq \frac{n r }{t} 
+ \| \frac{1}{t} \sum_{s=\tau_t(i) + 1}^{t} x_s \| = 
\frac{n r }{t} 
+  \frac{t - \tau_t(i)}{t} \| \frac{1}{t - \tau_t(i)} \sum_{s=\tau_t(i) + 1}^{t} x_s \|
\\ &
\leq \frac{2 n r}{t}.
\end{split}
\end{equation}
For the x-player, since it plays \BR, the regret is non-positive.

Combining the average regrets of both players leads to the result.

\end{proof}
\subsection{Proof of Theorem~\ref{thm:Heavy}} \label{app:thm:Heavy}

\begin{proof}
First, we can bound the norm of the gradient as
  \[
    \| \nabla \ell_t(y_t) \|^2 = \| x_t - \nabla f^*(y_t) \|^2 = \| x_t - \bar{x}_{t-1} \|^2 
  \]
  Combining this with Lemma~\ref{regret:FTL} we see that
  \begin{align*} 
    \avgregret{y}[\FTL] 
      & \leq \frac{1}{A_T}  \sum_{t=1}^T \frac{2 \alpha_t^2 \| \nabla \ell_t(y_t) \|^2}{\sum_{s=1}^{t} \alpha_s (1/L)} 
  \leq \frac{1}{A_T}  \sum_{t=1}^T \frac{2 \alpha_t^2 \| x_t - \bar{x}_{t-1} \|^2}{\sum_{s=1}^{t} \alpha_s (1/L)}\\
&  = O( \sum_{\tau=1}^T \frac{ L \| \bar{x}_{t-1} - x_t  \|^2 }{A_T}).
  \end{align*}
On the other hand, the x-player plays $\MD$, according to Lemma~\ref{lem:MD},
its regret satisfies
\begin{equation}
\avgregret{x}  
\leq  \frac{ \frac{1}{\gamma} D - \sum_{t=1}^{T}  \frac{1}{2 \gamma}   \| x_{t-1} - x_t \|^2 }{A_T}
\end{equation}
Since the distance terms may not cancel out, one can only bound the differences of the distance terms by a constant, which leads to the non-accelerated $O(1/T)$ rate.

\end{proof}

\subsection{Proof of Theorem~\ref{thm:conv}} \label{app:thmFW}

Assume that the spectral norm of the gradient $\| \nabla f(\cdot) \|_2$ over the nuclear-norm ball $\mathcal{NB}_{d_1,d_2}(r)$ satisfies 
$\| \nabla f(\cdot) \|_2 \leq G$.
Then, 
\begin{equation} \label{G}
\begin{aligned}
\| \nabla F_r(X) \|_{\infty} & :=  \max\{ |\lambda_{\min}( \nabla F_r(X) )|, | \lambda_{\max}( \nabla F_r(X) )| \} \\ & = \| \nabla  f(2r X^{(2)}) \|_2 \leq G,
\end{aligned}
\end{equation}
where the equality is due to the structure of the gradient matrix (\ref{eq:grad}).

\begin{proof} (of Theorem~\ref{thm:conv})
Following the discussion in the main text, Subsection~\ref{sub:design},
we consider an instance of Algorithm~\ref{alg:game} by setting the weight $\alpha_t=t$, the function $F(\cdot)$ in the definition of the payoff function (\ref{eq:fenchelgame}) of the game as  $F(\cdot)\leftarrow F_r(\cdot)$ as defined in (\ref{obj:eqv}). Furthermore, let  
$\alg^y$ and $\alg^x$ in Algorithm~\ref{alg:game} respectively as:
\begin{eqnarray}
\textstyle y_t  & \textstyle \leftarrow & \textstyle \argmin_{y} \left\{ \alpha_t \ell_{t-1}(y) +  \sum_{s=1}^{t-1} \alpha_s \ell_s(y)\right\} = 
\nabla F_r(\xof_{t}), \label{eq:yapp}
\\
\textstyle x_t & \textstyle \leftarrow & \textstyle 
 \frac{1}{m_t} \sum_{j_t=1}^{m_t}  \Psi_{u_{j_t}} (-\eta \sum_{s=1}^t \alpha_s y_s ), \text{ where each } u_{j_t} \sim \text{Uni}(\mathbb{S}_d) 
\label{eq:xapp},
\end{eqnarray}
 where $\xof_t := \textstyle \frac{1}{A_t}(\alpha_t x_{t-1} + \sum_{s=1}^{t-1} \alpha_s x_s )$.
We also need a \emph{ghost} sequence $\{ \hat{x}_s\}_{s=1}^t$ solely used for the analysis,
\begin{eqnarray}
\textstyle \hat{x}_t
& \leftarrow & \bar{\Psi}( - \eta \sum_{s=1}^t \alpha_s y_s ) = \bar{\Psi}( - \eta \sum_{s=1}^t \alpha_s \nabla F_r(\xof_s) )
\end{eqnarray}
and we use (\ref{eq:yalg}) that $y_s = \nabla F_r(\xof_s)$.
By Lemma~
the y-player's regret is
\begin{equation}
\textstyle \regret{y} \leq \hat{L} \sum_{t=1}^T \frac{\alpha_t^2}{A_t} \|x_{t-1} - x_t \|^2,
\end{equation}
where $\hat{L}$ is the smoothness constant of the underlying function $F_r(\cdot)$.

In the following, we denote the Bregman divergence with the distance generating function $\bar{\psi}(\cdot)$ (defined in the preliminary section): 
\begin{equation}
V_{C}(B) = \bar{\psi}(B) - \bar{\psi}(C) - \langle \bar{\Psi}(C), B -C \rangle,
\end{equation}
for any symmetric matrices $B,C \in \mathcal{S}_d$. Recall that $\bar{\Psi}(\cdot) = \nabla \bar{\psi}(\cdot)$.

Now we are going to analyze the regret of the x-player.
But before that, let us analyze the regret if the x-player would have played $\hat{x}_t$.
We have
\begin{equation} \label{eq:x1}
\begin{split}
\regret{\hat x} & := \sum_{t=1}^T \alpha_t \langle \hat{x}_t - x^*, y_t \rangle
\\ & \overset{(a)}{\leq}
\sum_{t=1}^T  \langle \hat{x}_t - x^*, t \nabla F_r(\xof_{t}) \rangle
\\ & \overset{(b)}{=}
\sum_{t=1}^T \frac{1}{\eta} \langle \hat{x}_t - x^*, G_{t-1} - G_t \rangle
\\ & \overset{(c)}{\leq}
\sum_{t=1}^T \frac{1}{\eta} 
\left( V_{G'}(G_{t-1}) - V_{G'}(G_t) - V_{G_t}(G_{t-1}) \right)
\leq \frac{1}{\eta} \left(  V_{G'}(G_0) - \sum_{t=1}^T V_{G_t}(G_{t-1} ) \right)
\\ & \overset{(d)}{\leq} \frac{1}{\eta} \left( \log 4 d - \sum_{t=1}^T V_{G_t}(G_{t-1} )  \right) 
\\ & \overset{(e)}{\leq}
\frac{1}{\eta} \left( \log 4 d - \sum_{t=1}^T \frac{1}{6} \| \hat{x}_t - \hat{x}_{t-1} \|^2  \right),
\end{split}
\end{equation}
where (a) is by $\alpha_t = t$ and the y-player strategy (\ref{eq:yapp}),
(b) we define $G_t = G_{t-1} - \eta t \nabla F_r(\xof_{t})$,
(c) we use the well-known three-point inequality:
\begin{equation}
\langle \bar{\Psi}(B_1) - \bar{\Psi}(B_0),  B_2 - B_1  \rangle =
V_{B_0}(B_2) - V_{B_0}(B_1) - V_{B_1}(B_2),
\end{equation}
and we let $B_1 \leftarrow G_t $, $\hat{x}_t=\bar{\Psi}(G_t)$, $B_2 \leftarrow G_{t-1} $,
$x^* = \bar{\Psi}(G')$ and $B_0 \leftarrow G'$ for some symmetric matrix $G' \in \mathcal{S}_d$,
(d) we use that $G_0 = 0_d$ and that $V_{G'}(0_d) \leq \log 4 d$ for any $G' \in \mathcal{S}_d$
by Proposition 1 of \citet{CDST19},
(e) we use that $V_{G_{t}}(G_{t-1}) \geq \frac{1}{6} \| \hat{x}_t - \hat{x}_{t-1} \|^2$ by Proposition 1 of \citet{CDST19} and Lemma 16 in \citet{K16},
as $\hat{x}_t = \bar{\Psi}(G_t)$ and $\hat{x}_{t-1} = \bar{\Psi}(G_{t-1})$.

So the regret of the x-player using strategy (\ref{eq:xapp}) can be bounded as
\begin{equation} \label{eq:x3}
\begin{split}
\textstyle \regret{x}  \textstyle & := \sum_{t=1}^T \alpha_t \langle x_t - x^* , y_t \rangle 
=  \sum_{t=1}^T \alpha_t \langle \hat{x}_t - x^* , y_t \rangle 
+ \alpha_t \langle x_t - \hat{x}_t, y_t \rangle
\\ & 
 \textstyle
\overset{(\ref{eq:x1})}{ \leq } \frac{ \log (4 d) - \frac{1}{6} \sum_{{t=1}}^{T} \| \hat{x}_{t} - \hat{x}_{t-1} \|^2}{\eta} + \sum_{t=1}^T \alpha_t \langle x_t - \hat{x}_t, y_t \rangle.
\end{split}
\end{equation}
For the terms $\{ \alpha_s \langle x_s - \hat{x}_s, y_s \rangle \}_{s=1}^t$, notice that it is a martingale difference sequence. Using the fact that $\alpha_s = s$ and that $y_s$ is a gradient at some point which is bounded, i.e. $\| \nabla F_r(\cdot) \|_{\infty} \leq G$, see (\ref{eq:grad}) and (\ref{G}), we have that $\alpha_s \langle x_{s,j_s}, y_s \rangle \leq \alpha_s \| y_s \|_{\infty} \leq \alpha_s G$. Hoeffding's lemma implies that
$\{ \alpha_s \langle x_{s,j_s} - \hat{x}_{s}, y_s \rangle \}$ is $\alpha_s^2 G^2$-sub-Gaussian, and consequently $\{ \alpha_s \langle x_s - \hat{x}_{s}, y_s \rangle \} = \{ \frac{1}{m_s} \sum_{j_s=1}^{m_s}  \alpha_s \langle x_{s,j_s} - \hat{x}_s, y_s \rangle \}$ is $\frac{\alpha_s^2 G^2}{m_s}$-sub-Gaussian.
The fact that $\xi_s:=\alpha_s \langle x_s - \hat{x}_s, y_s \rangle$ is $\frac{\alpha_s^2 G^2}{m_s}$-sub-Gaussian implies that
$\max \{ Pr(\xi_s \geq \theta), Pr( \xi_s \leq - \theta)   \} \leq 2 \exp( - \frac{m_s}{2 \alpha_s^2 G^2} \theta^2)$.
So we can apply a variant of Azuma-Hoeffding inequality (Lemma~\ref{lem:Azuma} in Section~\ref{app:Azuma}) to conclude that
\begin{equation} \label{eq:x2}
 \sum_{s=1}^t \alpha_s \langle x_s - \hat{x}_s, y_s \rangle \leq \sqrt{  112 G^2 \log (2/\delta) \sum_{s=1}^t \frac{\alpha_s^2}{m_s}  },
\end{equation}
  with probability at least $1-\delta/2$.
Therefore, the sum of the weighted average regret of both players is bounded by
\begin{equation} \label{eq:sumreg}
\begin{split}
&  \avgregret{x} + \avgregret{y} 
\\ & \overset{(\ref{eq:x2}),(\ref{eq:x3})}{\leq} \frac{ \frac{ \log (4d) }{\eta} + \sum_{t=1}^{T}\big( \frac{\alpha_t^2}{A_t} \hat{L} \| x_{t-1} - x_t \|^2 - \frac{1}{6 \eta} \|  \hat{x}_{t-1} - \hat{x}_{t}  \|^2 \big) + \sqrt{ 112 G^2 \log (2/\delta) \sum_{t=1}^T \frac{\alpha_t^2}{m_t}  } }{A_T}
\\ &
\overset{(a)}{\leq}
 \frac{ \frac{ \log (4d) }{\eta} + \sum_{t=1}^{T}\big( \frac{\alpha_t^2}{A_t} 3 \hat{L}  \| \hat{x}_{t-1} - \hat{x}_t \|^2  - \frac{1}{6 \eta} \|  \hat{x}_{t-1} - \hat{x}_{t}  \|^2 \big)   }{A_T}
\\ & \quad + \frac{  \sum_{t=1}^{T}  \frac{96 \hat{L} \log (4d /\delta)}{m_t} + \sqrt{ 112 G^2 \log (2/\delta) \sum_{t=1}^T \frac{\alpha_t^2}{m_t}  }     }{A_T} 
\\ & 
\overset{(b)}{\leq} O\left( \frac{ \hat{L} \log (d) }{T^2} \right) +
\frac{  \sum_{t=1}^{T}  \frac{96 \hat{L} \log (4d /\delta)}{m_t} + \sqrt{ 112 G^2 \log (2/\delta) \sum_{t=1}^T \frac{t^2}{m_t}  }    }{T(T+1)/2},
\end{split}
\end{equation}
where (a) is because 
$ \textstyle \| x_{t-1} - x_t \|^2 = \| x_{t-1} - \hat{x}_{t-1} + \hat{x}_{t-1} - \hat{x}_t + \hat{x}_t - x_t \|^2  
\leq 3 ( \| x_{t-1} -\hat{x}_{t-1} \|^2 +\| x_t - \hat{x}_t \|^2 + \| \hat{x}_{t-1} - \hat{x}_t \|^2 ) \overset{(\#)}{\leq}  \frac{24 \log (4d /\delta)}{m_t} +   \frac{24 \log (4d /\delta)}{m_{t-1}} + 3 \| \hat{x}_{t-1} - \hat{x}_t \|^2$ and the inequality $(\#)$ is due to Theorem 1.6.2 (Matrix Bernstein) of \citet{T14}:
$\text{Pr}( \| \frac{1}{m_t} \sum_{j_t=1}^{m_t} x_{t,j_t} - \hat{x}_t \| \geq \theta ) \leq 2 d \exp( - \frac{ m_t \theta^2}{ 4 (1+\frac{\theta}{3} )}    )$, which means that
with probability at least $1-\delta/2$, 
$\| \frac{1}{m_t} \sum_{j_t=1}^{m_t} x_{t,j_t} - \hat{x}_t \| \leq\sqrt{ \frac{8 \log (4d /\delta)}{m_t} }$ if $\sqrt{ \frac{8}{9 m_t} \log (4 d / \delta) } \leq 1$,
and (b) of (\ref{eq:sumreg})
is due to the constraint of $\eta \leq \frac{1}{36\hat{L}}$ so that the distance terms cancel out and that $A_T = \sum_t t$.
Thus, by Lemma~\ref{lem:eq}, Theorem~\ref{lem:fenchelgame}
we have established the convergence rate.

Since by Theorem~\ref{thm:equv}, Algorithm~\ref{alg:main} is exactly equivalent to the instance of Algorithm~\ref{alg:game}
here, we have completed the proof.

\end{proof}

\subsection{Proof of Theorem~\ref{thm:equv}} \label{app:thm:equv}
\begin{proof}[Proof of Theorem~\ref{thm:equv}]
We use proof by induction to show that
$W_t = \sum_{{s=1}}^{t} \frac{\alpha_s}{A_t} x_{s} = \bar{x}_t$,
$X_t = x_t$, and $Z_t = \tilde{x}_t$
for any $t>0$.

For the base case $t=1$, we have
$ W_1    = (1-\beta_1) W_0 + \beta_1 X_1    = X_1 = x_1 = \frac{\alpha_1}{A_1} x_1$,
as by the same initialization
$Z_1 = X_0 = x_0 = \tilde{x}_1$, one can ensure that $X_1 = x_1$.

Now assume that the one-to-one correspondence holds at $t-1$. We have that
$\textstyle  W_{t}  = (1-\beta_t) W_{t-1} + \beta_t x_{t}
= (1-\beta_t) ( \sum_{{s=1}}^{t-1} \frac{\alpha_s}{A_{t-1}} x_{s}   ) + \beta_t x_{t}
 \textstyle  =  (1 - \frac{2}{t+1}) ( \sum_{{s=1}}^{t-1} \frac{\alpha_s}{ \frac{t(t-1)}{2} } x_{s}   ) + \beta_t x_{t} = \sum_{{s=1}}^{t-1} \frac{\alpha_s}{ \frac{t(t+1)}{2} } x_{s} +  \frac{ \alpha_t}{A_t}  x_{t}
 = \sum_{{s=1}}^{t} \frac{\alpha_s}{A_s} x_s = \bar{x}_t$.
On the other hand,
we have that
$ \textstyle Z_{t}  = (1-\beta_t) W_{t-1} + \beta_t x_{t-1}
= (1-\beta_t) ( \sum_{{s=1}}^{t-1} \frac{\alpha_s}{A_{t-1}} x_{s}   ) + \beta_t x_{t-1}
 \textstyle =  (1 - \frac{2}{t+1}) ( \sum_{{t=1}}^{t-1} \frac{\alpha_t}{ \frac{t(t-1)}{2} } x_{t}   ) + \beta_t x_{t-1} = \sum_{{s=1}}^{t-1} \frac{\alpha_s}{ \frac{t(t+1)}{2} } x_{s} +  \beta_t x_{t-1}
 = \sum_{{s=1}}^{t-1} \frac{\alpha_s}{ A_{t} } x_{s} +  \frac{\alpha_t}{A_t} x_{t-1}
 = \xof_t.
 $
The result implies that $G_t = - \eta \sum_{s=1}^t s \nabla F_r(\xof_s) =  - \eta \sum_{s=1}^t \alpha_s y_s$; consequently $X_t = x_t$. We now have completed the proof.
\end{proof}

\subsection{Some supporting lemmas} \label{app:Azuma}

\begin{lemma} \label{lem:Azuma}
Let $\xi_1,\xi_2, \dots, \xi_T$ be a martingale difference sequence with respect to a sequence $\mathcal{F}_1, \dots, \mathcal{F}_T$, and suppose there are constants $\{b_t\} \geq 1$ and $\{ c_t \} > 0$ such that for any $\theta > 0$
\[
\max \{  Pr( \xi_t > \theta | \mathcal{F}_1, \dots, \mathcal{F}_{t-1} ),   Pr( \xi_t < - \theta | \mathcal{F}_1, \dots, \mathcal{F}_{t-1} )      \} \leq b_t \exp( - c_t \theta^2 )
\]
Then, for any $\delta$, it holds with probability at least $1-\delta$ that
\[
\frac{1}{T} \sum_{t=1}^T \xi_t \leq \sqrt{ \frac{ 28 \sum_{t=1}^T \frac{b_t}{c_t} \log (1/\delta) }{ T^2 } }.
\]

\end{lemma}

\begin{proof}
The lemma's statement is an extension of Theorem 2 in \citet{S11} which considers the case that $b_t=b$ and $c_t=c$ for some numbers $b>1,c>0$.

Denote $s$ a positive number.  
\begin{equation}
\begin{split}
& Pr ( \frac{1}{T} \sum_{t=1}^T \xi_t > \epsilon )
 = Pr\big(  \exp(  s \sum_{t=1}^T \xi_t    ) > \exp( s T \epsilon  )  \big)
\\ & \overset{(a)}{\leq} \exp( - s T \epsilon ) \mathbb{E} [  \exp( s \sum_{t=1}^T \xi_t ) ]
\leq \exp( - s T \epsilon ) \mathbb{E} [ \mathbb{E}[ \overset{T}{ \underset{t=1}{\Pi} } \exp( s \xi_t) | \mathcal{F}_1 , \dots, \mathcal{F}_T   ]  ]
\\ &
\leq \exp( - s T \epsilon ) \mathbb{E} \big[
 \mathbb{E}[ \exp( s \xi_T) | \mathcal{F}_1 , \dots, \mathcal{F}_{T-1}   ] 
 \mathbb{E}[ \overset{T-1}{ \underset{t=1}{\Pi} } \exp( s \xi_t) | \mathcal{F}_1 , \dots, \mathcal{F}_{T-1}   ]  \big]
\\ &
\overset{(b)}{\leq} \exp( - s T \epsilon ) \exp( \frac{7 b_T }{ c_T} s^2 )
\mathbb{E} \big[
 \mathbb{E}[ \overset{T-1}{ \underset{t=1}{\Pi} } \exp( s \xi_t) | \mathcal{F}_1 , \dots, \mathcal{F}_{T-1}   ]  \big]
 \\ &
\dots
\\ &
\leq \exp( - s T \epsilon + \sum_{t=1}^T \frac{7 b_t }{ c_t} s^2 ),
\end{split}
\end{equation}
where (a) is by Markov's inequality and (b) is due to Lemma~\ref{lem:shamir}.
By setting $s= \frac{T \epsilon}{ 2 \sum_{t=1}^T \frac{7 b_t }{ c_t}}$, we have that
$Pr ( \frac{1}{T} \sum_{t=1}^T \xi_t > \epsilon ) \leq \exp(  - \frac{T^2 \epsilon^2 }{28 \sum_{t=1}^T \frac{ b_t }{ c_t} })$.
Now setting $\delta= \exp(  - \frac{T^2 \epsilon^2 }{28 \sum_{t=1}^T \frac{ b_t }{ c_t} })$ and solving $\epsilon$ leads to the result.
\end{proof}

\begin{lemma} (Lemma 1 in \citet{S11}) \label{lem:shamir}
Let $\xi$ be a random variable with $\E[\xi] = 0$, and suppose there exist a constant $b \geq 1$ and a constant $c>0$ such that for all $\theta > 0$, it holds that
\[
\max \{ Pr(\xi \geq \theta), Pr( \xi \leq - \theta)   \} \leq b \exp( - c \theta^2).
\]
Then for any $s>0$.
\[
\E[ \exp(s \xi) ] \leq \exp( 7 b s^2 / c ).
\]
\end{lemma}

\section{Conclusion}
In this chapter, we present a modular analysis that bridges the online learning/no-regret learning and the classical \emph{offline} convex optimization.
The generic scheme also makes designing fast algorithms easier.
Simply pitting any two no-regret learning algorithms against each other with an appropriate weighting scheme will lead to an offline convex optimization with a guarantee implied by our meta theorem.
We believe our generic acceleration scheme can help to design new algorithms.
For example, in online learning there are many adaptive algorithms
which enjoy data-dependent regret guarantees and allow a different adaptive learning rate for a different coordinate
(e.g. \citet{L17} and \citet{M17}).
It is interesting to check if our approach of \emph{optimization as iteratively playing a game} can help to design a fast adaptive algorithm for \emph{offline} optimization.
    \chapter{A Modular Analysis of Provable Acceleration via Polyak's Momentum: Training a Wide ReLU Network and a Deep Linear Network}
\label{ch3}

\section{Introduction}

Momentum methods are very popular for training neural networks in various applications (e.g. \cite{Rnet16,attention17,KSH12}).
It has been widely observed that the use of momentum helps faster training in deep learning (e.g. \cite{KH1918,CO19}).
Among all the momentum methods, the most popular one seems to be Polyak's momentum (a.k.a. Heavy Ball momentum) \cite{P64}, which is the default choice of momentum in PyTorch and Tensorflow.
The success of Polyak's momentum in deep learning is widely appreciated
and almost all of the recently developed adaptive gradient methods like
Adam \cite{KB15}, AMSGrad \cite{RKK18}, and AdaBound \cite{LXLS19}    adopt the use of Polyak's momentum, instead of Nesterov's momentum. 

However, despite its popularity, little is known in theory about why Polyak's momentum helps to accelerate training neural networks.
Even for convex optimization, 
problems 
like strongly convex quadratic problems seem to be one of the few cases that discrete-time Polyak's momentum method provably achieves faster convergence than standard gradient descent (e.g. \cite{LRP16,goh2017why,GFJ15,GLZX19,LR17,LR18,CGZ19,SP20,NB15,WJR21,FSRV20,DJ19,SDJS18,H20}). 
On the other hand, the theoretical guarantees of
Adam, AMSGrad , or AdaBound     
are only worse if the momentum parameter $\beta$ is non-zero and 
the guarantees deteriorate as the momentum parameter increases,
which do not show any advantage of the use of momentum \cite{AMMC20}.
Moreover,
the convergence rates that have been established for Polyak's momentum in several
related works \cite{GPS16,SYLHGJ19,YLL18,LGY20,MJ20} do not improve upon those for vanilla gradient descent or vanilla SGD in the worst case.
\citet{LRP16,GFJ15} even show negative cases in \emph{convex} optimization 
that the use of Polyak's momentum results in divergence. 
Furthermore,
\citet{NKJK18} construct a problem instance for which the momentum method under its optimal tuning is outperformed by other algorithms. 
A solid understanding of the empirical success of Polyak's momentum in deep learning has eluded researchers for some time.

We begin this chapter by first revisiting the use of Polyak's momentum for 
the class of strongly convex quadratic problems,
\begin{equation} \label{obj:strc}
\textstyle
 \min_{w \in \reals^d} \frac{1}{2} w^\top \Gamma w + b^\top w,
\end{equation} 
where $\Gamma \in \reals^{d \times d}$ is a PSD matrix such that $\lambda_{\max}( \Gamma)= \alpha$, $\lambda_{\min}(\Gamma) =  \mu > 0$.
This is one of the few
known examples that Polyak's momentum has a provable \textit{globally} \textit{accelerated} linear rate in the \textit{discrete-time} setting. 
Yet even for this class of problems existing results only establish an accelerated linear rate in an asymptotic sense and several of them do not have an explicit rate in the non-asymptotic regime (e.g. \cite{P64,LRP16,M19,R18}).
Is it possible to prove a non-asymptotic accelerated linear rate in this case? We will return to this question soon. 

For general $\mu$-strongly convex, $\alpha$-smooth, and twice differentiable
functions (not necessarily quadratic), denoted as $F_{\mu,\alpha}^2$,
Theorem 9 in \citet{P64} shows an asymptotic accelerated linear rate
when the iterate is \textit{sufficiently} close to the minimizer so that the landscape can be well approximated by that of a quadratic function. However, 
the definition of the neighborhood was not very precise in the paper. In this work, we show a locally accelerated linear rate under a quantifiable definition of the neighborhood.

\begin{algorithm}[t]
\begin{algorithmic}[1]
\small
\caption{Gradient descent with Polyak's momentum \cite{P64} (Equivalent Version 1)
} \label{alg:HB1}{}
\STATE Required: Step size parameter $\eta$ and momentum parameter $\beta$.
\STATE Init: $w_{0} \in \reals^d $ and $M_{-1} = 0 \in \reals^d$.
\FOR{$t=0$ to $T$}
\STATE Given current iterate $w_t$, obtain gradient $\nabla \ell(w_t)$.
\STATE Update momentum $M_t = \beta M_{t-1} +  \nabla \ell(w_t)$.
\STATE Update iterate $w_{t+1} = w_t - \eta M_t$.
\ENDFOR
\end{algorithmic}
\end{algorithm}

\begin{algorithm}[t]
\begin{algorithmic}[1]
\small
\caption{Gradient descent with Polyak's momentum \cite{P64} (Equivalent Version 2) } 
\label{alg:HB2}
\STATE Required: step size $\eta$ and momentum parameter $\beta$.
\STATE Init: $w_{0} = w_{-1} \in \reals^d $
\FOR{$t=0$ to $T$}
\STATE Given current iterate $w_t$, obtain gradient $\nabla \ell(w_t)$.
\STATE Update iterate $w_{t+1} = w_t - \eta \nabla \ell(w_t) + \beta ( w_t - w_{t-1} )$.
\ENDFOR
\end{algorithmic}
\end{algorithm}
\vspace{-0.02in}

Furthermore, we provably show that Polyak's momentum helps to achieve a faster convergence for training two neural networks, compared to vanilla GD.
The first is training a one-layer ReLU network.
Over the past few years there have appeared an enormous number of works considering training a one-layer ReLU network, provably showing convergence results for vanilla (stochastic) gradient descent 
(e.g. \cite{LL18,JT20,LY17,DZPS19,DLLWZ16,ZL19_icml,ZY19,ZCZG19,ADHLSW19_icml,JGH18,LXSBSP19,COB19,OS19,BG17,CHS20,
T17,S17,BL20,LMZ20,HN20,Dan17,ZG19,DGM20,D20,WLLM19,YS20,FDZ19,SY19,CCGZ20}),
as well as for other algorithms (e.g. \cite{ZMG19,WDS19,Cetal19,ZSJBD17,GKLW16,BPSW20,LSSWY20,PE20}).
However, we are not aware of any theoretical works that study the momentum method in neural net training except the work \citet{KCH20}.
These authors show that SGD with Polyak's momentum (a.k.a. stochastic Heavy Ball) with infinitesimal step size, i.e. $\eta \rightarrow 0$, for training a one-hidden-layer network with an infinite number of neurons, i.e. $m \rightarrow \infty$, converges to a stationary solution.
However, the theoretical result does not show a faster convergence by momentum. 
In this work we consider the discrete-time setting
and
nets with finitely many neurons. 
We 
provide a non-asymptotic convergence rate of Polyak's momentum, establishing a concrete improvement relative to the best-known rates for vanilla gradient descent.

Our setting
of training a ReLU network follows the same framework as previous results, including \cite{DZPS19,ADHLSW19_icml,ZY19}. 
Specifically,
we study training a one-hidden-layer ReLU neural net of the form, 
\begin{equation} \label{eq:Network}
\textstyle
\N_{W}^{\text{ReLU}}(x) := \frac{1}{\sqrt{m} } \sum_{r=1}^m a_r \sigma( \langle w^{(r)},  x \rangle ),
\end{equation}
where $\sigma(z):= z \cdot \mathbbm{1}\{ z \geq 0\}$ is the ReLU activation,
$w^{(1)}, \dots, w^{(m)}  \in \reals^d$ are the weights of $m$ neurons on the first layer, $a_1, \dots, a_m \in \reals$ are weights on the second layer, 
 and $\N_{W}^{\text{ReLU}}(x) \in \reals$ is the output predicted on input $x$. 
Assume $n$ number of samples $\{ x_i \in \reals^d \}_{i=1}^n$ is given.
Following \cite{DZPS19,ADHLSW19_icml,ZY19},
we define a Gram matrix $H \in \reals^{n \times n}$ for the weights $W$
and its expectation $\bar{H} \in \reals^{n \times n}$ over the random draws of $w^{(r)} \sim N(0,I_d) \in \reals^d$ 
whose $(i,j)$ entries are defined
as follows,
\begin{equation}
\begin{aligned}
& H(W)_{i,j}   =  \sum_{r=1}^m \frac{x_i^\top x_j}{m} \mathbbm{1}\{ \langle w^{(r)}, x_i \rangle \geq 0 \text{ } \&  \text{ }   \langle w^{(r)}, x_j \rangle \geq 0 \}
\\ & \quad
\bar{H}_{i,j}  := \underset{ w^{(r)}}{\mathbbm{E}}
[ x_i^\top x_j \mathbbm{1}\{ \langle w^{(r)}, x_i \rangle \geq 0 \text{ } \&  \text{ }   \langle w^{(r)}, x_j \rangle \geq 0 \}     ] .
\end{aligned}
\end{equation}
The matrix $\bar{H}$ is also called a neural tangent kernel (NTK) matrix in the literature (e.g. \cite{JGH18,Y19,BM19}).
Assume that the smallest eigenvalue $\lambda_{\min}(\bar{H})$ is strictly positive and certain conditions about the step size and the number of neurons are satisfied.
Previous works \citet{DZPS19,ZY19} show a linear rate of vanilla gradient descent, while we show an accelerated linear rate 
\footnote{ 
We borrow the term ``accelerated linear rate'' from the convex optimization literature \cite{N13}, because the result here has a resemblance to those results in convex optimization, even though the neural network training is a non-convex problem.
}
of gradient descent with Polyak's momentum. As far as we are aware, our result is the first acceleration result of training an over-parametrized ReLU network. 

The second result is training a deep linear network.
The deep linear network is a canonical model for studying optimization and deep learning, and in particular for understanding gradient descent (e.g. \cite{BHL18,SMG14,HXP20}), studying the optimization landscape (e.g. \cite{K16,LvB18}), and establishing the effect of implicit regularization (e.g. \cite{MGWLSS20,JT19,LMZ18,RC20,ACHL19,GBL19,GWBNS17,LL20}).
In this work, following \citet{DH19}, \citet{HXP20}, we study training a $L$-layer linear network of the form,
\begin{equation} \label{eq:NetworkLinear}
\textstyle
\N_W^{L\text{-linear}}(x) := \frac{1}{\sqrt{m^{L-1} d_{y}}} \W{L} \W{L-1} \cdots \W{1} x,
\end{equation}
where $\W{l} \in \reals^{d_l \times d_{l-1}}$ is the weight matrix of the layer $l \in [L]$, and $d_0 = d$, $d_L = d_y$ and $d_l = m$ for $l \neq 1, L$.
Therefore, except the first layer $\W{1} \in \reals^{ m \times d}$ and the last layer $\W{L} \in \reals^{ d_y \times m}$, all the intermediate layers are $m \times m$ square matrices. 
The scaling $\frac{1}{\sqrt{m^{L-1} d_{y}}} $ is necessary to ensure that the network's output at the initialization 
$\N_{W_0}^{L\text{-linear}}(x)$ has the same size as that of the input $x$, in the sense that $\E[\| \N_{W_0}^{L\text{-linear}}(x) \|^2 ] = \| x \|^2$, where the expectation is taken over some appropriate random initialization of the network (see e.g. \cite{DH19,HXP20}).
\citet{HXP20} show vanilla gradient descent with orthogonal initialization converges linearly and the required width of the network $m$ is independent of the depth $L$, while we show an accelerated linear rate of Polyak's momentum and the width $m$ is also independent of $L$. To our knowledge, this is the first acceleration result 
of training a deep linear network.

A careful reader may be tempted by the following line of reasoning: a deep linear network (without activation) is effectively a simple linear model, and we already know that a linear model with the squared loss gives a quadratic objective for which Polyak's momentum exhibits an accelerated convergence rate. 
But this intuition, while natural, is not quite right: it is indeed nontrivial even to show that vanilla gradient descent provides a linear rate on deep linear networks \cite{HXP20,DH19,BHL18,ACGH19,HM16,WWM19,ZLG20}, as the optimization landscape is non-convex.
Existing works show that under certain assumptions, all the local minimum are global \cite{K16,LvB18,YSJ17,LK17,ZL18,HM16}. These results are not sufficient to explain the linear convergence of momentum, let alone the acceleration; see Section~\ref{app:exp}
for an empirical result.

Similarly, it is known that under the NTK regime the output of the ReLU network trained by gradient descent can be approximated by a linear model (e.g. \citet{HXAP20}). However, this result alone neither implies a global convergence of any algorithm nor characterizes the optimization landscape. 
While \citet{LZB20a} attempt to derive an algorithm-independent equivalence of a class of linear models and a family of wide networks, their result requires the activation function to be differentiable which does not hold for the most prevalent networks like ReLU. Also, their work heavily depends on the regularity of Hessian, making it hard to generalize beyond differentiable networks. Hence, while there has been some progress understanding training of wide networks through linear models, there remains a significant gap in applying this to the momentum dynamics of a non-differentiable networks.
\citet{LZB20b} establish an interesting connection between solving an over-parametrized non-linear system of equations and solving the classical linear system. They show that for smooth and twice differentiable activation, 
the optimization landscape of an over-parametrized network satisfies a (non-convex) notion called the Polyak-Lokasiewicz (PL) condition \cite{P63}, i.e. $
\frac{1}{2} \| \nabla \ell(w) \|^2 \geq \mu \left( \ell(w) - \ell(w_*) \right)$, where $w_*$ is a global minimizer and $\mu > 0$.
It is not clear whether their result can be extended to ReLU activation, however, and the 
existing result of \cite{DKB18} for the discrete-time Polyak's momentum under the PL condition does not give an accelerated rate nor is it better than that of vanilla GD.
\citet{ADR20} show a \emph{variant} of Polyak's momentum method having an accelerated rate in a \emph{continuous-time} limit for a problem that satisfies PL and has a unique global minimizer. 
It is unclear if their result is applicable to our problem. 
Therefore, showing the advantage of training the ReLU network and the deep linear network by using existing results of Polyak's momentum can be difficult.

To summarize, our contributions 
in the present work include
\begin{itemize}
\item 
In convex optimization, we show
an accelerated linear rate in the non-asymptotic sense for solving the class of the strongly convex quadratic problems via Polyak's momentum
(Theorem~\ref{thm:stcFull}).
We also provide an analysis of the accelerated local convergence  
for the class of functions in $F_{\mu,\alpha}^2$ (Theorem~\ref{thm:STC} in Section~\ref{app:thm:STC}).
We establish a technical result (Theorem~\ref{thm:akv}) that helps to obtain these non-asymptotic rates.
\item 
In non-convex optimization, we show 
accelerated linear rates of the discrete-time Polyak's momentum for training an over-parametrized ReLU network and a deep linear network. (Theorems~\ref{thm:acc} and~\ref{thm:LinearNet})
\end{itemize}
Furthermore, we will develop a modular analysis to show all the results in this work. We identify conditions and propose a meta theorem of acceleration when the momentum method exhibits a certain dynamic, which can be of independent interest.
We show that when applying Polyak's momentum for these problems, the induced dynamics exhibit a form where we can directly apply our meta theorem.

\section{Preliminaries} \label{sec:pre}

Throughout this work, $\| \cdot \|_F$ represents the Frobenius norm and $\| \cdot \|_2$ represents the spectral norm of a matrix, while $\| \cdot \|$ represents $l_2$ norm of a vector. We also denote $\otimes$ the Kronecker product, $\sigma_{\max}(\cdot)=\| \cdot \|_2$ and $\sigma_{\min}(\cdot)$ the largest and the smallest singular value of a matrix respectively.

For the case of training neural networks, we will consider minimizing the squared loss 
\begin{equation}  \label{eq:obj}
\textstyle \ell(W):= \frac{1}{2} \sum_{i=1}^n \big(  y_i -   \N_{W}(x_i)    \big)^2,
\end{equation}
where $x_i \in \reals^{d}$ is the feature vector, $y_i \in \reals^{d_y}$
is the label of sample $i$, and there are $n$ number of samples.
For training the ReLU network, we have
$\N_{W}(\cdot) := \N_W^{\text{ReLU}}(\cdot)$, $d_y = 1$,
and $W:= \{ w^{(r)}  \}_{r=1}^m$,
while for the deep linear network, we have 
$\N_{W}(\cdot) := \N_W^{L\text{-linear}}(\cdot)$,
and $W$ represents the set of all the weight matrices, i.e. $W:= \{ W^{(l)}  \}_{l=1}^L$.
The notation $A^k$ represents the $k_{th}$ matrix power of $A$.

\subsection{Prior result of Polyak's momentum}

Algorithm~\ref{alg:HB1} and Algorithm~\ref{alg:HB2} show two equivalent presentations of gradient descent with Polyak's momentum. Given the same initialization, one can show that 
Algorithm~\ref{alg:HB1} and Algorithm~\ref{alg:HB2} generate exactly the same iterates during optimization. 

Let us briefly describe a prior acceleration result of Polyak's momentum. 
The recursive dynamics of Poylak's momentum for solving the
strongly convex quadratic problems (\ref{obj:strc})
can be written as
\begin{equation} \label{eq:A}
\textstyle
\begin{bmatrix}
w_{t} - w_* \\
w_{t-1} - w_*
\end{bmatrix}
=
\underbrace{
\begin{bmatrix}
I_d - \eta \Gamma + \beta I_d &  - \beta I_d \\
I_d                      &  0_d
\end{bmatrix} }_{:=A}
\cdot
 \begin{bmatrix}
w_{t-1} - w_* \\
w_{t-2} - w_* 
\end{bmatrix},
\end{equation}
where $w_*$ is the unique minimizer.
By a recursive expansion, one can get
\begin{equation} \label{eq:B}
\textstyle
\|
\begin{bmatrix}
w_{t} - w_* \\
w_{t-1} - w_*
\end{bmatrix}
\| \leq \| A^{t} \|_2 \| 
\begin{bmatrix}
w_{0} - w_* \\
w_{-1} - w_*
\end{bmatrix}
\|.
\end{equation}
Hence, it suffices to control the spectral norm of the matrix power
$\| A^{t} \|_2$ for obtaining a convergence rate. In the literature, this is achieved by using 
Gelfand's formula.
\begin{theorem} (\citet{G41}; see also \citet{F18}) (Gelfand's formula) \label{thm:Gelfand}
Let $A$ be a $d \times d$ matrix. Define the spectral radius $\rho(A) := \max_{i \in [d]} | \lambda_i(A)|$, where $\lambda_i(\cdot)$ is the $i_{th}$ eigenvalue.
Then, there exists a non-negative sequence $\{ \epsilon_t \}$ such that 
$\| A^t \|_2 = ( \rho(A) + \epsilon_t)^t $
and
$\lim_{t \rightarrow \infty} \epsilon_t = 0$.
\end{theorem}
We remark that there is a lack of the convergence rate of $\epsilon_t$ in
Gelfand's formula in general.

Denote $\kappa:= \alpha / \mu$ the condition number.
One can control the spectral radius $\rho(A)$
as $\rho(A) \leq  1 - \frac{2}{\sqrt{\kappa}+1}$ by choosing $\eta$ and $\beta$ appropriately,  which leads to the following result.
\begin{theorem} (\citet{P64}; see also \cite{LRP16,R18,M19}) \label{thm:polyak}
Gradient descent with Polyak's momentum with the step size $\eta = \frac{4}{(\sqrt{\mu}+\sqrt{\alpha})^2}$ and the momentum parameter $\beta = \left( 1 - \frac{2}{\sqrt{\kappa}+1} \right)^2$ has
\[
\|
\begin{bmatrix}
w_{t} - w_* \\
w_{t-1} - w_*
\end{bmatrix}
\| \leq (  1 - \frac{2}{\sqrt{\kappa}+1}  + \epsilon_t )^{t}
\begin{bmatrix}
w_{0} - w_* \\
w_{-1} - w_*
\end{bmatrix}
\|,
\]
where $\epsilon_t$ is a non-negative sequence that goes to zero.
\end{theorem}
That is, when $t \rightarrow \infty$, Polyak's momentum has the
 $(  1 - \frac{2}{\sqrt{\kappa}+1})$ rate, which has a better dependency on the condition number $\kappa$ than the
$1-\Theta(\frac{1}{\kappa})$ rate of vanilla gradient descent.
A concern is that the bound is not quantifiable for a finite $t$.
On the other hand,
we are aware of a different analysis 
that leverages Chebyshev polynomials
instead of Gelfand's formula (e.g. \cite{LB18}), which manages to obtain a $t (1- \Theta(\frac{1}{\sqrt{\kappa}}) )^t$ convergence rate.
So the accelerated linear rate is still obtained in an asymptotic sense.
Theorem~9 in \citet{CGZ19} shows a rate
$\max \{\bar{C}_1,  t \bar{C}_2 \} (1- \Theta(\frac{1}{\sqrt{\kappa}})^t )$ 
for some constants $\bar{C}_1$ and $\bar{C}_2$ under the same choice of the momentum parameter and the step size as Theorem~\ref{thm:polyak}. 
However, for a large $t$, the dominant term could be $ t (1- \Theta(\frac{1}{\sqrt{\kappa}})^t)$. 
In this work, we aim at obtaining a bound that (I) holds for a wide range of values of the parameters, (II) has a dependency on the squared root of the condition number $\sqrt{\kappa}$, (III) is quantifiable in each iteration and is better than the rate $t (1- \Theta(\frac{1}{\sqrt{\kappa}}) )^t $.

Finally, we remark that, to our knowledge, the class of the strongly convex quadratic problems is one of the only known examples that Polyak's momentum has a provable \emph{accelerated linear rate} in terms of the \emph{global convergence} in the \emph{discrete-time} setting.
For general smooth, strongly convex, and differentiable functions, a linear rate of the global convergence via discrete-time Polyak's momentum is shown by \citet{GFJ15} and \citet{SDJS18}.
However, the rate is not an accelerated rate and is not better than that of the vanilla gradient descent.

\subsection{(One-layer ReLU network) Settings and Assumptions}
The ReLU activation is not differentiable at zero. So 
for solving (\ref{eq:obj}), 
we will replace the notion of gradient in Algorithm~\ref{alg:HB1} and \ref{alg:HB2} with subgradient
$\frac{ \partial \ell(W_t)}{ \partial w_t^{(r)} }  := \frac{1}{\sqrt{m}} \sum_{i=1}^n \big( \N_{W_t}(x_i) - y_i \big) a_r \cdot \mathbbm{1}[ \langle w_t^{(r)}, x_i \rangle \geq 0]  x_i $ 
and update the neuron $r$ as 
$
w_{t+1}^{(r)} = w_t^{(r)} - \eta \frac{ \partial \ell(W_t)}{ \partial w_t^{(r)} } + \beta \big(  w_t^{(r)} - w_{t-1}^{(r)}  \big).
$

As described in the introduction,
we assume that the smallest eigenvalue of the Gram matrix $\bar{H} \in \reals^{n \times n }$ is strictly positive, i.e. $\lambda_{\min}( \bar{H} ) > 0 $.
We will also denote the largest eigenvalue of the Gram matrix $\bar{H}$ as
$\lambda_{\max}( \bar{H} )$ and
denote the condition number of the Gram matrix as $\kappa := \frac{\lambda_{\max}(\bar{H}) }{ \lambda_{\min}(\bar{H}) }$.
\citet{DZPS19} show that the strict positiveness assumption is indeed mild.
Specifically, they show that if no two inputs are parallel, then the least eigenvalue is strictly positive. 
\citet{PSG2020} were able to provide a quantitative lower bound
under certain conditions.
Following the same framework of \citet{DZPS19},
we consider that each weight vector $w^{(r)} \in \reals^d$ is initialized according to the normal distribution, i.e. $w^{(r)} \sim N(0,I_d)$, 
and each $a_r \in R$ is sampled from the Rademacher distribution, 
i.e. $a_r =1$ with probability 0.5; and $a_r=-1$ with probability $0.5$.
We also assume $\| x_i \| \leq 1$ for all samples $i$.
As the previous works (e.g. \cite{LL18,JT20,DZPS19}), we consider only training the first layer $\{w^{(r)} \}$
and the second layer $\{a_r\}$ is fixed throughout the iterations. We will
denote $u_t \in \reals^n$ whose $i_{th}$ entry is the network's prediction for sample $i$, i.e. $u_t[i] := \N_{W_t}^{\text{ReLU}}(x_i)$ in iteration $t$
and denote $y \in \reals^n$ the vector whose $i_{th}$ element is the label of sample $i$. 
The following theorem
is a prior result due to \citet{DZPS19}.

\begin{theorem} (Theorem 4.1 in \citet{DZPS19}) \label{thm:du}
Assume that $\lambda := \lambda_{\min}( \bar{H}) / 2 > 0$ and that $w_0^{(r)} \sim N(0,I_d)$ and $a_r$ uniformly sampled from $\{-1,1\}$. 
Set the number of nodes $m = \Omega( \lambda^{-4} n^6 \delta^{-3})$ and the constant step size $\eta = O(\frac{\lambda}{n^2})$.
Then, with probability at least $1-\delta$ over the random initialization,
vanilla gradient descent, i.e. Algorithm~\ref{alg:HB1}\&~\ref{alg:HB2} with $\beta=0$, has
\[ \|
u_t - y 
\|^2 \leq \left( 1 - \eta  \lambda \right)^{t}
\cdot
\|
u_0 - y
\|^2.
\]
\end{theorem}
Later \citet{ZY19} improve the network size $m$ to
$m =\Omega( \lambda^{-4} n^4 \log^3 ( n / \delta))$.
\citet{WDW19} provide an improved analysis over \citet{DZPS19}, which shows that the step size $\eta$ of vanilla gradient descent can be set as $\eta =  \frac{1}{c_1 \lambda_{\max}(\bar{H})}$ for some quantity $c_1>0$. The result in turn leads to a convergence rate 
$( 1 - \frac{1}{c_2 \kappa}  )$ for some quantity $c_2>0$. 
However, the quantities $c_1$ and $c_2$ are not universal constants and actually depend on the problem parameters $\lambda_{\min}(\bar{H})$, $n$, and $\delta$.
A question that we will answer in this work is ``\textit{Can Polyak's momentum achieve an accelerated linear rate $\left( 1- \Theta(\frac{1}{\sqrt{\kappa}}) \right)$, where the factor $\Theta(\frac{1}{\sqrt{\kappa}})$ does not depend on any other problem parameter?}''. 

\begin{figure}[t]
  \centering
    \includegraphics[width=0.6\textwidth]{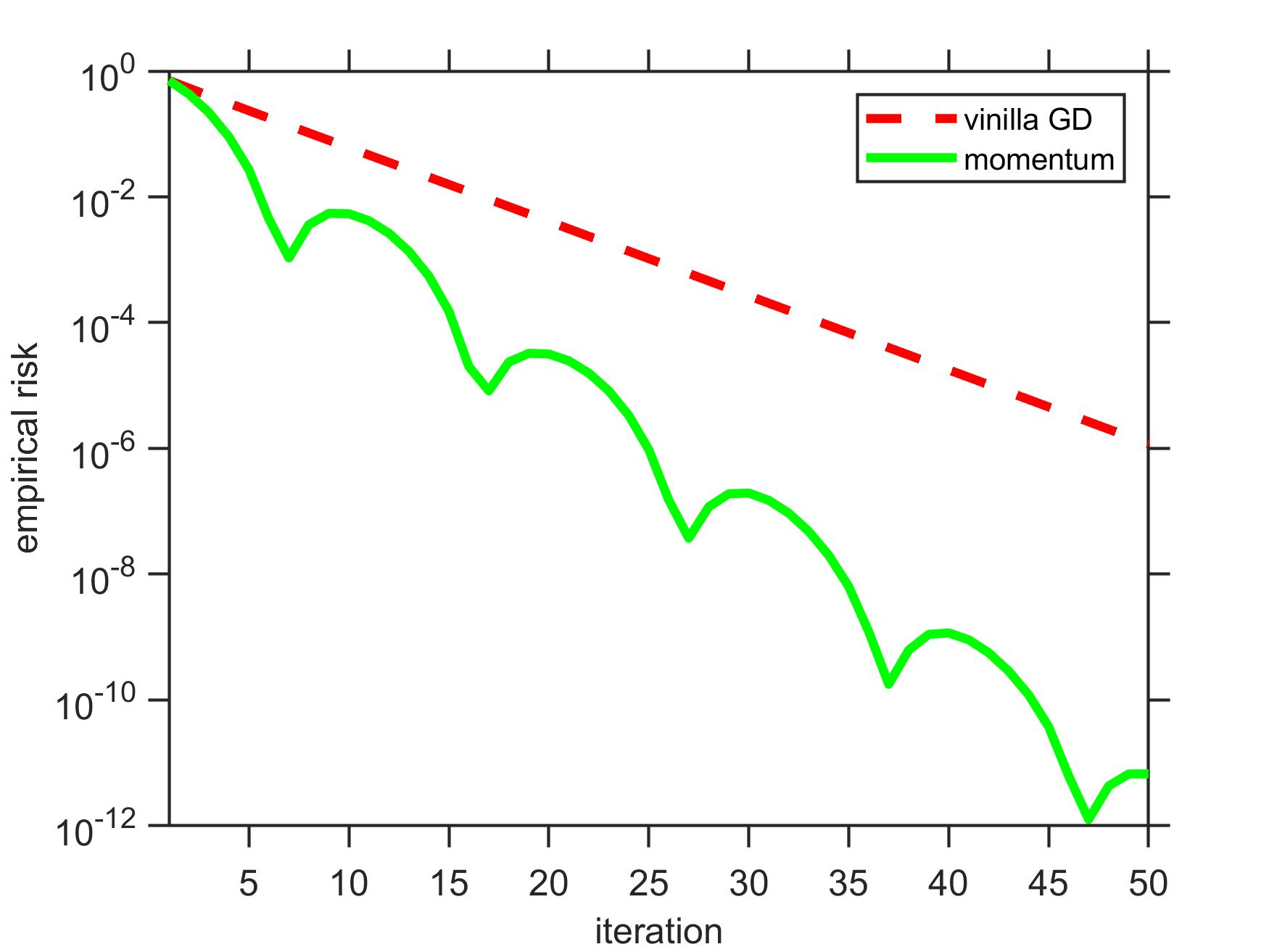}
    \caption{Empirical risk $\ell(W_t)$ vs. iteration $t$. Polyak's momentum accelerates the optimization process of training an over-parametrized one-layer ReLU network. Experimental details are available in Section~\ref{app:exp}. }
        \label{fig:exp} 
\end{figure}

\subsection{(Deep Linear network) Settings and Assumptions}

For the case of deep linear networks, we will denote  
$X := [x_1,\dots, x_n] \in \reals^{d \times n}$
the data matrix
 and $Y:= [y_1, \dots, y_n] \in \reals^{d_y \times n}$  
 the corresponding label matrix. 
We will also
denote $\bar{r} := rank(X)$ and the condition number $\kappa := \frac{\lambda_{\max}( X^\top X)}{ \lambda_{\bar{r}}(X^\top X)}$.
Following \citet{HXP20},
we will assume that the linear network
is initialized by
the orthogonal initialization,
which is conducted by sampling uniformly from (scaled) orthogonal matrices such that
 $(\W{1}_0)^\top \W{1}_0 = m I_{d}$,
 $\W{L}_0 (\W{L}_0)^\top  = m I_{d_y}$, and $(\W{l}_0)^\top \W{l}_0= \W{l}_0 (\W{l}_0)^\top  = m I_{m}$
for layer $2 \leq l \leq L-1$. 
We will denote $\W{j:i} := W_j W_{j-1} \cdots W_i = \Pi_{l=i}^j W_l$, where $1 \leq i \leq j \leq L$ and $\W{i-1:i} = I$. We also denote the network's output 
$\textstyle
U := \frac{1}{\sqrt{m^{L-1} d_y} } \W{L:1}X \in \reals^{d_y \times n}.
$

In our analysis, following \citet{DH19}, \citet{HXP20},
we will further assume that (A1) there exists a $W^*$ such that $Y = W^* X$, $X \in \reals^{d \times \bar{r}}$, and $\bar{r}=rank(X)$, which is actually without loss of generality (see e.g. the discussion in Section B of \citet{DH19}).
\begin{theorem} (Theorem 4.1 in \citet{HXP20}) \label{thm:hu}
Assume (A1) and the use of the orthogonal initialization.
Suppose the width of the deep linear network satisfies $m 
\geq C \frac{\|X \|^2_F}{\sigma^2_{\max}(X)}$ $ \kappa^2 \big( d_y (1 + \| W_* \|^2_2)$ $ + \log (\bar{r} /\delta)  \big)$ and $m \geq \max \{d_x,d_y \}$ 
for some $\delta \in (0,1)$ and a sufficiently large constant $C> 0$.
Set the constant step size $\eta = \frac{d_y}{2 L \sigma^2_{\max}(X) }$.
Then, with probability at least $1-\delta$ over the random initialization,
vanilla gradient descent, i.e. Algorithm~\ref{alg:HB1}\&~\ref{alg:HB2} with $\beta=0$, has
\[\|
U_t - Y 
\|^2_F \leq \left( 1 - \Theta(\frac{1}{\kappa}) \right)^{t}
\cdot
\|
U_0 - Y
\|^2_F.
\]
\end{theorem}

\section{Modular Analysis} \label{sec:meta}

In this section,
we will provide a meta theorem for the following dynamics of the residual vector $\xi_t \in \reals^{n_0} $,
\begin{equation} \label{eq:meta}
\begin{split}
\begin{bmatrix}
\xi_{t+1} \\
\xi_{t} 
\end{bmatrix}
 &
=
\begin{bmatrix}
I_{n_0} - \eta H + \beta I_{n_0} & - \beta  I_{n_0}   \\
I_{n_0} & 0_{n_0} 
\end{bmatrix}
\begin{bmatrix}
\xi_{t} \\
\xi_{t-1} 
\end{bmatrix}
+
\begin{bmatrix}
\varphi_t \\ 0_{n_0}
\end{bmatrix}
,
\end{split}
\end{equation}
where $\eta$ is the step size, $\beta$ is the momentum parameter,
$H \in \reals^{n_0 \times n_0}$ is a PSD matrix, $\varphi_t \in \reals^{n_0}$ is some vector, and $I_{n_0}$ is the $n_0 \times n_0$-dimensional identity matrix. Note that $\xi_t$ and $\varphi_t$ depend on the underlying model learned at iteration $t$, i.e. depend on $W_t$.

We first show that the residual dynamics of Polyak's momentum for solving all the four problems in this work are in the form of \eqref{eq:meta}.
The proof of the following lemmas (Lemma~\ref{lem:SC-residual},~\ref{lem:ReLU-residual}, and~\ref{lem:DL-residual}) are available in Section~\ref{app:sec:meta}.

\subsection{Realization: Strongly convex quadratic problems} \label{inst:stc}

One can easily see that the dynamics of Polyak's momentum (\ref{eq:A}) for solving the strongly convex quadratic problem (\ref{obj:strc}) is in the form of (\ref{eq:meta}). We thus have the following lemma.

\begin{lemma} \label{lem:stc-residual} 
Applying Algorithm~\ref{alg:HB1} or Algorithm~\ref{alg:HB2} to solving the class of strongly convex quadratic problems (\ref{obj:strc})  induces a residual dynamics in the form of 
(\ref{eq:meta}), where 
\[
\begin{aligned}
\xi_t & = w_t - w_* \quad  \text{and hence } n_0 = d
\\ H     & = \Gamma,
\\ \varphi_t & = 0_d.
\end{aligned}
\]
\end{lemma}

\subsection{Realization: Solving $F_{\mu,\alpha}^2$} \label{inst:SC}
A similar result holds for optimizing functions in $F_{\mu,\alpha}^2$.
\begin{lemma} \label{lem:SC-residual} 
Applying Algorithm~\ref{alg:HB1} or Algorithm~\ref{alg:HB2} to minimizing a function $f(w) \in F_{\mu,\alpha}^2$ induces a residual dynamics in the form of 
(\ref{eq:meta}), where 
\[
\begin{aligned}
\xi_t & = w_t - w_*
\\ H     & = \int_0^1 \nabla^2 f\big( (1-\tau) w_0 + \tau w_* \big) d \tau
\\ \varphi_t & = \eta \left( \int_0^1 \nabla^2 f\big( (1-\tau) w_0 + \tau w_* \big) d \tau -  \int_0^1 \nabla^2 f\big( (1-\tau) w_t + \tau w_* \big) d \tau \right) (w_t - w_*),
\end{aligned}
\]
where $w_*:=\arg\min_w f(w)$.
\end{lemma}

\subsection{Realization: One-layer ReLU network} \label{inst:ReLU}

\noindent
\textbf{More notations:}
For the analysis,
let us define the event 
$A_{ir} := \{  \exists w \in \reals^d: \| w - w_0^{(r)} \| \leq R^{\text{ReLU}}, \mathbbm{1} \{ x_i^\top w_0^{(r)}  \} \neq \mathbbm{1} \{ x_i^\top w \geq 0  \}  \},$
where $R^{\text{ReLU}}>0$ is a number to be determined later. The event $A_{ir}$ means that there exists a $w \in \reals^d$ which is within the $R^{\text{ReLU}}$-ball centered at the initial point $w_0^{(r)}$ such that its activation pattern of sample $i$ is different from that of $w_0^{(r)}$.
We also denote a random set $S_i := \{ r \in [m]: \mathbbm{1}\{ A_{ir} \}  = 0 \}$ 
and its complementary set $S_i^\perp := [m] \setminus S_i$.

Lemma~\ref{lem:ReLU-residual} below shows that training the ReLU network $\N_W^{\text{-ReLU}}(\cdot)$ via momentum induces the residual dynamics in the form of (\ref{eq:meta}).


\begin{lemma} \label{lem:ReLU-residual} (Residual dynamics of training the ReLU network $\N_W^{\text{ReLU}}(\cdot)$)
Denote 
\[
\begin{aligned}
 & 
  (H_t)_{i,j}:= H(W_t)_{i,j} = \frac{1}{m} \sum_{r=1}^m x_i^\top x_j 
 \times \mathbbm{1}\{ \langle w^{(r)}_t, x_i \rangle \geq 0 \text{ } \&  \text{ }   \langle w^{(r)}_t, x_j \rangle \geq 0 \}.
\end{aligned}
\]
Applying Algorithm~\ref{alg:HB1} or Algorithm~\ref{alg:HB2} to (\ref{eq:obj}) for training the ReLU network 
$\N_W^{\text{ReLU}}(x)$ induces a residual dynamics in the form of 
(\ref{eq:meta}) such that
\[
\begin{aligned}
\xi_t[i] &= \N_{W_t}^{\text{ReLU}}(x_i) - y_i \quad  (\text{and hence } n_0 = d)
\\ H     & = H_0
\\ \varphi_t & = \phi_t + \iota_t,
\end{aligned}
\]
where 
each element $i$ of $\xi_t \in \reals^n$ is the residual error of the sample $i$, and
the $i_{th}$-element of $\phi_t\in \reals^n$ satisfies \[ \textstyle |\phi_t[i]| \leq \frac{ 2 \eta \sqrt{n} |S_i^\perp|}{ m } \big(  \| u_t - y \| + \beta \sum_{s=0}^{t-1} \beta^{t-1-s}  \| u_s - y \| \big),\]
and $\iota_t = \eta  \left( H_0 - H_t \right) \xi_t \in \reals^n $.
\end{lemma}


\subsection{Realization: Deep Linear network} \label{inst:linear}

Lemma~\ref{lem:DL-residual} below shows that the residual dynamics due to Polyak's momentum
for training the deep linear network 
is indeed in the form of (\ref{eq:meta}). 
In the lemma, ``$\v$'' stands for the vectorization of the underlying matrix in column-first order. 

\begin{lemma} \label{lem:DL-residual} (Residual dynamics of training $\N_W^{L\text{-linear}}(\cdot)$)
Denote 
$M_{t,l}$ the momentum term of layer $l$ at iteration $t$, which is recursively defined as
$M_{t,l} = \beta M_{t,l-1} + \frac{ \partial \ell(\W{L:1}_t)}{ \partial \W{l}_t } $. Denote
\[
\begin{aligned}
& 
H_t \textstyle := \frac{1}{ m^{L-1} d_y } \sum_{l=1}^L [ (\W{l-1:1}_t X)^\top (\W{l-1:1}_t X ) 
\otimes
  \W{L:l+1}_t (\W{L:l+1}_t)^\top ]   \in \reals^{d_y n \times d_y n}.
\end{aligned}
\]
Applying Algorithm~\ref{alg:HB1} or Algorithm~\ref{alg:HB2} to (\ref{eq:obj}) for training the deep linear network 
$\N_W^{L\text{-linear}}(x)$ induces a residual dynamics in the form of 
(\ref{eq:meta}) such that
\[
\begin{aligned}
\xi_t &    = \text{vec}(U_t - Y) \in \reals^{d_y n} \text{, and hence } n_0 = d_y n \\
 H & =  H_0  \\
\varphi_t & = \phi_t + \psi_t  + \iota_t \in \reals^{d_y n},
\end{aligned}
\]
where
\[
\begin{aligned}
 \phi_t &  = \frac{1}{\sqrt{m^{L-1} d_y} } \v\left( \Phi_t X\right) \text{ with }
\Phi_t    = \Pi_l ( \W{l}_t - \eta M_{t,l} )
- \W{L:1}_t  +  \eta \sum_{l=1}^L \W{L:l+1}_t M_{t,l} \W{l-1:1}_t
\\   \psi_t & =\frac{1}{\sqrt{m^{L-1} d_y} } 
\v\left( (L-1) \beta \W{L:1}_{t}  X + \beta  \W{L:1}_{t-1} X 
 - \beta \sum_{l=1}^L \W{L:l+1}_t \W{l}_{t-1} \W{l-1:1}_{t} X \right)
\\ \iota_t  & = \eta (H_0 - H_t) \xi_t.
\end{aligned}
\]
\end{lemma}

\subsection{A key theorem of bounding a matrix-vector product}

Our meta theorem of acceleration will be based on 
Theorem~\ref{thm:akv} in the following, which upper-bounds the size of the matrix-vector product of a matrix power $A^k$ and a vector $v_0$.
Compared to Gelfand's formula (Theorem~\ref{thm:Gelfand}),
Theorem~\ref{thm:akv} below provides a better control of the size of the matrix-vector product, 
since it avoids the dependency on the unknown sequence $\{ \epsilon_t\}$.
The result can be of independent interest and might be useful for analyzing Polyak's momentum for other problems in future research. 

\begin{theorem} \label{thm:akv}
Let $A:=
\begin{bmatrix} 
(1 + \beta) I_n - \eta  H  & - \beta I_n \\
I_n & 0
\end{bmatrix}
\in \reals^{2 n \times 2 n}$. 
Suppose that $H \in \reals^{n \times n}$ is a positive semidefinite matrix. 
Fix a vector $v_0 \in \reals^n$.
If $\beta$ is chosen to satisfy 
$1 \geq \beta >  \max \{ \left( 1 - \sqrt{\eta \lambda_{\min}(H) } \right)^2, \left( 1 - \sqrt{\eta \lambda_{\max}(H)} \right)^2 \} $
then
\begin{equation}
\| A^k v_0 \| \leq \big( \sqrt{ \beta } \big)^k C_0 \| v_0 \|, 
\end{equation}
where the constant 
\begin{equation} \label{C_0}
C_0:= \frac{\sqrt{2} (\beta+1)}{
\sqrt{ \min\{  
h(\beta,\eta \lambda_{\min}(H)) , h(\beta,\eta \lambda_{\max}(H)) \} } }\geq 1,
\end{equation} 
and the function $h(\beta,z)$ is defined as
$  h(\beta,z):=-\left(\beta-\left(1-\sqrt{z}\right)^2\right)\left(\beta-\left(1+\sqrt{z}\right)^2\right).  $
\end{theorem}

Note that the constant $C_0$ in Theorem~\ref{thm:akv} depends on $\beta$ and $\eta H$. It should be written as $C_0(\beta,\eta H)$ to be precise. 
However, for the brevity, we will simply denote it as $C_0$ when the underlying choice of $\beta$ and $\eta H$ is clear from the context. 
The proof of Theorem~\ref{thm:akv} is available in Section~\ref{app:thm:akv}.
Theorem~\ref{thm:akv} allows us to derive a concrete upper bound of the residual errors
in each iteration of momentum, and consequently allows us to show an accelerated linear rate in the non-asymptotic sense. 
The favorable property of the bound will also help to analyze Polyak's momentum for training the neural networks. As shown later in this chapter, we will need to guarantee the progress of Polyak's momentum in each iteration, which is not possible if we only have a quantifiable bound in the limit. 
Based on Theorem~\ref{thm:akv}, we have the following corollary. The proof is in Section~\ref{app:corr}.
\begin{corollary}  \label{corr:1}
Assume that $\lambda_{\min}(H) > 0$.
Denote $\kappa:= \lambda_{\max}(H) / \lambda_{\min}(H)$.
Set $\eta = 1 / \lambda_{\max}(H)$ 
and set $\beta = \left( 1 - \frac{1}{2} \sqrt{\eta \lambda_{\min}(H)} \right)^2 = \left( 1 - \frac{1}{2 \sqrt{\kappa}} \right)^2$.
Then, $C_0 \leq  4 \sqrt{\kappa}$.
\end{corollary}


\subsection{Meta theorem}

Let $\lambda >0$ be the smallest eigenvalue of the matrix $H$ that appears on the residual dynamics (\ref{eq:meta}).
Our goal is to show that the residual errors satisfy
\begin{equation} \label{assump:eq}
\textstyle
\left\|
\begin{bmatrix}
\xi_{s} \\
\xi_{s-1} 
\end{bmatrix}
\right\|
\leq
\left(\sqrt{\beta } + \mathbbm{1}_{\varphi} C_2 \right)^{s} (C_{0} + \mathbbm{1}_{\varphi} C_1) 
\left\|
\begin{bmatrix}
\xi_{0} \\
\xi_{-1} 
\end{bmatrix}
\right\|,
\end{equation}
where $C_0$ is the constant defined on (\ref{C_0}),
and $C_1,C_2 \geq 0$ are some constants,
$\mathbbm{1}_{\varphi}$ is an indicator if any $\varphi_t$ on the residual dynamics (\ref{eq:meta}) is a non-zero vector.
For the case of training the neural networks, we have 
$\mathbbm{1}_{\varphi}= 1$. 

\begin{theorem} \label{thm:metas} (Meta theorem for the residual dynamics (\ref{eq:meta}))
Assume that
the step size $\eta$ and
the momentum parameter $\beta$ satisfying 
$1 \geq \beta >  \max \{ \left( 1 - \sqrt{\eta \lambda_{\min}(H) } \right)^2, \left( 1 - \sqrt{\eta \lambda_{\max}(H)} \right)^2 \} $
are set appropriately so that (\ref{assump:eq})
holds at iteration $s=0,1,\dots,t-1$
implies that
\begin{equation} \label{eq:thm1}
\textstyle
\| \sum_{s=0}^{t-1} A^{t-s-1} \begin{bmatrix}
\varphi_s \\
0 
\end{bmatrix}
\|
\leq \left(\sqrt{\beta} + \mathbbm{1}_{\varphi} C_2 \right)^{t}
C_3
\left\|
\begin{bmatrix}
\xi_{0} \\
\xi_{-1} 
\end{bmatrix}
\right\|.
\end{equation}
Then, we have
\begin{equation} \label{eq:thm2}
\textstyle
\left\|
\begin{bmatrix}
\xi_{t} \\
\xi_{t-1} 
\end{bmatrix}
\right\|
\leq
\left( \sqrt{\beta}  + \mathbbm{1}_{\varphi} C_2 \right)^{t} 
(C_0 + \mathbbm{1}_{\varphi} C_1) 
\left\|
\begin{bmatrix}
\xi_{0} \\
\xi_{-1} 
\end{bmatrix}
\right\|,
\end{equation}
holds for all $t$, where $C_0$ is defined on (\ref{C_0}) and $C_1,C_2,C_3 \geq 0$ are some constants satisfying:
\begin{equation} \label{eq:C-meta}
\begin{aligned}
& \textstyle 
\left( \sqrt{\beta} \right)^{t} C_0 +
\left( \sqrt{\beta} + \mathbbm{1}_{\varphi} C_2 \right)^{t} \mathbbm{1}_{\varphi} C_3 \leq 
\\ & 
\textstyle \qquad \qquad \qquad \qquad 
\left( \sqrt{\beta} + \mathbbm{1}_{\varphi} C_2 \right)^{t} (C_0 + \mathbbm{1}_{\varphi} C_1) .
\end{aligned}
\end{equation}
\end{theorem}

\begin{proof}
The proof is by induction. At $s=0$,  (\ref{assump:eq}) holds
since $C_0 \geq 1$ by Theorem~\ref{thm:akv}. Now assume that the inequality holds at $s=0,1,\dots, t-1$.
Consider iteration $t$.
Recursively expanding the dynamics (\ref{eq:meta}), we have 
\begin{equation} \label{eq:1-meta}
\begin{bmatrix}
\xi_{t} \\
\xi_{t-1} 
\end{bmatrix}
=  
A^t
\begin{bmatrix}
\xi_{0} \\
\xi_{-1} 
\end{bmatrix}
+ \sum_{s=0}^{t-1}  A^{t-s-1} \begin{bmatrix}
\varphi_s \\
0 
\end{bmatrix}
.
\end{equation}
By Theorem~\ref{thm:akv},
the first term on the r.h.s. of (\ref{eq:1-meta}) can be bounded by 
\begin{equation} \label{eq:2-meta}
\| A^t
\begin{bmatrix}
\xi_{0} \\
\xi_{-1} 
\end{bmatrix}
\| \leq \left( \sqrt{\beta} \right)^t C_0 \| \begin{bmatrix}
\xi_{0} \\
\xi_{-1} 
\end{bmatrix}
\|
\end{equation}
By assumption, given (\ref{assump:eq}) holds at $s=0,1,\dots, t-1$, we have (\ref{eq:thm1}).
Combining (\ref{eq:thm1}), (\ref{eq:C-meta}), (\ref{eq:1-meta}), and (\ref{eq:2-meta}),
we have (\ref{eq:thm2}) and hence the proof is completed.

\end{proof}

\noindent
\textbf{Remark:} As shown in the proof, we need the residual errors be tightly bounded as (\ref{assump:eq}) in each iteration.
Theorem~\ref{thm:akv} is critical for establishing the desired result.
On the other hand, it would become tricky if instead we use Gelfand's formula or other techniques 
in the related works
that lead to a convergence rate in the form of $O(t \theta^{t})$.

\section{Main results} \label{sec:main}

The important lemmas and theorems in the previous section help to show our main results in the following subsections. 
The high-level idea to obtain the results is by using the meta theorem (i.e. Theorem~\ref{thm:metas}). Specifically, we will need to show that if the underlying residual dynamics satisfy (\ref{assump:eq}) for all the previous iterations, then
the terms $\{ \varphi_s \}$ in the dynamics satisfy
 (\ref{eq:thm1}). 
This condition trivially holds for the case of the quadratic problems, since there is no such term.  
On the other hand,
for solving the other problems, we need to carefully show that the condition holds. 
For example, 
according to Lemma~\ref{lem:ReLU-residual},
showing acceleration for the ReLU network will require bounding terms like $\| (H_0 - H_s) \xi_s \|$ (and other terms as well), where
$H_0-H_s$ corresponds to the difference of the kernel matrix at two different time steps. By controlling the width of the network, we can guarantee that the change is not too much.
A similar result can be obtained for the problem of the deep linear network.
The high-level idea is simple but the analysis
of the problems of the neural networks can be tedious.

\subsection{Non-asymptotic accelerated linear rate for solving strongly convex quadratic problems}

\begin{theorem} \label{thm:stcFull}
Assume the momentum parameter $\beta$ satisfies
$1 \geq \beta >  \max \{ \left( 1 - \sqrt{\eta \mu } \right)^2, \left( 1 - \sqrt{\eta \alpha} \right)^2 \} $
Gradient descent with Polyak's momentum for solving (\ref{obj:strc}) has
\begin{equation} \label{eq:qqq0}
\|
\begin{bmatrix}
w_{t} - w_* \\
w_{t-1} - w_*
\end{bmatrix}
\| \leq \left(  \sqrt{\beta}  \right)^{t} C_0
\|
\begin{bmatrix}
w_{0} - w_* \\
w_{-1} - w_*
\end{bmatrix}
\|,
\end{equation}
where the constant $C_0$ is defined as
\begin{equation}
C_0:=\frac{\sqrt{2} (\beta+1)}{
\sqrt{ \min\{  
h(\beta,\eta \mu) , h(\beta,\eta \alpha) \} } } \geq 1,
\end{equation} 
and 
$   h(\beta,z)=-\left(\beta-\left(1-\sqrt{z}\right)^2\right)\left(\beta-\left(1+\sqrt{z}\right)^2\right).$  
Consequently, if the step size $\eta = \frac{1}{\alpha}$ and the momentum parameter $\beta = \left(1 - \frac{1}{2\sqrt{\kappa}}\right)^2$, then it has
\begin{equation}\label{eq:qqq1}
\|
\begin{bmatrix}
w_{t} - w_* \\
w_{t-1} - w_*
\end{bmatrix}
\| \leq \left(  1 - \frac{1}{2 \sqrt{\kappa}}   \right)^{t} 4 \sqrt{\kappa}
\|
\begin{bmatrix}
w_{0} - w_* \\
w_{-1} - w_*
\end{bmatrix}
\|.
\end{equation}
Furthermore, if $\eta = \frac{4}{(\sqrt{\mu}+\sqrt{\alpha})^2}$ 
and $\beta$ approaches $\beta \rightarrow \left( 1 - \frac{2}{\sqrt{\kappa}+1} \right)^2$ from above, then it has a convergence rate approximately
$ \left(  1 - \frac{2}{\sqrt{\kappa} + 1}   \right)$
as $t \rightarrow \infty$.
\end{theorem}
The convergence rates shown on (\ref{eq:qqq0}) and (\ref{eq:qqq1}) do not depend on the unknown sequence $\{\epsilon_t\}$. Moreover, the rates depend on the squared root of the condition number $\sqrt{\kappa}$.
We have hence
established a non-asymptotic accelerated linear rate of Polyak's momentum, which helps to show the advantage of Polyak's momentum over vanilla gradient descent in the finite $t$ regime.
Our result also recovers the rate
$\left(  1 - \frac{2}{\sqrt{\kappa}+1} \right)$ asymptotically under the same choices of the parameters as the previous works.
The detailed proof can be found
in Section~\ref{app:sec:stc}, which is actually a 
 trivial application of
Lemma~\ref{lem:stc-residual}, Theorem~\ref{thm:metas}, and Corollary~\ref{corr:1} with $C_1=C_2=C_3=0$.

In Section~\ref{app:thm:STC} (Theorem~\ref{thm:STC}), we also provide a local acceleration result for general smooth strongly convex and twice differentiable function $F_{\mu,\alpha}^2$ of the discrete-time Polyak's momentum.

\subsection{Acceleration for training $\N_W^{\text{ReLU}}(x)$}

Before introducing our result, we need the following lemma.
\begin{lemma} \label{lem:ReLU-A} 
[Lemma~3.1 in \citet{DZPS19} and \citet{ZY19}] 
Set $m= \Omega( \lambda^{-2} n^2 \log ( n / \delta )  )$. 
Suppose that the neurons $w^{(1)}_0, \dots, w^{(m)}_0$ are i.i.d. generated by $N(0,I_d)$ initially.
Then,
with probability at least $1-\delta$,
it holds that
\[
\begin{aligned}
& \qquad \qquad \qquad \qquad \| H_0 - \bar{H}\|_F \leq \frac{\lambda_{\min}( \bar{H} )}{4}, 
\\ & 
\lambda_{\min}\big(H_0\big) \geq \frac{3}{4}\lambda_{\min}( \bar{H} ),
\qquad 
\lambda_{\max}\big(H_0\big) \leq \lambda_{\max}( \bar{H} ) + \frac{ \lambda_{\min}( \bar{H} )}{4}.
\end{aligned}
\]
\end{lemma}

Lemma~\ref{lem:ReLU-A} shows that by the random initialization, with probability $1-\delta$, the least eigenvalue of the Gram matrix $H : = H_0$ defined in Lemma~\ref{lem:ReLU-residual} is lower-bounded and the largest eigenvalue is also close to $\lambda_{\max}( \bar{H} )$. 
Furthermore, Lemma~\ref{lem:ReLU-A} implies that
the condition number of the Gram matrix $H_0$ at the initialization $\hat{\kappa}:= \frac{\lambda_{\max}(H_0) }{\lambda_{\min}(H_0) }$ satisfies
\begin{equation}
\hat{\kappa} 
\leq \frac{4}{3} \kappa + \frac{1}{3},
\end{equation}
where $\kappa:= \frac{\lambda_{\max}(\bar{H}) }{\lambda_{\min}(\bar{H}) }$.

\begin{theorem} \label{thm:acc}
(One-layer ReLU network $\N_W^{\text{ReLU}}(x)$)
Assume that $\lambda := \frac{3\lambda_{\min}( \bar{H})}{4} > 0$ and that $w_0^{(r)} \sim N(0,I_d)$ and $a_r$ uniformly sampled from $\{-1,1\}$.
Denote $\lambda_{\max} := \lambda_{\max} (\bar{H}) + \frac{\lambda_{\min}(\bar{H})}{4}$ and denote $\hat{\kappa} := \lambda_{max} / \lambda = (4 \kappa + 1) / 3 $. 
Set a constant step size $\eta = \frac{1}{ \lambda_{\max} }$,
fix momentum parameter $\beta  
= \big( 1 - \frac{1}{2 \hat{\kappa} }  \big)^2$, and finally set the number of network nodes $m = \Omega( \lambda^{-4} n^{4} \kappa^2 \log^3 ( n / \delta))$.
Then, with probability at least $1-\delta$ over the random initialization,
gradient descent with Polyak's momentum satisfies for any $t$,
\begin{equation} \label{eq:thm0}
\left\|
\begin{bmatrix}
\xi_{t}  \\
\xi_{t-1}  
\end{bmatrix}
\right\| \leq \left( 1 - \frac{1}{4 \sqrt{\hat{\kappa} } }  \right)^{t}
\cdot 8  \sqrt{ \hat{\kappa} }
\left\|
 \begin{bmatrix}
\xi_0 \\
\xi_{-1}
\end{bmatrix}
\right\|.
\end{equation}
\end{theorem}

We remark that $\hat{\kappa}$, which is the condition number of the Gram matrix $H_0$, 
is within a constant factor of the condition number of $\bar{H}$ (recall that $\kappa := \frac{\lambda_{\max}(\bar{H} )}{  \lambda_{\min}(\bar{H}) }$).
Therefore, Theorem~\ref{thm:acc} essentially shows an accelerated linear rate 
$\left( 1 - \Theta ( \frac{1}{\sqrt{\kappa}}) \right)$.
The rate has an improved dependency on the condition number, i.e. $\sqrt{\kappa}$ instead of $\kappa$, which shows the advantage of Polyak's momentum over vanilla GD when the condition number is large.  
We believe this is an interesting result, as the acceleration is akin to that in convex optimization, e.g. \cite{N13,SDJS18}.

Our result also implies that
over-parametrization helps acceleration in optimization.
To our knowledge, in the literature, there is little theory of understanding why over-parametrization can help training a neural network faster.
The only exception that we are aware of is \citet{ACH18}, which shows that the dynamic of vanilla gradient descent for an over-parametrized objective function exhibits some momentum terms, although their message is very different from ours. The proof of Theorem~\ref{thm:acc} is in Section~\ref{app:sec:relu}.

\subsection{Acceleration for training $\N_W^{L\text{-linear}}(x)$}

\begin{theorem} \label{thm:LinearNet}
(Deep linear network $\N_W^{L\text{-linear}}(x)$)
Assume (A1) and denote $\lambda:= \frac{L \sigma_{\min}^2(X)} {d_y}$.  
Set a constant step size $\eta = \frac{d_y}{L \sigma_{\max}^2(X)}$,
fix momentum parameter $\beta 
= \big( 1 - \frac{1}{2 \sqrt{\kappa}} \big)^2 $, and finally set a parameter $m$ that controls the width $m \geq C \frac{ \kappa^5 }{\sigma^2_{\max}(X)} \left( 
d_y( 1 + \|W^* \|^2_2) + \log ( \bar{r} / \delta)  \right)$ and $m \geq \max \{ d_x, d_y\}$ for some constant $C>0$.
Then, with probability at least $1-\delta$ over the random orthogonal initialization,
gradient descent with Polyak's momentum satisfies for any $t$,
\begin{equation} \label{eq:thm-LinearNet}
\left\|
\begin{bmatrix}
\xi_t \\
\xi_{t-1} 
\end{bmatrix}
\right\| \leq \left( 1 - \frac{1}{4 \sqrt{\kappa} } \right)^{t}
\cdot 8 \sqrt{\kappa}
\left\|
 \begin{bmatrix}
\xi_0 \\
\xi_{-1}
\end{bmatrix}
\right\|.
\end{equation}

\end{theorem}

Compared with Theorem~\ref{thm:hu} of \citet{HXP20} for vanilla GD, our result clearly shows the acceleration via Polyak's momentum,
as it improves the dependency of the condition number
to $\sqrt{\kappa}$ (recall that $\kappa := \frac{\sigma_{\max}^2(x)}{\sigma_{\min}^2(x)}$ in this case).
Furthermore, the result suggests that the depth does not hurt optimization. Acceleration is achieved for any depth $L$ and the required width $m$
is independent of the depth $L$ as \cite{HXP20,ZLG20} (of vanilla GD).
The proof of Theorem~\ref{thm:LinearNet} is in Section~\ref{app:sec:linear}.

\section{Detailed proofs}

\subsection{Proof of Lemma~\ref{lem:SC-residual}, Lemma~\ref{lem:ReLU-residual},
and Lemma~\ref{lem:DL-residual}
 } \label{app:sec:meta}

\noindent
\textbf{Lemma}~\ref{lem:SC-residual}:
\textit{ 
Applying Algorithm~\ref{alg:HB1} or Algorithm~\ref{alg:HB2} to minimizing a function $f(w) \in F_{\mu,\alpha}^2$ induces a residual dynamics in the form of 
(\ref{eq:meta}), where 
\[
\begin{aligned}
\xi_t & = w_t - w_*
\\ H     & = \int_0^1 \nabla^2 f\big( (1-\tau) w_0 + \tau w_* \big) d \tau
\\ \varphi_t & = \eta \left( \int_0^1 \nabla^2 f\big( (1-\tau) w_0 + \tau w_* \big) d \tau -  \int_0^1 \nabla^2 f\big( (1-\tau) w_t + \tau w_* \big) d \tau \right) (w_t - w_*),
\end{aligned}
\]
where $w_*:=\arg\min_w f(w)$.
}

\begin{proof}
We have
\begin{equation} 
\begin{aligned}
& \begin{bmatrix}
w_{t+1} - w_* \\
w_{t} - w_*
\end{bmatrix}
 =
\begin{bmatrix}
I_d  + \beta I_d &  - \beta I_d \\
I_d                      &  0_d
\end{bmatrix}
\cdot
 \begin{bmatrix}
w_{t} - w_* \\
w_{t-1} - w_* 
\end{bmatrix}
+ 
 \begin{bmatrix}
-\eta \nabla f(w_t) \\
 0 
\end{bmatrix}
\\ & =
\begin{bmatrix}
I_d  - \eta \int_0^1 \nabla^2 f\big( (1-\tau) w_t + \tau w_* \big) d \tau + \beta I_d &  - \beta I_d \\
I_d                      &  0_d
\end{bmatrix}
\cdot
 \begin{bmatrix}
w_{t} - w_* \\
w_{t-1} - w_* 
\end{bmatrix}
\\ & =
\begin{bmatrix}
I_d  - \eta \int_0^1 \nabla^2 f\big( (1-\tau) w_0 + \tau w_* \big) d \tau + \beta I_d &  - \beta I_d \\
I_d                      &  0_d
\end{bmatrix}
\cdot
 \begin{bmatrix}
w_{t} - w_* \\
w_{t-1} - w_* 
\end{bmatrix}
\\ & + 
\eta \left( 
\int_0^1 \nabla^2 f\big( (1-\tau) w_0 + \tau w_* \big) d \tau
- \int_0^1 \nabla^2 f\big( (1-\tau) w_t + \tau w_* \big) d \tau
 \right) (w_t - w_*),
\end{aligned}
\end{equation}
where the second equality is by the fundamental theorem of calculus.
\begin{equation}
\nabla f(w_t) - \nabla f(w_*) = \left( \int_0^1  \nabla^2 f( (1-\tau) w_t + \tau w_* ) d\tau  \right) (w_t - w_*),
\end{equation}
and that $\nabla f(w_*) = 0$.
\end{proof}

\noindent
\textbf{Lemma}~\ref{lem:ReLU-residual}:
\textit{
 (Residual dynamics of training the ReLU network $\N_W^{\text{ReLU}}(\cdot)$)
Denote 
\[
(H_t)_{i,j}:= H(W_t)_{i,j} = \frac{1}{m} \sum_{r=1}^m x_i^\top x_j \mathbbm{1}\{ \langle w^{(r)}_t, x_i \rangle \geq 0 \text{ } \&  \text{ }   \langle w^{(r)}_t, x_j \rangle \geq 0 \}.
\]
Applying Algorithm~\ref{alg:HB1} or Algorithm~\ref{alg:HB2} to (\ref{eq:obj}) for training the ReLU network 
$\N_W^{\text{ReLU}}(x)$ induces a residual dynamics in the form of 
(\ref{eq:meta}) such that
\[
\begin{aligned}
\xi_t[i] &= \N_{W_t}^{\text{ReLU}}(x_i) - y_i \quad  \text{and hence } n_0 = d
\\ H     & = H_0
\\ \varphi_t & = \phi_t + \iota_t,
\end{aligned}
\]
where 
each element $i$ of $\xi_t \in \reals^n$ is the residual error of the sample $i$,
the $i_{th}$-element of $\phi_t\in \reals^n$ satisfies \[ \textstyle |\phi_t[i]| \leq \frac{ 2 \eta \sqrt{n} |S_i^\perp|}{ m } \big(  \| u_t - y \| + \beta \sum_{s=0}^{t-1} \beta^{t-1-s}  \| u_s - y \| \big),\]
and $\iota_t = \eta  \left( H_0 - H_t \right) \xi_t \in \reals^n $.
}

\begin{proof}
For each sample $i$, we will divide the contribution to $\N(x_i)$ into two groups.
\begin{equation} \label{eq:Ndiv}
\begin{aligned}
 \N(x_i)
& 
  = \frac{1}{\sqrt{m} }  \sum_{r=1}^m a_r \sigma( \langle w^{(r)}, x_i \rangle  )
\\ &   = \frac{1}{\sqrt{m} }  \sum_{r \in S_i} 
a_r \sigma( \langle w^{(r)}, x_i \rangle  )
+ \frac{1}{\sqrt{m} }  \sum_{r \in S_i^\perp} 
a_r \sigma( \langle w^{(r)}, x_i \rangle  ).
\end{aligned}
\end{equation}
To continue, let us recall some notations;
the subgradient with respect to $w^{(r)} \in \reals^d$ is
\begin{equation}
\frac{ \partial L(W)}{ \partial w^{(r)} }
 := \frac{1}{\sqrt{m} } 
\sum_{i=1}^n \big( \N(x_i) - y_i  \big) a_r x_i \mathbbm{1}\{
 \langle w^{(r)}, x \rangle \geq 0 \},
\end{equation}
and the Gram matrix $H_t$ whose $(i,j)$ element is
\begin{equation}
H_t[i,j] := \frac{1}{m} x_i^\top x_j \sum_{r=1}^m \mathbbm{1} \{ \langle w_{t}^{(r)} , x_i \rangle \geq 0 \text{ \& } \langle w_{t}^{(r)} , x_j \rangle \geq 0 \}.
\end{equation}
Let us also denote
\begin{equation}
H_t^\perp[i,j] := \frac{1}{m} x_i^\top x_j \sum_{r \in S_i^\perp} \mathbbm{1} \{ \langle w_{t}^{(r)} , x_i \rangle \geq 0 \text{ \& } \langle w_{t}^{(r)} , x_j \rangle \geq 0 \}.
\end{equation}
We have that
\begin{equation} \label{eq:a0}
\begin{split}
\xi_{t+1}[i]  & =  \N_{t+1}(x_i) - y_i
\\ & \overset{(\ref{eq:Ndiv})}{ = }
\underbrace{
\frac{1}{\sqrt{m} }  \sum_{r \in S_i} 
a_r \sigma( \langle w_{t+1}^{(r)}, x_i \rangle  ) }_{ \text{first term} }
+ \frac{1}{\sqrt{m} }  \sum_{r \in S_i^\perp} 
a_r \sigma( \langle w_{t+1}^{(r)}, x_i \rangle  ) - y_i.
\end{split}
\end{equation}
For the first term above, we have that
\begin{equation} \label{eq:a1}
\begin{split}
\textstyle  & \underbrace{
 \frac{1}{\sqrt{m} } \sum_{r \in S_i} 
a_r \sigma( \langle w_{t+1}^{(r)}, x_i \rangle  ) }_{ \text{first term} }
=
\frac{1}{\sqrt{m} }  \sum_{r \in S_i} 
a_r \sigma( \langle w_t^{(r)} - \eta \frac{ \partial L(W_t)}{ \partial w_t^{(r)} } + \beta ( w_{t}^{(r)} - w_{t-1}^{(r)} ), x_i \rangle  )
\\ 
= &
\frac{1}{\sqrt{m} }  \sum_{r \in S_i} 
a_r \langle w_t^{(r)} - \eta \frac{ \partial L(W_t)}{ \partial w_t^{(r)} } + \beta ( w_{t}^{(r)} - w_{t-1}^{(r)} ), x_i \rangle \cdot \mathbbm{1}\{ \langle w_{t+1}^{(r)} , x_i \rangle \geq 0 \}
\\ 
\overset{(a)}{=} & 
\frac{1}{\sqrt{m} }  \sum_{r \in S_i} 
a_r \langle w_t^{(r)}, x_i \rangle \cdot \mathbbm{1}\{ \langle w_{t}^{(r)} , x_i \rangle \geq 0 \}
+ 
\frac{\beta}{\sqrt{m} }  \sum_{r \in S_i} 
a_r \langle w_t^{(r)}, x_i \rangle \cdot \mathbbm{1}\{ \langle w_{t}^{(r)} , x_i \rangle \geq 0 \}
\\ 
& 
-
\frac{\beta}{\sqrt{m} }  \sum_{r \in S_i} 
a_r \langle w_{t-1}^{(r)}, x_i \rangle \cdot \mathbbm{1}\{ \langle w_{t-1}^{(r)} , x_i \rangle \geq 0 \}
- \eta \frac{1}{\sqrt{m} }  \sum_{r \in S_i} a_r \langle \frac{ \partial L(W_t)}{ \partial w_t^{(r)} }, x_i \rangle \mathbbm{1}\{ \langle w_{t}^{(r)} , x_i \rangle \geq 0 \}
\\  
=
&
\N_t(x_i) + \beta \big(  \N_t(x_i) -  \N_{t-1}(x_i)  \big)
- \frac{1}{\sqrt{m}} \sum_{r \in S_i^\perp} 
a_r \langle w_t^{(r)}, x_i \rangle \mathbbm{1}\{ \langle w_{t}^{(r)} , x_i \rangle \geq 0 \} 
\\ &
- 
 \frac{\beta}{\sqrt{m}} \sum_{r \in S_i^\perp}
a_r \langle w_{t}^{(r)}, x_i \rangle \mathbbm{1}\{ \langle w_{t}^{(r)} , x_i \rangle \geq 0 \}  
+ \frac{\beta}{\sqrt{m}} \sum_{r \in S_i^\perp} a_r \langle w_{t-1}^{(r)}, x_i \rangle \mathbbm{1}\{ \langle w_{t-1}^{(r)} , x_i \rangle \geq 0 \}
           \big) 
\\ &
           - \eta \underbrace{  \frac{1}{\sqrt{m} }  \sum_{r \in S_i} a_r \langle \frac{ \partial L(W_t)}{ \partial w_t^{(r)} }, x_i \rangle \mathbbm{1}\{ \langle w_{t}^{(r)} , x_i \rangle \geq 0 \} }_{\text{last term}},
\end{split}
 \end{equation}
 where (a) uses that for $r \in S_i$,
$\mathbbm{1}\{ \langle w_{t+1}^{(r)} , x_i \rangle \geq 0 \}
= \mathbbm{1}\{ \langle w_{t}^{(r)} , x_i \rangle \geq 0 \}
= \mathbbm{1}\{ \langle w_{t-1}^{(r)} , x_i \rangle \geq 0 \}$ as the 
 neurons in $S_i$ do not change their activation patterns.
We can further bound (\ref{eq:a1}) as
 \begin{equation} \label{eq:a3}
\begin{split}
\overset{(b)}{=}
&
\N_t(x_i) + \beta \big(  \N_t(x_i) -  \N_{t-1}(x_i)  \big)
- \eta
\sum_{j=1}^n \big( \N_t(x_j) - y_j \big)  H(W_t)_{i,j} 
\\ &
- \frac{\eta}{ m }
\sum_{j=1}^n x_i^\top x_j ( \N_t(x_j) - y_j )
 \sum_{r \in S_i^\perp}  
 \mathbbm{1}\{ \langle w_{t}^{(r)} , x_i \rangle \geq 0 \text{ \& } \langle w_{t}^{(r)} , x_j \rangle \geq 0 \}
\\ &
- \frac{1}{\sqrt{m}} \sum_{r \in S_i^\perp} 
a_r \langle w_t^{(r)}, x_i \rangle \mathbbm{1}\{ \langle w_{t}^{(r)} , x_i \rangle \geq 0 \} - 
 \frac{\beta}{\sqrt{m}} \sum_{r \in S_i^\perp}
a_r \langle w_{t}^{(r)}, x_i \rangle \mathbbm{1}\{ \langle w_{t}^{(r)} , x_i \rangle \geq 0 \}  
\\ &
+ \frac{\beta}{\sqrt{m}} \sum_{r \in S_i^\perp} a_r \langle w_{t-1}^{(r)}, x_i \rangle \mathbbm{1}\{ \langle w_{t-1}^{(r)} , x_i \rangle \geq 0 \}
           \big), 
\end{split}
 \end{equation}
where (b) is due to that
\begin{equation}
\begin{split}
 &  \underbrace{ \frac{1}{\sqrt{m} }  \textstyle  \sum_{r \in S_i} a_r \langle \frac{ \partial L(W_t)}{ \partial w_t^{(r)} }, x_i \rangle \mathbbm{1}\{ \langle w_{t}^{(r)} , x_i \rangle \geq 0 \} }_{\text{last term}}
\\   = &  
\frac{1}{ m }
\sum_{j=1}^n x_i^\top x_j ( \N_t(x_j) - y_j )
 \sum_{r \in S_i}  
 \mathbbm{1}\{ \langle w_{t}^{(r)} , x_i \rangle \geq 0 \text{ \& } \langle w_{t}^{(r)} , x_j \rangle \geq 0 \}
\\   = &  
\sum_{j=1}^n \big( \N_t(x_j) - y_j \big) H(W_t)_{i,j}
\\ & \qquad \qquad -  
\frac{1}{ m }
\sum_{j=1}^n x_i^\top x_j ( \N_t(x_j) - y_j )
 \sum_{r \in S_i^\perp}  
 \mathbbm{1}\{ \langle w_{t}^{(r)} , x_i \rangle \geq 0 \text{ \& } \langle w_{t}^{(r)} , x_j \rangle \geq 0 \}.
 \end{split}
\end{equation}
Combining (\ref{eq:a0}) and (\ref{eq:a3}), we have that
\begin{equation} \label{eq:a2}
\begin{split}
\xi_{t+1}[i] &   =   
\xi_t[i] + \beta \big(  \xi_t[i] -  \xi_{t-1}[i]  \big) - \eta
\sum_{j=1}^n  H_t[i,j]  
\xi_t[j] 
\\ & - \frac{\eta}{ m }
\sum_{j=1}^n x_i^\top x_j ( \N_t(x_j) - y_j )
 \sum_{r \in S_i^\perp}  
 \mathbbm{1}\{ \langle w_{t}^{(r)} , x_i \rangle \geq 0 \text{ \& } \langle w_{t}^{(r)} , x_j \rangle \geq 0 \}
\\ & 
 + \frac{1}{\sqrt{m} }  \sum_{r \in S_i^\perp} 
a_r \sigma( \langle w_{t+1}^{(r)}, x_i \rangle  )
- 
a_r \sigma( \langle w_{t}^{(r)}, x_i \rangle  )
- \beta a_r \sigma( \langle w_{t}^{(r)}, x_i \rangle  )
+\beta a_r \sigma( \langle w_{t-1}^{(r)}, x_i \rangle  ).
\end{split}
\end{equation}
So we can write the above into a matrix form.
\begin{equation} \label{k1}
\begin{split}
\xi_{t+1} & =  (I_n - \eta H_t )\xi_t + \beta ( \xi_t - \xi_{t-1} )
+ \phi_t
\\ & = (I_n - \eta H_0 )\xi_t + \beta ( \xi_t - \xi_{t-1} )
+ \phi_t + \iota_t,
\end{split}
\end{equation}
where 
the $i$ element of $\phi_t \in \reals^n$ is defined as
\begin{equation}
\begin{split}
 \phi_t[i] &  = 
- \frac{\eta}{ m }
\sum_{j=1}^n x_i^\top x_j ( \N_t(x_j) - y_j )
 \sum_{r \in S_i^\perp}  
 \mathbbm{1}\{ \langle w_{t}^{(r)} , x_i \rangle \geq 0 \text{ \& } \langle w_{t}^{(r)} , x_j \rangle \geq 0 \}
\\ &  +
\frac{1}{\sqrt{m} }   \sum_{r \in S_i^\perp} 
\big\{ 
a_r \sigma( \langle w_{t+1}^{(r)}, x_i \rangle  )
- 
a_r \sigma( \langle w_{t}^{(r)}, x_i \rangle  )
- \beta a_r \sigma( \langle w_{t}^{(r)}, x_i \rangle  )
+\beta a_r \sigma( \langle w_{t-1}^{(r)}, x_i \rangle  )
\big\}.
\end{split} 
\end{equation}
Now let us bound $\phi_t[i]$ as follows.
\begin{equation} \label{k2}
\begin{split}
  \phi_t[i] &  =  - \frac{\eta}{ m }
\sum_{j=1}^n x_i^\top x_j ( \N_t(x_j) - y_j )
 \sum_{r \in S_i^\perp}  
 \mathbbm{1}\{ \langle w_{t}^{(r)} , x_i \rangle \geq 0 \text{ \& } \langle w_{t}^{(r)} , x_j \rangle \geq 0 \}
\\ & +
 \frac{1}{\sqrt{m} }   \sum_{r \in S_i^\perp} 
\big\{ 
a_r \sigma( \langle w_{t+1}^{(r)}, x_i \rangle  )
- 
a_r \sigma( \langle w_{t}^{(r)}, x_i \rangle  )
- \beta a_r \sigma( \langle w_{t}^{(r)}, x_i \rangle  )
+\beta a_r \sigma( \langle w_{t-1}^{(r)}, x_i \rangle  )
\big\}
\\ & \overset{(a)}{ \leq } \frac{\eta \sqrt{n} |S_i^\perp | }{m} \| u_t - y \| +
\frac{ 1 }{ \sqrt{m} }  \sum_{r \in S_i^\perp}  
 \big(  \| w_{t+1}^{(r)} - w_t^{(r)}  \|
+ \beta \| w_{t}^{(r)} - w_{t-1}^{(r)} \| \big)
\\ & 
\overset{(b)}{=} \frac{\eta \sqrt{n} |S_i^\perp | }{m} \| u_t - y \| +
\frac{ \eta }{ \sqrt{m} }  \sum_{r \in S_i^\perp}   \big(  \| \sum_{s=0}^t \beta^{t-s} \frac{ \partial L(W_{s})}{ \partial w_{s}^{(r)} } \|
+ \beta \| \sum_{s=0}^{t-1} \beta^{t-1-s} \frac{ \partial L(W_{s})}{ \partial w_{s}^{(r)} } \| \big)
\\ & 
\overset{(c)}{\leq} \frac{\eta \sqrt{n} |S_i^\perp | }{m} \| u_t - y \| +
\frac{ \eta }{ \sqrt{m} }  \sum_{r \in S_i^\perp}   \big(  \sum_{s=0}^t \beta^{t-s}  \|  \frac{ \partial L(W_{s})}{ \partial w_{s}^{(r)} } \|
+ \beta \sum_{s=0}^{t-1} \beta^{t-1-s} \|  \frac{ \partial L(W_{s})}{ \partial w_{s}^{(r)} } \| \big)
\\ & 
\overset{(d)}{\leq} \frac{\eta \sqrt{n} |S_i^\perp | }{m} \| u_t - y \| +
\frac{ \eta \sqrt{n} |S_i^\perp|}{ m } \big(  \sum_{s=0}^t \beta^{t-s}  \| u_s - y \|
+ \beta \sum_{s=0}^{t-1} \beta^{t-1-s} \| u_s - y \| \big)
\\ & = 
\frac{ 2 \eta \sqrt{n} |S_i^\perp|}{ m }
\big(  \| u_t - y \| + \beta \sum_{s=0}^{t-1} \beta^{t-1-s}  \| u_s - y \| \big),
\end{split}
\end{equation}
where (a) is because
$-\frac{\eta}{ m }
\sum_{j=1}^n x_i^\top x_j ( \N_t(x_j) - y_j )
 \sum_{r \in S_i^\perp}  
 \mathbbm{1}\{ \langle w_{t}^{(r)} , x_i \rangle \geq 0 \text{ \& } \langle w_{t}^{(r)} , x_j \rangle \geq 0 \}
\leq \frac{\eta |S_i^\perp | }{m} \sum_{j=1}^n | \N_t(x_j) - y_j |
\leq \frac{\eta \sqrt{n} |S_i^\perp | }{m} \| u_t - y \|,
$
and that $\sigma(\cdot)$ is $1$-Lipschitz
so that 
\[
\begin{aligned}
& \frac{1}{ \sqrt{m} }   \sum_{r \in S_i^\perp}  
 \big ( 
a_r \sigma( \langle w_{t+1}^{(r)}, x_i \rangle  )
- 
a_r \sigma( \langle w_{t}^{(r)}, x_i \rangle ) \big) 
\leq \frac{1}{ \sqrt{m} }  \sum_{r \in S_i^\perp}  
| \langle w_{t+1}^{(r)}, x_i \rangle -  \langle w_{t}^{(r)}, x_i \rangle |
\\ & \leq \frac{1}{ \sqrt{m} }  \sum_{r \in S_i^\perp}  \| w_{t+1}^{(r)} -  w_{t}^{(r)} \| \| x_i \| \leq \frac{1}{\sqrt{m}} \sum_{r \in S_i^\perp}   \| w_{t+1}^{(r)} -  w_{t}^{(r)} \|,
\end{aligned}
\]
similarly, 
$\frac{-\beta}{ \sqrt{m} }   \sum_{r \in S_i^\perp}  
 \big ( 
a_r \sigma( \langle w_{t}^{(r)}, x_i \rangle  )
- 
a_r \sigma( \langle w_{t-1}^{(r)}, x_i \rangle ) \big) 
\leq 
\beta \frac{ 1}{ \sqrt{m} } \sum_{r \in S_i^\perp}   \| w_{t}^{(r)} - w_{t-1}^{(r)} \| 
$, (b) is by the update rule (Algorithm~\ref{alg:HB1}),
(c) is by Jensen's inequality, (d) is because
$|\frac{ \partial L(W_s)}{ \partial w_s^{(r)} }
| =| \frac{1}{\sqrt{m} } 
\sum_{i=1}^n \big( u_s[i] - y_i  \big) a_r x_i \mathbbm{1}\{
 x^\top w_t^{(r)} \geq 0 \} | \leq \frac{\sqrt{n}}{m} \| u_s - y \|$.

\end{proof}

\noindent
\textbf{Lemma: \ref{lem:DL-residual}}
\textit{
 (Residual dynamics of training $\N_W^{L\text{-linear}}(\cdot)$)
Denote 
$M_{t,l}$ the momentum term of layer $l$ at iteration $t$, which is recursively defined as
$M_{t,l} = \beta M_{t,l-1} + \frac{ \partial \ell(\W{L:1}_t)}{ \partial \W{l}_t } $. Denote
\[
\textstyle H_t \textstyle := \frac{1}{ m^{L-1} d_y } \sum_{l=1}^L [ (\W{l-1:1}_t X)^\top (\W{l-1:1}_t X ) 
\otimes
  \W{L:l+1}_t (\W{L:l+1}_t)^\top ]   \in \reals^{d_y n \times d_y n}.
\]
Applying Algorithm~\ref{alg:HB1} or Algorithm~\ref{alg:HB2} to (\ref{eq:obj}) for training the deep linear network 
$\N_W^{L\text{-linear}}(x)$ induces a residual dynamics in the form of 
(\ref{eq:meta}) such that
\[
\begin{aligned}
\xi_t &    = \text{vec}(U_t - Y) \in \reals^{d_y n} \text{, and hence } n_0 = d_y n \\
 H & =  H_0  \\
\varphi_t & = \phi_t + \psi_t  + \iota_t \in \reals^{d_y n},
\end{aligned}
\]
where
\[
\begin{aligned}
 \phi_t &  = \frac{1}{\sqrt{m^{L-1} d_y} } \v\left( \Phi_t X\right) \text{ with }
\Phi_t    = \Pi_l ( \W{l}_t - \eta M_{t,l} )
- \W{L:1}_t  +  \eta \sum_{l=1}^L \W{L:l+1}_t M_{t,l} \W{l-1:1}_t
\\   \psi_t & =\frac{1}{\sqrt{m^{L-1} d_y} } 
\v\left( (L-1) \beta \W{L:1}_{t}  X + \beta  \W{L:1}_{t-1} X 
 - \beta \sum_{l=1}^L \W{L:l+1}_t \W{l}_{t-1} \W{l-1:1}_{t} X \right)
\\ \iota_t  & = \eta (H_0 - H_t) \xi_t.
\end{aligned}
\]
}

\begin{proof}
According to the update rule of gradient descent with Polyak's momentum,
we have
\begin{equation} \label{eq:t1}
\W{L:1}_{t+1} = \Pi_l \left( \W{l}_t - \eta M_{t,l} \right)
= \W{L:1}_t - \eta \sum_{l=1}^L \W{L:l+1}_t M_{t,l} \W{l-1:1} + \Phi_t,
\end{equation}
where $M_{t,l}$ stands for the momentum term of layer $l$, which is
$M_{t,l} = \beta M_{t,l-1} + \frac{ \partial \ell(\W{L:1}_t)}{ \partial \W{l}_t } = 
\sum_{s=0}^t \beta^{t-s} \frac{ \partial \ell(\W{L:1}_s)}{ \partial \W{l}_s }$,
and 
$\Phi_t$ contains all the high-order terms (in terms of $\eta$), e.g. those with $\eta M_{t,i}$ and $ \eta M_{t,j}$, $i \neq j \in [L]$, or higher. Based on the equivalent update expression of gradient descent with Polyak's momentum
$- \eta M_{t,l} = - \eta \frac{ \partial \ell(\W{L:1}_t)}{ \partial \W{l}_t } +
\beta ( \W{l}_t - \W{l}_{t-1} )$,
we can rewrite (\ref{eq:t1}) as
\begin{equation}
\begin{aligned}
&\W{L:1}_{t+1} 
\\ & = \W{L:1}_t - \eta \sum_{l=1}^L \W{L:l+1}_t \frac{ \partial \ell(\W{L:1}_t)}{ \partial \W{l}_t } \W{l-1:1}_t 
+ \sum_{l=1}^L \W{L:l+1}_t \beta ( \W{l}_t - \W{l}_{t-1} ) \W{l-1:1}_t
 + \Phi_t
\\ & = \W{L:1}_t - \eta \sum_{l=1}^L \W{L:l+1}_t \frac{ \partial \ell(\W{L:1}_t)}{ \partial \W{l}_t } \W{l-1:1}_t 
+ \beta ( \W{L:1}_t - \W{L:1}_{t-1} )
+ \phi_t
\\ &
\quad + (L-1) \beta \W{L:1}_{t} + \beta  \W{L:1}_{t-1}
- \beta \sum_{l=1}^L \W{L:l+1}_t \W{l}_{t-1} \W{l-1:1}_{t}. 
\end{aligned}
\end{equation}
Multiplying the above equality with $\frac{1}{ \sqrt{ m^{L-1} d_y} } X$, we get
\begin{equation}
\begin{split}
U_{t+1} & = U_t - \eta \frac{1}{ m^{L-1} d_y } \sum_{l=1}^L \W{L:l+1}_t  (\W{L:l+1}_t)^\top
( U_t - Y )   (\W{l-1:1}_t X)^\top  \W{l-1:1}_t X 
\\ & \quad + \frac{1}{ \sqrt{ m^{L-1} d_y} } \left( (L-1) \beta \W{L:1}_{t} + \beta  \W{L:1}_{t-1}
- \beta \sum_{l=1}^L \W{L:l+1}_t \W{l}_{t-1} \W{l-1:1}_{t} \right) X
\\ &  \quad +
\frac{1}{ \sqrt{ m^{L-1} d_y} } \Phi_t X + \beta  (U_t - U_{t-1})
\end{split}
\end{equation}
Using $\text{vec}(ACB) = (B^\top \otimes A) \text{vec}(C)$,
where $ \otimes$ stands for the Kronecker product,
 we can apply a vectorization of the above equation and obtain
\begin{equation} \label{eq:L1}
\begin{split}
& \v(U_{t+1}) - \v(U_t)  = - \eta H_t \v( U_t - Y ) +
\beta  \left( \v(U_{t}) - \v(U_{t-1})  \right)
\\ & \quad + 
\v(  \frac{1}{ \sqrt{ m^{L-1} d_y} } \left( (L-1) \beta \W{L:1}_{t} + \beta  \W{L:1}_{t-1}
- \beta \sum_{l=1}^L \W{L:l+1}_t \W{l}_{t-1} \W{l-1:1}_{t} \right) X )
\\ & \quad +
\frac{1}{ \sqrt{ m^{L-1} d_y} } \v( \Phi_t X),
\end{split}
\end{equation}
where 
\begin{equation}
H_t = \frac{1}{  m^{L-1} d_y } \sum_{l=1}^L \left[ \left( (\W{l-1:1}_t X)^\top (\W{l-1:1}_t X)
 \right)  \otimes \W{L:l+1}_t (\W{L:l+1}_t)^\top    \right],
\end{equation}
which is a positive semi-definite matrix. 

In the following, we will denote
$\xi_{t} := \v(U_t - Y) $ 
 as the vector of the residual errors. 
Also, we denote $\phi_t: = \frac{1}{ \sqrt{ m^{L-1} d_y} } \v( \Phi_t X)$
\text{ with } 
$\Phi_t    = \Pi_l ( \W{l}_t - \eta M_{t,l} )
- \W{L:1}_t  +  \eta \sum_{l=1}^L \W{L:l+1}_t M_{t,l} \W{l-1:1}_t$,
 and \\ $\psi_t:= 
\v(  \frac{1}{ \sqrt{ m^{L-1} d_y} } \left( (L-1) \beta \W{L:1}_{t} + \beta  \W{L:1}_{t-1}
- \beta \sum_{l=1}^L \W{L:l+1}_t \W{l}_{t-1} \W{l-1:1}_{t} \right) X )$.
 Using the notations, we can rewrite (\ref{eq:L1}) as
\begin{equation} \label{eq:L2}
\begin{split}
\begin{bmatrix}
\xi_{t+1} \\
\xi_{t} 
\end{bmatrix}
& = 
\begin{bmatrix}
I_{d_y n} - \eta H_t + \beta  I_{d_y n} & - \beta  I_{d_y n}   \\
I_{d_y n} & 0_{d_y n} 
\end{bmatrix}
\begin{bmatrix}
\xi_{t} \\
\xi_{t-1} 
\end{bmatrix}
+
\begin{bmatrix}
\phi_t + \psi_t \\ 0_{d_y n}
\end{bmatrix}
\\ & = 
\begin{bmatrix}
I_{d_y n} - \eta H_0 + \beta  I_{d_y n} & - \beta  I_{d_y n}   \\
I_{d_y n} & 0_{d_y n} 
\end{bmatrix}
\begin{bmatrix}
\xi_{t} \\
\xi_{t-1} 
\end{bmatrix}
+
\begin{bmatrix}
\varphi_t \\ 0_{d_y n}
\end{bmatrix}
,
\end{split}
\end{equation}
where $\varphi_t = \phi_t + \psi_t + \iota_t \in \reals^{d_y n}$ and 
 $I_{d_y n}$ is the $d_y n \times d_y n$-dimensional identity matrix.

\end{proof}

\subsection{Proof of Theorem~\ref{thm:akv}} \label{app:thm:akv}

\noindent
\textbf{Theorem~\ref{thm:akv}}
\textit{
Let $A:=
\begin{bmatrix} 
(1 + \beta) I_n - \eta  H  & - \beta I_n \\
I_n & 0
\end{bmatrix}
\in \reals^{2 n \times 2 n}$. 
Suppose that $H \in \reals^{n \times n}$ is a positive semi-definite matrix. 
Fix a vector $v_0 \in \reals^n$.
If $\beta$ is chosen to satisfy 
$1 \geq \beta >  \max \{ \left( 1 - \sqrt{\eta \lambda_{\min}(H) } \right)^2, \left( 1 - \sqrt{\eta \lambda_{\max}(H)} \right)^2 \} $
then
\begin{equation}\label{thm:5rate}
\| A^k v_0 \| \leq \big( \sqrt{ \beta } \big)^k C_0 \| v_0 \|, 
\end{equation}
where the constant
\begin{equation} 
 C_0:=\frac{\sqrt{2} (\beta+1)}{
\sqrt{ \min\{  
h(\beta,\eta \lambda_{\min}(H)) , h(\beta,\eta \lambda_{\max}(H)) \} } } \geq 1,
\end{equation} 
and the function $h(\beta,z)$ 
is defined as
\begin{align}
     h(\beta,z) :=-\left(\beta-\left(1-\sqrt{z}\right)^2\right)\left(\beta-\left(1+\sqrt{z}\right)^2\right).  
\end{align}
}

We would first prove some lemmas for the analysis.
\begin{lemma} \label{lem:diagonal}
Under the assumption of Theorem~\ref{thm:akv}, $A$ is diagonalizable with respect to complex field $\mathbb{C}$ in $\mathbb{C}^n$, i.e., $\exists P$ such that $A = PDP^{-1}$ for some diagonal matrix $D$. Furthermore, the diagonal elements of $D$ all have magnitudes bounded by $\sqrt{\beta}$.
\end{lemma}

\begin{proof}
In the following, we will use the notation/operation ${\rm Diag}( \cdots )$
to represents a block-diagonal matrix that has the arguments on its main diagonal. 
Let $U{\rm Diag}([\lambda_1,\dots,\lambda_n]) U^*$ be the singular-value-decomposition of $H$, then 
\begin{align}
    A = 
\begin{bmatrix}
U& 0\\0& U
\end{bmatrix}
\begin{bmatrix} 
(1 + \beta) I_n - \eta {\rm Diag}([\lambda_1,\dots,\lambda_n])  & - \beta I_n \\
I_n & 0 
\end{bmatrix}
\begin{bmatrix}
U^*& 0\\0& U^*
\end{bmatrix}.
\end{align}
Let $\tilde{U} = \begin{bmatrix}
U& 0\\0& U
\end{bmatrix}$. Then, after applying some permutation matrix $\tilde{P}$, $A$ can be further simplified into 
\begin{align}\label{decompose1}
A = \tilde{U}\tilde{P} \Sigma \tilde{P}^T\tilde{U}^*,
\end{align}
where $\Sigma$ is a block diagonal matrix consisting of $n$ 2-by-2
matrices $\tilde{\Sigma}_i := \begin{bmatrix}
1+\beta-\eta \lambda_i& -\beta\\1& 0
\end{bmatrix}$. The characteristic polynomial of $\tilde{\Sigma}_i$ is $x^2 - (1+\beta -\lambda_i)x +\beta$. Hence it can be shown that when $\beta > (1-\sqrt{\eta \lambda_i})^2$ then the roots of polynomial are conjugate and have magnitude $\sqrt{\beta}$. These roots are exactly the eigenvalues of $\tilde{\Sigma}_i \in \reals^{2 \times 2}$. On the other hand, the corresponding eigenvectors $q_i,\bar{q}_i$ are also conjugate to each other as $\tilde{\Sigma}_i \in \reals^{2 \times 2}$ is a real matrix. As a result, $\Sigma \in \reals^{2n \times 2n}$ admits a block eigen-decomposition as follows,
\begin{align}\label{decompose2}
    \Sigma = & {\rm Diag}(\tilde{\Sigma}_i,\dots,\tilde{\Sigma}_n)\nonumber\\
    =&
    {\rm Diag}(Q_1,\dots,Q_n)
    {\rm Diag}\left(\begin{bmatrix}
    z_{1}&0\\0&\bar{z}_{1}
    \end{bmatrix},\dots,\begin{bmatrix}
    z_{n}&0\\0&\bar{z}_{n}
    \end{bmatrix}\right){\rm Diag}(Q^{-1}_1,\dots,Q^{-1}_n),
\end{align}
where $Q_i= [q_i,\bar{q}_i]$ and $z_{i}, \bar{z}_{i}$ are eigenvalues of $\tilde{\Sigma}_i$ (they are conjugate by the condition on $\beta$). 
Denote $Q:={\rm Diag}(Q_1,\dots,Q_n)$ and
\begin{align}
D:={\rm Diag}\left(\begin{bmatrix}
    z_{1}&0\\0&\bar{z}_{1}
    \end{bmatrix},\dots,\begin{bmatrix}
    z_{n}&0\\0&\bar{z}_{n}
    \end{bmatrix}\right).
\end{align}
By combining (\ref{decompose1}) and (\ref{decompose2}), we have
\begin{align}
    A & = P {\rm Diag}\left(\begin{bmatrix}
    z_{1}&0\\0&\bar{z}_{1}
    \end{bmatrix},\dots,\begin{bmatrix}
    z_{n}&0\\0&\bar{z}_{n}
    \end{bmatrix}\right) P^{-1}
     = P D P^{-1} ,
\end{align}
where 
\begin{align} \label{def:P}
P = \tilde{U}\tilde{P}Q,
\end{align} by the fact that $\tilde{P}^{-1} = \tilde{P}^T$ and $\tilde{U}^{-1}=\tilde{U}^{*}$.  
\end{proof}

\begin{proof} (of Theorem~\ref{thm:akv})
Now we proceed the proof of Theorem~\ref{thm:akv}.
In the following, we denote $v_k := A^k v_0$ (so $v_k = A v_{k-1} )$.
Let $P$ be the matrix in Lemma~\ref{lem:diagonal}, and $u_k:=P^{-1}v_k$, the dynamic can be rewritten as $u_k=P^{-1} A v_{k-1} = P^{-1}APu_{k-1}=Du_{k-1}$. As $D$ is diagonal, we immediately have 
\begin{align}
    &\|u_k\|\leq \max_{i \in [n]}|D_{ii}|^k \|u_0\|\nonumber\\
    \Rightarrow ~&\|P^{-1}v_k\|\leq \max_{i \in [n]}|D_{ii}|^k \|P^{-1}v_0\|\nonumber\\
    \Rightarrow~ &\sigma_{\rm min}(P^{-1})\|v_k\|\leq \sqrt{\beta}^k \sigma_{\rm max}(P^{-1})\|v_0\|~~~(\rm Lemma~\ref{lem:diagonal}.)\nonumber\\
   \Rightarrow ~& \sigma^{-1}_{\rm max}(P)\|v_k\|\leq \sqrt{\beta}^k \sigma^{-1}_{\rm min}(P)\|v_0\|\nonumber\\
   \Rightarrow ~& \|v_k\|\leq \sqrt{\beta}^k \frac{\sigma_{\rm max}(P)}{\sigma_{\rm min}(P)}\|v_0\|\nonumber\\
   \Rightarrow ~& \|v_k\|\leq \sqrt{\beta}^k \sqrt{\frac{\lambda_{\rm max}(PP^*)}{\lambda_{\rm min}(PP^*)}}\|v_0\|.
\end{align}
Hence, now it suffices to prove upper bound and lower bound of $\lambda_{\rm max}$ and $\lambda_{\rm min}$, respectively. By using Lemma~\ref{lem:eigbound} in the following, we obtain the inequality of (\ref{thm:5rate}).
We remark that as $C_0$ is an upper-bound of the squared root of the condition number $\sqrt{\frac{\lambda_{\rm max}(PP^*)}{\lambda_{\rm min}(PP^*)}}$, it is lower bounded by $1$.
\end{proof}

\begin{lemma} \label{lem:eigbound}
Let $P$ be the matrix in Lemma~\ref{lem:diagonal}, then we have $\lambda_{\rm max}(PP^*)\leq 2(\beta + 1)$ and $\lambda_{\rm min}(PP^*) \geq \min\{h(\beta,\eta\lambda_{\min}(H)),h(\beta,\eta \lambda_{\max}(H))\}/(1+\beta)$, where 
\begin{align}
    h(\beta,z)=-\left(\beta-\left(1-\sqrt{z}\right)^2\right)\left(\beta-\left(1+\sqrt{z}\right)^2\right).
\end{align}
\end{lemma}
\begin{proof}
As (\ref{def:P}) in the proof of Lemma 2, $P = \tilde{U}\tilde{P}{\rm Diag}(Q_1,\dots,Q_n)$. Since $\tilde{U}\tilde{P}$ is unitary, it does not affect the spectrum of $P$, therefore, it suffices to analyze the eigenvalues of $QQ^*$, where $Q ={\rm Diag}(Q_1,\dots,Q_n)$. Observe that $QQ^*$ is a block diagonal matrix with blocks $Q_iQ_i^*$, the eigenvalues of it are exactly that of $Q_iQ_i^*$, i.e., $\lambda_{\rm max}(QQ^*)=\max_{i\in[n]}\lambda_{\rm max}(Q_iQ_i^*)$ and likewise for the minimum. Recall $Q_i = [q_i,\bar{q_i}]$ consisting of eigenvectors of $\tilde{\Sigma}_i := \begin{bmatrix}
1+\beta-\eta \lambda_i& -\beta\\1& 0
\end{bmatrix}$ with corresponding eigenvalues $z_{i},\bar{z}_{i}$. The eigenvalues satisfy 
\begin{align}\label{val}
    z_i + \bar{z}_i = 2\Re{z_i} &= 1+\beta -\eta\lambda_i, \\
    z_i\bar{z}_i &= |z_i|^2 = \beta. 
\end{align}
On the other hand, the eigenvalue equation $\tilde{\Sigma}_iq_i = z_i q_i$ together with (\ref{val}) implies $q_i = [z_i,1]^T$.
Furthermore, $Q_iQ_i^* = q_iq_i^* + \bar{q}_i\bar{q}_i^*=2\Re{q_iq_i^*} = 2\Re{q_i}\Re{q_i}^T + 2\Im{q_i}\Im{q_i}^T$. Thus, 
\begin{align}\label{vec}
    Q_iQ_i^* &= 2\Re{q_i}\Re{q_i}^T + 2\Im{q_i}\Im{q_i}^T\nonumber\\
    &= 2\left(\begin{bmatrix}
    \Re{z_i}\\1
    \end{bmatrix}\begin{bmatrix}
    \Re{z_i}~ 1
    \end{bmatrix}
    +
    \begin{bmatrix}
    \Im{z_i}\\0
    \end{bmatrix}\begin{bmatrix}
    \Im{z_i}~ 0
    \end{bmatrix}\right)\nonumber\\
    &=2\begin{bmatrix}
    |z_i|^2 & \Re{z_i}\\
     \Re{z_i} & 1
    \end{bmatrix}.
\end{align}
Let the eigenvalues of $Q_iQ_i^*$ be $\theta_1,\theta_2$, then by (\ref{val})-(\ref{vec}) we must have 
\begin{align}\label{xi}
    \theta_1 +\theta_2 &= 2(\beta +1), \\
    \theta_1\theta_2 &= 4\left(\beta -(\frac{1+\beta -\eta\lambda_i}{2})^2\right)\nonumber\\
    &=-\left(\beta-\left(1-\sqrt{\eta\lambda_i}\right)^2\right)\left(\beta-\left(1+\sqrt{\eta\lambda_i}\right)^2\right) \geq 0.
\end{align}
From (\ref{xi}), as both eigenvalues are nonnegative, we deduce that \begin{align}2(1+\beta) \geq\max\{\theta_1,\theta_2\}\geq \beta +1.
\end{align}
On the other hand, from (\ref{xi}) we also have \begin{align}\label{minbound}
\min\{\theta_1,\theta_2\}
= & \theta_1\theta_2/\max\{\theta_1,\theta_2\} \nonumber\\
\geq& -\left(\beta-\left(1-\sqrt{\eta\lambda_i}\right)^2\right)\left(\beta-\left(1+\sqrt{\eta\lambda_i}\right)^2\right)/(1+\beta)\nonumber\\
:=&h(\beta, \eta \lambda_i)/(1+\beta).
\end{align}
Finally, as the eigenvalues of $QQ^*$ are composed of exactly that of $Q_iQ_i^*$, applying the bound of (\ref{minbound}) to each $i$ we have
\begin{align}
    \lambda_{\rm min} (PP^*)\geq&\min_{i\in[n]}h(\beta,\eta \lambda_i)/(1+\beta)\nonumber\\
    \geq& \min\{h(\beta,\eta \lambda_{\min}(H)),h(\beta,\eta \lambda_{\max}(H))\}/(1+\beta),
\end{align}
where the last inequality follows from the facts that $\lambda_{\min}(H) \leq \lambda_i \leq \lambda_{\max}(H)$ and $h$ is concave quadratic function of of $\lambda$ in which the minimum must occur at the boundary.
\end{proof}

\subsection{Proof of Corollary~\ref{corr:1}} \label{app:corr}

\noindent
\textbf{Corollary~\ref{corr:1}}
\textit{
Assume that $\lambda_{\min}(H) > 0$. Denote $\kappa:= \lambda_{\max}(H) / \lambda_{\min}(H)$.
Set $\eta = 1 / \lambda_{\max}(H)$ 
and set $\beta = \left( 1 - \frac{1}{2} \sqrt{ \eta \lambda_{\min} (H) } \right)^2 = \left( 1 - \frac{1}{2 \sqrt{\kappa}}  \right)^2$.
Then, $C_0 \leq \max \{ 4  , 2 \sqrt{\kappa} \} \leq 4 \sqrt{\kappa}$.
}


\begin{proof}
For notation brevity, in the following, we let $\mu:= \lambda_{\min}(H)$
and $\alpha := \lambda_{\max}(H)$.
Recall that
$h(\beta,z)=-\left(\beta-\left(1-\sqrt{z}\right)^2\right)\left(\beta-\left(1+\sqrt{z}\right)^2\right)$.
We have
\begin{equation} \label{eq:h1}
\begin{split}
h(\beta, \eta \mu)
& = - \left( (1 - \frac{1}{2} \sqrt{\eta \mu})^2 -\left(1-\sqrt{\eta \mu}\right)^2\right)\left( (1 - \frac{1}{2} \sqrt{\eta \mu})^2 -\left(1+\sqrt{\eta \mu}\right)^2\right)  
\\ & = 3 \left( \sqrt{\eta \mu} - \frac{3}{4} \eta \mu \right)
\left( \sqrt{\eta \mu} + \frac{1}{4} \eta \mu \right)
= 3 \left( \frac{1}{\sqrt{\kappa}} - \frac{3}{4 \kappa}  \right)
\left( \frac{1}{\sqrt{\kappa}}  + \frac{1}{4 \kappa}  \right)
\end{split}
\end{equation}
and
\begin{equation} \label{eq:h2}
\begin{split}
h(\beta, \eta \alpha)
& = - \left( (1 - \frac{1}{2} \sqrt{\eta \mu})^2 -\left(1-\sqrt{\eta \alpha}\right)^2\right)\left( (1 - \frac{1}{2} \sqrt{\eta \mu})^2 -\left(1+\sqrt{\eta \alpha}\right)^2\right)  
\\ & =  \left( 2 \sqrt{\eta \alpha} - \sqrt{\eta \mu}
-  \eta \alpha + \frac{1}{4} \eta \mu \right)
\left( \sqrt{\eta \mu} + 2 \sqrt{\eta \alpha} + \eta \alpha - \frac{1}{4} \eta \mu \right)
\\ &
= \left(1 - \frac{1}{\sqrt{\kappa}} + \frac{1}{4 \kappa} \right)
\left(3 + \frac{1}{\sqrt{\kappa}} - \frac{1}{4 \kappa} \right).
\end{split}
\end{equation}
We can simplify it to get that
$h(\beta, \eta \alpha) = 3 - \frac{2}{\sqrt{\kappa}} - \frac{1}{2 \kappa} + \frac{1}{2 \kappa^{3/2}} - \frac{1}{16 \kappa^2}  \geq 0.5 $.

Therefore,
we have
\begin{equation}
\begin{split}
 & 
 \frac{\sqrt{2} (\beta+1)}{
\sqrt{ 
h(\beta,\eta \mu) } }
= 
\frac{ \sqrt{2} (\beta+1) }{  
\sqrt{3 \eta \mu (1 - \frac{1}{2} \sqrt{\eta \mu} - \frac{3}{16} \eta \mu ) } }
= 
\frac{ \sqrt{2} (\beta+1) }{  
\sqrt{3  (1 - \frac{1}{2} \sqrt{\eta \mu} - \frac{3}{16} \eta \mu ) } } \sqrt{\kappa}
\\ & 
\leq 
\frac{1 }{  
\sqrt{ (1 - \frac{1}{2}  - \frac{3}{16}  ) } } \sqrt{\kappa}
\leq 2 \sqrt{\kappa},
\end{split}
\end{equation} 
where  we use $\eta \mu = \frac{1}{\kappa}$. On the other hand,
 $\frac{\sqrt{2} (\beta+1)}{
\sqrt{ 
h(\beta,\eta \alpha) } }
\leq 4 .$
We conclude that
\begin{equation} 
 C_0 =\frac{\sqrt{2} (\beta+1)}{
\sqrt{ \min\{  
h(\beta,\eta \mu) , h(\beta,\eta \alpha) \} }}\leq \max \{ 4  , 2 \sqrt{\kappa} \} \leq 4 \sqrt{\kappa}.
\end{equation}

\end{proof}


\subsection{Proof of Theorem~\ref{thm:stcFull} } \label{app:sec:stc}

\noindent
\textbf{Theorem~\ref{thm:stcFull}}
\textit{
Assume the momentum parameter $\beta$ satisfies 
$1 \geq \beta > \big( \max \{ 1 - \sqrt{\eta \mu }, 1 - \sqrt{\eta \alpha } \} \big)^2$.
Gradient descent with Polyak's momentum has
\begin{equation} \label{eq:qq0}
\|
\begin{bmatrix}
w_{t} - w_* \\
w_{t-1} - w_*
\end{bmatrix}
\| \leq \left(  \sqrt{\beta}  \right)^{t} C_0
\|
\begin{bmatrix}
w_{0} - w_* \\
w_{-1} - w_*
\end{bmatrix}
\|,
\end{equation}
where the constant
\begin{equation}
 C_0:=\frac{\sqrt{2} (\beta+1)}{
\sqrt{ \min\{  
h(\beta,\eta \mu) , h(\beta,\eta \alpha) \} } },
\end{equation} 
and
$   h(\beta,z)=-\left(\beta-\left(1-\sqrt{z}\right)^2\right)\left(\beta-\left(1+\sqrt{z}\right)^2\right).$  
Consequently, if the step size $\eta = \frac{1}{\alpha}$ and the momentum parameter $\beta = \left(1 - \sqrt{\eta \mu}\right)^2$, then it has
\begin{equation}\label{eq:qq1}
\|
\begin{bmatrix}
w_{t} - w_* \\
w_{t-1} - w_*
\end{bmatrix}
\| \leq \left(  1 - \frac{1}{2 \sqrt{\kappa}}   \right)^{t} 4 \sqrt{\kappa}
\|
\begin{bmatrix}
w_{0} - w_* \\
w_{-1} - w_*
\end{bmatrix}
\|.
\end{equation}
Furthermore, if $\eta = \frac{4}{(\sqrt{\mu}+\sqrt{\alpha})^2}$ 
and $\beta$ approaches $\beta \rightarrow \left( 1 - \frac{2}{\sqrt{\kappa}+1} \right)^2$ from above, then it has a convergence rate approximately
$ \left(  1 - \frac{2}{\sqrt{\kappa} + 1}   \right)$
as $t \rightarrow \infty$.
}

\begin{proof}
The result (\ref{eq:qq0}) and (\ref{eq:qq1}) is due to a trivial combination of
Lemma~\ref{lem:stc-residual}, Theorem~\ref{thm:metas}, and Corollary~\ref{corr:1}.

On the other hand, set
$\eta = \frac{4}{(\sqrt{\mu}+\sqrt{\alpha})^2}$, 
the lower bound on $\beta$ becomes $\big( \max \{ 1 - \sqrt{\eta \mu }, 1 - \sqrt{\eta \alpha} \} \big)^2 = \left( 1 - \frac{2}{\sqrt{\kappa}+1} \right)^2$.
Since the rate is $r=\lim_{t\rightarrow\infty}\frac{1}{t}\log(\sqrt{\beta}^{t+1}C_0)=\sqrt{\beta}$, setting $\beta \downarrow  \left( 1 - \frac{2}{\sqrt{\kappa}+1} \right)^2$ from above leads to the rate of $\left(  1 - \frac{2}{\sqrt{\kappa} + 1}   \right)$. 
Formally, it is straightforward to show that $C_0=\Theta\left(1/\sqrt{\beta-(1-\frac{2}{1+\sqrt{\kappa}})^2}\right)$, hence, for any $\beta$ converges to $(1-\frac{2}{\sqrt{\kappa}+1})^2$ slower than inverse exponential of $\kappa$, i.e.,
$\beta = (1-\frac{2}{\sqrt{\kappa}+1})^2+(\frac{1}{\kappa})^{o(t)}$, we have $r =  1 - \frac{2}{\sqrt{\kappa} + 1}   $. 


\end{proof}


\subsection{Proof of Theorem~\ref{thm:acc} } \label{app:sec:relu}

We will need some supporting lemmas in the following for the proof.
In the following analysis,
we denote 
$
 C_0:=\frac{\sqrt{2} (\beta+1)}{
\sqrt{ \min\{  
h(\beta,\eta \lambda_{\min}(H)) , h(\beta,\eta \lambda_{\max}(H)) \} } }$,
where $h(\beta,\cdot)$ is defined in Theorem~\ref{thm:akv} and $H = H_0$ whose $(i,j)$ entry is
$(H_0)_{i,j}:= H(W_0)_{i,j} = \frac{1}{m} \sum_{r=1}^m x_i^\top x_j \mathbbm{1}\{ \langle w^{(r)}_0, x_i \rangle \geq 0 \text{ } \&  \text{ }   \langle w^{(r)}_0, x_j \rangle \geq 0 \}$, as defined in Lemma~\ref{lem:ReLU-residual}.
In the following, we also denote $\beta = (1- \frac{1}{2} \sqrt{\eta \lambda})^2 := \beta_*^2$. We summarize the notations in Table~\ref{table:1}.

\begin{table*}[h]
\footnotesize
 \begin{tabular}{|c | c | c|} 
 \hline
 Notation & definition (or value)  & meaning  \\  
 \hline\hline
$ \N_{W}^{\text{ReLU}}(x)$ & $ \N_{W}^{\text{ReLU}}(x):= \frac{1}{\sqrt{m} } \sum_{r=1}^m a_r \sigma( \langle w^{(r)},  x \rangle )$ & the ReLU network's output given $x$ \\ \hline
$\bar{H}$ &
$\bar{H}_{i,j}  := \underset{ w^{(r)}}{\mathbbm{E}}
[ x_i^\top x_j \mathbbm{1}\{ \langle w^{(r)}, x_i \rangle \geq 0 \text{ } \&  \text{ }   \langle w^{(r)}, x_j \rangle \geq 0 \}     ] .
$ & the expectation of the Gram matrix \\ \hline
$H_0$ & \shortstack{ $:= H(W_0)_{i,j}$\\$ = \frac{1}{m} \sum_{r=1}^m x_i^\top x_j  \mathbbm{1}\{ \langle w^{(r)}_0, x_i \rangle \geq 0 \text{ } \&  \text{ }   \langle w^{(r)}_0, x_j \rangle \geq 0 \}$ } & the Gram matrix at the initialization \\ \hline
$\lambda_{\min}(\bar{H})$ &  $\lambda_{\min}(\bar{H}) > 0 $ (by assumption)  & the least eigenvalue of $\bar{H}$.\\ \hline
$\lambda_{\max}(\bar{H})$ &       & the largest eigenvalue of $\bar{H}$\\ \hline
$\kappa$ &  $\kappa:= \lambda_{\max}(\bar{H}) / \lambda_{\min}(\bar{H})$      & the condition number of $\bar{H}$\\ \hline
$\lambda$ &  $\lambda := \frac{3}{4} \lambda_{\min}(\bar{H}) $     &
\shortstack{ (a lower bound of) \\ the least eigenvalue of $H_0$}.\\ \hline
$\lambda_{\max}$ &  $\lambda_{\max} := \lambda_{\max} (\bar{H}) + \frac{\lambda_{\min}(\bar{H})}{4}$     &
\shortstack{ (an upper bound of) \\ the largest eigenvalue of $H_0$}.\\ \hline
$\hat{\kappa}$ &  $\hat{\kappa}:=\frac{\lambda_{\max}}{\lambda} = \frac{4}{3} \kappa + \frac{1}{3} $      &
the condition number of $H_0$.\\ \hline
$\eta$  & $\eta = 1 / \lambda_{\max} $ & step size \\ \hline
$\beta$ & $\beta = (1- \frac{1}{2} \sqrt{\eta \lambda})^2 = (1 - \frac{1}{2 \sqrt{\hat{\kappa}}})^2 := \beta_*^2$ & momentum parameter \\ \hline
$\beta_*$ & $\beta_* = \sqrt{\beta} = 1- \frac{1}{2} \sqrt{\eta \lambda}$ & squared root of $\beta$ \\ \hline 
$\theta$ & $\theta = \beta_*  + \frac{1}{4} \sqrt{\eta \lambda} = 1 -\frac{1}{4} \sqrt{\eta \lambda} = 1 - \frac{1}{4 \sqrt{\hat{\kappa}}} $ & the convergence rate \\ \hline
$C_0$ &
$ C_0:=\frac{\sqrt{2} (\beta+1)}{
\sqrt{ \min\{  
h(\beta,\eta \lambda_{\min}(H_0)) , h(\beta,\eta \lambda_{\max}(H_0)) \} } }$
& the constant used in Theorem~\ref{thm:akv} 
\\ \hline \hline
\end{tabular}
\caption{Summary of the notations for proving Theorem~\ref{thm:acc}.} \label{table:1}
\end{table*}


\begin{lemma} \label{lem:ReLU-B} 
Suppose that 
the neurons $w^{(1)}_0, \dots, w^{(m)}_0$ are i.i.d. generated by $N(0,I_d)$ initially.
Then, for any set of weight vectors $W_t:=\{ w^{(1)}_t, \dots, w^{(m)}_t \}$ that satisfy for any $r\in [m]$, $\| w^{(r)}_t - w^{(r)}_0 \| \leq R^{\text{ReLU}} := \frac{\lambda}{1024 n C_0}$, 
it holds that
\[
\begin{aligned}
&
 \|  H_t - H_0 \|_F \leq 2 n R^{\text{ReLU}} = \frac{ \lambda}{512 C_0 }, 
\end{aligned}
\]
with probability at least $1 - n^2 \cdot \exp( -m R^{\text{ReLU}} / 10)$.
\end{lemma}

\begin{proof}
This is an application of Lemma~3.2 in \citet{ZY19}.
\end{proof}

Lemma~\ref{lem:ReLU-B} shows that if the distance between the current iterate $W_t$ and its initialization $W_0$ is small, then
the distance between the Gram matrix $H(W_t)$ and $H(W_0)$ should also be small. Lemma~\ref{lem:ReLU-B} allows us to obtain the following lemma, which bounds the size of $\varphi_t$ (defined in Lemma~\ref{lem:ReLU-residual}) in the residual dynamics.

\begin{lemma} \label{lem:ReLU-deviate2}
Following the setting as Theorem~\ref{thm:acc},
denote $\theta := \beta_* + \frac{1}{4} \sqrt{ \eta\lambda } = 1 - \frac{1}{4} \sqrt{ \eta\lambda }$.
Suppose that $\forall i \in [n], |S_i^\perp| \leq 4 m R^{\text{ReLU}}$ for some constant $R^{\text{ReLU}}:= \frac{\lambda}{1024n C_0} >0$.
If we have (I) for any $s \leq t$, the residual dynamics satisfies
$ \|
\begin{bmatrix}
\xi_{s} \\
\xi_{s-1} 
\end{bmatrix}
\| \leq 
\theta^{s} 
\cdot \nu C_0
\|
 \begin{bmatrix}
\xi_{0} \\
\xi_{-1} 
\end{bmatrix}
\|$, for some constant $\nu>0$,
and (II)
for any $r\in [m]$ and any $s\leq t$, $\| w^{(r)}_s - w^{(r)}_0 \| \leq R^{\text{ReLU}}$,
then $\phi_t$ and $\iota_t$ in Lemma~\ref{lem:ReLU-residual} satisfies
\[
\begin{aligned}
\| \phi_t \| & \leq 
\frac{ \sqrt{\eta \lambda} }{16} \theta^t
\nu 
 \| \begin{bmatrix} \xi_0 \\ \xi_{-1} \end{bmatrix} \|
, \text{ and } \| \iota_t \|  \leq \frac{\eta \lambda}{512} 
\theta^t \nu 
\| \begin{bmatrix} \xi_0 \\ \xi_{-1} \end{bmatrix} \|.
\end{aligned}
\]
Consequently, $\varphi_t$ in Lemma~\ref{lem:ReLU-residual} satisfies
\[
\| \varphi_t \| \leq \left( \frac{ \sqrt{\eta \lambda} }{16} 
 +  \frac{\eta \lambda }{512 } \right)
\theta^t \nu 
\| \begin{bmatrix} \xi_0 \\ \xi_{-1} \end{bmatrix} \|.
\]
\end{lemma}

\begin{proof}
Denote $\beta_* := 1 - \frac{1}{2} \sqrt{ \eta\lambda }$
and $\theta := \beta_* + \frac{1}{4} \sqrt{ \eta\lambda } = 1 - \frac{1}{4} \sqrt{ \eta\lambda }$.
We have by Lemma~\ref{lem:ReLU-residual}
\begin{equation} \label{eq:main4}
\begin{split}
 \| \phi_t \|
&  = \sqrt{ \sum_{i=1}^n \phi_t[i]^2 } \leq \sqrt{  \sum_{i=1}^n  \big( \frac{ 2 \eta \sqrt{n} |S_i^\perp|}{ m } \big( \| \xi_t \| +  \beta \sum_{\tau=0}^{t-1} \beta^{t-1-\tau}  \| \xi_{\tau}  \| \big) \big)^2 }
\\ & 
\overset{(a)}{\leq} 8 \eta n R^{\text{ReLU}} \big( \| \xi_t  \| + \beta \sum_{\tau=0}^{t-1} \beta^{t-1-\tau}  \| \xi_{\tau}  \| \big)
\\ & 
\overset{(b)}{\leq} 
8 \eta n R^{\text{ReLU}}
\left( \theta^{t} 
 \nu C_0 \| \begin{bmatrix} \xi_0 \\ \xi_{-1} \end{bmatrix} \| 
+ 
\beta \sum_{\tau=0}^{t-1} \beta^{t-1-\tau}  
\theta^{\tau} 
 \nu C_0 
 \| \begin{bmatrix} \xi_0 \\ \xi_{-1} \end{bmatrix} \| \right)
\\ & 
\overset{(c)}{=} 
8 \eta n R^{\text{ReLU}}
\left( 
\theta^{t} 
 \nu C_0 \| \begin{bmatrix} \xi_0 \\ \xi_{-1} \end{bmatrix} \|
 + \beta_*^{2} \nu C_0
  \sum_{\tau=0}^{t-1} \beta_*^{2(t-1-\tau)}
  \theta^{\tau} 
 \| \begin{bmatrix} \xi_0 \\ \xi_{-1} \end{bmatrix} \| \right)
\\ & 
\overset{(d)}{\leq}
8 \eta n R^{\text{ReLU}}
\left( 
\theta^{t} 
 \nu C_0 \| \begin{bmatrix} \xi_0 \\ \xi_{-1} \end{bmatrix} \|
 + \beta_*^{2} \nu C_0
   \theta^{t-1} 
  \sum_{\tau=0}^{t-1} 
  \theta^{t-1-\tau} 
 \| \begin{bmatrix} \xi_0 \\ \xi_{-1} \end{bmatrix} \| \right)
\\ & \leq
8 \eta n R^{\text{ReLU}} \theta^{t} ( 1 + \beta_* \sum_{\tau=0}^{t-1} \theta^{\tau} ) 
\nu C_0
\| \begin{bmatrix} \xi_0 \\ \xi_{-1} \end{bmatrix} \|
\\ & \leq
8 \eta n R^{\text{ReLU}} \theta^{t} ( 1 + \frac{\beta_*}{1-\theta} ) 
\nu C_0
\| \begin{bmatrix} \xi_0 \\ \xi_{-1} \end{bmatrix} \|
\\ & \overset{(e)}{\leq} 
\frac{\sqrt{\eta \lambda}}{16} \theta^t \nu 
 \| \begin{bmatrix} \xi_0 \\ \xi_{-1} \end{bmatrix} \|,
\end{split}
\end{equation}
where (a) is by
$|S_i^\perp| \leq 4 m R^{\text{ReLU}}$,
(b) is by induction that $\| \xi_t \| \leq \theta^t 
 \nu C_0 
 \| \begin{bmatrix} \xi_0 \\ \xi_{-1} \end{bmatrix} \|
$ as $u_0 = u_{-1}$, (c) uses that $\beta= \beta_*^2 $,
(d) uses $\beta_* \leq \theta$,
(e) uses $1+\frac{\beta_*}{1-\theta} \leq \frac{2}{1-\theta} \leq \frac{8}{\sqrt{\eta \lambda}}$ and $R^{\text{ReLU}}:= \frac{\lambda}{1024 n C_0}$.

Now let us switch to bound $\| \iota_t \|$.
\begin{equation} 
\begin{split}
\| \iota_t \| & \leq \eta \|  H_0 - H_t \|_2 \| \xi_t \|
\leq \frac{ \eta \lambda }{512 C_0}
 \theta^{t} \nu C_0 
  \| \begin{bmatrix} \xi_0 \\ \xi_{-1} \end{bmatrix} \|, 
\end{split}
\end{equation}
where we uses Lemma~\ref{lem:ReLU-B} that 
 $\|  H_0 - H_t \|_2 \leq \frac{ \lambda }{512 C_0}$
 and the induction that
$
\| \begin{bmatrix} \xi_t \\ \xi_{t-1} \end{bmatrix} \| \leq \theta^t  \nu C_0 
 \| \begin{bmatrix} \xi_0 \\ \xi_{-1} \end{bmatrix} \| $.

\end{proof}

The assumption of Lemma~\ref{lem:ReLU-deviate2}, $\forall i \in [n], |S_i^\perp| \leq 4 m R^{\text{ReLU}}$ only depends on the initialization. 
Lemma~\ref{lem:3.12} 
shows that it holds
with probability at least $1- n \cdot \exp( - m R^{\text{ReLU}})$.

\begin{lemma} \label{lem:ReLU-deviate1} 
Following the setting as Theorem~\ref{thm:acc},
denote $\theta := \beta_* + \frac{1}{4} \sqrt{ \eta\lambda } = 1 - \frac{1}{4} \sqrt{ \eta\lambda }$.
Suppose that the initial error satisfies $\| \xi_0 \|^2 = O(  n \log ( m / \delta) \log^2 (n / \delta) )$.
If for any $s < t$, the residual dynamics satisfies
$ \|
\begin{bmatrix}
\xi_{s} \\
\xi_{s-1} 
\end{bmatrix}
\| \leq \theta^{s} \cdot \nu C_0 
\|
 \begin{bmatrix}
\xi_{0} \\
\xi_{-1} 
\end{bmatrix}
\|$, for some constant $\nu>0$,
then \[ 
 \| w_{t}^{(r)} - w_0^{(r)} \|  \leq R^{\text{ReLU}} := \frac{\lambda}{1024n C_0} .\]
 \end{lemma}

\begin{proof}
We have 
\begin{equation}
\begin{split}
 \| w_{t+1}^{(r)} - w_0^{(r)} \| 
& \overset{(a)}{\leq} 
\eta \sum_{s=0}^t \|  M_s^{(r)} \|  
\overset{(b)}{=} 
\eta \sum_{s=0}^t
\| \sum_{\tau=0}^s \beta^{s-\tau}   \frac{ \partial L(W_{\tau})}{ \partial w_{\tau}^{(r)} }  \|
\leq 
\eta \sum_{s=0}^t \sum_{\tau=0}^s \beta^{s-\tau} 
\|   \frac{ \partial L(W_{\tau})}{ \partial w_{\tau}^{(r)} }  \|
\\ &
\overset{(c)}{\leq} \eta \sum_{s=0}^t \sum_{\tau=0}^s \beta^{s-\tau} 
\frac{\sqrt{n}}{\sqrt{m}} \| y - u_{\tau} \|
\\ &
\overset{(d)}{\leq} \eta \sum_{s=0}^t \sum_{\tau=0}^s \beta^{s-\tau} 
\frac{\sqrt{2n}}{\sqrt{m}}  \theta^{\tau} 
\nu C_0
 \| y - u_0 \|  
\\ &
\overset{(e)}{\leq} \frac{\eta \sqrt{2 n} }{ \sqrt{m} } 
\sum_{s=0}^t \frac{ \theta^s}{1 - \theta}
\nu C_0  \| y - u_0 \| 
\leq 
 \frac{\eta \sqrt{2 n}  }{ \sqrt{m} } \left(   
\frac{\nu C_0}{(1 - \theta)^2}  \right) \| y - u_0 \|
\\ &
\overset{(f)}{=} 
\frac{\eta \sqrt{2 n} }{ \sqrt{m} } 
\left(   \frac{16 \nu C_0}{\eta \lambda}  \right)
\| y - u_0 \|
\\ &
\overset{(g)}{=}
\frac{ \eta \sqrt{2n} }{ \sqrt{m} }
\left(   \frac{16 \nu C_0}{\eta \lambda}  \right)
O(\sqrt{ n \log ( m / \delta) \log^2 (n / \delta)} )
\\ & \textstyle
\overset{(h)}{\leq} \frac{\lambda}{1024 n C_0},
\end{split}
\end{equation}
where (a), (b) is by the update rule of momentum, which is
$ w_{t+1}^{(r)} - w_t^{(r)} = - \eta M_t^{(r)}$, where $M_t^{(r)}:= \sum_{s=0}^t \beta^{t-s} \frac{ \partial L(W_{s})}{ \partial w_{s}^{(r)} }$, (c) is because $\|\frac{ \partial L(W_{s})}{ \partial w_{s}^{(r)} }\| = \| \sum_{i=1}^n (y_i - u_s[i]) \frac{1}{\sqrt{m}} a_r x_i \cdot \mathbbm{1}\{ \langle w_s^{(r)}, x \rangle \geq 0 \} \| \leq \frac{1}{\sqrt{m}} \sum_{i=1}^n | y_i - u_s[i] | \leq \frac{\sqrt{n}}{\sqrt{m}} \| y - u_s \|$,
(d) is by
$ \|
\begin{bmatrix}
\xi_{s} \\
\xi_{s-1} 
\end{bmatrix}
\| \leq \theta^{s}  \nu C_0
\|
 \begin{bmatrix}
\xi_{0} \\
\xi_{-1} 
\end{bmatrix}
\|
$
(e) is because that $\beta=\beta_*^2 \leq \theta^2$,
(f) we use $\theta := (1 - \frac{1}{4} \sqrt{ \eta\lambda } )$, 
so that $\frac{1}{(1-\theta)^2} = \frac{16}{\eta \lambda}$,
(g) 
is by that the initial error satisfies $\| y - u_0 \|^2 = O(  n \log ( m / \delta) \log^2 (n / \delta) ),$
and (h) is by the choice of the number of neurons $m = \Omega( \lambda^{-4} n^{4 } C_0^4 \log^3 ( n / \delta)   ) = \Omega( \lambda^{-4} n^{4 } \kappa^2 \log^3 ( n / \delta)   )$, as $C_0 = \Theta( \sqrt{\kappa} )$ by Corollary~\ref{corr:1}.

The proof is completed.

\end{proof}

Lemma~\ref{lem:ReLU-deviate1} basically says that if the size of the residual errors is bounded and decays over iterations, then the distance between the current iterate $W_t$ and its initialization $W_0$ is well-controlled. The lemma will allows us to invoke Lemma~\ref{lem:ReLU-B} and Lemma~\ref{lem:ReLU-deviate2}
when proving Theorem~\ref{thm:acc}. 
The proof of Lemma~\ref{lem:ReLU-deviate1} is in Section~\ref{app:sec:relu}.
The assumption of Lemma~\ref{lem:ReLU-deviate1},
$\|  \xi_0 \|^2 = O(  n \log ( m / \delta)$  $\log^2 (n / \delta) )$,
  is satisfied by the random initialization with probability at least $1-\delta / 3$ according to   
Lemma~\ref{lem:3.10} .

\begin{lemma} (Claim 3.12 of \citet{ZY19}) \label{lem:3.12}
Fix a number $R_1 \in (0,1)$.
Recall that $S_i^\perp$ is a random set defined in Subsection~\ref{inst:ReLU}.
With probability at least $1 - n \cdot \exp(-mR_1)$, we have that
for all $i \in [n]$,
\[
\displaystyle
| S_i^\perp | \leq 4 m R_1.
\]
\end{lemma}
A similar lemma also appears in \citet{DZPS19}.
Lemma~\ref{lem:3.12} says that the number of neurons whose activation patterns for a sample $i$ could change during the execution is only a small faction of $m$ if $R_1$ is a small number,
i.e. $| S_i^\perp | \leq 4 m R_1 \ll m$. 

\begin{lemma} \label{lem:3.10} (Claim 3.10 in \citet{ZY19})
Assume that $w_0^{(r)} \sim N(0,I_d)$ and $a_r$ uniformly sampled from $\{-1,1\}$. For $0 < \delta < 1$, we have that 
\[
\| y - u_0 \|^2 = O( n \log ( m / \delta) \log^2 (n / \delta) ),
\]
with probability at least $1-\delta$.
\end{lemma}

\subsection{Proof of Theorem~\ref{thm:acc}}

\begin{proof} 
 (of Theorem~\ref{thm:acc})
Denote $\lambda:= \frac{3}{4} \lambda_{\min}(\bar{H})>0$.
Lemma~\ref{lem:ReLU-A} shows that $\lambda$
is a lower bound of $\lambda_{\min}(H)$ of the matrix $H$ defined in Lemma~\ref{lem:ReLU-residual}.
Also, denote $\beta_*:=1 - \frac{1}{2} \sqrt{\eta \lambda}$ (note that $\beta=\beta_*^2$)
and $\theta := \beta_* + \frac{1}{4} \sqrt{\eta \lambda} = 1 - \frac{1}{4} \sqrt{\eta \lambda} $. In the following,
we let $\nu = 2 $ in Lemma~\ref{lem:ReLU-deviate2},
~\ref{lem:ReLU-deviate1},
and let $C_1 = C_3 = C_0$ and
$C_2 = \frac{1}{4} \sqrt{\eta \lambda}$ in
Theorem~\ref{thm:metas}. The goal is to show that
$
\left\|
\begin{bmatrix}
\xi_{t} \\
\xi_{t-1} 
\end{bmatrix}
\right\|
\leq \theta^{t} 2 C_0  \left\|
\begin{bmatrix}
\xi_{0} \\
\xi_{-1} 
\end{bmatrix}
\right\|
$ for all $t$ by induction. 
To achieve this, we will also use induction to show that for all iterations $s$,
\begin{equation} \label{induct:relu}
 \textstyle \forall r \in [m], \| w^{(r)}_s - w^{(r)}_0 \|  \textstyle  \leq  \textstyle  \textstyle R^{\text{ReLU}} :=\frac{\lambda}{1024 n C_0},  
\end{equation}
which is clear true in the base case $s=0$.

By Lemma~\ref{lem:ReLU-residual},~\ref{lem:ReLU-A},~\ref{lem:ReLU-B},~\ref{lem:ReLU-deviate2}, Theorem~\ref{thm:metas}, and Corollary~\ref{corr:1},
 it suffices to show that
given
$
\left\|
\begin{bmatrix}
\xi_{s} \\
\xi_{s-1} 
\end{bmatrix}
\right\|
\leq \theta^{s} 2 C_0  \left\|
\begin{bmatrix}
\xi_{0} \\
\xi_{-1} 
\end{bmatrix}
\right\|
$ and (\ref{induct:relu}) hold at $s=0,1,\dots,t-1$, one has
\begin{eqnarray} 
\textstyle
\| \sum_{s=0}^{t-1} A^{t-s-1} \begin{bmatrix}
\varphi_s \\
0 
\end{bmatrix}
\|
& \textstyle \leq & \textstyle 
\theta^{t}
C_0
\left\|
\begin{bmatrix}
\xi_{0} \\
\xi_{-1} 
\end{bmatrix}
\right\|,
\label{eq:thm-ReLU-1}
\\ \textstyle \forall r \in [m], \| w^{(r)}_t - w^{(r)}_0 \|  &\textstyle  \leq  \textstyle & \textstyle R^{\text{ReLU}} :=\frac{\lambda}{1024 n C_0},  
\label{eq:thm-ReLU-2}
\end{eqnarray}
where the matrix $A$ and the vector $\varphi_t$ are defined in Lemma~\ref{lem:ReLU-residual}. The inequality (\ref{eq:thm-ReLU-1}) is the required  condition for using the result of Theorem~\ref{thm:metas}, while the inequality (\ref{eq:thm-ReLU-2}) helps us to show (\ref{eq:thm-ReLU-1}) through invoking Lemma~\ref{lem:ReLU-deviate2} to bound the terms $\{\varphi_s\}$ as shown in the following.

We have
\begin{equation} \label{eq:main5}
\begin{split}
& 
\| \sum_{s=0}^{t-1} A^{t-s-1} \begin{bmatrix}
\varphi_s \\
0 
\end{bmatrix}
\|
\overset{(a)}{\leq} 
 \sum_{s=0}^{t-1} \beta_*^{t-s-1} C_0 \| \varphi_s \| 
\\ &
\overset{(b)}{ \leq }
\left( \frac{ \sqrt{\eta \lambda}}{16}  
+ \frac{\eta \lambda}{512}  
\right)  2 C_0 \|
\begin{bmatrix}
\xi_{0} \\
\xi_{-1} 
\end{bmatrix}
\| \left(
 \sum_{s=0}^{t-1}
\beta_*^{t-1-s}  \theta^{s}  \right)
\\ & 
\overset{(c)}{ \leq }
\left( 
\frac{1 }{2} 
+ \frac{1}{64} \sqrt{\eta \lambda} 
\right) \theta^{t-1}
 C_0 \|
\begin{bmatrix}
\xi_{0} \\
\xi_{-1} 
\end{bmatrix}
\| 
\overset{(d)}{ \leq }  \theta^t C_0  \|
\begin{bmatrix}
\xi_{0} \\
\xi_{-1} 
\end{bmatrix}
\|, 
\end{split}
\end{equation}

where (a) uses Theorem~\ref{thm:akv}, (b) is due to Lemma~\ref{lem:ReLU-deviate2}, Lemma~\ref{lem:3.12},
(c) is because $\sum_{s=0}^{t-1} \beta_*^{t-1-s} \theta^s = \theta^{t-1} \sum_{s=0}^{t-1} \left( \frac{\beta_*}{\theta}  \right)^{t-1-s} \leq \theta^{t-1} \sum_{s=0}^{t-1} \theta^{t-1-s}$ $\leq \theta^{t-1} \frac{4}{\sqrt{\eta\lambda}}$, (d) uses that $\theta \geq \frac{3}{4}$ and $\eta \lambda \leq 1$. Hence, we have shown (\ref{eq:thm-ReLU-1}). 
 Therefore, by Theorem~\ref{thm:metas}, we have
$ \left\|
\begin{bmatrix}
\xi_{t} \\
\xi_{t-1} 
\end{bmatrix}
\right\|
\leq \theta^{t} 2 C_0  \left\|
\begin{bmatrix}
\xi_{0} \\
\xi_{-1} 
\end{bmatrix}
\right\|.
$

By Lemma~\ref{lem:ReLU-deviate1} and Lemma~\ref{lem:3.10}, 
we have (\ref{eq:thm-ReLU-2}). 
Furthermore, with the choice of $m$, we have $3 n^2 \exp( - m R^{\text{ReLU}} /10 ) \leq \delta$.  Thus, we have completed the proof.

\end{proof}


\subsection{Proof of Theorem~\ref{thm:LinearNet}} \label{app:sec:linear}

We will need some supporting lemmas in the following for the proof.
In the following analysis,
we denote 
$
 C_0:=\frac{\sqrt{2} (\beta+1)}{
\sqrt{ \min\{  
h(\beta,\eta \lambda_{\min}(H)) , h(\beta,\eta \lambda_{\max}(H)) \} } }$,
where $h(\beta,\cdot)$ is the constant defined in Theorem~\ref{thm:akv} and $\textstyle H = H_0 \textstyle := \frac{1}{ m^{L-1} d_y } \sum_{l=1}^L [ (\W{l-1:1}_0 X)^\top (\W{l-1:1}_0 X ) \otimes
  \W{L:l+1}_0 (\W{L:l+1}_0)^\top ]   \in \reals^{d_y n \times d_y n},
$ as defined in Lemma~\ref{lem:DL-residual}.
We also denote $\beta = (1- \frac{1}{2} \sqrt{\eta \lambda})^2 := \beta_*^2$.
As mentioned in the main text, following \citet{DH19}, \citet{HXP20},
we will further assume that (A1) there exists a $W^*$ such that $Y = W^* X$, $X \in \reals^{d \times \bar{r}}$, and $\bar{r}=rank(X)$, which is actually without loss of generality (see e.g. the discussion in Section B of \citet{DH19}). We summarize the notions in Table~\ref{table:2}.

\begin{table*}[h]
\footnotesize
\centering
 \begin{tabular}{|c | c | c|} 
 \hline
 Notation & definition (or value)  & meaning  \\  
 \hline\hline
$ \N_{W}^{L\text{-linear}}(x)$ & $
\N_W^{L\text{-linear}}(x) := \frac{1}{\sqrt{m^{L-1} d_{y}}} \W{L} \W{L-1} \cdots \W{1} x,$ &output of the deep linear network \\ \hline
$H_0$ & \shortstack{
$H_0 \textstyle := \frac{1}{ m^{L-1} d_y } \sum_{l=1}^L [ (\W{l-1:1}_0 X)^\top (\W{l-1:1}_0 X ) $ \\ \qquad \qquad $ \otimes
  \W{L:l+1}_0 (\W{L:l+1}_0)^\top ] \in \reals^{d_y n \times d_y n}$ } 
 & $H$ in (\ref{eq:meta}) is $H=H_0$ (Lemma~\ref{lem:DL-residual}) \\ \hline
$\lambda_{\max}(H_0)$ & $\lambda_{\max}(H_0)\leq L \sigma^2_{\max}(X) / d_y$  (Lemma~\ref{lem:DL-A})    & the largest eigenvalue of $H_0$\\ \hline
$\lambda_{\min}(H_0)$ & $\lambda_{\min}(H_0)\geq L \sigma^2_{\min}(X) / d_y$ (Lemma~\ref{lem:DL-A})     & the least eigenvalue of $H_0$\\ \hline
$\lambda$ & $\lambda:= L \sigma^2_{\min}(X) / d_y$     & \shortstack{ (a lower bound of) \\ the least eigenvalue of $H_0$}\\ \hline
$\kappa$ &  $\kappa:= \frac{\lambda_{1}(X^\top X) }{ \lambda_{\bar{r}}(X^\top X)} = \frac{\sigma^2_{\max}(X)}{ \sigma^2_{\min}(X) } $ (A1)      & the condition number of $X$ \\ \hline
$\hat{\kappa}$ & $\hat{\kappa} := \frac{\lambda_{\max}(H_0)}{ \lambda_{\min}(H_0)} \leq \frac{\sigma^2_{\max}(X) }{ \sigma^2_{\min}(X)} = \kappa$ (Lemma~\ref{lem:DL-A})   & the condition number of $H_0$ \\ \hline
$\eta$  & $\eta = \frac{d_y}{L \sigma^2_{\max}(X)} $ & step size \\ \hline
$\beta$ & $\beta = (1- \frac{1}{2} \sqrt{\eta \lambda})^2 = (1 - \frac{1}{2 \sqrt{\kappa}})^2 := \beta_*^2$ & momentum parameter \\ \hline
$\beta_*$ & $\beta_* = \sqrt{\beta} = 1- \frac{1}{2} \sqrt{\eta \lambda}$ & squared root  of $\beta$ \\ \hline 
$\theta$ & $\theta = \beta_*  + \frac{1}{4} \sqrt{\eta \lambda} = 1 -\frac{1}{4} \sqrt{\eta \lambda} = 1 - \frac{1}{4 \sqrt{\kappa}} $ & the convergence rate \\ \hline
$C_0$ &
$ C_0:=\frac{\sqrt{2} (\beta+1)}{
\sqrt{ \min\{  
h(\beta,\eta \lambda_{\min}(H_0)) , h(\beta,\eta \lambda_{\max}(H_0)) \} } }$
& the constant used in Theorem~\ref{thm:akv} 
\\ \hline \hline
\end{tabular}
\caption{Summary of the notations for proving Theorem~\ref{thm:LinearNet}.
We will simply use $\kappa$ to represent the condition number of the matrix $H_0$ in the analysis since we have $\hat{\kappa} \leq \kappa$. 
} \label{table:2}
\end{table*}

\begin{lemma} \label{lem:DL-A}
 [Lemma 4.2 in \citet{HXP20}] 
 By the orthogonal initialization,
we have
\[
\begin{aligned}
& \lambda_{\min}(H_0)  \geq L \sigma^2_{\min}(X) / d_y, \quad
   \lambda_{\max}(H_0)  \leq L \sigma^2_{\max}(X) / d_y .
\\ &   \sigma_{\max}( \W{j:i}_0 ) = m^{ \frac{j-i+1}{2} }, \quad
\sigma_{\min}( \W{j:i}_0 ) = m^{ \frac{j-i+1}{2} }
\end{aligned}
\]
Furthermore, with probability $1-\delta$,
\[
\begin{aligned}
 \ell(W_0)  \leq B_0^2 = O\left( 1 + \frac{\log(\bar{r}/\delta)}{d_y} + \| W_* \|^2_2  \right),
\end{aligned}
\]
for some constant $B_0 > 0$.
\end{lemma}

We remark that Lemma~\ref{lem:DL-A} implies that the condition number of $H_0$ satisfies
\begin{equation}
\hat{\kappa} := \frac{\lambda_{\max}(H_0) }{\lambda_{\min}(H_0)} 
\leq \frac{\sigma^2_{\max}(X) }{ \sigma^2_{\min}(X)} = \kappa.
\end{equation}
\begin{lemma} \label{lem:linear-deviate1}
Following the setting as Theorem~\ref{thm:LinearNet},
denote $\theta := \beta_* + \frac{1}{4} \sqrt{ \eta\lambda } = 1 - \frac{1}{4} \sqrt{ \eta\lambda }$.
If we have (I) for any $s \leq t$, the residual dynamics satisfies
$ \|
\begin{bmatrix}
\xi_{s} \\
\xi_{s-1} 
\end{bmatrix}
\| \leq \theta^{s} 
\cdot \nu C_0
\|
 \begin{bmatrix}
\xi_{0} \\
\xi_{-1} 
\end{bmatrix}
\|, $ for some constant $\nu > 0$,
and (II) for all $l \in [L]$ and for any $s \leq t$,
$\| \W{l}_s - \W{l}_0 \|_F \leq R^{L\text{-linear}} := 
\frac{64 \| X \|_2 \sqrt{d_y}}{ L \sigma_{\min}^2(X) } \nu C_0 B_0$,
then 
\[ 
\begin{aligned}
& \| \phi_t \| 
\leq 
\frac{ 43  \sqrt{d_y} }{\sqrt{m} \| X \|_2}   
 \theta^{2t} \nu^2 C_0^2
\left( 
 \frac{ \| \xi_0 \|  }{1 -\theta}  \right)^2, \quad
 \| \psi_t \| 
\leq 
\frac{ 43  \sqrt{d_y} }{\sqrt{m} \| X \|_2}   
 \theta^{2(t-1)} \nu^2 C_0^2
\left( 
 \frac{ \| \xi_0 \|  }{1 -\theta}  \right)^2,
 \quad
\\ & \qquad \qquad \qquad \| \iota_t \|  \leq  \frac{ \eta \lambda }{80} \theta^{t} \nu C_0 
  \| \begin{bmatrix} \xi_0 \\ \xi_{-1} \end{bmatrix} \| .
\end{aligned}  
\]
Consequently,  $\varphi_t$ in Lemma~\ref{lem:DL-residual} satisfies
\[
\| \varphi_t \| \leq
\frac{ 1920 \sqrt{d_y} }{\sqrt{m} \| X \|_2} \frac{1}{\eta \lambda}   
 \theta^{2t}  \nu^2 C_0^2  \| \begin{bmatrix} \xi_0 \\ \xi_{-1} \end{bmatrix} \|^2
 +
\frac{\eta \lambda}{80} \theta^{t} \nu C_0 
  \| \begin{bmatrix} \xi_0 \\ \xi_{-1} \end{bmatrix} \| .
\]
\end{lemma}

\begin{proof}
By Lemma~\ref{lem:DL-residual},
$\varphi_t = \phi_t + \psi_t  + \iota_t \in \reals^{d_y n}$,
we have
\begin{equation}
\begin{aligned}
 \phi_t & := \frac{1}{\sqrt{m^{L-1} d_y}} \v( \Phi_t X)
\\ & \qquad \qquad \qquad \text{ , with } 
\Phi_t   := \Pi_l \left( \W{l}_t - \eta M_{t,l} \right)
- \W{L:1}_t +  \eta \sum_{l=1}^L \W{L:l+1}_t M_{t,l} \W{l-1:1}_t,
\end{aligned}
\end{equation}
and
\begin{equation}
\begin{aligned}
& \psi_t:= \frac{1}{\sqrt{m^{L-1} d_y}} 
\v\left(    (L-1) \beta \W{L:1}_{t}  X + \beta  \W{L:1}_{t-1} X
- \beta \sum_{l=1}^L \W{L:l+1}_t \W{l}_{t-1} \W{l-1:1}_{t} X \right).
\end{aligned}
\end{equation}
and
\begin{equation}
\begin{aligned}
& \iota_t:= \eta (H_0 - H_t) \xi_t.
\end{aligned}
\end{equation}

So if we can bound $\| \phi_t \|$, $\| \psi_t \|$, and $\| \iota_t\|$ respectively, then
we can bound $\| \varphi_t \| $
 by the triangle inequality.
\begin{equation} \label{eq:var}
\| \varphi_t \| \leq \| \phi_t \| + \|\psi_t \| + \| \iota_t \| .
\end{equation}

Let us first upper-bound $\| \phi_t \|$.
Note that $\Phi_t$ is the sum of all the high-order (of $\eta$'s) term in the product,
\begin{equation}
\W{L:1}_{t+1} = \Pi_l \left( \W{l}_t - \eta M_{t,l} \right)
= \W{L:1}_t - \eta \sum_{l=1}^L \W{L:l+1}_t M_{t,l} \W{l-1:1} + \Phi_t.
\end{equation}
By induction, we can bound the gradient norm of each layer as
\begin{equation} \label{eq:gnorm-linear}
\begin{split}
 \| \frac{ \partial \ell(\W{L:1}_s)}{ \partial \W{l}_s } \|_F
 &
\leq \frac{1}{\sqrt{m^{L-1} d_y} } \| \W{L:l+1}_s \|_2 \| U_s - Y \|_F \| \W{l-1:1}_s \|_2 \| X \|_2
\\ &
\leq \frac{1}{\sqrt{m^{L-1} d_y} } 1.1 m^{\frac{L-l}{2}}  
\theta^s \nu C_0 2 \sqrt{2} \| U_0 - Y \|_F
 1.1 m^{\frac{l-1}{2}} \| X \|_2
\\ &
\leq
\frac{4 \| X \|_2}{\sqrt{d_y}} \theta^s \nu C_0 \| U_0 - Y \|_F,
\end{split}
\end{equation}
where the second inequality we use Lemma~\ref{lem:linear-eigen} and that 
$\| \begin{bmatrix} \xi_s \\ \xi_{s-1} \end{bmatrix} \| \leq \theta^s \nu C_0 \| \begin{bmatrix} \xi_0 \\ \xi_{-1} \end{bmatrix} \|$ and $\| \xi_s \| = \| U_s - Y \|_F$.

So the momentum term of each layer can be bounded as
\begin{equation} \label{eq:tmpM2}
\begin{split}
\| M_{t,l} \|_F & = \| \sum_{s=0}^t \beta^{t-s}  \frac{ \partial \ell(\W{L:1}_s)}{ \partial \W{l}_s }  \|_F
\leq \sum_{s=0}^t \beta^{t-s} \| \frac{ \partial \ell(\W{L:1}_s)}{ \partial \W{l}_s }  \|_F
\\ & \leq 
\frac{4 \| X \|_2}{\sqrt{d_y}} 
\sum_{s=0}^t \beta^{t-s} \theta^s \nu C_0  \| U_0 - Y \|_F.
\\ & \leq 
\frac{4 \| X \|_2}{\sqrt{d_y}} 
\sum_{s=0}^t \theta^{2(t-s)} \theta^s \nu C_0  \| U_0 - Y \|_F.
\\ &  \leq
\frac{4 \| X \|_2}{\sqrt{d_y}}  \frac{ \theta^t }{1 -\theta} \nu C_0 \| U_0 - Y \|_F, 
\end{split}
\end{equation}
where in the second to last inequality we use $\beta= \beta_*^2 \leq \theta^2$.

Combining all the pieces together, we can bound 
$\| \frac{1}{\sqrt{m^{L-1} d_y} } \Phi_t X \|_F$ as
\begin{equation}  
\begin{split}
& \| \frac{1}{\sqrt{m^{L-1} d_y} } \Phi_t X \|_F
\\ &
\overset{(a)}{\leq} \frac{1}{ \sqrt{ m^{L-1} d_y}  } \sum_{j=2}^L {L \choose j} 
\left( \eta 
\frac{4 \| X \|_2}{\sqrt{d_y}}  \frac{ \theta^{t} }{1 -\theta}  \nu C_0    \| U_0 - Y \|_F  \right)^j (1.1)^{j+1}  m^{\frac{L-j}{2}} \| X \|_2
\\ &
\overset{(b)}{\leq} 1.1 \frac{1}{ \sqrt{ m^{L-1} d_y}  } \sum_{j=2}^L L^j 
\left( \eta
\frac{4.4 \| X \|_2}{\sqrt{d_y}}  \frac{ \theta^{t} }{1 -\theta} \nu C_0 \| U_0 - Y \|_F\right)^j  m^{\frac{L-j}{2}} \| X \|_2
\\ &
\leq 1.1  \sqrt{ \frac{ m}{ d_y}  } \| X \|_2 \sum_{j=2}^L  
\left( \eta
\frac{ 4.4  L \| X \|_2}{\sqrt{m d_y}}  \frac{ \theta^{t} }{1 -\theta} \nu C_0   \| U_0 - Y \|_F\right)^j,   
\end{split}
\end{equation}
where (a) uses (\ref{eq:tmpM2}) and Lemma~\ref{lem:linear-eigen}
for bounding a $j \geq 2$ higher-order terms like
\[\frac{1}{ \sqrt{ m^{L-1} d_y} }\beta \W{L:k_j+1}_{t} \cdot (-\eta M_{t,k_j}) \W{k_j-1:k_{j-1}+1}_{t} 
\cdot (-\eta M_{t,k_{j-1}}) \cdots  (-\eta M_{t,k_{1}}) \cdot \W{k_1-1:1}_{t} \], where $1 \leq k_1 < \cdots < k_j \leq L$
 and (b) uses that ${L \choose j  } \leq \frac{L^j}{j!} $

To proceed, 
let us bound 
$\eta
\frac{ 4.4 L \| X \|_2}{\sqrt{m d_y}}  \frac{ \theta^{t} }{1 -\theta} \nu C_0  \| U_0 - Y \|_F$ in the sum above. We have
\begin{equation}
\begin{aligned}
\eta \frac{ 4.4  L \| X \|_2}{\sqrt{m d_y}}  \frac{ \theta^{t} }{1 -\theta} \nu C_0 \| U_0 - Y \|_F
 &
\leq 4.4
\sqrt{ \frac{  d_y}{ m }  } \frac{1}{ \| X \|_2 }  \frac{ \theta^{t} }{1 -\theta} \nu C_0  \| U_0 - Y \|_F
\\ & \leq 0.5,  
\end{aligned}
\end{equation}
where the last inequality uses that $\tilde{C}_1 \frac{d_y B_0^2 C_0^2  }{ \| X \|_2^2 } \frac{1}{ \left( 1 - \theta \right)^2}  \leq \tilde{C}_1
\frac{d_y B_0^2 C_0^2}{ \| X \|_2^2 } \frac{1}{ \eta \lambda}
\leq \tilde{C}_2
\frac{d_y B_0^2 \kappa^{2}  }{  \| X \|_2^2 } 
   \leq m $, for some sufficiently large constant $\tilde{C}_1, \tilde{C}_2 >0$.
Combining the above results, we have 
\begin{equation} \label{eq:phi}
\begin{split}
\| \phi_t \| &  = \| \frac{1}{ \sqrt{ m^{L-1} d_y}  } \Phi_t X \|_F
\\ & \leq 1.1  \sqrt{ \frac{ m}{ d_y}  } \| X \|_2 
\left( \eta
\frac{4.4  L \| X \|_2}{\sqrt{m d_y}}  \frac{ \theta^{t} }{1 -\theta} \nu C_0  \| U_0 - Y \|_F \right)^2  
\sum_{j=2}^{L-2} 
\left(  0.5 \right)^{j-2}   
\\ & \leq 2.2  \sqrt{ \frac{ m}{ d_y}  } \| X \|_2
\left( \eta
\frac{4.4  L \| X \|_2}{\sqrt{m d_y}}  \frac{ \theta^{t} }{1 -\theta} \nu C_0    \| U_0 - Y \|_F \right)^2   
\\ & \leq  \frac{ 43  \sqrt{d_y} }{ \sqrt{m} \| X \|_2}   
\left( 
 \frac{ \theta^{t} }{1 -\theta} \nu C_0  \| U_0 - Y \|_F \right)^2.   
\end{split}
\end{equation}

Now let us switch to upper-bound $\| \psi_t \|$.
It is equivalent to upper-bounding the Frobenius norm of
$\frac{1}{ \sqrt{ m^{L-1} d_y} }\beta (L-1)  \W{L:1}_{t} X + \frac{1}{ \sqrt{ m^{L-1} d_y} }\beta  \W{L:1}_{t-1} X \newline
- \frac{1}{ \sqrt{ m^{L-1} d_y} } \beta \sum_{l=1}^L \W{L:l+1}_t \W{l}_{t-1} \W{l-1:1}_{t} X$,
which can be rewritten as
\begin{equation} \label{eq:import}
\begin{aligned}
& \underbrace{\frac{1}{ \sqrt{ m^{L-1} d_y} }   \beta (L-1)  \cdot \Pi_{l=1}^L \left( \W{l}_{t-1} - \eta M_{t-1,l} \right) X }_{\text{first term} } + \underbrace{ \frac{1}{ \sqrt{ m^{L-1} d_y} }\beta  \W{L:1}_{t-1} X }_{\text{second term}}
\\ & 
\underbrace{
- \frac{1}{ \sqrt{ m^{L-1} d_y} } \beta \sum_{l=1}^L \Pi_{i=l+1}^L \left( \W{i}_{t-1} - \eta M_{t-1,i} \right)  \W{l}_{t-1} \Pi_{j=1}^{l-1} \left( \W{j}_{t-1} - \eta M_{t-1,j} \right)  X }_{\text{third term}}.
\end{aligned}
\end{equation}
The above can be written as $B_0 + \eta B_1 + \eta^2 B_2 + \dots + \eta^L B_L$ for some matrices $B_0,\dots, B_L \in \reals^{d_y \times n}$.
Specifically, we have
\begin{equation}
\begin{split}
B_0 & = \underbrace{ \frac{1}{ \sqrt{ m^{L-1} d_y} }(L-1) \beta \W{L:1}_{t-1} X }_{ \text{due to the first term} }+ 
 \underbrace{  \frac{1}{ \sqrt{ m^{L-1} d_y} }\beta  \W{L:1}_{t-1} X }_{ \text{due to the second term} }
 \underbrace{
- \frac{1}{ \sqrt{ m^{L-1} d_y} }\beta L \W{L:1}_{t-1} X }_{ \text{due to the third term} } = 0
\\
B_1 & = \underbrace{ - \frac{1}{ \sqrt{ m^{L-1} d_y} }(L-1) \beta \sum_{l=1}^L \W{L:l+1}_{t-1}  M_{t-1,l} \W{l-1:1}_{t-1} }_{ \text{due to the first term} }
\\ & \qquad + \underbrace{ \frac{1}{ \sqrt{ m^{L-1} d_y} }\beta \sum_{l=1}^L \sum_{k \neq l}
\W{L:k+1}_{t-1}  M_{t-1,k} \W{k-1:1}_{t-1} }_{ \text{due to the third term} }
= 0.
\end{split}
\end{equation}
So what remains on (\ref{eq:import}) are all the higher-order terms (in terms of the power of $\eta$), i.e. those with $\eta M_{t-1,i}$ and  $\eta M_{t-1,j}$, $\forall i \neq j$ or higher.

To continue, observe that for a fixed $(i,j)$, $i < j$, 
the second-order term that involves $\eta M_{t-1,i}$ and $\eta M_{t-1,j}$ on 
(\ref{eq:import}) is with coefficient
$\frac{1}{ \sqrt{ m^{L-1} d_y} }\beta$,
because the first term on (\ref{eq:import})
contributes to
$\frac{1}{ \sqrt{ m^{L-1} d_y} }(L-1) \beta $, while the third term on (\ref{eq:import})
contributes to
$- \frac{1}{ \sqrt{ m^{L-1} d_y} }(L-2) \beta$. 
Furthermore,
for a fixed $(i,j,k)$, $i < j < k$, 
the third-order term that involves $\eta M_{t-1,i}$, $\eta M_{t-1,j}$, and 
$\eta M_{t-1,k}$
on (\ref{eq:import}) is with coefficient $-2 \frac{1}{ \sqrt{ m^{L-1} d_y} }\beta$,
as the first term on (\ref{eq:import})
contributes to
$- \frac{1}{ \sqrt{ m^{L-1} d_y} }(L-1) \beta $, while the third term on (\ref{eq:import})
contributes to
$\frac{1}{ \sqrt{ m^{L-1} d_y} }(L-3) \beta$. 
Similarly, for a $p$-order term $\eta \underbrace{ M_{t-1,*}, \cdots, \eta M_{t-1,**} }_{ \text{p terms} }$,
the coefficient is 
$(p-1) \frac{1}{ \sqrt{ m^{L-1} d_y} }\beta (-1)^{p}$.

By induction (see (\ref{eq:tmpM2})), we can bound the norm of the momentum
at layer $l$ as
\begin{equation} \label{eq:tmpM}
\| M_{t-1,l} \|_F \leq \frac{4 \| X \|_2}{\sqrt{d_y}}  \frac{ \theta^{t-1} }{1 -\theta} \nu C_0 \| U_0 - Y \|_F .
\end{equation}
Combining all the pieces together, we have
\begin{equation} \label{eq:qqq1}
\begin{split}
& \| \frac{1}{ \sqrt{ m^{L-1} d_y} }\beta (L-1)  \W{L:1}_{t} X + \frac{1}{ \sqrt{ m^{L-1} d_y} }\beta  \W{L:1}_{t-1} X
\\ & \qquad \qquad
- \frac{1}{ \sqrt{ m^{L-1} d_y} } \beta \sum_{l=1}^L \W{L:l+1}_t \W{l}_{t-1} \W{l-1:1}_{t} X \|_F
\\ &
\overset{(a)}{\leq} \frac{\beta}{ \sqrt{ m^{L-1} d_y}  } \sum_{j=2}^L \left(j-1\right) {L \choose j} 
\left( \eta 
\frac{4 \| X \|_2}{\sqrt{d_y}}  \frac{ \theta^{t-1} }{1 -\theta}  \nu C_0   \| U_0 - Y \|_F  \right)^j (1.1)^{j+1}  m^{\frac{L-j}{2}} \| X \|_2
\\ &
\overset{(b)}{\leq} 1.1 \frac{\beta}{ \sqrt{ m^{L-1} d_y}  } \sum_{j=2}^L L^j 
\left( \eta
\frac{4.4 \| X \|_2}{\sqrt{d_y}}  \frac{ \theta^{t-1} }{1 -\theta} \nu C_0    \| U_0 - Y \|_F\right)^j  m^{\frac{L-j}{2}} \| X \|_2
\\ &
\leq 1.1 \beta \sqrt{ \frac{ m}{ d_y}  } \| X \|_2 \sum_{j=2}^L  
\left( \eta
\frac{ 4.4  L \| X \|_2}{\sqrt{m d_y}}  \frac{ \theta^{t-1} }{1 -\theta} \nu C_0    \| U_0 - Y \|_F\right)^j,   
\end{split}
\end{equation}
where (a) uses (\ref{eq:tmpM}), the above analysis of the coefficients of the higher-order terms
 and Lemma~\ref{lem:linear-eigen}
for bounding a $j \geq 2$ higher-order terms like
$\frac{1}{ \sqrt{ m^{L-1} d_y} }\beta (j-1) (-1)^{j} \W{L:k_j+1}_{t-1} \cdot (-\eta M_{t-1,k_j}) \W{k_j-1:k_{j-1}+1}_{t-1} 
\cdot (-\eta M_{t-1,k_{j-1}}) \cdots  (-\eta M_{t-1,k_{1}}) \cdot \W{k_1-1:1}_{t-1} $, where $1 \leq k_1 < \cdots < k_j \leq L$
 and (b) uses that ${L \choose j  } \leq \frac{L^j}{j!} $

Let us bound 
$\eta
\frac{ 4.4  L \| X \|_2}{\sqrt{m d_y}}  \frac{ \theta^{t-1} }{1 -\theta} \nu C_0   \| U_0 - Y \|_F$ in the sum above. We have
\begin{equation} \label{eq:qqq2}
\begin{aligned}
\eta \frac{ 4.4 L \| X \|_2}{\sqrt{m d_y}}  \frac{ \theta^{t-1} }{1 -\theta} \nu C_0  \| U_0 - Y \|_F
 &
\leq 4.4
\sqrt{ \frac{  d_y}{ m }  } \frac{1}{ \| X \|_2 }  \frac{ \theta^{t-1} }{1 -\theta} \nu C_0  \| U_0 - Y \|_F
\\ & \leq 0.5,  
\end{aligned}
\end{equation}
where the last inequality uses that $\tilde{C}_1 \frac{d_y B_0^2 C_0^2  }{ \| X \|_2^2 } \frac{1}{ \left( 1 - \theta \right)^2}  \leq \tilde{C}_1
\frac{d_y B_0^2 C_0^2 }{ \| X \|_2^2 } \frac{1}{ \eta \lambda}
\leq \tilde{C}_2
\frac{d_y B_0^2 \kappa^{2}  }{  \| X \|_2^2 } 
   \leq m $, for some sufficiently large constant $\tilde{C}_1, \tilde{C}_2 >0$.
Combining the above results, i.e. (\ref{eq:qqq1}) and (\ref{eq:qqq2}), we have 
\begin{equation} \label{eq:psi}
\begin{split}
\| \psi_t \|& \leq \| \frac{1}{ \sqrt{ m^{L-1} d_y} }\beta (L-1)  \W{L:1}_{t} X + \frac{1}{ \sqrt{ m^{L-1} d_y} }\beta  \W{L:1}_{t-1} X
\\ & \qquad \qquad \qquad \qquad - \frac{1}{ \sqrt{ m^{L-1} d_y} } \beta \sum_{l=1}^L \W{L:l+1}_t \W{l}_{t-1} \W{l-1:1}_{t} X \|_F
\\ & \leq 1.1 \beta \sqrt{ \frac{ m}{ d_y}  } \| X \|_2 
\left( \eta
\frac{4.4  L \| X \|_2}{\sqrt{m d_y}}  \frac{ \theta^{t-1} }{1 -\theta} \nu C_0   \| U_0 - Y \|_F \right)^2  
\sum_{j=2}^{L-2} 
\left(  0.5 \right)^{j-2}   
\\ & \leq 2.2 \beta \sqrt{ \frac{ m}{ d_y}  } \| X \|_2
\left( \eta
\frac{4.4  L \| X \|_2}{\sqrt{m d_y}}  \frac{ \theta^{t-1} }{1 -\theta} \nu C_0   \| U_0 - Y \|_F \right)^2   
\\ & \leq  \frac{ 43  \sqrt{d_y} }{ \sqrt{m} \| X \|_2}   
\left( 
 \frac{ \theta^{t-1} }{1 -\theta} \nu C_0   \| U_0 - Y \|_F \right)^2,
\end{split}
\end{equation}
where the last inequality uses $\eta \leq \frac{d_y}{ L \| X \|_2^2}$.

Now let us switch to bound $\| \iota_t \|$.
We have
\begin{equation} 
\begin{aligned} \label{eq:jj1}
& \| \iota_t \| = \| \eta (H_t - H_0) \xi_t \|
\\ & =
\frac{\eta}{m^{L-1} d_y} 
\| \sum_{l=1}^L \W{L:l+1}_t (\W{L:l+1}_t)^\top ( U_t - Y) (\W{l-1:1}_t X)^\top \W{l-1:1}_t X
\\ & \qquad \qquad \qquad \qquad -
\sum_{l=1}^L \W{L:l+1}_0 (\W{L:l+1}_0)^\top ( U_t - Y) (\W{l-1:1}_0 X)^\top \W{l-1:1}_0 X\|_F
\\ & \leq 
\frac{\eta}{m^{L-1} d_y} 
\sum_{l=1}^L  
\| \W{L:l+1}_t (\W{L:l+1}_t)^\top ( U_t - Y) (\W{l-1:1}_t X)^\top \W{l-1:1}_t  X 
\\ & \qquad \qquad \qquad \qquad -
\W{L:l+1}_0  (\W{L:l+1}_0 )^\top ( U_t - Y) (\W{l-1:1}_0 X)^\top \W{l-1:1}_0 X\|_F
\\ & \leq 
\frac{\eta}{m^{L-1} d_y} 
\\ & 
\times 
\sum_{l=1}^L  \big( 
\underbrace{ 
\|\left( \W{L:l+1}_t (\W{L:l+1}_t)^\top - \W{L:l+1}_0 (\W{L:l+1}_0)^\top \right)  ( U_t - Y) (\W{l-1:1}_t X)^\top \W{l-1:1}_t  X \|_F }_{\text{ first term} }
\\ & + 
\underbrace{ 
\| \W{L:l+1}_0 (\W{L:l+1}_0)^\top  ( U_t - Y) \left( \W{l-1:1}_t X)^\top \W{l-1:1}_t X  - (\W{l-1:1}_0 X)^\top \W{l-1:1}_0 X   \right)  \|_F \big)
}_{\text{ second term} }.
\end{aligned}
\end{equation}

Now let us bound the first term. We have
\begin{equation} \label{eq:j0}
\begin{aligned}
&\underbrace{ \|\left( \W{L:l+1}_t (\W{L:l+1}_t)^\top - \W{L:l+1}_0 (\W{L:l+1}_0)^\top \right)  ( U_t - Y) (\W{l-1:1}_t X)^\top \W{l-1:1}_t X  \|_F }_{\text{ first term} }
\\ & \leq 
\| \W{L:l+1}_t (\W{L:l+1}_t)^\top - \W{L:l+1}_0 (\W{L:l+1}_0)^\top \|_2
\| U_t - Y \|_F \|  (\W{l-1:1}_t X)^\top \W{l-1:1}_t X  \|_2.
\end{aligned}
\end{equation}
For $\|  (\W{l-1:1}_t X)^\top \W{l-1:1}_t X  \|_2$, by using 
Lemma~\ref{lem:linear-deviate2} and Lemma~\ref{lem:linear-eigen},
we have
\begin{equation} \label{eq:j00}
\| (\W{l-1:1}_t X)^\top \W{l-1:1}_t X \|_2 
\leq \left( \sigma_{\max}( \W{l-1:1}_t X) \right)^2 \leq
\left( 1.1 m^{\frac{l-1}{2}} \sigma_{\max}(X) \right)^2.
\end{equation}

For $\| \W{L:l+1}_t (\W{L:l+1}_t)^\top - \W{L:l+1}_0 (\W{L:l+1}_0)^\top \|_2
$, denote $\W{L:l+1}_t = \W{L:l+1}_0 + \Delta^{(L:l+1)}_t$, we have
\begin{equation} \label{eq:j1}
\begin{split}
& \| \W{L:l+1}_t (\W{L:l+1}_t)^\top - \W{L:l+1}_0 (\W{L:l+1}_0)^\top \|_2
\\ & \leq \| \Delta_t^{(L:l+1)}  (\W{L:l+1}_t)^\top  +  \W{L:l+1}_t (\Delta_t^{(L:l+1)})^\top + \Delta_t^{(L:l+1)} ( \Delta_t^{(L:l+1)})^\top \|_2
\\ & \leq 2 \| \Delta_t^{(L:l+1)} \|_2 \cdot \sigma_{\max} (\W{L:l+1}_t)  +
\| \Delta_t^{(L:l+1)} \|_2^2 
\\ & \leq 2 \| \Delta_t^{(L:l+1)} \|_2 \cdot 
\left( 1.1 m^{\frac{L-l}{2}}\right)  +
\| \Delta_t^{(L:l+1)} \|_2^2.
\end{split}
\end{equation}

Therefore, we have to bound $\| \Delta_t^{(L:l+1)} \|_2$.
We have for any $1\leq i \leq j \leq L$.
\begin{equation} \label{eq:tw}
\W{j:i}_t = \left( \W{j}_0  + \Delta_j \right)
 \cdots \left( \W{i}_0 + \Delta_i  \right),
\end{equation}
where $\| \Delta_i \|_2 \leq \|\W{i}_t - \W{i}_0 \|_F \leq D:
= \frac{64 \| X \|_2 \sqrt{d_y}}{ L \sigma_{\min}^2(X) } \nu  C_0 B_0$
by Lemma~\ref{lem:linear-deviate2}.
The product (\ref{eq:tw}) above minus $\W{j:i}_0$ can be written as a finite sum of some terms of the form
\begin{equation}
\W{j:k_l+1}_0 \Delta_{k_l} \W{k_l-1: k_{l-1}+1}_0 \Delta_{k_{l-1}} \cdots \Delta_{k_1} \W{k_1-1:i}_0,
\end{equation}
where $i \leq k_1 < \cdots < k_l \leq j$. Recall that 
$\| \W{j':i'}_0 \|_2 = m^{\frac{j'-i'+1}{2} }$ by Lemma~\ref{lem:DL-A}.
Thus, we can bound 
\begin{equation} \label{eq:delta}
\begin{split}
& \| \Delta_t^{(j:i)} \|_2 \leq
\| \W{j:i}_t - \W{j:i}_0 \|_F 
\\ & \leq 
\sum_{l=1}^{j-i+1} { j-i+1 \choose l } (D)^l m^{\frac{j-i+1-l}{2}}
= (\sqrt{m} + D)^{j-i+1} - (\sqrt{m})^{j-i+1}
\\ & = (\sqrt{m})^{j-i+1} \left(  (1+D/\sqrt{m})^{j-i+1} -1  \right)
\leq (\sqrt{m})^{j-i+1} \left(  (1+D/\sqrt{m})^{L} -1  \right)
\\ & 
\overset{(a)}{=} \left(  (1+\frac{1}{\sqrt{C'}L\kappa})^{L} -1  \right)( \sqrt{m} )^{j-i+1} \overset{(b)}{\leq} \left(\exp\left(\frac{1}{\sqrt{C'}\kappa}\right)-1\right)( \sqrt{m} )^{j-i+1} 
\\ & \overset{(c)}{\leq} \left(1 + (e-1)\frac{1}{\sqrt{C'}\kappa} - 1\right)( \sqrt{m} )^{j-i+1}  \overset{(d)}{\leq} 
 \frac{1}{480 \kappa } ( \sqrt{m} )^{j-i+1} ,
\end{split}
\end{equation}
where (a) 
uses $\frac{D}{\sqrt{m}} \leq \frac{1}{\sqrt{C'}L\kappa}$, 
for some constant $C'>0$, since $C' \frac{d_y C_0^2 B_0^2 \kappa^4}{ \| X \|^2_2  } \leq C \frac{d_y  B_0^2 \kappa^5}{ \| X \|^2_2  } \leq m$, (b) follows by the inequality $(1+x/n)^n\leq e^x, \forall x \geq 0, n >0$, (c) from Bernoulli's inequality $e^r\leq 1 + (e-1)r,\forall 0 \leq r\leq 1$, and (d) by choosing any sufficiently larger $C'$. 

From (\ref{eq:delta}), we have 
$\| \Delta_t^{(L:l+1)} \|_2 \leq  \frac{1}{480 \kappa } ( \sqrt{m} )^{L-l}.$
Combining this with (\ref{eq:j0}), (\ref{eq:j00}), and (\ref{eq:j1}), we have
\begin{equation} \label{eq:jj2}
\begin{split}
&\underbrace{ \|\left( \W{L:l+1}_t (\W{L:l+1}_t)^\top - \W{L:l+1}_0 (\W{L:l+1}_0)^\top \right)  ( U_t - Y) (\W{l-1:1}_t X)^\top \W{l-1:1}_t X  \|_F }_{\text{ first term} }
\\ & 
\leq \left( 2 \| \Delta_t^{(L:l+1)} \|_2 \cdot 
\left( 1.1 m^{\frac{L-l}{2}}\right)  +
\| \Delta_t^{(L:l+1)} \|_2^2 \right)
\left( 1.1 m^{\frac{l-1}{2}} \sigma_{\max}(X) \right)^2 \| U_t - Y \|_F
\\ & 
\leq \left( 2 \frac{1}{480 \kappa } ( \sqrt{m} )^{L-l} \cdot 
\left( 1.1 m^{\frac{L-l}{2}} \right)  +
\big( \frac{1}{480 \kappa } ( \sqrt{m} )^{L-l} \big)^2 \right)
\left( 1.1 m^{\frac{l-1}{2}} \sigma_{\max}(X) \right)^2 \| U_t - Y \|_F
\\ & 
\leq \frac{\sigma_{\min}^2(X)}{160} m^{L-1}  \| U_t - Y \|_F,
\end{split}
\end{equation}
where in the last inequality we use $\kappa := \frac{\sigma_{\max}^2(X)}{\sigma_{\min}^2(X) }$.

Now let us switch to bound the second term, we have
\begin{equation} \label{eq:j33}
\begin{split}
& \underbrace{ 
\| (\W{L:l+1}_0 (\W{L:l+1}_0)^\top  ( U_t - Y) \left( \W{l-1:1}_t X)^\top \W{l-1:1}_t X  - (\W{l-1:1}_0 X)^\top \W{l-1:1}_0 X   \right)  \|_F \big)
}_{\text{ second term} }
\\ & 
\leq 
\| (\W{L:l+1}_0 (\W{L:l+1}_0)^\top \|_2 \| U_t - Y \|_F 
\| (\W{l-1:1}_t X)^\top \W{l-1:1}_t X  - (\W{l-1:1}_0 X)^\top \W{l-1:1}_0 X   \|_2.
\end{split}
\end{equation}
For $\| \W{L:l+1}_0 (\W{L:l+1}_0)^\top \|_2 $, based on Lemma~\ref{lem:DL-A}, we have
\begin{equation} \label{eq:j30}
\| \W{L:l+1}_0 (\W{L:l+1}_0)^\top \|_2 \leq m^{L-l}.
\end{equation}
To bound 
$\| (\W{l-1:1}_t X)^\top \W{l-1:1}_t X  - (\W{l-1:1}_0 X)^\top \W{l-1:1}_0 X   \|_2 $, we proceed as follows.
Denote $\W{l-1:1}_t = \W{l-1:1}_0 + \Delta^{(l-1:1)}_t$, we have
\begin{equation} \label{eq:j3}
\begin{split}
& \| (\W{l-1:1}_t X)^\top \W{l-1:1}_t X  - (\W{l-1:1}_0 X)^\top \W{l-1:1}_0 X   \|_2 
\\ & \leq 2 \| (\Delta^{(l-1:1)}_t X)^\top \W{l-1:1}_t X \|_2 + \| \Delta^{(l-1:1)}_t X \|^2_2
\\ &
\leq \left( 2 \ \| \Delta^{(l-1:1)}_t \| \| \W{l-1:1}_t \|_2
+  \| \Delta^{(l-1:1)}_t \|^2_2
\right) \| X \|^2_2
\\ &
\leq \left(2  \frac{1}{480 \kappa} m^{\frac{l-1}{2}}  1.1 m^{\frac{l-1}{2}}
+  \left( \frac{1}{480 \kappa} m^{\frac{l-1}{2}} \right)^2
\right) \| X \|^2_2
\\ &
\leq \frac{\sigma_{\min}^2(X)}{160} m^{l-1}  ,
\end{split}
\end{equation}
where the second to last inequality uses (\ref{eq:delta}),
Lemma~\ref{lem:linear-deviate2}, and
Lemma~\ref{lem:linear-eigen}, while the last inequality uses
$\kappa := \frac{\sigma_{\max}^2(X)}{\sigma_{\min}^2(X) }$.
Combining (\ref{eq:j33}), (\ref{eq:j30}), (\ref{eq:j3}), we have
\begin{equation} \label{eq:jj3}
\begin{split}
& \underbrace{ 
\| (\W{L:l+1}_0 (\W{L:l+1}_0)^\top  ( U_t - Y) \left( \W{l-1:1}_t X)^\top \W{l-1:1}_t X  - (\W{l-1:1}_0 X)^\top \W{l-1:1}_0 X   \right)  \|_F \big)
}_{\text{ second term} }
\\ & \leq \frac{\sigma_{\min}^2(X)}{160} m^{L-1} \| U_t - Y \|_F .
\end{split}
\end{equation}

Now combing (\ref{eq:jj1}), (\ref{eq:jj2}), and (\ref{eq:jj3}), we have
\begin{equation} \label{eq:iota}
\begin{split}
\| \iota_t \| \leq \frac{\eta}{m^{L-1} d_y} L \frac{\sigma_{\min}^2(X)}{80} m^{L-1} \| U_t - Y \|_F = \frac{\eta \lambda}{80} \| \xi_t \|,
\end{split}
\end{equation} 
where we use $\lambda:= \frac{L \sigma^2_{\min}(X)}{ d_y }$.

Now we have (\ref{eq:phi}), (\ref{eq:psi}), and  (\ref{eq:iota}),  
which leads to
\begin{equation}
\begin{split}
\| \varphi_t \| & \leq 
\| \phi_t \| + \|\psi_t \| + \| \iota_t \| 
\\ & \leq 
\frac{ 43  \sqrt{d_y} }{\sqrt{m} \| X \|_2}   
( \theta^{2t} +  \theta^{2(t-1)}) \nu^2 C_0^2
\left( 
 \frac{ \| \xi_0 \|  }{1 -\theta}  \right)^2
 + \frac{\eta \lambda}{80} \nu C_0  \| \begin{bmatrix} \xi_0 \\ \xi_{-1} \end{bmatrix} \| .
\\ & \leq
\frac{ 1920 \sqrt{d_y} }{\sqrt{m} \| X \|_2} \frac{1}{\eta \lambda}   
 \theta^{2t} \nu^2 C_0^2  \| \begin{bmatrix} \xi_0 \\ \xi_{-1} \end{bmatrix}\|^2
 + \frac{\eta \lambda}{80} \nu C_0  \| \begin{bmatrix} \xi_0 \\ \xi_{-1} \end{bmatrix} \| .
\end{split}
\end{equation}
where the last inequality uses that $1 \leq \frac{16}{9} \theta^2$ as $\eta \lambda \leq 1$ so that $\theta \geq \frac{3}{4}$.
\end{proof}

\begin{lemma}~\label{lem:linear-deviate2}
Following the setting as Theorem~\ref{thm:LinearNet},
denote $\theta := \beta_* + \frac{1}{4} \sqrt{ \eta\lambda } = 1 - \frac{1}{4} \sqrt{ \eta\lambda }$.
If for any $s \leq t$, the residual dynamics satisfies
$\textstyle \|
\begin{bmatrix}
\xi_{s} \\
\xi_{s-1} 
\end{bmatrix}
\| \leq \theta^{s} 
\cdot \nu C_0
\|
 \begin{bmatrix}
\xi_{0} \\
\xi_{-1} 
\end{bmatrix}
\|,
$ 
for some constant $\nu > 0$,
then 
\[
\| \W{l}_t - \W{l}_0 \|_F \leq R^{L\text{-linear}} := 
\frac{64 \| X \|_2 \sqrt{d_y}}{ L \sigma_{\min}^2(X) } \nu C_0 B_0.
\]
\end{lemma}

\begin{proof}
We have 
\begin{equation}
\begin{split}
 \| \W{l}_{t+1} - \W{l}_0 \|_F 
& \overset{(a)}{\leq} 
\eta \sum_{s=0}^t \|  M_{s,l} \|_F  
\overset{(b)}{=} 
\eta \sum_{s=0}^t
\| \sum_{\tau=0}^s \beta^{s-\tau}   \frac{ \partial \ell(\W{L:1}_{\tau})}{ \partial \W{L}_{\tau} }  \|_F
\\ & \leq 
\eta \sum_{s=0}^t \sum_{\tau=0}^s \beta^{s-\tau} 
\| \frac{ \partial \ell(\W{L:1}_{\tau})}{ \partial \W{L}_{\tau} }   \|_F
\\ &
\overset{(c)}{\leq} \eta \sum_{s=0}^t \sum_{\tau=0}^s \beta_*^{2(s-\tau)} 
\frac{4 \| X \|_2}{\sqrt{d_y}} \theta^{\tau} \nu C_0 \| U_0 - Y \|_F.
\\ &
\overset{(d)}{\leq} \eta \sum_{s=0}^t \frac{ \theta^{s} }{ 1 - \theta}
\frac{4 \| X \|_2}{\sqrt{d_y}} \nu C_0 \| U_0 - Y \|_F.
\\ &
\leq 
\frac{4 \eta \| X \|_2}{\sqrt{d_y}} \frac{1}{(1 - \theta) (1 - \theta)} \nu C_0 \| U_0 - Y \|_F
\\ &
\overset{(e)}{\leq} \frac{64 \| X \|_2}{ \lambda \sqrt{d_y}} \nu C_0 \| U_0 - Y \|_F
\\ &
\overset{(f)}{\leq} \frac{64 \| X \|_2 \sqrt{d_y}}{ L \sigma_{\min}^2(X) } \nu C_0 B_0,
\end{split}
\end{equation}
where (a), (b) is by the update rule of momentum, which is
$ \W{l}_{t+1} - \W{l}_t = - \eta M_{t,l}$, where $M_{t,l}:= \sum_{s=0}^t \beta^{t-s} \frac{ \partial \ell(W_{L:1})}{ \partial \W{l}_s }$, (c) is because $\|\frac{ \partial \ell(W_{L:1})}{ \partial \W{l}_s }\|_F = 
\frac{4 \| X \|_2}{\sqrt{d_y}} \theta^{s} \nu C_0 \| U_0 - Y \|_F$ (see (\ref{eq:gnorm-linear})),
(d) is because that $\beta= \beta_*^2 \leq \theta^2$,
(e) is because that $\frac{1}{(1-\theta)^2} = \frac{16}{\eta \lambda}$,
and (f) uses the upper-bound $B_0 \geq \| U_0 - Y \|$ defined in Lemma~\ref{lem:DL-A} and $\lambda:= \frac{L \sigma_{\min}^2(X)}{d_y}$.
The proof is completed.

\end{proof}

\begin{lemma}  \citet{HXP20} \label{lem:linear-eigen}
Let $R^{L\text{-linear}}$ be an upper bound that satisfies $\| \W{l}_t - \W{l}_t \|_F \leq R^{L\text{-linear}}$ for all $l$ and $t$.
Suppose the width $m$ satisfies $m > C (LR^{L\text{-linear}})^2$, where $C$ is any sufficiently large constant.
Then, 
\[
\begin{aligned}
\textstyle \sigma_{\max}( \W{j:i}_t ) \leq 1.1 m^{ \frac{j-i+1}{2} },  &\textstyle \text{ } 
\sigma_{\min}( \W{j:i}_t ) \geq 0.9 m^{ \frac{j-i+1}{2} }.
\end{aligned}
\]
\end{lemma}
\begin{proof}
The lemma has been proved in proof of Claim 4.4 and Claim 4.5 in \citet{HXP20}. For completeness, let us replicate the proof here.

We have for any $1\leq i \leq j \leq L$.
\begin{equation}
\W{j:i}_t = \left( \W{j}_0  + \Delta_j \right)
 \cdots \left( \W{i}_0 + \Delta_i  \right),
\end{equation}
where $\Delta_i = \W{i}_t - \W{i}_0$.
The product above minus $\W{j:i}_0$ can be written as a finite sum of some terms of the form
\begin{equation}
\W{j:k_l+1}_0 \Delta_{k_l} \W{k_l-1: k_{l-1}+1}_0 \Delta_{k_{l-1}} \cdots \Delta_{k_1} \W{k_1-1:i}_0,
\end{equation}
where $i \leq k_1 < \cdots < k_l \leq j$. Recall that 
$\| \W{j':i'}_0 \|_2 = m^{\frac{j'-i'+1}{2} }$.
Thus, we can bound 
\begin{equation} \label{eq:wdist}
\begin{aligned}
\| \W{j:i}_t - \W{j:i}_0 \|_F & \leq 
\sum_{l=1}^{j-i+1} { j-i+1 \choose l } (R^{L\text{-linear}})^l m^{\frac{j-i+1-l}{2}}
\\ & = (\sqrt{m} + R^{L\text{-linear}})^{j-i+1} - (\sqrt{m})^{j-i+1}
\\ & = (\sqrt{m})^{j-i+1} \left(  (1+R^{L\text{-linear}}/\sqrt{m})^{j-i+1} -1  \right)
\\ & \leq (\sqrt{m})^{j-i+1} \left(  (1+R^{L\text{-linear}}/\sqrt{m})^{L} -1  \right)
\\ & \leq 0.1 ( \sqrt{m} )^{j-i+1},
\end{aligned}
\end{equation}
where the last step uses $m > C (LR^{L\text{-linear}})^2$.
By combining this with Lemma~\ref{lem:DL-A}, one can obtain the result.

\end{proof}
\noindent
\textbf{Remark:}
In the proof of Lemma~\ref{lem:linear-deviate1}, we obtain a tighter bound of the distance $\| \W{j:i}_t - \W{j:i}_0 \|_F \leq O(\frac{1}{\kappa} ( \sqrt{m} )^{j-i+1} )$. 
However,
to get the upper-bound $\sigma_{\max}( \W{j:i}_t )$ shown in Lemma~\ref{lem:linear-eigen}, (\ref{eq:wdist}) is sufficient for the purpose.

\subsection{Proof of Theorem~\ref{thm:LinearNet}} \label{sec:linear}

\begin{proof} (of Theorem~\ref{thm:LinearNet})
Denote $\lambda:= L \sigma_{\min}^2(X) / d_y$.
By Lemma~\ref{lem:DL-A}, $\lambda_{\min}(H) \geq \lambda$.
Also, denote $\beta_*:=1 - \frac{1}{2} \sqrt{\eta \lambda}$
and $\theta := \beta_* + \frac{1}{4} \sqrt{\eta \lambda} = 1 - \frac{1}{4} \sqrt{\eta \lambda} $.
Let $\nu = 2 $ in Lemma~\ref{lem:linear-deviate1},
~\ref{lem:linear-deviate2},
and let $C_1 = C_3 = C_0$ and
$C_2 = \frac{1}{4} \sqrt{\eta \lambda}$ in
Theorem~\ref{thm:metas}.
The goal is to show that
$
\left\|
\begin{bmatrix}
\xi_{t} \\
\xi_{t-1} 
\end{bmatrix}
\right\|
\leq \theta^{t} 2 C_0  \left\|
\begin{bmatrix}
\xi_{0} \\
\xi_{-1} 
\end{bmatrix}
\right\|
$ for all $t$ by induction.
To achieve this, we will also use induction to show that for all iterations $s$,
\begin{equation} \label{induct:linear}
\forall l \in [L], \| \W{l}_t - \W{l}_0 \|
 \leq   R^{L\text{-linear}}:= \frac{64 \| X \|_2 \sqrt{d_y}}{ L \sigma_{\min}^2(X) }  C_0 B_0,  
\end{equation}
which is clearly true in the base case $s=0$.

By Lemma~\ref{lem:DL-residual},~\ref{lem:DL-A}, \ref{lem:linear-deviate1},~\ref{lem:linear-deviate2}, Theorem~\ref{thm:metas} and Corollary~\ref{corr:1}, it suffices to show that
$
\left\|
\begin{bmatrix}
\xi_{s} \\
\xi_{s-1} 
\end{bmatrix}
\right\|
\leq
\theta^{s}  \cdot 2C_0 \left\|
\begin{bmatrix}
\xi_{0} \\
\xi_{-1} 
\end{bmatrix}
\right\|
$
and 
$
\forall l \in [L], \| \W{l}_s - \W{l}_0 \|
\leq   R^{L\text{-linear}}$
 hold at $s=0,1,\dots,t-1$, one has 
\begin{eqnarray}
\| \sum_{s=0}^{t-1} A^{t-s-1} \begin{bmatrix}
\varphi_s \\
0 
\end{bmatrix}
\|
& \leq & 
\theta^{t}
 C_0
\left\|
\begin{bmatrix}
\xi_{0} \\
\xi_{-1} 
\end{bmatrix}
\right\|,
\label{eq:thm-DL-1}
\\ \forall l \in [L], \| \W{l}_t - \W{l}_0 \| & \leq & R^{L\text{-linear}}:= \frac{64 \| X \|_2 \sqrt{d_y}}{ L \sigma_{\min}^2(X) }  C_0 B_0,
 \label{eq:thm-DL-2}
\end{eqnarray}
where the matrix $A$ and the vector $\varphi_t$ are defined in Lemma~\ref{lem:DL-residual}, and $B_0$ is a constant such that $B_0 \geq \| Y - U_0 \|_F$ with probability $1-\delta$ by Lemma~\ref{lem:DL-A}.
 The inequality (\ref{eq:thm-DL-1}) is the required  condition for using the result of Theorem~\ref{thm:metas}, while the inequality (\ref{eq:thm-DL-2}) helps us to show (\ref{eq:thm-DL-1}) through invoking Lemma~\ref{lem:linear-deviate1} to bound the terms $\{\varphi_s\}$ as shown in the following.

Let us show (\ref{eq:thm-DL-1}) first.
We have
\begin{equation} \label{eq:6-linear}
\begin{aligned}
 \| \sum_{s=0}^{t-1} A^{t-1-s} \| \begin{bmatrix}
\varphi_{s} \\ 0 \end{bmatrix} \|
& \overset{(a)}{ \leq} 
\sum_{s=0}^{t-1} \beta_*^{t-1-s} C_0 \| \varphi_s \|
\\ & \overset{(b)}{ \leq} 
\frac{ 1920 \sqrt{d_y} }{\sqrt{m} \| X \|_2} \frac{1}{\eta \lambda}   
\sum_{s=0}^{t-1} \beta_*^{t-1-s}
 \theta^{2s}  4 C_0^3  \| \begin{bmatrix} \xi_0 \\ \xi_{-1} \end{bmatrix} \|^2
\\ & \qquad +  \sum_{s=0}^{t-1} \beta_*^{t-1-s}
\frac{\eta \lambda}{80}  \theta^{s} 2 C_0^2 
  \| \begin{bmatrix} \xi_0 \\ \xi_{-1} \end{bmatrix} \| 
\\ & \overset{(c)}{ \leq} 
\frac{ 1920 \sqrt{d_y} }{\sqrt{m} \| X \|_2} \frac{1}{\eta \lambda}   
\sum_{s=0}^{t-1} \beta_*^{t-1-s}
 \theta^{2s}  4 C_0^3  \| \begin{bmatrix} \xi_0 \\ \xi_{-1} \end{bmatrix} \|^2
\\ & \qquad + \frac{2 \sqrt{\eta \lambda}}{15}  \theta^{t} C_0^2 
  \| \begin{bmatrix} \xi_0 \\ \xi_{-1} \end{bmatrix} \| 
 \\ & \overset{(d)}{ \leq} 
\frac{ 1920 \sqrt{d_y} }{\sqrt{m} \| X \|_2} \frac{16}{3(\eta \lambda)^{3/2}}   
\theta^{t} 4 C_0^3
\| \begin{bmatrix} \xi_0 \\ \xi_{-1} \end{bmatrix} \|^2 +  
\frac{2 \sqrt{\eta \lambda}}{15}  \theta^{t} C_0^2 
  \| \begin{bmatrix} \xi_0 \\ \xi_{-1} \end{bmatrix} \| 
\\ & \overset{(e)}{ \leq} 
\frac{1}{3} \theta^{t} C_0
\| \begin{bmatrix} \xi_0 \\ \xi_{-1} \end{bmatrix} \|
 + 
\frac{2 \sqrt{\eta \lambda}}{15}  \theta^{t} C_0^2 
  \| \begin{bmatrix} \xi_0 \\ \xi_{-1} \end{bmatrix} \| 
\\ & \overset{(f)}{ \leq}
\theta^{t}   C_0  \| \begin{bmatrix} \xi_0 \\ \xi_{-1} \end{bmatrix} \|,
\end{aligned}
\end{equation}
where (a) uses Theorem~\ref{thm:akv} with $\beta = \beta_*^2$,
(b)
is by Lemma~\ref{lem:linear-deviate1}, (c)
uses
$\sum_{s=0}^{t-1} \beta_*^{t-1-s} \theta^s = \theta^{t-1} \sum_{s=0}^{t-1} \left( \frac{\beta_*}{\theta}  \right)^{t-1-s} \leq \theta^{t-1} \sum_{s=0}^{t-1} \theta^{t-1-s}$ $\leq \theta^{t-1} \frac{4}{\sqrt{\eta\lambda}} \leq \theta^{t} \frac{16}{3\sqrt{\eta\lambda}} $, $\beta_* = 1 - \frac{1}{2} \sqrt{\eta \lambda} \geq \frac{1}{2}$
, and $\theta = 1 - \frac{1}{4} \sqrt{\eta \lambda} \geq \frac{3}{4}$,
(d) uses $\sum_{s=0}^{t-1} \beta_*^{t-1-s} \theta^{2s} \leq \sum_{s=0}^{t-1} \theta^{t-1+s} \leq \frac{\theta^{t-1}}{1-\theta}  \leq \theta^{t} \frac{16}{3\sqrt{\eta\lambda}} $,
(e) is because
$C' \frac{d_y  C_0^4 B_0^2 }{  \| X \|_2^2 } \frac{1}{ (\eta \lambda)^3 } 
\leq C \frac{d_y  B_0^2 }{ \| X \|_2^2 } \kappa^5  \leq m
$ for some sufficiently large constants $C', C > 0$, and (f) uses that $\eta \lambda = \frac{1}{\kappa}$ and $C_0 \leq 4 \sqrt{\kappa}$ by Corollary~\ref{corr:1}.
Hence, we have shown (\ref{eq:thm-DL-1}). 
 Therefore, by Theorem~\ref{thm:metas}, we have
$ \left\|
\begin{bmatrix}
\xi_{t} \\
\xi_{t-1} 
\end{bmatrix}
\right\|
\leq \theta^{t} 2 C_0  \left\|
\begin{bmatrix}
\xi_{0} \\
\xi_{-1} 
\end{bmatrix}
\right\|.
$

By Lemma~\ref{lem:linear-deviate2}, we have (\ref{eq:thm-DL-2}). 
Thus, we have completed the proof.

\end{proof}
\subsection{Non-asymptotic accelerated linear rate of the local convergence
for solving $f(\cdot) \in F_{\mu,\alpha}^2$}  \label{app:thm:STC}
\begin{theorem} \label{thm:STC}
Assume that the function $f(\cdot) \in F_{\mu,\alpha}^2$ and its Hessian is $\alpha$-Lipschitz. Denote the condition number $\kappa:= \frac{\alpha}{\mu}$.
Suppose that the initial point satisfies $\| \begin{bmatrix} w_0 - w_* \\ w_{-1} - w_* \end{bmatrix} \| \leq
\frac{1}{ 683 \kappa^{3/2}}$. Then,
Gradient descent with Polyak's momentum with the step size $\eta = \frac{1}{\alpha}$ and the momentum parameter $\beta = \left(1 - \frac{1}{2 \sqrt{\kappa}}\right)^2$ for solving $\min_{w} f(w)$ has
\begin{equation}\label{thm:qq1}
\|
\begin{bmatrix}
w_{t+1} - w_* \\
w_{t} - w_*
\end{bmatrix}
\| \leq \left(  1 - \frac{1}{4 \sqrt{\kappa}}   \right)^{t+1} 8 \sqrt{\kappa}
\|
\begin{bmatrix}
w_{0} - w_* \\
w_{-1} - w_*
\end{bmatrix}
\|,
\end{equation}
where $w_* = \arg\min_w f(w)$.
\end{theorem}
\noindent
\textbf{Remark:}
Compared to Theorem~9 of \citet{P64},
Theorem~\ref{thm:STC} clearly indicates the required distance that ensures 
an acceleration when the iterate is in the neighborhood of the global minimizer.
Furthermore,
the rate is in the non-asymptotic sense instead of the asymptotic one.

\begin{proof}
In the following, we denote $\xi_t:= w_t - w_*$ and
denote $\lambda := \mu >0$, which is a lower bound of $\lambda_{\min}(H)$ of the matrix $H:= \int_0^1 \nabla^2 f\big( (1-\tau) w_0 + w_* \big) d \tau$ defined in Lemma~\ref{lem:SC-residual}, i.e.
$\lambda_{\min}(H) \geq \lambda.$
Also, denote $\beta_*:=1 - \frac{1}{2} \sqrt{\eta \lambda}$
and $\theta := \beta_* + \frac{1}{4} \sqrt{\eta \lambda} = 1 - \frac{1}{4} \sqrt{\eta \lambda} $. Suppose $\eta = \frac{1}{\alpha }$, where $\alpha$ is the smoothness constant.
Denote 
$
 C_0:=\frac{\sqrt{2} (\beta+1)}{
\sqrt{ \min\{  
h(\beta,\eta \lambda_{\min}(H)) , h(\beta,\eta \lambda_{\max}(H)) \} } }
\leq 4 \sqrt{\kappa}
$ by Corollary~\ref{corr:1}. 
Let $C_1 = C_3 = C_0$ and
$C_2 = \frac{1}{4} \sqrt{\eta \lambda}$ in
Theorem~\ref{thm:metas}.
The goal is to show that
$
\left\|
\begin{bmatrix}
\xi_{t} \\
\xi_{t-1} 
\end{bmatrix}
\right\|
\leq \theta^{t} 2 C_0  \left\|
\begin{bmatrix}
\xi_{0} \\
\xi_{-1} 
\end{bmatrix}
\right\|
$ for all $t$ by induction. 
To achieve this, we will also use induction to show that for all iterations $s$,
\begin{equation} \label{induct:stc}
 \textstyle \| w_s - w_* \|  \textstyle  \leq  \textstyle  \textstyle R:= \frac{3}{64  \sqrt{\kappa} C_0 }.  
\end{equation}
A sufficient condition for the base case $s=0$ of (\ref{induct:stc}) to hold is
\begin{equation} \label{eq:h3}
\| \begin{bmatrix} w_0 - w_* \\ w_{-1} - w_* \end{bmatrix} \| \leq \frac{R}{2 C_0} = \frac{3}{128 \sqrt{\kappa} C_0^2},
\end{equation}
as $C_0 \geq 1$ by Theorem~\ref{thm:akv},
which in turn can be guaranteed if
$\| \begin{bmatrix} w_0 - w_* \\ w_{-1} - w_* \end{bmatrix} \|  \leq \frac{1}{ 683 \kappa^{3/2}}$ by using the upper bound $C_0 \leq 4 \sqrt{\kappa}$ of Corollary~\ref{corr:1}.

From Lemma~\ref{lem:SC-residual}, we have
\begin{equation} \label{eq:varphi-sc}
\begin{split}
\| \phi_s \| & \leq \eta \| \int_0^1 \nabla^2 f( (1-\tau) w_s + \tau w_* ) d\tau - \int_0^1 \nabla^2 f( (1-\tau) w_0 + \tau w_* ) d \tau \| \| \xi_s \|
\\ & \overset{(a)}{\leq } \eta  \alpha \left( \int_0^1 (1-\tau) \| w_s - w_0 \| d \tau \right)  \| \xi_s \|
\leq \eta \alpha \| w_s - w_0 \| \| \xi_s \|
\\ & \overset{(b)}{ \leq} 
\eta \alpha \left( \| w_s - w_* \| + \| w_0 - w_* \| \right) \| \xi_s \|, 
\end{split}
\end{equation}
where (a) is by $\alpha$-Lipschitzness of the Hessian and (b) is by the triangle inequality.
By \eqref{induct:stc}, \eqref{eq:varphi-sc}, Lemma~\ref{lem:SC-residual}, Theorem~\ref{thm:metas}, and Corollary~\ref{corr:1},
  it suffices to show that
given
$
\left\|
\begin{bmatrix}
\xi_{s} \\
\xi_{s-1} 
\end{bmatrix}
\right\|
\leq \theta^{s} 2 C_0  \left\|
\begin{bmatrix}
\xi_{0} \\
\xi_{-1} 
\end{bmatrix}
\right\|
$
and 
$\textstyle \| w_s - w_* \|  \textstyle  \leq  \textstyle  \textstyle R:= \frac{3}{64  \sqrt{\kappa} C_0 }$
 hold at $s=0,1,\dots,t-1$, one has
\begin{eqnarray}
\| \sum_{s=0}^{t-1} A^{t-s-1} \begin{bmatrix}
\varphi_s \\
0 
\end{bmatrix}
\|
& \leq &  \theta^{t}
C_0
\left\|
\begin{bmatrix}
\xi_{0} \\
\xi_{-1}  
\end{bmatrix}
\right\| \label{eq:Avarphi} \\
 \| w_t - w_* \| & \leq & R:= \frac{3}{64  \sqrt{\kappa} C_0 }, \label{eq:wR}
\end{eqnarray}
where $A:= \begin{bmatrix} 
(1 + \beta) I_n - \eta   \int_0^1 \nabla^2 f\big( (1-\tau) w_0 + w_* \big) d \tau & - \beta I_n \\
I_n & 0
\end{bmatrix}
$. 

We have
\begin{equation}
\begin{aligned}
\| \sum_{s=0}^{t-1} A^{t-s-1} \begin{bmatrix}
\varphi_s \\
0 
\end{bmatrix}
\|
& \leq  \sum_{s=0}^{t-1} \| A^{t-s-1} \begin{bmatrix}
\varphi_s \\
0 
\end{bmatrix}
\|
\\ & \overset{(a)}{\leq} \sum_{s=0}^{t-1} \beta_*^{t-s-1} C_0 \| \varphi_s \|
\\ & \overset{(b)}{\leq} 4 \eta \alpha R C_0^2 
 \sum_{s=0}^{t-1} \beta_*^{t-s-1} \theta^{s} 
 \|
\begin{bmatrix}
\xi_{0}  \\
\xi_{-1}  
\end{bmatrix}
\|
\\ &
\overset{(c)}{ \leq} R C_0^2  \frac{64}{3 \sqrt{\eta \lambda} } 
\theta^{t} \|
\begin{bmatrix}
\xi_{0}  \\
\xi_{-1}  
\end{bmatrix} 
\|
\\ &
\overset{(d)}{ \leq} C_0 \theta^{t} \|
\begin{bmatrix}
\xi_{0}  \\
\xi_{-1}  
\end{bmatrix} 
\|,
 \end{aligned}
\end{equation}
where (a) uses Theorem~\ref{thm:akv} with $\beta = \beta_*^2$,
(b) is by (\ref{eq:varphi-sc}), (\ref{induct:stc}), and the induction that $\| \xi_s \| \leq \theta^s 2 C_0 \| \begin{bmatrix}
\xi_{0}  \\
\xi_{-1}  
\end{bmatrix} 
\|
$,
(c) is because
 $\sum_{s=0}^{t-1} \beta_*^{t-1-s} \theta^s = \theta^{t-1} \sum_{s=0}^{t-1} \left( \frac{\beta_*}{\theta}  \right)^{t-1-s}$  $\leq \theta^{t-1} \sum_{s=0}^{t-1} \theta^{t-1-s}$ $\leq \theta^{t-1} \frac{4}{\sqrt{\eta\lambda}} \leq \theta^t \frac{16}{3 \sqrt{\eta\lambda} } $,
 and (d) is due to the definition of $R := \frac{3}{64 \sqrt{\kappa} C_0}$.
 Therefore, by Theorem~\ref{thm:metas}, we have
$ \left\|
\begin{bmatrix}
\xi_{t} \\
\xi_{t-1} 
\end{bmatrix}
\right\|
\leq \theta^{t} 2 C_0  \left\|
\begin{bmatrix}
\xi_{0} \\
\xi_{-1} 
\end{bmatrix}
\right\|.
$

Now let us switch to show (\ref{eq:wR}).
We have
\begin{equation}
\| \xi_t \| := \| w_t - w_* \| \overset{\text{induction}}{ \leq} \theta^t 2 C_0 \| \begin{bmatrix} w_0 - w_* \\ w_{-1} - w_* \end{bmatrix} \| \leq R,
\end{equation}
where the last inequality uses the constraint $\| \begin{bmatrix} w_0 - w_* \\ w_{-1} - w_* \end{bmatrix} \| \leq \frac{R}{2 C_0}$ by (\ref{eq:h3}).
\end{proof}

\section{Experiments} \label{app:exp}

\subsection{ReLU network}
We report a proof-of-concept experiment for training the ReLU network. We sample $n=5$ points from the normal distribution, and then scale the size to the unit norm. We generate the labels uniformly random from $\{1,-1\}$. We let $m=1000$ and $d=10$. We compare vanilla GD and gradient descent with Polyak's momentum. 
Denote $\hat{\lambda}_{\max}:= \lambda_{\max}(H_0)$, $\hat{\lambda}_{\min} := \lambda_{\min}(H_0)$, and $\hat{\kappa}:= \hat{\lambda}_{\max} /\hat{\lambda}_{\min} $. Then, for gradient descent with Polyak's momentum, we set the step size $\eta =  1 / \left(  \hat{\lambda}_{\max} \right)$ and set the momentum parameter $\beta = (1 - \frac{1}{2} \frac{1}{\sqrt{ \hat{\kappa} }}) ^2$.  
For gradient descent, we set the same step size. The result is shown on Figure~\ref{fig:exp}.

We also report the percentiles of pattern changes over iterations. Specifically, we report the quantity
\[
\frac{ \sum_{i=1}^n \sum_{r=1}^m \mathbbm{1} \{ \text{sign}( x_i^\top w_t^{(r)} ) \neq  \text{sign}( x_i^\top w_0^{(r)} ) \} }{  m n  },
\]
as there are $m n$ patterns. For gradient descent with Polyak's momentum, the percentiles of pattern changes is approximately $0.76\%$; while for vanilla gradient descent, the percentiles of pattern changes is $0.55\%$. 

\begin{figure}[t]
  \centering
    \includegraphics[width=0.5\textwidth]{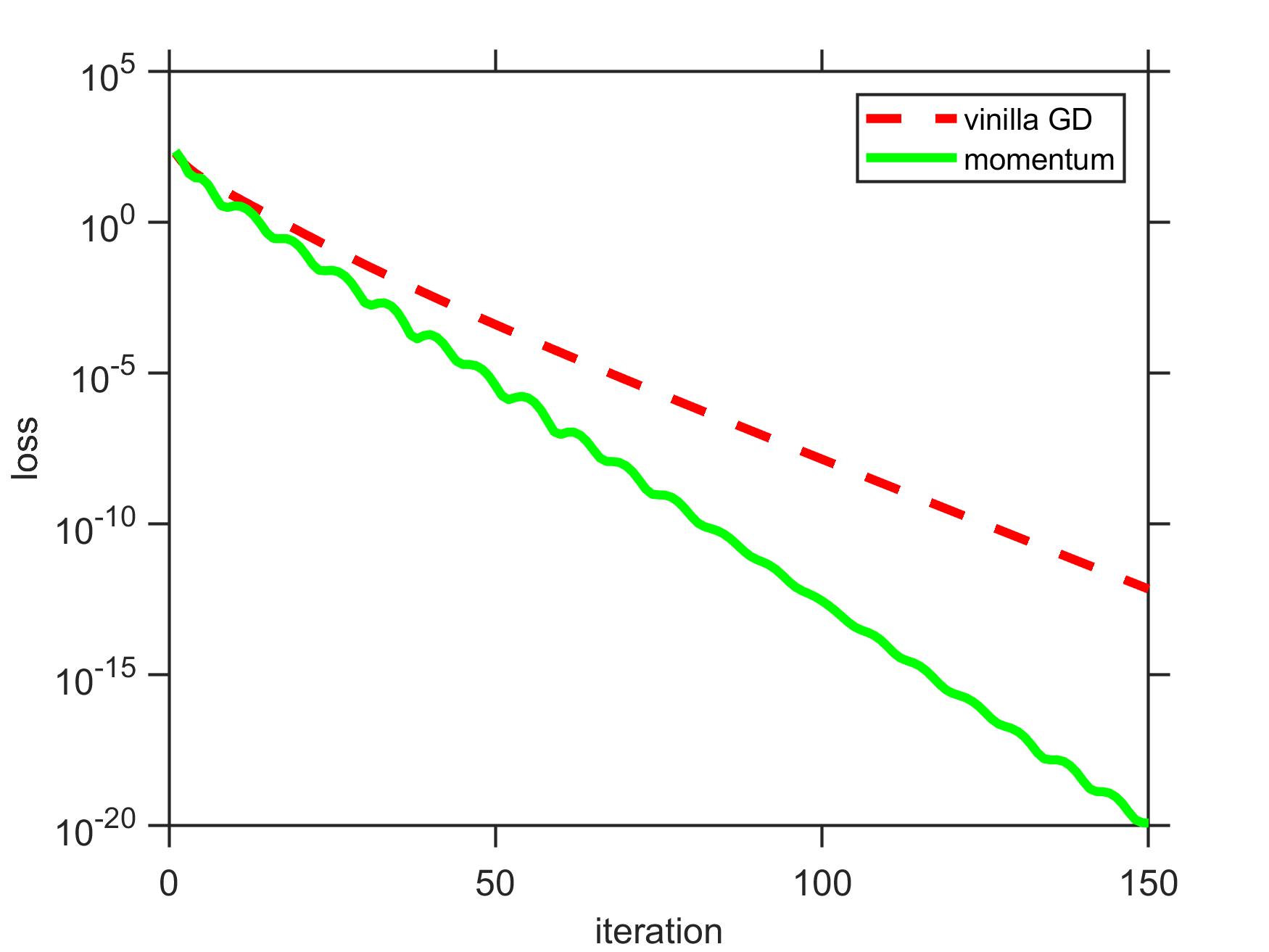}
    \caption{Training a $100$-layer deep linear network. Here ``momentum'' stands for gradient descent with Polyak's momentum. 
    } 
        \label{fig:exp2} 
\end{figure}

\subsection{Deep linear network}

We let the input and output dimension $d = d_y = 20$, the width of the intermediate layers $m=50$, the depth $L=100$. We sampled a $X \in \reals^{20 \times 5}$ from the normal distribution. We let  $W^* = I_{20} + 0.1\bar{W}$, where $\bar{W} \in \reals^{20 \times 20}$ is sampled from the normal distribution.
Then, we have $Y = W^* X$,
$\eta = \frac{d_y}{ L \sigma_{\max}^2(X)}$ and $\beta = (1-\frac{1}{2} \sqrt{\eta \lambda})^2$, where $\lambda= \frac{L \sigma_{\min}^2(X)}{ d_y}$.
Vanilla GD also uses the same step size. 
The network is initialized by the orthogonal initialization and both algorithms start from the same initialization.
The result is shown on Figure~\ref{fig:exp2}. 

\section{Conclusion}

We show some non-asymptotic acceleration results of the discrete-time Polyak's momentum in this work.
The results not only improve the previous results in convex optimization but also establish the first time that Polyak's momentum has provable acceleration for training certain neural networks.  
We analyze all the acceleration results from a modular framework.
We hope the framework can serve as a building block towards understanding Polyak's momentum in a more unified way.

One of the possible future work is considering applying Polyak's momentum to the Nesterov-Polyak cubic-regularized problem \cite{N06},
\begin{equation} \label{obj:reg}
\textstyle \min_w f(w) := \frac{1}{2} w^\top A w + b^\top w + \frac{\rho}{3} \| w \|^3,
\end{equation}
where the matrix $A \in \reals^{d \times d}$ is symmetric and possibly indefinite.
At the first glance, it looks a bit like the quadratic problems.
However, due to the presence of the cubic-regularized term,
the Hessian is changing and can change significantly during the optimization process. Empirically (Figure~\ref{fig:exp}) we observe that Polyak's momentum leads to acceleration. Yet, as far as we know, no theoretical result in the literature is able to explain this observation. Therefore, it is interesting to check if our modular analysis can be extended to this problem as well. 

\begin{figure}[H]
  \centering
    \includegraphics[width=0.5\textwidth]{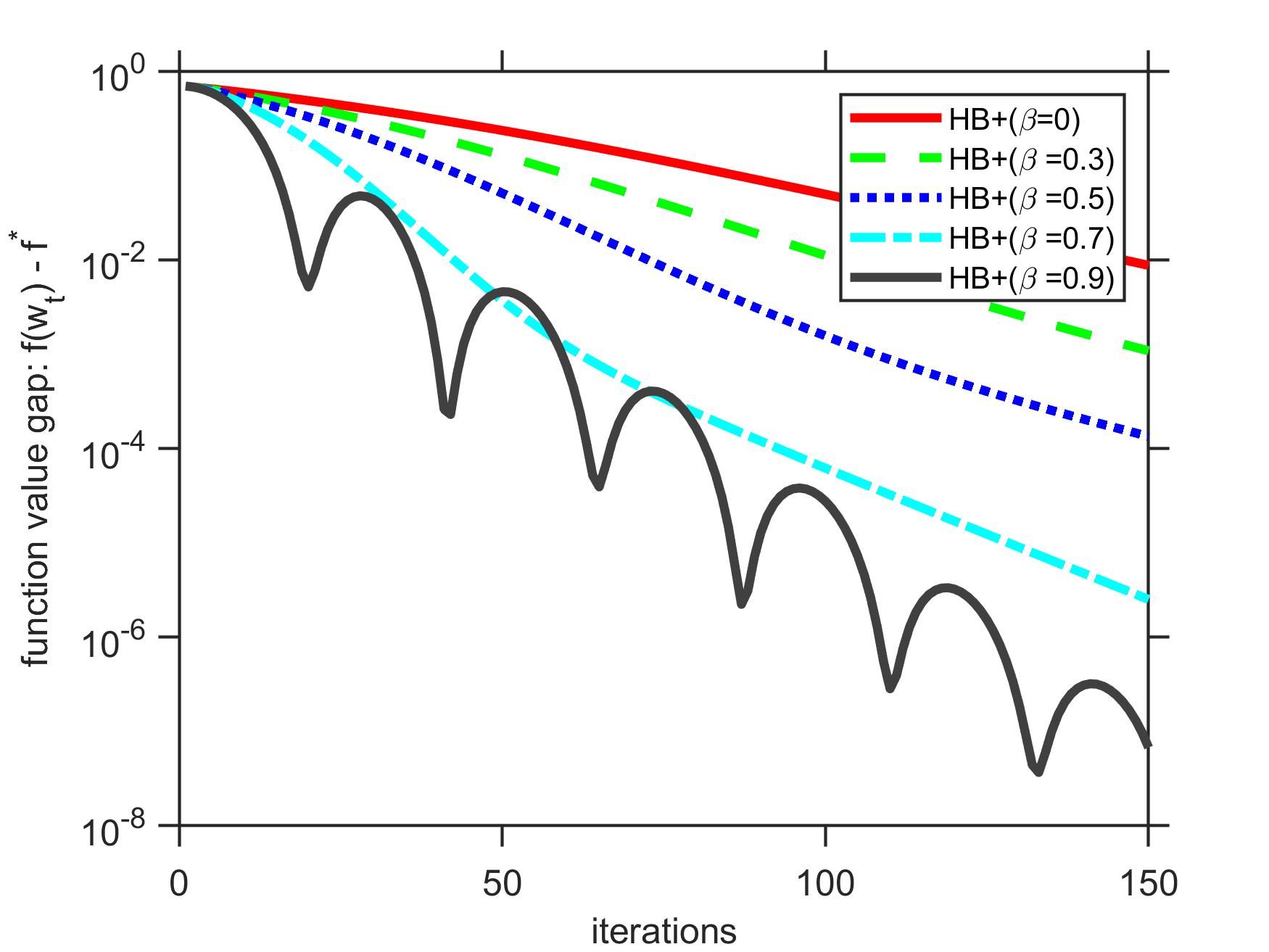}
    \caption{Optimality gap $f(w_t) - \min f(w)$ vs. iteration $t$. We see that a larger momentum leads to an acceleration.
The setup of
the experiment is as follows.
We first set step size $\eta = 0.01$,
dimension $d=4$, $\rho = \| w_* \| = \| A \|_2 = 1$, $\lambda_{\min}(A) =- 0.2$ and $\textbf{gap} = 5 \times 10^{-3}$. Then we set $A = \text{diag}( [\lambda_{\min}(A); \lambda_{\min}(A) + \textbf{gap}; a_{33}; a_{44} ]         )$, where the entries $a_{33}$ and $a_{44}$ are sampled uniformly random in
$[\lambda_{\min}(A) + \textbf{gap}; \| A \|_2]$.
We draw $\tilde{w} = (A + \rho \| w_* \| I_d)^{-\xi} \theta$, where $\theta \sim \mathcal{N}(0; I_d)$ and $\log_2 \xi$ is uniform on $[-1,1]$. We set $w_* = \frac{ \|w_*\| }{ \| \tilde{w}\| } \tilde{w}$ and $b= - (A + \rho \| w_*\| I_d ) w_*$. The procedure makes $w_*$ the global minimizer of problem instance $(A,b,\rho)$. 
Patterns shown on this figure exhibit for other random problem instances as well.}
        \label{fig:exp} 
\end{figure}


\chapter{Escaping Saddle Points Faster via Polyak's Momentum} \label{ch4}

\section{Introduction}
SGD with stochastic momentum has been a de facto algorithm in nonconvex optimization and deep learning.
It has been widely adopted for training machine learning models in various applications.
Modern techniques in computer vision (e.g.\cite{KSH12,Rnet16,CZMVL18,G17}), speech recognition (e.g. \cite{Baidu16}), natural language processing (e.g. \cite{attention17}),
and reinforcement learning (e.g. \cite{silver2017}) use SGD with stochastic momentum to train models.
The advantage of SGD with stochastic momentum has been widely observed (\cite{HHS17,KH1918}).
\citet{SMDH13} demonstrate that training deep neural nets by
SGD with stochastic momentum helps achieving in faster convergence compared with the standard SGD (i.e. without momentum).
The success of momentum makes it a necessary tool for designing new optimization algorithms in optimization and deep learning. For example,
all the popular variants of adaptive stochastic gradient methods like Adam (\cite{KB15}) or AMSGrad (\cite{RKK18}) include the use of momentum.

Despite the wide use of stochastic momentum (Algorithm~\ref{alg:sgd-momentum}) in practice,
\footnote{Heavy ball momentum is the default choice of momentum method in PyTorch and Tensorflow, instead of Nesterov's momentum. See the manual pages \url{https://pytorch.org/docs/stable/_modules/torch/optim/sgd.html}
and \url{https://www.tensorflow.org/api_docs/python/tf/keras/optimizers/SGD}.} justification for the clear empirical improvements has remained elusive, as has any mathematical guidelines for actually setting the momentum parameter---it has been observed that large values (e.g. $\beta=0.9$) work well in practice. It should be noted that Algorithm~\ref{alg:sgd-momentum} is the default momentum-method in popular software packages such as PyTorch and Tensorflow.
In this work we provide a theoretical analysis for SGD with momentum. We identify some mild conditions that guarantees SGD with stochastic momentum will provably escape saddle points faster than the standard SGD, which provides clear evidence for the benefit of using stochastic momentum.
For stochastic heavy ball momentum, a weighted average of stochastic gradients at the visited points is maintained. The new update is computed as the current update minus a step in the direction of the momentum.
Our analysis shows that these updates can amplify a component in an escape direction of the saddle points.

\begin{algorithm}[t]
\begin{algorithmic}[1]
\small
\caption{SGD with stochastic heavy ball momentum 
} \label{alg:sgd-momentum}{}
\STATE Required: Step size parameter $\eta$ and momentum parameter $\beta$.
\STATE Init: $w_{0} \in \reals^d $ and $m_{-1} = 0 \in \reals^d$.
\FOR{$t=0$ to $T$}
\STATE Given current iterate $w_t$, obtain stochastic gradient $g_t := \nabla f(w_t; \xi_t)$.
\STATE Update stochastic momentum $m_t := \beta m_{t-1} +  g_t$.
\STATE Update iterate $w_{t+1} := w_t - \eta m_t$.
\ENDFOR
\end{algorithmic}
\end{algorithm}

In this work, we focus on finding a second-order stationary point for smooth non-convex optimization by SGD with stochastic heavy ball momentum.
Specifically, we consider the stochastic nonconvex optimization problem,
$\min_{w \in \reals^d} f(w) := \E_{\xi \sim \mathcal{D} }[ f(w; \xi)]$,
where we overload the notation so that
$f(w; \xi)$ represents a stochastic function induced
by the randomness $\xi$
while $f(w)$ is the expectation of the stochastic functions.
An $(\epsilon,\epsilon)$-second-order stationary point $w$ satisfies
\begin{equation} \label{eq:condition}
 \| \nabla f(w) \| \leq \epsilon  \text{ and }  \nabla^2 f(w) \succeq -\epsilon I.
\end{equation}
Obtaining a second order guarantee has emerged as a desired goal in the nonconvex optimization community. Since finding a global minimum or even a local minimum in general nonconvex optimization can be NP hard (\cite{AG16,JN15,MK87,N00}), most of the papers in nonconvex optimization target at reaching an approximate second-order stationary point with additional assumptions like Lipschitzness in the gradients and the Hessian (e.g.
\cite{AL18,CD17,Curtis17,DKLH18,DJLJPS18,FLLZ18,FLZCOLT19,GHJY15,JGNKJ17,JNGKJ19,KL17,LJCJ17,LPPSJR19,KL16,MOJ16,N06,RZSPBSS18,SRRKKS19,TSJRJ18,XRY18}). 
We follow these related works for the goal and aim at showing the benefit of the use of the momentum in reaching an $(\epsilon,\epsilon)$-second-order stationary point.

We introduce a required condition, akin to a model assumption made in (\citet{DKLH18}), that ensures the dynamic procedure in Algorithm~\ref{alg:0} produces updates with suitable correlation with the negative curvature directions of the function $f$.

\begin{definition}
  Assume, at some time $t$, that the Hessian $H_t = \nabla^2 f(w_t)$ has some eigenvalue smaller than $-\epsilon$ and $\| \nabla f(w_t) \| \leq \epsilon$. Let $v_{t}$ be the eigenvector corresponding to the smallest eigenvalue of $\nabla^2 f(w_t)$. The stochastic momentum $m_t$ satisfies \textbf{Correlated Negative Curvature (CNC)} at $t$ with parameter $\gamma > 0$ if 
\begin{equation}
\EE{t}[ \langle m_{t}, v_{t} \rangle ^2 ] \geq \gamma.
\end{equation}
\end{definition}

As we will show, the recursive dynamics of SGD with heavy ball momentum helps in amplifying the escape signal $\gamma$, which allows it to escape saddle points faster.

\noindent
\textbf{Contribution: }
We show that, under CNC assumption and some minor constraints that upper-bound parameter $\beta$, if SGD with momentum has properties called \textit{Almost Positively Aligned with Gradient} (APAG), \textit{Almost Positively Correlated with Gradient} (APCG), and \textit{Gradient Alignment or Curvature Exploitation} (GrACE), defined in the later section, then it takes $T = O( (1-\beta) \log( 1 / (1-\beta) \epsilon )\epsilon^{-10} )$ iterations
to return an $(\epsilon, \epsilon)$ second order stationary point.
Alternatively, one can obtain an $(\epsilon, \sqrt{\epsilon})$ second order stationary point in $T = O( (1-\beta) \log( 1 / (1-\beta) \epsilon )\epsilon^{-5} )$ iterations.
Our theoretical result demonstrates that a larger momentum parameter $\beta$ can help in escaping saddle points faster. As saddle points are pervasive in the loss landscape of optimization and deep learning (\cite{dauphin14,CHMAL15}),
the result sheds light on explaining why SGD with momentum enables training faster in optimization and deep learning.

\noindent
\textbf{Notation:} In this chapter we use $\E_{t}[ \cdot ]$ to represent conditional expectation
$\E[ \cdot | w_1, w_2, \dots, w_{t} ]$,
which is about fixing the randomness upto but not including $t$ and notice that $w_t$ was determined at $t-1$. 

\begin{figure}[t]
  \centering
    \includegraphics[width=0.70\textwidth]{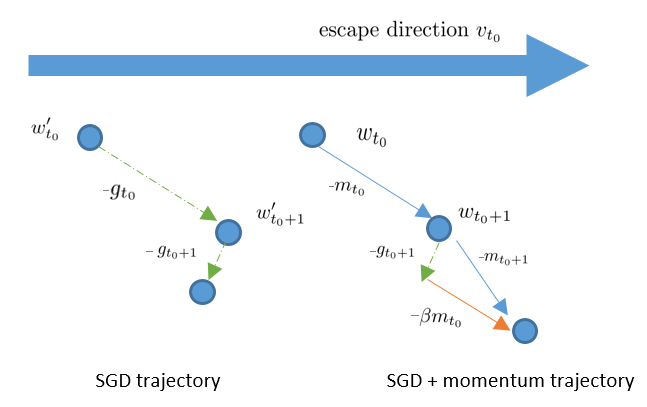}
    \caption{The trajectory of the standard SGD (left) and SGD with momentum (right).}
    \label{s0}
\end{figure}
\section{Background}
\subsection{A thought experiment.}

Let us provide some high-level intuition about the benefit of stochastic momentum with respect to escaping saddle points. In an iterative update scheme, at some time $\t$ the parameters $w_{\t}$ can enter a \emph{saddle point region}, that is a place where Hessian $\nabla^2 f(w_{\t})$ has a non-trivial negative eigenvalue, say $\lambda_{\min}(\nabla^2 f(w_{\t})) \leq -\epsilon$, and the gradient $\nabla f(w_{\t})$ is small in norm, say $\|\nabla f(w_{\t})\| \leq \epsilon$. The challenge here is that gradient updates may drift only very 
slowly away from the saddle point, and may not escape this region; see (\cite{DJLJPS18,LPPSJR19}) for additional details. On the other hand, if the iterates were to move in one 
particular direction, namely along $v_{\t}$ the direction of the smallest eigenvector of $\nabla^2 f(w_{\t})$, then a fast escape is guaranteed under certain constraints on the step size $\eta$; see e.g. (\cite{CDHS18}). While the negative eigenvector could be computed directly, this 2nd-order method is prohibitively expensive and hence we typically aim to rely on gradient methods. With this in mind, \citet{DKLH18}, who study non-momentum SGD, make an assumption akin to our CNC property described above that each stochastic gradient $g_{\t}$ is strongly non-orthogonal to $v_{\t}$ the direction of large negative curvature. This suffices to drive the updates out of the saddle point region.

In the present work we study stochastic momentum, and our CNC property requires that the update direction $m_{\t}$ is strongly non-orthogonal to $v_{\t}$; more precisely, $\EE{t_0}[ \langle m_{\t}, v_{\t} \rangle ^2 ] \geq \gamma > 0$.
We are able to take advantage of the analysis of (\citet{DKLH18}) to establish that updates begin to escape a saddle point region for similar reasons.
Further, this effect is \emph{amplified} in successive iterations through the momentum update when $\beta$ is close to 1. Assume that at some $w_{\t}$ we have $m_{\t}$ which possesses significant correlation with the negative curvature direction $v_{\t}$, then on successive rounds $m_{\t+1}$ is quite close to $\beta m_{\t}$, $m_{\t+2}$ is quite close to $\beta^2 m_{\t}$, and so forth; see Figure~\ref{s0} for an example. This provides an intuitive perspective on how momentum might help accelerate the escape process. Yet one might ask \textit{does this procedure provably contribute to the escape process} and, if so, \textit{what is the aggregate performance improvement of the momentum?} We answer the first question in the affirmative, and we answer the second question essentially by showing that momentum can help speed up saddle-point escape by a multiplicative factor of $1-\beta$. On the negative side, we also show that $\beta$ is constrained and may not be chosen arbitrarily close to 1.

\subsection{Momentum helps escape saddle points: an empirical view}

Let us now establish, empirically, the clear benefit of stochastic momentum on the problem of saddle-point escape. We construct two stochastic optimization tasks, and each exhibits at least one significant saddle point. The two objectives are as follows.
\begin{eqnarray} 
\label{obj:simu}
\min_w f(w) & := &   \frac{1}{n} \sum_{i=1}^n \big(  \frac{1}{2} w^\top H w + b_i^\top w + \| w \|^{10}_{10}  \big),\\
\label{obj:phase}
 \min_{w} f(w) & := &    \frac{1}{n} \sum_{i=1}^n \big( (a_i^\top w)^2 - y    \big)^2.
\end{eqnarray}
Problem \eqref{obj:simu} of these was considered by (\citet{SRRKKS19}, \citet{RZSPBSS18}) and represents a very straightforward non-convex optimization challenge, with an embedded saddle given by the matrix $H:= \text{diag}([1, -0.1])$, and stochastic gaussian perturbations given by $b_i \sim \mathcal{N}( 0, \text{diag}([0.1, 0.001]))$; the small variance in the second component provides lower noise in the escape direction. 
Here we have set $n=10$.
Observe that 
the origin is in the neighborhood of saddle points and has objective value zero.
SGD and SGD with momentum are initialized at the origin in the experiment so that they have to escape saddle points before the convergence.
The second objective \eqref{obj:phase} appears in the \emph{phase retrieval} problem,  that has real applications in physical sciences (\cite{CESV12,SECCMS15}).
In phase retrieval\footnote{It is known that phase retrieval is nonconvex and has the so-called strict saddle property: (1) every local minimizer $\{w^*, -w^*\}$ is global up to phase, (2) each saddle exhibits negative curvature
(see e.g. (\cite{SQW15,SQW16,CCFMY18}))}, one wants to find an unknown $w^* \in \reals^d$ with access to but a few samples $y_i = (a_i^\top w^*)^2$; the design vector $a_i$ is known a priori. Here we have sampled $w^* \sim \mathcal{N}( 0, \mathcal{I}_d / d)$ and $a_i \sim \mathcal{N}( 0, \mathcal{I}_d )$ with $d=10$ and $n=200$.
\begin{figure}[t]
\centering
   \begin{subfigure}[t]{0.4\textwidth}
        \label{subfig-1:func1} \includegraphics[width=1.0\textwidth]{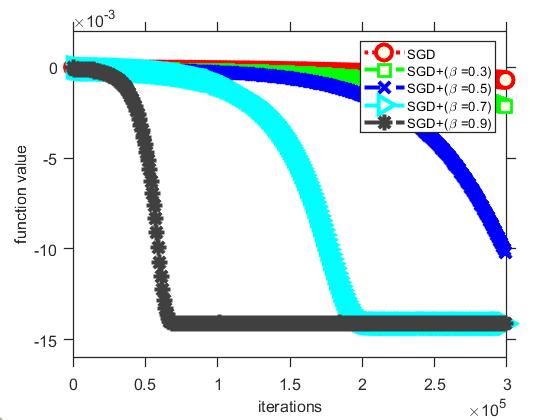}
        \caption{Solving objective (\ref{obj:simu})}
   \end{subfigure} %
    \begin{subfigure}[t]{0.4\textwidth}
        \label{subfig-2:func2} \includegraphics[width=1.0\textwidth]{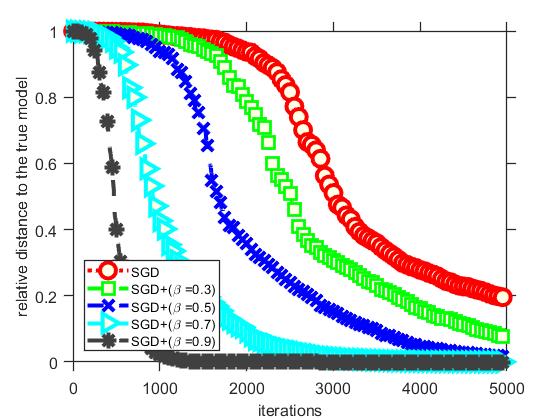}
        \caption{Solving objective (\ref{obj:phase}). (phase retrieval)}
    \end{subfigure} %
     \caption{Performance of SGD with different values of $\beta= \{ 0, 0.3, 0.5, 0.7, 0.9\}$; $\beta =0$ corresponds to standard SGD. 
      Figure of the left: We plot convergence in function value $f(\cdot)$ given in \eqref{obj:simu}. Initialization is always set as $w_0 = \mathbf{0}$. All the algorithms use the same step size $\eta=5 \times 10^{-5}$.
      Figure of the right: We plot convergence in \emph{relative distance} to the true model $w^*$, defined as $\min ( \| w_t - w^* \|, \| w_t + w^* \| ) / \| w^* \|$, which more appropriately captures progress as the global sign of the objective \eqref{obj:phase} is unrecoverable.
     All the algorithms are initialized at the same point $w_0 \sim \mathcal{N}( 0, \mathcal{I}_d / (10000 d))$ and use the same step size $\eta = 5 \times 10^{-4}$.}
     \label{simu}
\end{figure}

The empirical findings, displayed in Figure~\ref{simu}, are quite stark: for both objectives, convergence is significantly accelerated by larger choices of $\beta$. In the first objective (Subfigure (a) of Figure~\ref{simu}), we see each optimization trajectory entering a saddle point region, apparent from the ``flat'' progress, yet we observe that large-momentum trajectories escape the saddle much more quickly than those with smaller momentum. A similar affect appears in Subfigure (b) of Figure~\ref{simu}. To the best of our knowledge, this is the first reported empirical finding that establishes the dramatic speed up of stochastic momentum for finding an optimal solution in phase retrieval.

\subsection{Related works.}

\textbf{Heavy ball method:} The heavy ball method was originally proposed by \citet{P64}.
It has been observed that this algorithm, even in the deterministic setting, provides no convergence speedup over standard gradient descent, except in some highly structure cases such as convex quadratic objectives where an ``accelerated'' rate is possible (\cite{LRP16,goh2017why,GFJ15,SYLHGJ19,LR17,LR18,GPS16,YLL18,NKJK18,CGZ19}). 
In recent years, some works make some efforts in analyzing heavy ball method for other classes of optimization problems besides the quadratic functions. For example, \citet{GFJ15} prove an $O(1/T)$ ergodic convergence rate when the problem is smooth convex, while \citet{SYLHGJ19} provide a non-ergodic convergence rate for certain classes of convex problems.
\citet{OCBP14} combine the technique of forward-backward splitting with heavy ball method for a specific class of nonconvex optimization problem. For stochastic heavy ball method,
\citet{LR17} analyze a class of linear regression problems and shows a linear convergence rate of stochastic momentum, in which
the linear regression problems actually belongs to the case of strongly convex quadratic functions. Other works includes (\cite{GPS16}), which shows almost sure convergence to the critical points by stochastic heavy ball for general non-convex coercive functions. Yet, the result does not show any advantage of stochastic heavy ball over other optimization algorithms like SGD. \citet{CGZ19} show an accelerated linear convergence to a stationary distribution under Wasserstein distance for strongly convex quadratic functions
by SGD with stochastic heavy ball momentum. 
\citet{YLL18} provide a unified analysis of stochastic heavy ball momentum and Nesterov's momentum for smooth non-convex objective functions.
They show that the expected gradient norm converges at rate $O(1/\sqrt{t})$. Yet, the rate is not better than that of the standard SGD. We are also aware of the works of \cite{GL16}, \cite{GL13}, which propose some variants of stochastic accelerated algorithms with first order stationary point guarantees. Yet, the framework in \cite{GL16,GL13} does not capture the stochastic heavy ball momentum used in practice. There is also a negative result about the heavy ball momentum.
\citet{NKJK18} show that for a specific strongly convex and strongly smooth problem, SGD with heavy ball momentum fails to achieving the best convergence rate while some algorithms can.

\noindent
\textbf{Reaching a second order stationary point:}
As we mentioned earlier, there are many works aim at reaching a second order stationary point.
We classify them into two categories: specialized algorithms and simple GD/SGD variants. Specialized algorithms are those designed to exploit the negative curvature explicitly and escape saddle points faster than the ones without the explicit exploitation (e.g. \cite{CDHS18,AABHM17,AL18,XRY18}). Simple GD/SGD variants are those with minimal tweaks of standard GD/SGD or their variants (e.g. \cite{GHJY15,KL16,FLZCOLT19,JGNKJ17,CNJ18,JNGKJ19,DKLH18,SRRKKS19}).
Our work belongs to this category.
In this category, perhaps the pioneer works are \citet{GHJY15} and \citet{JGNKJ17}.
\citet{JGNKJ17} show that explicitly adding isotropic noise in each iteration guarantees that GD escapes saddle points and finds a second order stationary point with high probability.
Following \citet{JGNKJ17}, \citet{DKLH18} assume that stochastic gradient inherently has a component to escape. Specifically, they make assumption of the Correlated Negative Curvature (CNC) for stochastic gradient $g_t$ so that
$\EE{t}[ \langle g_{t}, v_{t} \rangle ^2 ] \geq \gamma > 0$.
The assumption allows the algorithm to avoid the procedure of perturbing the updates by adding isotropic noise.
Our work is motivated by \citet{DKLH18} but assumes CNC for the stochastic momentum $m_t$ instead.
Very recently, \citet{JNGKJ19} consider perturbing the update of SGD and provide a second order guarantee. \citet{SRRKKS19} consider a variant of RMSProp \cite{TH12}, in which the gradient $g_t$ is multiplied by a preconditioning matrix $G_t$
and the update is $w_{t+1}= w_t - G_t^{-1/2} g_t$.
The work shows that the algorithm can help in escaping saddle points faster compared to standard SGD under certain conditions. \citet{FLZCOLT19} propose average-SGD, in which a suffix averaging scheme is conducted for the updates. They also assume an inherent property of stochastic gradients that allows SGD to escape saddle points.

We summarize the iteration complexity results of the related works for simple SGD variants on Table~\ref{table:complexity}. \footnote{We follow the work \citet{DKLH18} for reaching an $(\epsilon,\epsilon)$-stationary point, while some works are for an $(\epsilon,\sqrt{\epsilon})$-stationary point. We translate them into the complexity of getting an $(\epsilon,\epsilon)$-stationary point. }
The readers can see that the iteration complexity of \citet{FLZCOLT19} and \citet{JNGKJ19} are better than \citet{DKLH18,SRRKKS19} and our result.  So, we want to explain the results and clarify the differences. First,
we focus on explaining why the popular algorithm, SGD with heavy ball momentum, works well in practice, which is without the suffix averaging scheme used in \citet{FLZCOLT19}
and is without the explicit perturbation used in \citet{JNGKJ19}. 
Specifically, we focus on studying the effect of stochastic heavy ball momentum and showing the advantage of using it. Furthermore,
our analysis framework is built on the work of \citet{DKLH18}.
We believe that, based on the insight in our work, one can also show the advantage of stochastic momentum by modifying the assumptions and algorithms
in (\cite{FLZCOLT19}) or (\cite{JNGKJ19}) and consequently get a better dependency on $\epsilon$.

\begin{table}[t]
\small
\centering
\begin{tabular}{r | l}
Algorithm    &  Complexity \\ \hline \hline
Perturbed SGD (\cite{GHJY15})   &  $\mathcal{O}(\epsilon^{-16})$ \\ \hline
Average-SGD (\cite{FLZCOLT19})     &  $\mathcal{O}(\epsilon^{-7})$ \\ \hline
Perturbed SGD (\cite{JNGKJ19})  &  $\mathcal{O}(\epsilon^{-8})$ \\ \hline
CNC-SGD      (\cite{DKLH18})       &  $\mathcal{O}(\epsilon^{-10})$ \\ \hline
Adaptive SGD (\cite{SRRKKS19}) &  $\mathcal{O}(\epsilon^{-10})$ \\ \hline \hline
SGD+momentum (this work) &  $\mathcal{O}( (1-\beta) \log( \frac{1}{(1-\beta)\epsilon}) \epsilon^{-10})$
\end{tabular}
\caption{Iteration complexity to find an $(\epsilon, \epsilon)$ second-order stationary point
.} \label{table:complexity}
\end{table}

\section{Main Results}

We assume that the gradient $\nabla f$ is $L$-Lipschitz; that is, $f$ is $L$-smooth.
Further, we assume that the Hessian $\nabla^2 f$ is $\rho$-Lipschitz.
These two properties ensure that $\| \nabla f(w) - \nabla f(w')  \| \leq L \| w - w' \|$
and that $\| \nabla^2 f(w) - \nabla^2 f(w') \| \leq \rho \| w - w' \|$, $\forall w, w'$.
The $L$-Lipschitz gradient assumption implies that
$| f(w') - f(w) - \langle \nabla f(w), w' - w \rangle | \leq \frac{L}{2} \|w - w' \|^2, \forall w, w'$, while the $\rho$-Lipschitz Hessian assumption implies that
$| f(w') - f(w) - \langle \nabla f(w), w' - w \rangle - (w' - w)^\top \nabla^2 f(w) (w'-w )  | \leq \frac{\rho}{6} \|w - w' \|^3$, $\forall w, w'$.
Furthermore, we assume that the stochastic gradient has bounded noise
 $ \| \nabla f(w) -  \nabla f(w;\xi) \|^2 \leq \sigma^2$
 and that the norm of stochastic momentum is bounded so that $\|m_t\| \leq c_m$.
We denote $\Pi_i M_i$ as the matrix product of matrices $\{ M_i \}$
and we use $\sigma_{max}(M) = \| M \|_2 := \sup_{ x \ne 0} \frac{ \langle x, M x\rangle}{\langle x, x \rangle} $ to denote the spectral norm of the matrix $M$.

\subsection{Required properties with empirical validation}

Our analysis of stochastic momentum relies on three properties of the stochastic momentum dynamic.
These properties are somewhat unusual, but we argue they should hold in natural settings, and later we aim to demonstrate that they hold empirically in a couple of standard problems of interest.

\begin{definition}
  We say that
  SGD with stochastic momentum 
  satisfies \textbf{A}lmost \textbf{P}ositively \textbf{A}ligned with \textbf{G}radient (APAG) \footnote{Note that our analysis still go through if one replaces $\frac{1}{2}$ on r.h.s. of (\ref{ass3}) with any larger number $c < 1$; the resulted iteration complexity would be only a constant multiple worse.}
  if we have
  \begin{equation} \label{ass3}
   \EE{t}[ \langle \nabla f(w_t), m_t - g_t \rangle ] \geq - \frac{1}{2} \| \nabla f(w_t) \|^2.
  \end{equation}
  We say that SGD with stochastic momentum 
  satisfies \textbf{A}lmost \textbf{P}ositively \textbf{C}orrelated with \textbf{G}radient (APCG) with parameter $\tau$ if $\exists c'>0$ such that,
\begin{equation} \label{ass4}
\EE{t}[ \langle  \nabla f(w_t), M_t m_{t} \rangle ]
 \geq - c' \eta \sigma_{max}(M_t) \| \nabla f(w_t) \|^2,
\end{equation}
where the PSD matrix $M_t$ is defined as 
\[ \textstyle
  M_t = ( \Pi_{s=1}^{\tau-1} G_{s,t} )( \Pi_{s=k}^{\tau-1} G_{s,t}) \quad \text{ with } \quad G_{s,t}: = I -\eta  \sum_{j=1}^s \beta^{s-j} \nabla^2 f(w_t)
  = I - \frac{\eta (1 -\beta^s)}{ 1 - \beta } \nabla^2 f(w_t)
\]
for any integer $1\leq k \leq \tau-1$, and $\eta$ is any step size chosen that guarantees each $G_{s,t}$ is PSD.
\end{definition}

\begin{definition}
We say that the SGD with momentum exhibits
\textbf{Gr}adient \textbf{A}lignment or \textbf{C}urvature \textbf{E}xploitation (GrACE)
if $\exists c_h \geq 0$ such that
\begin{equation} \label{GrACE}
\textstyle \EE{t}[ \eta \langle \nabla f(w_t), g_t - m_t \rangle
+ \frac{\eta^2}{2} m_t^\top \nabla^2 f(w_t) m_t ] \leq \eta^2 c_h.
\end{equation}
\end{definition}

APAG requires that the momentum term $m_{t}$ must, in expectation, not be significantly misaligned with the gradient $\nabla f(w_t)$. This is a very natural condition when one sees that the momentum term is acting as a biased estimate of the gradient of the deterministic $f$. APAG demands that the bias can not be too large relative to the size of $\nabla f(w_t)$. Indeed this property is only needed in our analysis when the gradient is large
(i.e. $\| \nabla f(w_t) \| \geq \epsilon$) as it guarantees that the algorithm makes progress; our analysis does not require APAG holds when gradient is small.

APCG is a related property, but requires that the current momentum term $m_{t}$ is almost positively correlated with the the gradient $\nabla f(w_t)$, but \emph{measured in the Mahalanobis norm induced by $M_t$.}
It may appear to be an unusual object, but one can view the PSD matrix $M_t$ as measuring something about the local curvature of the function with respect to the trajectory of the SGD with momentum dynamic. 
We will show that this property holds empirically on two natural problems for a reasonable constant $c'$.
APCG is only needed in our analysis when the update is in a saddle region with significant negative curvature, $\|\nabla f(w)\| \leq \epsilon$ and
$\lambda_{\min} ( \nabla^2 f(w))  \leq -\epsilon$.
Our analysis does not require APCG holds when the gradient is large or the update is at an $(\epsilon,\epsilon)$-second order stationary point.

For GrACE, the first term on l.h.s of (\ref{GrACE}) measures the alignment between stochastic momentum $m_t$ and the gradient $\nabla f(w_t)$, while the second term on l.h.s measures the curvature exploitation.
The first term is small (or even negative) when the stochastic momentum $m_t$ is aligned with the gradient $\nabla f(w_t)$, while the second term is small (or even negative)
when the stochastic momentum $m_t$ can exploit a negative curvature (i.e. the subspace of eigenvectors that corresponds to the negative eigenvalues of the Hessian $ \nabla^2 f(w_t)$ if exists). Overall, a small sum of the two terms (and, consequently, a small $c_h$)
allows one to bound the function value of the next iterate (see Lemma~\ref{lem:ch}).

On Figure~\ref{fig:properties}, we report some quantities related to APAG and APCG as well as the gradient norm when solving the previously discussed problems (\ref{obj:simu}) and (\ref{obj:phase}) using SGD with momentum. We also report a quantity regarding GrACE on Figure~\ref{fig:gace}.

\begin{figure}[ht!]
\tiny
\centering
   \begin{subfigure}[t]{0.4\textwidth}
        \label{suba}        \includegraphics[width=1.0\textwidth]{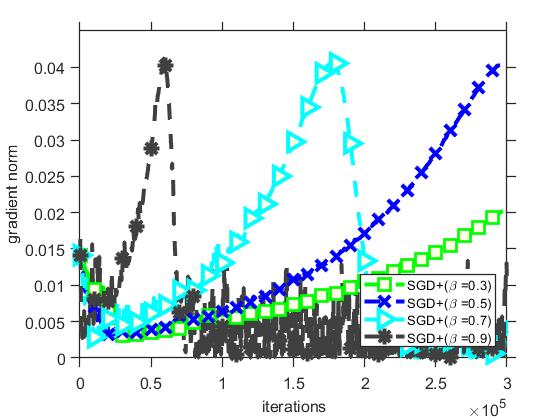}
        \caption{Gradient norm $\| \nabla  f(w_t)\|$.\label{suba}}
   \end{subfigure} %
   \begin{subfigure}[t]{0.4\textwidth}
        \label{subb}        \includegraphics[width=1.0\textwidth]{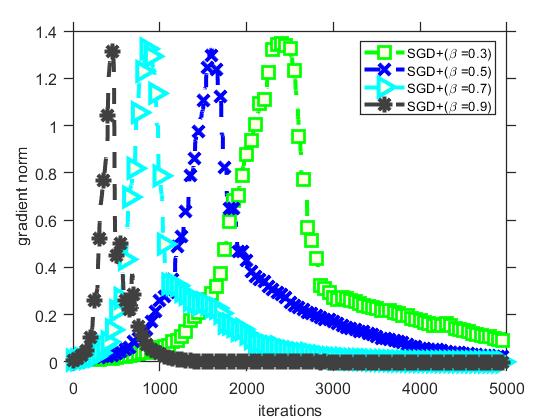}
        \caption{Gradient norm $\| \nabla  f(w_t)\|$.\label{subb}}
   \end{subfigure} \\
   \begin{subfigure}[t]{0.4\textwidth}
        \label{subc}        \includegraphics[width=1.0\textwidth]{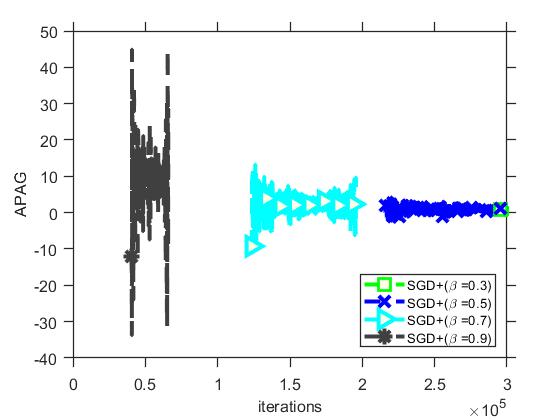}
        \caption{About APAG. \label{subc}}
   \end{subfigure} 
   \begin{subfigure}[t]{0.4\textwidth}
        \label{subd}        \includegraphics[width=1.0\textwidth]{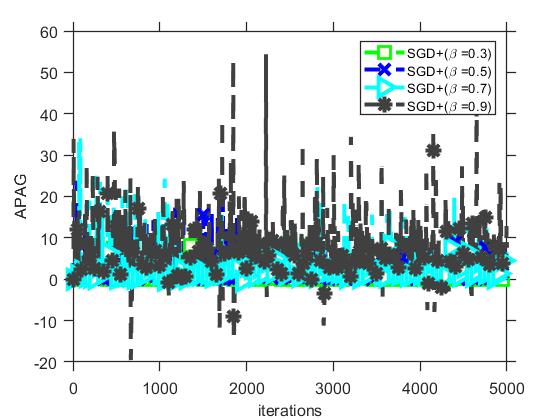}
        \caption{About APAG. \label{subd}}
   \end{subfigure} \\
   \begin{subfigure}[t]{0.4\textwidth}
        \label{sube}        \includegraphics[width=1.0\textwidth]{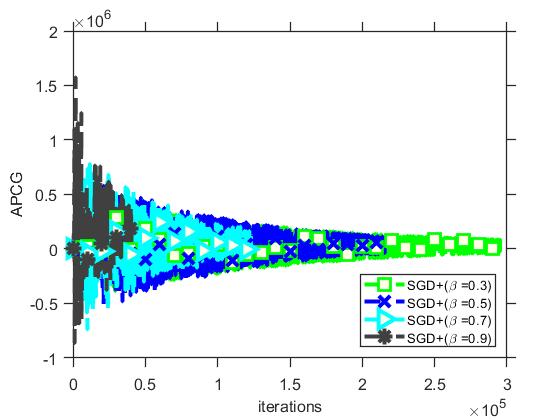}
        \caption{About APCG. \label{sube}}
   \end{subfigure} %
   \begin{subfigure}[t]{0.4\textwidth}
        \label{subf}        \includegraphics[width=1.0\textwidth]{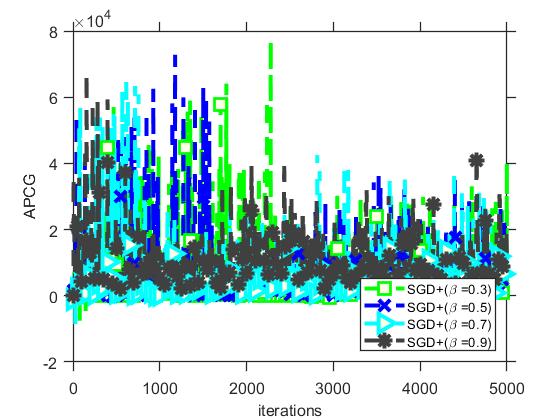}
        \caption{About APCG. \label{subf}}
   \end{subfigure} 
     \caption{
Plots of the related properties. Sub-figures on the left are regarding solving (\ref{obj:simu}) and sub-figures on the right are regarding solving (\ref{obj:phase}) (phase retrieval).
Note that the function value/relative distance to $w^*$ are plotted on Figure~\ref{simu}. Above, 
sub-figures (a) and (b): We plot the gradient norms versus iterations.
Sub-figures (c) and (d): We plot the values of $\langle \nabla f(w_t), m_t - g_t \rangle / \| \nabla f(w_t) \|^2 $ versus iterations. 
 For (c),
 we only report them when the gradient is large ($\| \nabla f(w_t)\| \geq 0.02$). It shows that the value is large than $-0.5$ except the transition. 
 For (d), we observe that the value is almost always nonnegative.
Sub-figures (e) and (f):
 We plot the value of $ \langle  \nabla f(w_t), M_t m_{t} \rangle / (\eta \sigma_{max}(M_t) \| \nabla f(w_t) \|^2)$.
 For (e), we let $M_t = ( \Pi_{s=1}^{3 \times 10^5} G_{s,t} )( \Pi_{s=1}^{3 \times 10^5} G_{s,t})$ and we only report the values when the update is in the region of saddle points.
 For (f), we let $M_t= ( \Pi_{s=1}^{500} G_{s,t} )( \Pi_{s=1}^{500} G_{s,t})$
 and we observe that the value is almost always nonnegative.
 The figures implies that SGD with momentum has APAG and APCG properties in the experiments.
 Furthermore, an interesting observation is that, for the phase retrieval problem, the expected values might actually be nonnegative.
 }
     \label{fig:properties}
\end{figure}
\clearpage

\begin{figure}[t]
\centering
   \begin{subfigure}[t]{0.4\textwidth}
          \includegraphics[width=1.0\textwidth]{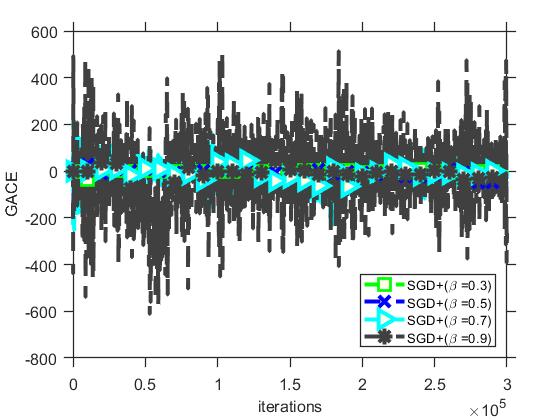}
        \caption{About GrACE for problem (\ref{obj:simu}).\label{suba}}
   \end{subfigure} %
   \begin{subfigure}[t]{0.4\textwidth}
          \includegraphics[width=1.0\textwidth]{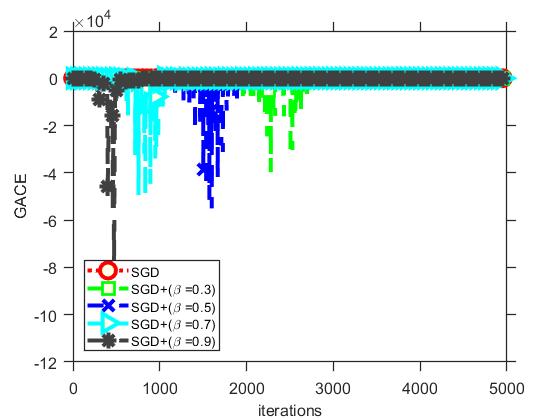}
        \caption{About GrACE for problem (\ref{obj:phase}). (phase retrieval)\label{subb}}
   \end{subfigure} %
     \caption{Plot regarding the GrACE property. We plot the values of $\big( \eta \langle \nabla f(w_t), g_t - m_t \rangle $ $ + \frac{1}{2} \eta^2 m_t^\top H_t m_t  \big) / \eta^2$ versus iterations. An interesting observation is that the value is well upper-bounded by zero for the phase retrieval problem.
The results imply that the constant $c_h$ is indeed small. 
     }
     \label{fig:gace}
\end{figure}

\subsection{Convergence results}

The high level idea of our analysis follows as a similar template to (\cite{JGNKJ17,DKLH18,SRRKKS19}). Our proof is structured into three cases: either (a)  $\| \nabla f(w) \| \geq \epsilon$, or (b) $\| \nabla f(w) \| \leq \epsilon$ and $\lambda_{\min}( \nabla^2 f(w) ) \leq - \epsilon$, or otherwise (c) $\| \nabla f(w) \| \leq \epsilon$ and $\lambda_{\min}( \nabla^2 f(w) ) \geq - \epsilon$, meaning we have arrived in a second-order stationary region.
The precise algorithm we analyze is Algorithm~\ref{alg:0}, which identical to Algorithm~\ref{alg:sgd-momentum}
except that we boost the step size to a larger value $r$ on occasion.
We will show that the algorithm makes progress in cases (a) and (b). In case (c), when the goal has already been met, further execution of the algorithm only weakly hurts progress.
Ultimately, we prove that a second order stationary point is arrived at with high probability.
While our proof borrows tools from (\cite{DKLH18,SRRKKS19}), much of the momentum analysis is entirely novel to our knowledge.

\begin{algorithm}[t]
\begin{algorithmic}[1]
\small
\caption{SGD with stochastic heavy ball momentum} \label{alg:0}
\STATE Required: Step size parameters $r$ and $\eta$, momentum parameter $\beta$, and period parameter $\TT$.
\STATE Init: $w_{0} \in \reals^d $ and $m_{-1} = 0 \in \reals^d$.
\FOR{$t=0$ to $T$}
\STATE Get stochastic gradient $g_t$ at $w_t$, and set stochastic momentum $m_t := \beta m_{t-1} +  g_t$.
\STATE Set learning rate: $\hat{\eta} := \eta$ \textbf{unless} $(t \text{ mod } \TT) = 0$ in which case $\hat \eta := r$
\STATE $w_{t+1} = w_t - \hat \eta m_t$.
\ENDFOR
\end{algorithmic}
\end{algorithm}

\begin{theorem} \label{thm:main_escape}
Assume that the stochastic momentum satisfies CNC.
Set \footnote{See Table~\ref{table:parameters} in Section~\ref{app:lem:escape} for the precise expressions of the parameters. Here, we hide the parameters' dependencies on $\gamma$, $L$, $c_m$, $c'$, $\sigma^2$, $\rho$, $c_h$, and $\delta$. 
W.l.o.g, we also assume that $c_m$, $L$, $\sigma^2$, $c'$, $c_h$, and $\rho$ are not less than one 
and $\epsilon \leq 1$. 
} 
$r=O(\epsilon^2)$, $\eta = O(\epsilon^5)$, and $\TT =
\frac{c (1-\beta)}{\eta \epsilon} \log( \frac{L c_m \sigma^2 \rho c' c_h }{ (1-\beta)  \delta \gamma \epsilon }  ) = 
 O( (1-\beta) \log( \frac{L c_m \sigma^2 \rho c' c_h }{ (1-\beta)  \delta \gamma \epsilon }  ) \epsilon^{-6}  )$ for some constant $c > 0$.
If SGD with momentum (Algorithm~\ref{alg:0}) has
APAG property when gradient is large ($\| \nabla f(w) \| \geq \epsilon$), APCG$_{\TT}$
property when it enters a region of saddle points that exhibits a negative curvature ($\| \nabla f(w) \| \leq \epsilon$ and $\lambda_{\min}( \nabla^2 f(w) ) \leq - \epsilon$),
and GrACE property throughout the iterations,
then it
reaches an $(\epsilon, \epsilon)$ second order stationary point in $T = 2 \TT ( f(w_0) - \min_w f(w) ) / ( \delta \FT ) =  O( (1-\beta) \log( \frac{L c_m \sigma^2 \rho c' c_h }{ (1-\beta)  \delta \gamma \epsilon }  )\epsilon^{-10} )$ iterations with high probability $1-\delta$, where $\FT = O(\epsilon^4)$.
\end{theorem}

The theorem implies the advantage of using stochastic momentum for SGD. Higher $\beta$ leads to reaching a second order stationary point faster. As we will show in the following, this is due to that higher $\beta$ enables escaping the saddle points faster.
In Subsection 3.2.1, we provide some key details of the proof of Theorem~\ref{thm:main_escape}. The interested reader can read a high-level sketch of the proof, as well as the detailed version, in Section~\ref{app:thm}. 

\noindent
\noindent 
\textbf{Remark 1: (constraints on $\beta$)}
.

\begin{table}[h] 
\small
\caption{Constraints and choices of the parameters.} \label{table:parameters}
\centering
\begin{tabular}{r | c | l | c }
Parameter    &  Value  &  Constraint origin  & constant \\ \hline \hline
$r$      & $\delta \gamma \epsilon^2 c_r$ & (\ref{e:0}), (\ref{e:1}), (\ref{e:2})   &
 \begin{tabular}{c} $c_r \leq \frac{c_0}{c_m^3 \rho L \sigma^2 c_h }$, $c_0 = \frac{1}{1152}$ \\  
$\frac{c_0}{c_m^3 \rho L \sigma^2 c' (1-\beta)^2 c_h} \leq c_r$ \\
$\frac{c_0}{c_m^3 \rho L \sigma^4 (1-\beta)^3 c_h} \leq c_r$ 
   \end{tabular} \\ \hline
$r$      &  '' & $r \leq \sqrt{ \frac{\delta \FT}{8 c_h } } \text{ from } (\ref{need:0})$ &  '' \\ \hline \hline
$\eta$      & $\delta^2 \gamma^2 \epsilon^5 c_{\eta}$ & (\ref{e:0}) & $c_{\eta} \leq \frac{c_1}{c_m^5 \rho L^2 \sigma^2 c' c_h}$, $c_1 = \frac{c_0}{24}$ \\ \hline
$\eta$      & '' & \shortstack{ $\eta \leq r / \sqrt{ \TT}$ \\ from (\ref{q3d}),(\ref{qq:3}),(\ref{need:-1}),(\ref{need:0}) } & '' \\ \hline 
$\eta$      & '' & \shortstack{ $\eta \leq \min \{  \frac{(1-\beta)}{L}, \frac{(1-\beta)}{\epsilon} \}$ \\ from (\ref{eta:another}), (\ref{eta:another2})}   &  \\ \hline \hline
$\FT$      & $\delta \gamma^2 \epsilon^4 c_{F}$ & (\ref{e:1})  &
\begin{tabular}{c} $c_{F} \leq \frac{c_2}{c_m^4 \rho^2 L \sigma^4 c_h}$, $c_2 = \frac{c_0}{576}$ 
\\ $c_F \geq \frac{8 c_0^2}{ c_m^6 \rho^2 L^2 \sigma^4 c_h}$
\end{tabular}\\ \hline
$\FT$      & '' & $\FT \leq \frac{\epsilon^2 r}{4}$  from (\ref{f:2})  & '' \\ \hline \hline
$\TT$  &   & \shortstack{ $\TT \geq \frac{c (1-\beta)}{\eta \epsilon} \log ( \frac{L c_m \sigma^2 \rho c' c_h}{(1-\beta) \delta \gamma \epsilon }  )$ \\ from (\ref{TT}) } & \\ \hline
\end{tabular} 
\end{table}

W.l.o.g, we assume that $c_m$, $L$, $\sigma^2$, $c'$, $c_h$, and $\rho$ are not less than one and that $\epsilon \leq 1$.
\footnote{We assume that $\beta$ is chosen so that $1-\beta$ is not too small and consequently the choice of $\eta$ satisfies $\eta \leq \min \{  \frac{(1-\beta)}{L}, \frac{(1-\beta)}{\epsilon} \}$.}
We require that parameter $\beta$ is not too close to 1 so that the following holds,
\begin{itemize}
\item 1) $L(1-\beta)^3 > 1$.
\item 2) $\sigma^2 (1-\beta)^3 > 1$.
\item 3) $c' (1-\beta)^2 > 1$.
\item 4) $\eta  \leq \frac{1-\beta}{L}$.
\item 5) $\eta \leq \frac{1-\beta}{\epsilon}$.
\item 6) $\TT \geq \frac{c (1-\beta)}{\eta \epsilon} \log ( \frac{L c_m \sigma^2 \rho c' c_h}{(1-\beta) \delta \gamma \epsilon }  ) \geq 1 + \frac{2 \beta}{1-\beta}$.
\end{itemize}  
The constraints upper-bound the value of $\beta$. That is, $\beta$ cannot be too close to 1.
We note that the $\beta$ dependence on $L$, $\sigma$, and $c'$ are only artificial. We use these constraints in our proofs but they are mostly artefacts of the analysis. For example, if a function is $L$-smooth, and $L < 1$, then it is also $1$-smooth, so we can assume without loss of generality that $L > 1$. Similarly, the dependence on $\sigma$ is not highly relevant, since we can always increase the variance of the stochastic gradient, for example by adding an $O(1)$ gaussian perturbation.


\noindent
\textbf{Remark 2: (escaping saddle points)}
Note that Algorithm 2 reduces to CNC-SGD of \cite{DKLH18} when $\beta = 0$ (i.e. without momentum).
Therefore, let us compare the results.
We show that the escape time of Algorithm 2 is $T_{thred} :=\tilde{O}\big( \frac{(1-\beta)}{\eta \epsilon} \big)$ (see Section~\ref{app:TT}, especially (\ref{inq:TT}-\ref{TT})). On the other hand, for CNC-SGD, based on Table 3 in their paper, is $T_{thred} = \tilde{O}\big( \frac{1}{\eta \epsilon}  \big)$.
One can clearly see that $T_{thred}$ of our result has a dependency $1-\beta$, which makes it smaller than that of \cite{DKLH18} for any same $\eta$ and consequently demonstrates escaping saddle point faster with momentum.

\noindent
\textbf{Remark 3: (finding a second order stationary point)}
Denote $\ell$ a number such that $\forall t, \| g_t \| \leq \ell$.
In the following, we will show that in the high momentum regime where $(1-\beta) \ll \frac{ \rho^2 \ell^{10} }{  c_m^9  c_h^2 c' }$,
Algorithm~\ref{alg:0} is strictly better than CNC-SGD of \cite{DKLH18},
which means that a higher momentum can help find a second order stationary point faster.
Empirically, we find out that $c' \approx 0$ (Figure~\ref{fig:properties}) and $c_h \approx 0$ (Figure~\ref{fig:gace}) in the phase retrieval problem, so the condition is easily satisfied for a wide range of $\beta$.

\paragraph{Comparison to \cite{DKLH18}} 

Theorem 2 in \cite{DKLH18} states that, for CNC-SGD to find an $(\epsilon,\rho^{1/2} \epsilon)$ stationary point, the total number of iterations is 
\[
T = O( \frac{ \ell^{10} L^3  } { \delta^4 \gamma^4} \log^2 ( \frac{\ell L}{\epsilon \delta \gamma} ) \epsilon^{-10}  ),
\]
where $\ell$ is the bound of the stochastic gradient norm $\| g_t \| \leq \ell$ which can be viewed as the counterpart of $c_m$ in our work.
By translating their result for finding an $(\epsilon,\epsilon)$ stationary point, it is
$T = O( \frac{ \ell^{10} L^3 \rho^5} { \delta^4 \gamma^4} \log^2 ( \frac{\rho \ell L}{\epsilon \delta \gamma} ) \epsilon^{-10}  )$.
On the other hand, using the parameters value on Table~\ref{table:parameters}, we have that $T = 2 \TT \big( f(w_0) - \min_w f(w) \big) / (\delta  \FT)
= O( \frac{ (1-\beta) c_m^9 L^3 \rho^3 (\sigma^2)^3 c_h^2 c'  }{ \delta^4 \gamma^4} \log( \frac{L c_m \sigma^2  c' c_h}{ (1-\beta)  \delta \gamma \epsilon }  )\epsilon^{-10} )$ for Algorithm~\ref{alg:0}.

Before making a comparison, we note that their result does not have a dependency on the variance of stochastic gradient (i.e. $\sigma^2$), which is because they assume that the variance is also bounded by the constant $\ell$ (can be seen from (86) in the supplementary of their paper
where the variance terms $\| \zeta_i \|$ are bounded by $\ell$). 
Following their treatment, if we assume that $\sigma^2 \leq c_m$,
then on (\ref{eq:hi}) we can instead replace $(\sigma^2 + 3 c_m^2)$ with $4 c_m^2$ and on (\ref{r:2}) it becomes $1 \geq  \frac{ 576 c_m^3 \rho c_r \epsilon }{  (1-\beta)^3  }$. This will remove all the parameters' dependency on $\sigma^2$.
Now by comparing
$\tilde{O}( (1-\beta) c_m^9  c_h^2 c' \cdot  \frac{ \rho^3 L^3  }{ \delta^4 \gamma^4} \epsilon^{-10} )$
of ours and $T = \tilde{O}( \rho^2 \ell^{10} \cdot \frac{ \rho^3 L^3   }{ \delta^4 \gamma^4}  \epsilon^{-10}  )$ of \cite{DKLH18},
we see that in the high momentum regime where $(1-\beta) << \frac{ \rho^2 \ell^{10} }{  c_m^9  c_h^2 c' }$, Algorithm~\ref{alg:0} is strictly better than that of \cite{DKLH18}, which means that a higher momentum can help to find a second order stationary point faster.

\subsection{Escaping saddle points} \label{sub:hessian}

In this subsection, we analyze the process of escaping saddle points by SGD with momentum.
Denote $\t$ any time such that $( \t \mod \TT)  = 0$.
Suppose that it enters the region exhibiting a small gradient but a
large negative eigenvalue of the Hessian (i.e.
$\| \nabla f(w_{\t}) \| \leq \epsilon$ and $ \lambda_{\min} ( \nabla^2 f(w_{\t}) ) \leq - \epsilon$). We want to show that it takes at most $\TT$ iterations to escape the region
and whenever it escapes, the function value decreases at least by $\FT=O(\epsilon^4)$ on expectation,
where the precise expression of $\FT$ will be determined later in Section~\ref{app:lem:escape}.
The technique that we use is proving by contradiction.
Assume that the function value on expectation does not decrease at least $\FT$ in $\TT$ iterations.
Then, we get an upper bound of the expected distance $\EE{t_0}[ \| w_{\t+ \TT} - w_{\t} \|^2] \leq C_{\text{upper}}$. Yet,
by leveraging the negative curvature, we also show a lower bound of the form $\EE{t_0}[ \| w_{\t+ \TT} - w_{\t} \|^2 ]\geq C_{\text{lower}}$.
The analysis will show that the lower bound is larger than the upper bound (namely, $C_{\text{lower}} > C_{\text{upper}}$), which leads
to the contradiction and concludes that
the function value must decrease at least $\FT$ in $\TT$ iterations on expectation.
Since $\TT = O( (1-\beta) \log( \frac{1}{ (1-\beta) \epsilon } ) \epsilon^6  )$, the dependency on $\beta$ suggests that larger $\beta$ can leads to smaller $\TT$, which implies that larger momentum helps in escaping saddle points faster.

Lemma~\ref{lem:3} below provides an upper bound of the expected distance. The proof is in Section~\ref{app:lem:3}.
\begin{lemma} \label{lem:3}
  Denote $\t$ any time such that $( \t \mod \TT)  = 0$.
Suppose that $\EE{t_0}[ f(w_{\t}) - f(w_{\t+t}) ]\leq \FT$ for any $0\leq t \leq \TT$.
Then,
$\EE{t_0}[ \| w_{\t+t} - w_{\t} \|^2 ] \leq C_{\text{upper},t}  
  :=
  \frac{8 \eta t \big(  \FT  + 2 r^2 c_h + \frac{\rho}{3} r^3 c_m^3 \big)
}{ (1-\beta)^2}
 +
8 \eta^2 \frac{t \sigma^2}{(1-\beta)^2}  +  4 \eta^2 \big( \frac{\beta }{1- \beta}  \big)^2 c_m^2
  + 2 r^2 c_m^2.$
\end{lemma}
We see that $C_{\text{upper,t}}$ in Lemma~\ref{lem:3} is monotone increasing with $t$,
so we can define $C_{\text{upper}} := C_{\text{upper},\TT}$.
Now let us switch to obtaining the lower bound of $\EE{t_0}[ \| w_{\t+ \TT} - w_{\t} \|^2 ]$. The key to get the lower bound comes from the recursive dynamics of SGD with momentum.
\begin{lemma} \label{lem:recursive}
Denote $\t$ any time such that $( \t \mod \TT)  = 0$.
Let us define a quadratic approximation at $w_{\t}$,
$Q(w) := f(w_{\t}) + \langle w - w_{\t}, \nabla f(w_{\t}) \rangle + \frac{1}{2}
(w - w_{\t} )^\top H (w - w_{\t})$,
where $H:= \nabla^2 f(w_{\t})$. Also, define
$G_s  := (I - \eta  \sum_{k=1}^{s} \beta^{s-k}  H)$. Then we can write $w_{\t + t} - w_{\t}$ exactly using the following decomposition.
\begin{eqnarray*}
  &  &  \overbrace{  \big(  \Pi_{j=1}^{t-1} G_j \big) \big( -r m_{\t} \big)}^{q_{v,t-1}}
  + \eta \overbrace{(-1) \sum_{s=1}^{t-1} \big( \Pi_{j=s+1}^{t-1} G_{j} \big) \beta^{s} m_{\t}}^{q_{m,t-1}} \\
  & + &  \eta \overbrace{ (-1)\sum_{s=1}^{t-1} \big( \Pi_{j=s+1}^{t-1} G_{j} \big) \sum_{k=1}^{s} \beta^{s-k} \big( \nabla f(w_{\t+k}) - \nabla Q(w_{\t+s})  \big) }^{q_{q,t-1}} \\
  & + &  \eta \overbrace{ (-1) \sum_{s=1}^{t-1} \big( \Pi_{j=s+1}^{t-1} G_{j} \big) \sum_{k=1}^{s} \beta^{s-k} \nabla f(w_{\t})}^{q_{w,t-1}} 
  + \eta \overbrace{ (-1)  \sum_{s=1}^{t-1} \big( \Pi_{j=s+1}^{t-1} G_{j} \big)  \sum_{k=1}^{s} \beta^{s-k} \xi_{\t+k}}^{q_{\xi,t-1}}.
\end{eqnarray*}
\end{lemma}
The proof of Lemma~\ref{lem:recursive} is in Section~\ref{app:lem:recursive}. Furthermore, we will use the quantities $q_{v,t-1}, q_{m,t-1}, q_{q,t-1}, q_{w,t-1}, q_{\xi,t-1}$ as defined above throughout the analysis.
\begin{lemma} \label{lem:00}
Following the notations of Lemma~\ref{lem:recursive}, we have that
$$
\EE{t_0}[  \| w_{\t + t} - w_{\t} \|^2 ] \geq \EE{t_0}[ \| q_{v,t-1} \|^2  ] + 2 \eta \EE{t_0}[ \langle
q_{v,t-1} , q_{m,t-1} + q_{q,t-1} +  q_{w,t-1} + q_{\xi,t-1} \rangle ] =: C_{\text{lower}}.
$$
\end{lemma}

We are going to show that the dominant term in the lower bound of $\EE{t_0}[  \| w_{\t + t} - w_{\t} \|^2 ]$ is
$\EE{t_0}[ \| q_{v,t-1} \|^2  ]$, which is the critical component for ensuring that the lower bound is larger than the upper bound of the expected distance.

\begin{lemma} \label{lem:101}
Denote $\theta_j := \sum_{k=1}^{j} \beta^{j-k} = \sum_{k=1}^{j} \beta^{k-1}$
and $\lambda:= -\lambda_{\min}(H)$.
Following the conditions and notations in Lemma~\ref{lem:3} and Lemma~\ref{lem:recursive}, 
we have that
\begin{equation} \label{eq:lem:101}
\begin{aligned}
 \EE{t_0}[ \| q_{v,t-1} \|^2  ]\geq
\big(  \Pi_{j=1}^{t-1} (1+\eta \theta_j \lambda) \big)^2  r^2 \gamma.
\end{aligned}
\end{equation}
\end{lemma}

\begin{proof}

We know that $\lambda_{\min}(H) \leq -\epsilon < 0$.
Let $v$ be the eigenvector of the Hessian $H$ with unit norm that corresponds to $\lambda_{\min}(H)$
so that $H v = \lambda_{\min}(H) v$. We have
$ ( I - \eta H ) v = v  - \eta \lambda_{\min}(H) v = (1 - \eta \lambda_{\min}(H)  ) v.$
Then,
\begin{equation}
\begin{aligned}
&  \EE{t_0}[ \| q_{v,t-1} \|^2 ] \overset{(a)}{=} \EE{t_0}[ \| q_{v,t-1} \|^2 \| v \|^2 ]
\overset{(b)}{\geq} \EE{t_0}[ \langle q_{v,t-1}, v \rangle^2 ]
\overset{(c)}{=} \EE{t_0}[ \langle \big(  \Pi_{j=1}^{t-1} G_j \big) r m_{\t} , v \rangle^2 ]
\\ &  \overset{(d)}{ = }\EE{t_0}[ \langle \big(  \Pi_{j=1}^{t-1} (I-\eta \theta_j H) \big)  r m_{\t} , v \rangle^2 ]
 =\EE{t_0}  \langle \big(  \Pi_{j=1}^{t-1} (1-\eta \theta_j \lambda_{\min}(H)) \big)  r m_{\t} , v \rangle^2 ]
\\ & \overset{(e)}{\geq } 
\big(  \Pi_{j=1}^{t-1} (1+\eta \theta_j \lambda) \big)^2  r^2 \gamma,
\end{aligned}
\end{equation}
where $(a)$ is because $v$ is with unit norm, $(b)$ is by Cauchy–Schwarz inequality,
$(c)$, $(d)$ are by the definitions,
and $(e)$ is by the CNC assumption
so that $\EE{t_0}[  \langle m_{\t}, v   \rangle^2 ] \geq \gamma$.
\end{proof}
Observe that the lower bound in (\ref{eq:lem:101}) is monotone increasing with $t$ and the momentum parameter $\beta$.
Moreover, it actually grows exponentially in $t$.
To get the contradiction, we have to show that the lower bound is larger than the upper bound.
By Lemma~\ref{lem:3} and Lemma~\ref{lem:00}, it suffices to prove the following lemma. We provide its proof in Section~\ref{app:lem:escape}.

\begin{lemma} \label{lem:escape}
Let $\FT = O(\epsilon^4)$ and $\eta^2 \TT \leq r^2$.
By following the conditions and notations in Theorem~\ref{thm:main_escape}, Lemma~\ref{lem:3} and Lemma~\ref{lem:recursive}, we conclude that
if SGD with momentum (Algorithm~\ref{alg:0}) has the APCG property,
then we have that
$C_{\text{lower}}:=\EE{t_0}[ \| q_{v,\TT-1} \|^2  ] + 2 \eta \EE{t_0}[ \langle
q_{v,\TT-1} ,  q_{m,\TT-1} + q_{q,\TT-1} +  q_{w,\TT-1} + q_{\xi,\TT-1} \rangle]
> C_{\text{upper}}.$
\end{lemma}





\section{Detailed proofs}

\subsection{Lemma~\ref{lem:0a},~\ref{lem:0b}, and~\ref{lem:ch} } \label{app:lemmas:grad}

In the following,
Lemma~\ref{lem:0b} says that under the APAG property, when the gradient norm is large, on expectation SGD with momentum decreases the function value by a constant and consequently makes progress.
On the other hand, Lemma~\ref{lem:ch} upper-bounds the increase of function value of the next iterate (if happens) by leveraging the GrACE property.

\begin{lemma} \label{lem:0a}
If SGD with momentum has the APAG property,
then, considering the update step $w_{t+1}= w_t - \eta m_t$, we have that
$ \EE{t} [ f(w_{t+1} ) ] \leq f(w_t) - \frac{\eta}{2} \| \nabla f(w_t) \|^2 + \frac{L \eta^2  c_m^2 }{2}.$
\end{lemma}

\begin{proof}
By the $L$-smoothness assumption,
\begin{equation}
\begin{split}
 f(w_{t+1} ) &  \leq f(w_t) - \eta \langle \nabla f(w_t), m_t \rangle + \frac{L \eta^2 }{2} \| m_t \|^2
\\   \leq &  f(w_t) - \eta \langle \nabla f(w_t), g_t \rangle -
\eta \langle \nabla f(w_t),  m_t - g_{t} \rangle +
 \frac{L \eta^2  c_m^2 }{2}.
\end{split}
\end{equation}
Taking the expectation on both sides. We have
\begin{equation}
\begin{split}
& \EE{t}[ f(w_{t+1}) ]   \leq f(w_t) - \eta \| \nabla f(w_t) \|^2
- \eta  \EE{t}[ \langle \nabla f(w_t),  m_t - g_{t} \rangle   ]
+ \frac{L \eta^2  c_m^2 }{2}
\\ &  \leq  f(w_t)  - \frac{\eta}{2} \| \nabla f(w_t) \|^2 + \frac{L \eta^2  c_m^2 }{2}.
\end{split}
\end{equation}
where we use the APAG property in the last inequality.

\end{proof}

\begin{lemma} \label{lem:0b}
Assume that the step size $\eta$ satisfies $\eta \leq \frac{\epsilon^2}{8 L c_m^2} $.
If SGD with momentum has the APAG property,
then, considering the update step $w_{t+1}= w_t - \eta m_t$,
we have that
$\EE{t} [ f(w_{t+1} ) ] \leq f(w_t) - \frac{ \eta}{4} \epsilon^2$
when $\| \nabla f(w_t) \| \geq \epsilon$.
\end{lemma}

\begin{proof}
$\EE{t}[ f(w_{t+1}) - f(w_t)] \overset{Lemma~\ref{lem:0a}}{ \leq} - \frac{\eta}{2} \| \nabla f(w_t) \|^2 + \frac{L \eta^2  c_m^2 }{2} \overset{\| \nabla f(w_t) \| \geq \epsilon}{\leq} - \frac{\eta}{2} \epsilon^2 + \frac{L \eta^2  c_m^2 }{2}
  \leq  - \frac{\eta}{4} \epsilon^2$,
  where the last inequality is due to the constraint of $\eta$.
\end{proof}

\begin{lemma} \label{lem:ch}
If SGD with momentum has the GrACE property,
then, considering the update step $w_{t+1}= w_t - \eta m_t$, we have that
$\EE{t}[ f( w_{t+1} ) ]  \leq
f( w_t ) + \eta^2 c_h + \frac{\rho \eta^3 }{6} c_m^3$.
\end{lemma}

\begin{proof}
Consider the update rule $w_{t+1} = w_t - \eta m_t$, where $m_t$ represents the stochastic momentum and $\eta$ is the step size.
By $\rho$-Lipschitzness of Hessian, we have
$f(w_{t+1}) \leq f(w_t) - \eta \langle \nabla f(w_t) , g_t \rangle 
+ \eta \langle \nabla f(w_t), g_t - m_t \rangle
+ \frac{\eta^2}{2} m_t^\top \nabla^2 f(w_t) m_t 
+ \frac{\rho \eta^3 }{6} \| m_t \|^3 $. 
Taking the conditional expectation, one has
\begin{equation}
\begin{split}
\EE{t}[ f( w_{t+1} ) ] &  \leq
f( w_t ) - \EE{t}[ \eta \| \nabla f(w_t) \|^2 ] +  \EE{t}[ \eta \langle \nabla f(w_t), g_t - m_t \rangle
+ \frac{\eta^2}{2} m_t^\top \nabla^2 f(w_t) m_t ]
\\ & \quad
+ \frac{\rho \eta^3 }{6} c_m^3.
\\ & \leq
f( w_t ) + 0 + \eta^2 c_h + \frac{\rho \eta^3 }{6} c_m^3.
\end{split}
\end{equation}

\end{proof}

\subsection{Proof of Lemma~\ref{lem:3}} \label{app:lem:3}
\textbf{Lemma~\ref{lem:3}}
\textit{
  Denote $\t$ any time such that $( \t \mod \TT)  = 0$. 
Suppose that $\EE{t_0}[ f(w_{\t}) - f(w_{\t+t}) ]\leq \FT$ for any $0\leq t \leq \TT$.
Then, 
\begin{equation}
\begin{aligned}
& \EE{t_0}[ \| w_{\t+t} - w_{\t} \|^2 ] \\ &  \leq C_{\text{upper},t}   
\\ &   :=
  \frac{ 8 \eta t \big( \FT + 2 r^2 c_h +  \frac{\rho}{3} r^3 c_m^3 \big)}{(1-\beta)^2} 
  + 8 \eta^2 \frac{t \sigma^2}{(1-\beta)^2}
 + 4 \eta^2 \big( \frac{\beta }{1- \beta}  \big)^2 c_m^2
 + 2 r^2 c_m^2.
\end{aligned}
\end{equation}
}

\begin{proof}

Recall that the update is 
$w_{\t + 1} = w_{\t} - r m_{\t}$,
and 
$w_{\t + t} = w_{\t+t-1} - \eta m_{\t+t-1}$, 
for $t > 1$.
We have that
\begin{equation} \label{here:-1}
\begin{aligned}
 \| w_{\t+t} - w_{\t} \|^2  & \leq  2 ( \| w_{\t+t} - w_{\t+1} \|^2  
 +  \| w_{\t+1} - w_{\t} \|^2  )
 \leq 2 \| w_{\t+t} - w_{\t+1} \|^2  + 2 r^2 c_m^2,
\end{aligned}
\end{equation}
where the first inequality is by the triangle inequality and the second one is due to the assumption that $\| m_t\| \leq c_m$ for any $t$.
Now let us denote 
\begin{itemize}
\item $\alpha_s := \sum_{j=0}^{t-1-s} \beta^j$
\item $A_{t-1} := \sum_{s=1}^{t-1} \alpha_s$
\end{itemize}
and let us rewrite $g_{t} = \nabla f(w_{t}) + \xi_{t}$, where $\xi_{t}$ is the zero-mean noise. 
We have that
\begin{equation} \label{here:0}
\begin{split}
& \EE{t_0}[ \| w_{\t+t} - w_{\t+1} \|^2 ] =
\EE{t_0}[ \|  \sum_{s=1}^{t-1}  - \eta m_{\t + s }         \|^2   ]
= \EE{t_0}[ \eta^2 \|  \sum_{s=1}^{t-1} \big(   ( \sum_{j=1}^s \beta^{s-j} g_{\t+j} ) + \beta^{s} m_{\t}         \big)        \|^2   ] 
\\ &
\leq \EE{t_0}[ 2 \eta^2 \|  \sum_{s=1}^{t-1}   \sum_{j=1}^s \beta^{s-j} g_{\t+j}  \|^2
+ 2 \eta^2 \| \sum_{s=1}^{t-1}  \beta^{s} m_{\t}   \|^2   ]
\\ & \leq \EE{t_0}[ 2 \eta^2 \|  \sum_{s=1}^{t-1}   \sum_{j=1}^s \beta^{s-j} g_{\t+j}  \|^2 ]
+ 2 \eta^2 \big( \frac{\beta }{1- \beta}  \big)^2 c_m^2 
\\ & = \EE{t_0}[ 2 \eta^2 \|  \sum_{s=1}^{t-1} \alpha_s  g_{\t+s}  \|^2 ]
+ 2 \eta^2 \big( \frac{\beta }{1- \beta}  \big)^2 c_m^2  
\\ & = \EE{t_0}[ 2 \eta^2 \|  \sum_{s=1}^{t-1} \alpha_s \big( \nabla f(w_{\t+s}) + \xi_{\t+s} \big)  \|^2 ] 
+ 2 \eta^2 \big( \frac{\beta }{1- \beta}  \big)^2 c_m^2 
\\ & \leq 
\EE{t_0}[ 4 \eta^2 \|  \sum_{s=1}^{t-1} \alpha_s \nabla f(w_{\t+s})
\|^2 ] + \EE{t_0}[ 4 \eta^2 \| \sum_{s=1}^{t-1} \alpha_s \xi_{\t+s} \|^2 ]
+ 2 \eta^2 \big( \frac{\beta }{1- \beta}  \big)^2 c_m^2.  
\end{split}
\end{equation}
To proceed, we need to upper bound
$\EE{t_0}[ 4 \eta^2 \|  \sum_{s=1}^{t-1} \alpha_s \nabla f(w_{\t+s})
\|^2 ]$.
We have that
\begin{equation} \label{here:1}
\begin{split}
 \EE{t_0}[ 4 \eta^2 \|  \sum_{s=1}^{t-1} \alpha_s \nabla f(w_{\t+s})
\|^2 ] 
&  \overset{(a)}{ \leq} \EE{t_0}[ 4 \eta^2 A_{t-1}^2 \sum_{s=1}^{t-1} \frac{\alpha_s}{A_{t-1}}  \|   \nabla f(w_{\t+s}) \|^2 ]
\\ & \overset{(b)}{ \leq} \EE{t_0}[ 4 \eta^2 \frac{A_{t-1}}{1-\beta} \sum_{s=1}^{t-1} \|  \nabla f(w_{\t+s}) \|^2 ]
\\ & \overset{(c)}{ \leq} \EE{t_0}[ 4 \eta^2 \frac{t}{(1-\beta)^2} \sum_{s=1}^{t-1} \|  \nabla f(w_{\t+s}) \|^2 ].
\end{split}
\end{equation}
where $(a)$ is by Jensen's inequality, $(b)$ is by $\max_{s} \alpha_s \leq \frac{1}{1-\beta}$, and $(c)$ is by $A_{t-1} \leq \frac{t}{1-\beta}$.
Now let us switch to bound the other term.
\begin{equation} \label{here:2}
\begin{split}
& \EE{t_0} [  4 \eta^2 \|  \sum_{s=1}^{t-1} \alpha_s \xi_{\t+s}\|^2 ]
= 4 \eta^2 \big( \EE{t_0} [ \sum_{i\neq j}^{t-1} \alpha_{i} \alpha_{j} \xi_{\t+i}^\top \xi_{\t+j}  ] 
+ \EE{t_0} [ \sum_{s=1}^{t-1} \alpha_{s}^2 \xi_{\t+s}^\top \xi_{\t+s}  ] \big) 
\\ & 
\overset{(a)}{=} 4 \eta^2 \big( 0
+ \EE{t_0} [ \sum_{s=1}^{t-1} \alpha_{s}^2 \xi_{\t+s}^\top \xi_{\t+s}  ] \big), 
\\ & 
\overset{(b)}{\leq} 4 \eta^2 \frac{t \sigma^2}{(1-\beta)^2}.
\end{split}
\end{equation}
where $(a)$ is because $\EE{t_0} [ \xi_{\t+i}^\top \xi_{\t+j} ] = 0$ for $i \neq j$,
$(b)$ is by that $\| \xi_t \|^2 \leq \sigma^2$ and $\max_{t} \alpha_t \leq \frac{1}{1-\beta}$. 
Combining (\ref{here:-1}), (\ref{here:0}), (\ref{here:1}), (\ref{here:2}),
\begin{equation} \label{q3}
\begin{split}
& \EE{t_0}[ \| w_{\t+t} - w_{\t} \|^2 ] \\ & \leq
 \EE{t_0}[ 8 \eta^2 \frac{t}{(1-\beta)^2} \sum_{s=1}^{t-1} \|  \nabla f(w_{\t+s}) \|^2 ]
 + 8 \eta^2 \frac{t \sigma^2}{(1-\beta)^2}
 + 4 \eta^2 \big( \frac{\beta }{1- \beta}  \big)^2 c_m^2
 + 2 r^2 c_m^2.
\end{split}
\end{equation}
Now we need to bound  $\EE{t_0}[ \sum_{s=1}^{t-1} \|  \nabla f(w_{\t+s}) \|^2 ]$.
By using $\rho$-Lipschitzness of Hessian, we have that
\begin{equation}
\begin{split}
& f(w_{\t+s} ) \leq f( w_{\t+s-1}) - \eta \langle \nabla f( w_{\t+s-1} ), m_{\t+s-1} \rangle
+ \frac{1}{2} \eta^2 m_{\t+s-1}^\top \nabla^2 f( w_{\t+s-1} ) m_{\t+s-1} \\ & \qquad \qquad + \frac{\rho}{6} \eta^3 \| m_{\t+s-1} \|^3.
\end{split}
\end{equation}
By adding $\eta \langle \nabla f(w_{\t+s-1}), g_{\t+s-1} \rangle$ on both sides, we have
\begin{equation}
\begin{split}
\eta \langle \nabla f(w_{\t+s-1}), g_{\t+s-1} \rangle
& \leq f( w_{\t+s-1}) - f( w_{\t+s} )
+ \eta \langle \nabla f( w_{\t+s-1} ), g_{\t+s-1} - m_{\t+s-1} \rangle
\\ & + \frac{1}{2} \eta^2 m_{\t+s-1}^\top \nabla^2 f( w_{\t+s-1} ) m_{\t+s-1} + \frac{\rho}{6} \eta^3 \| m_{\t+s-1} \|^3.
\end{split}
\end{equation}
Taking conditional expectation on both sides leads to
\begin{equation} \label{q3a}
\begin{split}
\EE{\t+s-1}[ \eta \| \nabla f(w_{\t+s-1}) \|^2 ] \leq  \EE{\t+s-1}[ f( w_{\t+s-1}) - f( w_{\t+s} ) ]
+ \eta^2 c_h + \frac{\rho}{6} \eta^3 c_m^3,
\end{split}
\end{equation}
where  
$\EE{\t+s-1}[ \eta \langle \nabla f( w_{\t+s-1} ), g_{\t+s-1} - m_{\t+s-1} \rangle
+ \frac{1}{2} \eta^2 m_{\t+s-1}^\top \nabla^2 f( w_{\t+s-1} ) m_{\t+s-1} ] 
\leq \eta^2 c_h $ by the GrACE property.
We have that for $\t \leq \t + s-1$
\begin{equation}
\begin{split}
& \EE{\t}[ \eta \| \nabla f(w_{\t+s-1}) \|^2 ]
= \EE{\t}[ \EE{\t+s-1}[  \eta \| \nabla f(w_{\t+s-1}) \|^2    ]    ]
\\ & \overset{(\ref{q3a})}{ \leq } 
\EE{\t}[ \EE{\t+s-1}[  f( w_{\t+s-1}) - f( w_{\t+s} ) ] ]
+ \eta^2 c_h + \frac{\rho}{6} \eta^3 c_m^3
\\ & = 
\EE{\t}[ f( w_{\t+s-1}) - f( w_{\t+s} ) ] 
+ \eta^2 c_h + \frac{\rho}{6} \eta^3 c_m^3.
\end{split}
\end{equation}
Summing the above inequality from $s=2,3,\dots,t$ leads to
\begin{equation} \label{q3b}
\begin{split}
& \EE{\t}[ \sum_{s=1}^{t-1} \eta \| \nabla f(w_{\t+s}) \|^2 ]
 \leq \EE{\t}[ f(w_{\t+1}) - f( w_{\t+t} ) ]
+ \eta^2 (t-1) c_h + \frac{\rho}{6} \eta^3 (t-1) c_m^3
\\ & =\EE{\t}[ f(w_{\t+1}) - f(w_{\t}) + f(w_{\t}) - f( w_{\t+t} ) ]
+ \eta^2 (t-1) c_h + \frac{\rho}{6} \eta^3 (t-1) c_m^3
\\ & \overset{(a)}{\leq} \EE{\t}[ f(w_{\t+1}) - f(w_{\t})] + \FT 
+ \eta^2 (t-1) c_h + \frac{\rho}{6} \eta^3 (t-1) c_m^3,
\end{split}
\end{equation}
where $(a)$ is by the assumption (made for proving by contradiction) that
$\EE{\t}[ f(w_{\t}) - f(w_{\t+s})] \leq \FT$ for any $0 \leq s \leq \TT$.
By (\ref{q3a}) with $s=1$ and $\eta=r$, 
we have
\begin{equation} \label{q3c}
\begin{split}
\EE{\t}[ r \| \nabla f(w_{\t}) \|^2 ] \leq  \EE{\t}[ f( w_{\t}) - f( w_{\t+1} ) ]
+ r^2 c_h + \frac{\rho}{6} r^3 c_m^3.
\end{split}
\end{equation}
By (\ref{q3b}) and (\ref{q3c}), we know that
\begin{equation} \label{q3d}
\begin{split}
& \EE{\t}[ \sum_{s=1}^{t-1} \eta \| \nabla f(w_{\t+s}) \|^2 ]
\leq \EE{\t}[ r \| f(w_{\t}) \|^2 ]  +  \EE{\t}[ \sum_{s=1}^{t-1} \eta \| \nabla f(w_{\t+s}) \|^2 ]
\\ & \leq  \FT  + r^2 c_h + \frac{\rho}{6} r^3 c_m^3
+ \eta^2 t c_h + \frac{\rho}{6} \eta^3 t c_m^3
\\ & \overset{(a)}{ \leq } \FT + 2 r^2 c_h +  \frac{\rho}{6} r^3 c_m^3  + \frac{\rho}{6} r^2 \eta c_m^3.
\\ & \overset{(b)}{ \leq } \FT + 2 r^2 c_h +  \frac{\rho}{3} r^3 c_m^3,  
\end{split}
\end{equation}
where $(a)$ is by the constraint that $\eta^2 t \leq r^2$ for $0 \leq t \leq \TT$
and $(b)$ is by the constraint that $r \geq \eta$.
By combining (\ref{q3d}) and (\ref{q3})
\begin{equation} 
\begin{split}
& \EE{t_0}[ \| w_{\t+t} - w_{\t} \|^2 ]  
\\& \leq
 \EE{t_0}[ 8 \eta^2 \frac{t}{(1-\beta)^2} \sum_{s=1}^{t-1} \|  \nabla f(w_{\t+s}) \|^2 ]
 + 8 \eta^2 \frac{t \sigma^2}{(1-\beta)^2}
 + 4 \eta^2 \big( \frac{\beta }{1- \beta}  \big)^2 c_m^2
 + 2 r^2 c_m^2
\\ & \leq \frac{ 8 \eta t \big( \FT + 2 r^2 c_h +  \frac{\rho}{3} r^3 c_m^3 \big)}{(1-\beta)^2} 
  + 8 \eta^2 \frac{t \sigma^2}{(1-\beta)^2}
 + 4 \eta^2 \big( \frac{\beta }{1- \beta}  \big)^2 c_m^2
 + 2 r^2 c_m^2.
\end{split}
\end{equation}
\end{proof}


\subsection{Proof of Lemma~\ref{lem:recursive} and Lemma~\ref{lem:00}} \label{app:lem:recursive}
\noindent 
\textbf{Lemma~\ref{lem:recursive}}
\textit{
Denote $\t$ any time such that $( \t \mod \TT)  = 0$. 
Let us define a quadratic approximation at $w_{\t}$,
$Q(w) := f(w_{\t}) + \langle w - w_{\t}, \nabla f(w_{\t}) \rangle + \frac{1}{2}
(w - w_{\t} )^\top H (w - w_{\t})$,
where $H:= \nabla^2 f(w_{\t})$. Also, define
$G_s  := (I - \eta  \sum_{k=1}^{s} \beta^{s-k}  H)$
and
\begin{itemize}
  \small
\item $q_{v,t-1} := \big(  \Pi_{j=1}^{t-1} G_j \big) \big( -r m_{\t} \big).$
\item $q_{m,t-1} :=   - \sum_{s=1}^{t-1} \big( \Pi_{j=s+1}^{t-1} G_{j} \big) \beta^{s} m_{\t} $.
\item $q_{q,t-1} :=  - \sum_{s=1}^{t-1} \big( \Pi_{j=s+1}^{t-1} G_{j} \big) \sum_{k=1}^{s} \beta^{s-k} \big( \nabla f(w_{\t+k}) - \nabla Q(w_{\t+s})  \big)  $.
\item $q_{w,t-1} :=  - \sum_{s=1}^{t-1} \big( \Pi_{j=s+1}^{t-1} G_{j} \big) \sum_{k=1}^{s} \beta^{s-k} \nabla f(w_{\t}).$
\item $q_{\xi,t-1} :=
- \sum_{s=1}^{t-1} \big( \Pi_{j=s+1}^{t-1} G_{j} \big)  \sum_{k=1}^{s} \beta^{s-k} \xi_{\t+k}.$
\end{itemize}
Then,
$w_{\t + t} - w_{\t}
= q_{v,t-1} + \eta q_{m,t-1} + \eta q_{q,t-1} + \eta q_{w,t-1} + \eta q_{\xi,t-1}.$
}


\begin{mdframed}
\textbf{Notations:} \\
Denote $\t$ any time such that $( \t \mod \TT)  = 0$. 
Let us define a quadratic approximation at $w_{\t}$,
\begin{equation} \label{eq:rec0}
Q(w) := f(w_{\t}) + \langle w - w_{\t}, \nabla f(w_{\t}) \rangle + \frac{1}{2}
(w - w_{\t} )^\top H (w - w_{\t}),
\end{equation}
where $H:= \nabla^2 f(w_{\t})$.
Also, we denote 
\begin{equation} \label{eq:rec11}
\begin{aligned}
G_s & := (I - \eta  \sum_{k=1}^{s} \beta^{s-k}  H)
\\ v_{m,s} & := \beta^{s} m_{\t}
\\ v_{q,s} & := \sum_{k=1}^{s} \beta^{s-k} \big( \nabla f(w_{\t+k}) - \nabla Q(w_{\t+s})  \big)
\\ v_{w,s} & := \sum_{k=1}^{s} \beta^{s-k} \nabla f(w_{\t}) 
\\ v_{\xi,s} & := \sum_{k=1}^{s} \beta^{s-k} \xi_{\t+k}
\\ \theta_s & := \sum_{k=1}^s \beta^{s-k}.
\end{aligned}
\end{equation}

\end{mdframed}

\begin{proof}

First, we rewrite $m_{\t+j}$ for any $j \geq 1$ as follows.
\begin{equation} \label{rec:0}
\begin{aligned}
m_{\t+j} & = \beta^j m_{\t} +  \sum_{k=1}^j \beta^{j-k} g_{\t + k} 
\\ & =  \beta^j m_{\t} + \sum_{k=1}^j \beta^{j-k} \big( \nabla f(w_{\t +k}) + \xi_{\t+k} \big).
\end{aligned}
\end{equation}
We have that
\begin{equation} \label{rec:1}
\begin{aligned}
 w_{\t + t} - w_{\t} & =  w_{\t + t -1 } - w_{\t} - \eta m_{\t + t -1}
\\ & \overset{(a)}{=}  w_{\t + t -1 } - w_{\t} - \eta \big( \beta^{t-1} m_{\t} +
 \sum_{k=1}^{t-1} \beta^{t-1-k} \big( \nabla f(w_{\t+k}) + \xi_{\t+k} \big) \big)
 \\ & \overset{(b)}{ =}  w_{\t + t -1 } - w_{\t} 
- \eta  \sum_{k=1}^{t-1} \beta^{t-1-k} \nabla Q(w_{\t+t-1})
\\ &-  \eta \big( \beta^{t-1} m_{\t} +
 \sum_{k=1}^{t-1} \beta^{t-1-k} \big( \nabla f(w_{\t+k})  - \nabla Q(w_{\t+t-1}) + \xi_{\t+k} \big) \big)
 \\ & \overset{(c)}{=}  w_{\t + t -1 } - w_{\t} 
- \eta  \sum_{k=1}^{t-1} \beta^{t-1-k} \big( H ( w_{\t+t-1} - w_{\t})
+ \nabla f(w_{\t}) \big)
\\ & 
-  \eta \big( \beta^{t-1} m_{\t} +
 \sum_{k=1}^{t-1} \beta^{t-1-k} \big( \nabla f(w_{\t+k})  - \nabla Q(w_{\t+t-1}) + \xi_{\t+k} \big) \big)
 \\ & = (I - \eta  \sum_{k=1}^{t-1} \beta^{t-1-k}  H) \big(  w_{\t + t -1 } - w_{\t} \big)
\\ & -  \eta \big( \beta^{t-1} m_{\t}  +
 \sum_{k=1}^{t-1} \beta^{t-1-k} \big( \nabla f(w_{\t+k})  - \nabla Q(w_{\t+t-1}) + \nabla f(w_{\t}) +  \xi_{\t+k} \big) \big),
\end{aligned}
\end{equation} \label{rec:2}
where
$(a)$ is by using (\ref{rec:0}) with $j=t-1$, $(b)$ is by subtracting and adding back the same term, and $(c)$ is 
by $\nabla Q(w_{\t+t-1}) = \nabla f(w_{\t}) + H ( w_{\t+t-1} - w_{\t} )$.

To continue, by using the nations in (\ref{eq:rec11}), we can rewrite (\ref{rec:1}) as
\begin{equation} \label{rec:4}
\begin{aligned}
w_{\t + t} - w_{\t}  =  
G_{t-1} \big(  w_{\t + t -1 } - w_{\t} \big)
-\eta \big( v_{m,t-1} + v_{q,t-1} + v_{w,t-1} + v_{\xi,t-1}  \big).
\end{aligned}
\end{equation} 
Recursively expanding (\ref{rec:4}) leads to
\begin{equation} \label{rec:5}
\begin{aligned}
& w_{\t + t} - w_{\t}   =  
G_{t-1} \big(  w_{\t + t -1 } - w_{\t} \big)
-\eta \big( v_{m,t-1} + v_{q,t-1} + v_{w,t-1} + v_{\xi,t-1}  \big)
\\ & =   G_{t-1} \big(  G_{t-2} \big( w_{\t + t -2 } - w_{\t}  \big)
- \eta \big( v_{m,t-2} + v_{q,t-2} + v_{w,t-2} + v_{\xi,t-2}  \big) \big)
\\ & -\eta \big( v_{m,t-1} + v_{q,t-1} + v_{w,t-1} + v_{\xi,t-1}  \big)
\\ & \overset{(a)}{=} \big(  \Pi_{j=1}^{t-1} G_j \big) \big( w_{\t+1} - w _{\t} ) - \eta \sum_{s=1}^{t-1} \big( \Pi_{j=s+1}^{t-1} G_{j} \big)
\big(  v_{m,s} + v_{q,s} + v_{w,s} + v_{\xi,s}  \big),
\\ & \overset{(b)}{=} \big(  \Pi_{j=1}^{t-1} G_j \big) \big( -r m_{\t} \big)  - \eta \sum_{s=1}^{t-1} \big( \Pi_{j=s+1}^{t-1} G_{j} \big)
\big(  v_{m,s} + v_{q,s} + v_{w,s} + v_{\xi,s}  \big),
\end{aligned}
\end{equation} 
where $(a)$ we use the notation that $\Pi_{j=s}^{t-1} G_j := G_s \times G_{s+1} \times \dots  \dots G_{t-1}$ and the notation that $\Pi_{j=t}^{t-1} G_j = 1$
and $(b)$ is by the update rule. By using the definitions of $\{ q_{\star,t-1} \}$ in the lemma statement, we complete the proof.

\end{proof}
\noindent
\textbf{Lemma~\ref{lem:00}}
\textit{
Following the notations of Lemma~\ref{lem:recursive}, we have that
\begin{equation} 
\begin{aligned}
 \EE{t_0}[  \| w_{\t + t} - w_{\t} \|^2 ] & \geq \EE{t_0}[ \| q_{v,t-1} \|^2  ] + 2 \eta \EE{t_0}[ \langle 
q_{v,t-1} , q_{m,t-1} + q_{q,t-1} +  q_{w,t-1} + q_{\xi,t-1} \rangle ] \\ & := C_{\text{lower}}
\end{aligned}
\end{equation}
}
\begin{proof}
Following the proof of Lemma~\ref{lem:recursive}, we have
\begin{equation}
\begin{aligned}
& w_{\t + t} - w_{\t}  
= q_{v,t-1}  + \eta \big( q_{m,t-1} + q_{q,t-1} +  q_{w,t-1} + q_{\xi,t-1} \big).
\end{aligned}
\end{equation}
Therefore, by using $\| a + b \|^2 \geq \| a \|^2 + 2 \langle a , b \rangle $,
\begin{equation} \label{eq:total}
\begin{aligned}
& \EE{t_0}[ \| w_{\t + t} - w_{\t} \|^2 ] \geq \EE{t_0}[ \| q_{v,t-1} \|^2 ] + 2 \eta \EE{t_0}[ \langle 
q_{v,t-1} ,  q_{m,t-1} + q_{q,t-1} +  q_{w,t-1} + q_{\xi,t-1} \rangle].
\end{aligned}
\end{equation}
\end{proof}


\subsection{Proof of Lemma~\ref{lem:escape}} \label{app:lem:escape}

\textbf{Lemma~\ref{lem:escape}}
\textit{ 
Let $\FT = O(\epsilon^4)$ and $\eta^2 \TT \leq r^2$.
By following the conditions and notations in Theorem~\ref{thm:main_escape}, Lemma~\ref{lem:3} and Lemma~\ref{lem:recursive}, we conclude that
if SGD with momentum (Algorithm~\ref{alg:0}) has the APCG property,
then we have that
$C_{\text{lower}} := \EE{t_0}[ \| q_{v,\TT-1} \|^2  ] + 2 \eta \EE{t_0}[ \langle 
q_{v,\TT-1} ,  q_{m,\TT-1} + q_{q,\TT-1} +  q_{w,\TT-1} + q_{\xi,\TT-1} \rangle]
> C_{\text{upper}}.$
}

\paragraph{Some supporting lemmas}
To prove Lemma~\ref{lem:escape}, we need a series of lemmas with the choices of parameters on Table~\ref{table:parameters}.

\noindent
\textbf{Upper bounding $\EE{t_0}[ \| q_{q,t-1} \| ]$:}

\begin{mdframed}

\begin{lemma} \label{lem:up_qq}
Following the conditions in Lemma~\ref{lem:3} and Lemma~\ref{lem:recursive}, we have
\begin{equation} 
\begin{aligned}
\EE{t_0}[ \| q_{q,t-1} \|  ]   
 \leq &  \big( \Pi_{j=1}^{t-1}  (1+ \eta \theta_{j} \lambda ) \big)
\frac{ \beta L  c_m}{ \epsilon (1-\beta)^2 }
\\ & 
+ 
\frac{\big( \Pi_{j=1}^{t-1}  (1+ \eta \theta_{j} \lambda ) \big) }{1-\beta}
\frac{\rho}{ \eta \epsilon^2 }  
 \frac{ 8  \big( \FT + 2 r^2 c_h +  \frac{\rho}{3} r^3 c_m^3 \big)}{(1-\beta)^2}
 \\ & +    
\frac{\big( \Pi_{j=1}^{t-1}  (1+ \eta \theta_{j} \lambda ) \big) }{1-\beta}
\frac{\rho \big( 8  \frac{r^2 \sigma^2}{(1-\beta)^2} + 4 \eta^2 \big( \frac{\beta }{1- \beta}  \big)^2 c_m^2 + 2 r^2 c_m^2 \big)
 }{2 \eta \epsilon}.
\end{aligned}
\end{equation}
\end{lemma}

\end{mdframed}

\begin{proof}

\begin{equation}  \label{qq:1}
\begin{aligned}
& \EE{t_0}[ \| q_{q,t-1} \|  ] = 
\EE{t_0}[ \| - \sum_{s=1}^{t-1} \big( \Pi_{j=s+1}^{t-1} G_{j} \big) \sum_{k=1}^{s} \beta^{s-k} \big( \nabla f(w_{\t+k}) - \nabla Q(w_{\t+s})  \big) \|  ]
\\ &
\overset{(a)}{\leq}
\EE{t_0}[ \sum_{s=1}^{t-1}  \| \big( \Pi_{j=s+1}^{t-1} G_{j} \big) \sum_{k=1}^{s} \beta^{s-k} \big( \nabla f(w_{\t+k}) - \nabla Q(w_{\t+s})  \big) \|  ]
\\ &
\overset{(b)}{\leq}
\EE{t_0}[ \sum_{s=1}^{t-1}  \| \big( \Pi_{j=s+1}^{t-1} G_{j} \big) \|_2 \| \sum_{k=1}^{s} \beta^{s-k} \big( \nabla f(w_{\t+k}) - \nabla Q(w_{\t+s})  \big) \|  ]
\\ &
\overset{(c)}{\leq} 
\EE{t_0}[ \sum_{s=1}^{t-1} \| \big( \Pi_{j=s+1}^{t-1} G_{j} \big) \|_2 \sum_{k=1}^{s} \beta^{s-k} \|  \big( \nabla f(w_{\t+k}) - \nabla Q(w_{\t+s})  \big) \|  ]
\\ &
\overset{(d)}{\leq} 
\EE{t_0}[ \sum_{s=1}^{t-1} \| \big( \Pi_{j=s+1}^{t-1} G_{j} \big) \|_2  \sum_{k=1}^{s} \beta^{s-k}  \big( \|  \nabla f(w_{\t+k}) - \nabla f(w_{\t+s}) \| 
\\ & \qquad  + \|
\nabla f(w_{\t+s}) - 
\nabla Q(w_{\t+s})   \| \big)  ],
\end{aligned}
\end{equation}
where $(a)$, $(c)$, $(d)$ is by triangle inequality, $(b)$ is by the fact that $\|A x \|_2 \leq \| A \|_2 \| x \|_2$ for any matrix $A$ and vector $x$.
Now that we have an upper bound of $\|  \nabla f(w_{\t+k}) - \nabla f(w_{\t+s}) \|$,
\begin{equation} \label{qq:2}
\begin{aligned}
\|  \nabla f(w_{\t+k}) - \nabla f(w_{\t+s}) \|
\overset{(a)}{\leq} L \| w_{\t+k} - w_{\t+s} \| \overset{(b)}{\leq} L \eta (s-k) c_m.
\end{aligned}
\end{equation}
where $(a)$ is by the assumption of L-Lipschitz gradient and $(b)$ is by applying the triangle inequality $(s-k)$ times and that $\| w_{t} - w_{t-1} \| \leq \eta \| m_{t-1} \| \leq \eta c_m$, for any $t$.
We can also derive an upper bound of $\EE{t_0}[ \|
\nabla f(w_{\t+s}) - 
\nabla Q(w_{\t+s})   \| ]$,
\begin{equation} \label{qq:3}
\begin{aligned}
& \EE{t_0}[ \|
\nabla f(w_{\t+s}) - 
\nabla Q(w_{\t+s})   \| ]
\\ & \overset{(a)}{\leq} \EE{t_0} [  \frac{\rho}{2} \| w_{\t+s} - w_{\t} \|^2 ]
\\ & \overset{(b)}{\leq}  
\frac{ \rho}{2 } \big( \frac{8 \eta s \big( \FT + 2 r^2 c_h +  
\frac{\rho}{3} r^3 c_m^3 \big) }{(1-\beta)^2}
  + 8 \frac{r^2 
   \sigma^2}{(1-\beta)^2}
 + 4 \eta^2 \big( \frac{\beta }{1- \beta}  \big)^2 c_m^2
 + 2 r^2 c_m^2.
\big)
 \end{aligned}
\end{equation}
Above, $(a)$ is by the fact that 
if a function $f(\cdot)$ has $\rho$ Lipschitz Hessian, then
\begin{equation}
\| \nabla f(y) - \nabla f(x) - \nabla^2 f(x) ( y-x ) \| \leq \frac{\rho}{2} \| y -x \|^2
\end{equation}
(c.f. Lemma 1.2.4 in (\cite{N13}))
and using the definition that
$$Q(w) := f(w_{\t}) + \langle w - w_{\t}, \nabla f(w_{\t}) \rangle + \frac{1}{2}
(w - w_{\t} )^\top H (w - w_{\t}),$$
(b) is by Lemma~\ref{lem:3} and $\eta^2 t \leq r^2$ for $0\leq t \leq \TT$
\begin{equation} 
\begin{aligned}
&\EE{t_0}[ \| w_{\t+t} - w_{\t} \|^2 ] 
\\& \leq  \frac{ 8 \eta t \big( \FT + 2 r^2 c_h +  \frac{\rho}{3} r^3 c_m^3 \big)}{(1-\beta)^2} 
  + 8 \eta^2 \frac{t \sigma^2}{(1-\beta)^2}
 + 4 \eta^2 \big( \frac{\beta }{1- \beta}  \big)^2 c_m^2
 + 2 r^2 c_m^2
\\ &  \leq
 \frac{ 8 \eta t \big( \FT + 2 r^2 c_h +  \frac{\rho}{3} r^3 c_m^3 \big)}{(1-\beta)^2} 
  + 8 \frac{ r^2 \sigma^2}{(1-\beta)^2}
 + 4 \eta^2 \big( \frac{\beta }{1- \beta}  \big)^2 c_m^2
 + 2 r^2 c_m^2.
\end{aligned}
\end{equation}
Combing (\ref{qq:1}), (\ref{qq:2}), (\ref{qq:3}), we have that
\begin{equation} \label{qq:4}
\begin{aligned}
&  \EE{t_0}[ \| q_{q,t-1} \|  ]
\\ & \overset{ (\ref{qq:1}) }{\leq} 
\EE{t_0}[ \sum_{s=1}^{t-1} \| \big( \Pi_{j=s+1}^{t-1} G_{j} \big) \|_2  \sum_{k=1}^{s} \beta^{s-k}  \big( \|  \nabla f(w_{\t+k}) - \nabla f(w_{\t+s}) \| 
\\& \qquad + \|
\nabla f(w_{\t+s}) - 
\nabla Q(w_{\t+s})   \| \big)  ]
\\ &  \overset{ (\ref{qq:2}),(\ref{qq:3}) }{\leq}\sum_{s=1}^{t-1} \| \big( \Pi_{j=s+1}^{t-1} G_{j} \big) \|_2
\sum_{k=1}^{s} \beta^{s-k} L \eta (s-k) c_m  
\\ & + \text{ }
 \sum_{s=1}^{t-1} \| \big( \Pi_{j=s+1}^{t-1} G_{j} \big) \|_2
\sum_{k=1}^{s} \beta^{s-k} \frac{\rho}{2}
\big(  \frac{ 8 \eta s \big( \FT + 2 r^2 c_h +  \frac{\rho}{3} r^3 c_m^3 \big)}{(1-\beta)^2} 
  + 8  \frac{r^2 \sigma^2}{(1-\beta)^2}
\\ &
 + 4 \eta^2 \big( \frac{\beta }{1- \beta}  \big)^2 c_m^2
 + 2 r^2 c_m^2 \big)
\\ & 
:= \text{ } \sum_{s=1}^{t-1} \| \big( \Pi_{j=s+1}^{t-1} G_{j} \big) \|_2
\sum_{k=1}^{s} \beta^{s-k} L \eta (s-k) c_m   
+
\sum_{s=1}^{t-1} \| \big( \Pi_{j=s+1}^{t-1} G_{j} \big) \|_2
\sum_{k=1}^{s} \beta^{s-k} \frac{\rho}{2} ( \nu_s + \nu ),
\end{aligned}
\end{equation}
where on the last line we use the notation that  
\begin{equation}
\begin{aligned}
 \nu_s & := 
 \frac{ 8 \eta s \big( \FT + 2 r^2 c_h +  \frac{\rho}{3} r^3 c_m^3 \big)}{(1-\beta)^2}
\\  
\nu & :=  
8  \frac{r^2 \sigma^2}{(1-\beta)^2}
 + 4 \eta^2 \big( \frac{\beta }{1- \beta}  \big)^2 c_m^2
 + 2 r^2 c_m^2.
 \end{aligned}  
\end{equation}
To continue, let us analyze $\| \big( \Pi_{j=s+1}^{t-1} G_{j} \big) \|_2$ first.
\begin{equation} \label{qq:5}
\begin{split}
& \| \big( \Pi_{j=s+1}^{t-1} G_{j} \big) \|_2
= \|  \Pi_{j=s+1}^{t-1} (I - \eta  \sum_{k=1}^{j} \beta^{j-k}  H ) \|_2
\\ & \overset{(a)}{\leq} \Pi_{j=s+1}^{t-1}  (1+ \eta \theta_{j} \lambda )
= \frac{\Pi_{j=1}^{t-1}  (1+ \eta \theta_{j} \lambda ) }{ \Pi_{j=1}^{s}  (1+ \eta \theta_{j} \lambda )} \overset{(b)}{\leq}
\frac{\Pi_{j=1}^{t-1}  (1+ \eta \theta_{j} \lambda ) }{ (1+ \eta \epsilon)^s }.
\end{split}
\end{equation}
Above, we use the notation that $\theta_{j} :=\sum_{k=1}^{j} \beta^{j-k}$.
For (a), it is due to that $\lambda := - \lambda_{min}(H)$,
$\lambda_{\max}(H) \leq L$,
and the choice of $\eta$ so that $1 \geq \frac{\eta L}{1-\beta}$, 
or equivalently, 
\begin{equation} \label{eta:another}
 \eta \leq \frac{1-\beta}{L}.
\end{equation}
For $(b)$, it is due to that $\theta_j \geq 1$ for any $j$ and $\lambda \geq \epsilon$.
Therefore, we can upper-bound the first term on r.h.s of (\ref{qq:4}) as
\begin{equation} \label{qq:6}
\begin{split}
& \sum_{s=1}^{t-1} \| \big( \Pi_{j=s+1}^{t-1} G_{j} \big) \|_2
\sum_{k=1}^{s} \beta^{s-k} L \eta (s-k) c_m
= \sum_{s=1}^{t-1} \| \big( \Pi_{j=s+1}^{t-1} G_{j} \big) \|_2
\sum_{k=1}^{s-1} \beta^k k L \eta  c_m
\\ & \overset{(a)}{\leq}
\sum_{s=1}^{t-1} \| \big( \Pi_{j=s+1}^{t-1} G_{j} \big) \|_2
\frac{ \beta }{ (1-\beta)^2 } L \eta  c_m
\\ & \overset{(b)}{\leq} \big(  \Pi_{j=1}^{t-1}  (1+ \eta \theta_{j} \lambda ) \big)
\frac{ \beta L \eta  c_m}{ (1-\beta)^2 } 
\sum_{s=1}^{t-1} \frac{1}{ (1+ \eta \epsilon)^s }
\\ & \overset{(c)}{\leq} 
\big( \Pi_{j=1}^{t-1}  (1+ \eta \theta_{j} \lambda ) \big)
\frac{ \beta L \eta  c_m}{ (1-\beta)^2 } \frac{1}{\eta \epsilon} 
= \big( \Pi_{j=1}^{t-1}  (1+ \eta \theta_{j} \lambda ) \big)
\frac{ \beta L  c_m}{ \epsilon (1-\beta)^2 }, 
\end{split}
\end{equation}
where $(a)$ is by that fact that $\sum_{k=1}^\infty \beta^k k \leq \frac{\beta}{(1-\beta)^2} $ for any $0\leq \beta < 1$, $(b)$ is by using $(\ref{qq:5})$, and $(c)$ is by using that
$\sum_{s=1}^{\infty} ( \frac{1}{1+\eta \epsilon} )^s \leq \frac{1}{\eta \epsilon}$.
Now let us switch to bound 
$\sum_{s=1}^{t-1} \| \big( \Pi_{j=s+1}^{t-1} G_{j} \big) \|_2
\sum_{k=1}^{s} \beta^{s-k}
\frac{\rho}{2} ( \nu_s + \nu )
$ on (\ref{qq:4}).
We have that
\begin{equation} \label{qq:7}
\begin{aligned}
& \sum_{s=1}^{t-1} \| \big( \Pi_{j=s+1}^{t-1} G_{j} \big) \|_2
\sum_{k=1}^{s} \beta^{s-k}
\frac{\rho}{2} ( \nu_s + \nu )
\overset{(a)}{ \leq } \frac{1}{1-\beta} \sum_{s=1}^{t-1} \| \big( \Pi_{j=s+1}^{t-1} G_{j} \big) \|_2
\frac{\rho}{2} ( \nu_s + \nu )
\\ & \overset{(b)}{\leq} 
\frac{\big( \Pi_{j=1}^{t-1}  (1+ \eta \theta_{j} \lambda ) \big) }{1-\beta}
 \sum_{s=1}^{t-1} \frac{1}{(1+\eta \epsilon)^s }
\frac{\rho}{2} \nu_s +    
\frac{\big( \Pi_{j=1}^{t-1}  (1+ \eta \theta_{j} \lambda ) \big) }{1-\beta}
 \sum_{s=1}^{t-1} \frac{1}{(1+\eta \epsilon)^s }
\frac{\rho}{2} \nu
\\ & \overset{(c)}{\leq} 
\frac{\big( \Pi_{j=1}^{t-1}  (1+ \eta \theta_{j} \lambda ) \big) }{1-\beta}
 \sum_{s=1}^{t-1} \frac{1}{(1+\eta \epsilon)^s }
\frac{\rho}{2} \nu_s +    
\frac{\big( \Pi_{j=1}^{t-1}  (1+ \eta \theta_{j} \lambda ) \big) }{1-\beta}
\frac{\rho \nu}{ 2 \eta \epsilon} 
\\ & = 
\frac{\big( \Pi_{j=1}^{t-1}  (1+ \eta \theta_{j} \lambda ) \big) }{1-\beta}
 \sum_{s=1}^{t-1} \frac{1}{(1+\eta \epsilon)^s }
\frac{\rho}{2} \nu_s 
\\ & \quad +    \frac{\big( \Pi_{j=1}^{t-1}  (1+ \eta \theta_{j} \lambda ) \big) }{1-\beta}
\frac{\rho \big( 8  \frac{r^2 \sigma^2}{(1-\beta)^2} + 4 \eta^2 \big( \frac{\beta }{1- \beta}  \big)^2 c_m^2 + 2 r^2 c_m^2 \big)
 }{2 \eta \epsilon} 
\\ & \overset{(d)}{\leq}
\frac{\big( \Pi_{j=1}^{t-1}  (1+ \eta \theta_{j} \lambda ) \big) }{1-\beta}
\frac{\rho}{ (\eta \epsilon)^2 }   
 \frac{ 8 \eta \big( \FT + 2 r^2 c_h +  \frac{\rho}{3} r^3 c_m^3 \big)}{(1-\beta)^2}
\\ & +    \frac{\big( \Pi_{j=1}^{t-1}  (1+ \eta \theta_{j} \lambda ) \big) }{1-\beta}
\frac{\rho \big( 8  \frac{r^2 \sigma^2}{(1-\beta)^2} + 4 \eta^2 \big( \frac{\beta }{1- \beta}  \big)^2 c_m^2 + 2 r^2 c_m^2 \big)
 }{2 \eta \epsilon} 
\end{aligned}
\end{equation}
where $(a)$ is by the fact that $\sum_{k=1}^{s} \beta^{s-k} \leq 1 / (1-\beta)$,
$(b)$ is by (\ref{qq:5}),
$(c)$ is by using that
$\sum_{s=1}^{\infty} ( \frac{1}{1+\eta \epsilon} )^s \leq \frac{1}{\eta \epsilon}$,
$(d)$ is by 
$\sum_{k=1}^\infty z^k k \leq \frac{z}{(1-z)^2}$
for any $|z| \leq 1$ and substituting $z = \frac{1}{1+\eta \epsilon}$,
which leads to 
$\sum_{k=1}^\infty z^k k \leq \frac{z}{(1-z)^2} =  \frac{ 1 / (1 + \eta \epsilon) }{ ( 1 - 1 / (1+\eta \epsilon) )^2}   = \frac{ 1+\eta \epsilon }{ ( \eta \epsilon)^2} \leq \frac{2}{( \eta \epsilon)^2}$ in which the last inequality is by chosen the step size $\eta$ so that
$\eta \epsilon \leq 1$.

By combining (\ref{qq:4}), (\ref{qq:6}), and (\ref{qq:7}), we have that

\begin{equation} \label{qq:8}
\begin{aligned}
& \EE{t_0}[ \| q_{q,t-1} \| \overset{(\ref{qq:4})}{ \leq } \sum_{s=1}^{t-1} \| \big( \Pi_{j=s+1}^{t-1} G_{j} \big) \|_2
\sum_{k=1}^{s} \beta^{s-k} L \eta (s-k) c_m   
\\ & \qquad \qquad \qquad + 
\sum_{s=1}^{t-1} \| \big( \Pi_{j=s+1}^{t-1} G_{j} \big) \|_2
\sum_{k=1}^{s} \beta^{s-k}
\frac{\rho}{2} ( \nu_s + \nu )
\\ & 
\overset{ (\ref{qq:6}), (\ref{qq:7}) }{\leq} \big( \Pi_{j=1}^{t-1}  (1+ \eta \theta_{j} \lambda ) \big)
\frac{ \beta L  c_m}{ \epsilon (1-\beta)^2 }
\\ & 
 \text{ } +
\frac{\big( \Pi_{j=1}^{t-1}  (1+ \eta \theta_{j} \lambda ) \big) }{1-\beta}
\frac{\rho}{ \eta \epsilon^2 }   
 \frac{ 8  \big( \FT + 2 r^2 c_h +  \frac{\rho}{3} r^3 c_m^3 \big)}{(1-\beta)^2}
\\ & +    
\frac{\big( \Pi_{j=1}^{t-1}  (1+ \eta \theta_{j} \lambda ) \big) }{1-\beta}
 \frac{\rho \big( 8  \frac{r^2 \sigma^2}{(1-\beta)^2} + 4 \eta^2 \big( \frac{\beta }{1- \beta}  \big)^2 c_m^2 + 2 r^2 c_m^2 \big)
 }{2 \eta \epsilon} ,
\end{aligned}
\end{equation}
which completes the proof.

\end{proof}


\noindent
\textbf{Upper bounding $ \| q_{v,t-1} \|$:}
\begin{mdframed}
\begin{lemma} \label{lem:up_qv}
Following the conditions in Lemma~\ref{lem:3} and Lemma~\ref{lem:recursive}, we have
\begin{equation} 
\begin{aligned}
\| q_{v,t-1} \| \leq \big( \Pi_{j=1}^{t-1}  (1+ \eta \theta_{j} \lambda ) \big)  r c_m.
\end{aligned}
\end{equation}
\end{lemma}

\end{mdframed}

\begin{proof}
\begin{equation} 
\begin{aligned}
& \| q_{v,t-1} \| \leq 
\| \big(  \Pi_{j=1}^{t-1} G_j \big) \big( -r m_{\t} \big) \|
\leq \| \big(  \Pi_{j=1}^{t-1} G_j \big) \|_2 \| -r m_{\t} \|
\leq \big( \Pi_{j=1}^{t-1}  (1+ \eta \theta_{j} \lambda ) \big)  r c_m,
\end{aligned}
\end{equation}
where the last inequality is because $\eta$ is chosen so that $1 \geq \frac{\eta L}{1-\beta}$
and the fact that $\lambda_{\max}(H) \leq L$.

\end{proof}

\noindent
\textbf{Lower bounding $\EE{t_0}[ 2 \eta \langle q_{v,t-1} , q_{q,t-1} \rangle ]$:}
\begin{mdframed}
\begin{lemma} \label{lem:lb_vq}
Following the conditions in Lemma~\ref{lem:3} and Lemma~\ref{lem:recursive}, we have
\begin{equation} 
\begin{aligned}
& \EE{t_0}[ 2 \eta \langle q_{v,t-1} , q_{q,t-1} \rangle ] 
\\ & \geq - 2 \eta \big( \Pi_{j=1}^{t-1}  (1+ \eta \theta_{j} \lambda ) \big)^2  r c_m \times \\  & 
\big[
\frac{ \beta L  c_m}{ \epsilon (1-\beta)^2 }
 + 
\frac{\rho}{ \eta \epsilon^2 }  
 \frac{ 8  \big( \FT + 2 r^2 c_h +  \frac{\rho}{3} r^3 c_m^3 \big)}{(1-\beta)^3}
+
\frac{\rho \big( 8  \frac{r^2 \sigma^2}{(1-\beta)^2} + 4 \eta^2 \big( \frac{\beta }{1- \beta}  \big)^2 c_m^2 + 2 r^2 c_m^2 \big)
 }{2 \eta \epsilon (1-\beta)}
].
\end{aligned}
\end{equation}
\end{lemma}

\end{mdframed}

\begin{proof}
By the results of Lemma~\ref{lem:up_qq} and Lemma~\ref{lem:up_qv}
\begin{equation} 
\begin{aligned}
& \EE{t_0}[ 2 \eta \langle q_{v,t-1} , q_{q,t-1} \rangle ] 
 \geq  - \EE{t_0}[ 2 \eta \|  q_{v,t-1} \| \| q_{q,t-1} \|  ]
\\ & \overset{Lemma~\ref{lem:up_qv}  }{\geq}  - \EE{t_0}[ 2 \eta \big( \Pi_{j=1}^{t-1}  (1+ \eta \theta_{j} \lambda ) \big)  r c_m  \| q_{q,t-1} \|  ]
\\ & \overset{Lemma~\ref{lem:up_qq}  }{\geq} - 2 \eta \big( \Pi_{j=1}^{t-1}  (1+ \eta \theta_{j} \lambda ) \big)^2  r c_m \times \\ &
\big[
\frac{ \beta L  c_m}{ \epsilon (1-\beta)^2 }
 + 
\frac{\rho}{ \eta \epsilon^2 }  
 \frac{ 8  \big( \FT + 2 r^2 c_h +  \frac{\rho}{3} r^3 c_m^3 \big)}{(1-\beta)^3}
+
\frac{\rho \big( 8  \frac{r^2 \sigma^2}{(1-\beta)^2} + 4 \eta^2 \big( \frac{\beta }{1- \beta}  \big)^2 c_m^2 + 2 r^2 c_m^2 \big)
 }{2 \eta \epsilon (1-\beta)}
].
\end{aligned}
\end{equation}
\end{proof}


\noindent
\textbf{Lower bounding $\EE{t_0} [ 2 \eta \langle q_{v,t-1}, q_{ \xi,t-1} \rangle]$:}

\begin{mdframed}
\begin{lemma} \label{lem:lb:xi}
Following the conditions in Lemma~\ref{lem:3} and Lemma~\ref{lem:recursive}, we have
\begin{equation} \label{lb:xi}
\begin{aligned}
& \EE{t_0}[ 2 \eta \langle q_{v,t-1}, q_{ \xi,t-1} \rangle  ] =  0.
\end{aligned}
\end{equation}
\end{lemma}
\end{mdframed}

\begin{proof}

\begin{equation}
\begin{aligned}
& \EE{t_0}[ 2 \eta \langle q_{v,t-1}, q_{ \xi,t-1} \rangle  ] 
= \EE{t_0}[ 2 \eta \langle q_{v,t-1}, - \sum_{s=1}^{t-1} \big( \Pi_{j=s+1}^{t-1} G_{j} \big)  \sum_{k=1}^{s} \beta^{s-k} \xi_{\t+k} \rangle  ] 
\\ &
\overset{(a)}{=} \EE{t_0}[ 2 \eta \langle q_{v,t-1},  \sum_{k=1}^{s} \alpha_k \xi_{\t+k} \rangle  ]  
\\ &
\overset{(b)}{=} \EE{t_0}[ 2 \eta \sum_{k=1}^s \EE{\t+k-1}[ \langle q_{v,t-1}, \alpha_k  \xi_{\t+k} \rangle  ]  ]
\\ &
\overset{(c)}{=} \EE{t_0}[ 2 \eta \sum_{k=1}^s \langle q_{v,t-1}, \EE{\t+k-1}[  \alpha_k  \xi_{\t+k}] \rangle    ] 
\\ &
= \EE{t_0}[ 2 \eta \sum_{k=1}^s \alpha_k \langle q_{v,t-1}, \EE{\t+k-1}[    \xi_{\t+k}] \rangle    ] 
\\ & 
\overset{(d)}{=} 0,
\end{aligned}
\end{equation}
where $(a)$ holds for some coefficients $\alpha_k$,
$(b)$ is by the tower rule, $(c)$ is because $q_{v,t-1}$ is measureable with $\t$,
and $(d)$ is by the zero mean assumption of $\xi$'s.

\end{proof}


\noindent
\textbf{Lower bounding $\EE{t_0} [ 2 \eta \langle q_{v,t-1}, q_{m,t-1} \rangle]$:}

\begin{mdframed}
\begin{lemma} \label{lem:lb_qm}
Following the conditions in Lemma~\ref{lem:3} and Lemma~\ref{lem:recursive}, we have
\begin{equation}
\begin{aligned}
\EE{t_0} [ 2 \eta \langle q_{v,t-1}, q_{m,t-1} \rangle] \geq 0.
\end{aligned}
\end{equation}
\end{lemma}
\end{mdframed}

\begin{proof}

\begin{equation}
\begin{aligned}
& \EE{t_0} [ 2 \eta \langle q_{v,t-1}, q_{m,t-1} \rangle]
\\ & = 2 \eta r \EE{t_0}[  \langle \big(  \Pi_{j=1}^{t-1} G_j \big)  m_{\t} , 
\sum_{s=1}^{t-1} \big( \Pi_{j=s+1}^{t-1} G_{j} \big) \beta^{s}   m_{\t} \rangle ]
\\ & \overset{(a)}{=} 2 \eta r \EE{t_0}[  \langle   m_{\t} , B  m_{\t} \rangle ] 
\overset{(b)}{\geq} 0,
\end{aligned}
\end{equation}
where $(a)$ is by defining the matrix $B:= \big(  \Pi_{j=1}^{t-1} G_j \big)^\top 
\big( \sum_{s=1}^{t-1} \big( \Pi_{j=s+1}^{t-1} G_{j} \big) \beta^{s}  \big) $.
For (b), notice that the matrix $B$ is symmetric positive semidefinite.
To see that the matrix $B$ is symmetric positive semidefinite,
observe that each $G_j := (I - \eta  \sum_{k=1}^{j} \beta^{j-k}  H)$
can be written in the form of $G_j = U D_j U^\top$ for some orthonormal matrix $U$ and a diagonal matrix $D_j$. Therefore, the matrix product 
$\big(  \Pi_{j=1}^{t-1} G_j \big)^\top \big( \Pi_{j=s+1}^{t-1} G_{j} \big)= 
U (  \Pi_{j=1}^{t-1} D_j ) ( \Pi_{j=s+1}^{t-1} D_{j} ) U^\top$
is symmetric positive semidefinite as long as each $G_j$ is.
So, $(b)$ is by the property of a matrix being symmetric positive semidefinite.

\end{proof}


\noindent
\textbf{Lower bounding $ 2 \eta \EE{t_0}[ \langle q_{v,t-1},  q_{w,t-1} \rangle ]$:}

\begin{mdframed}
\begin{lemma}
\label{lem:lb_qvw}
Following the conditions in Lemma~\ref{lem:3} and Lemma~\ref{lem:recursive},
if SGD with momentum has the APCG property, 
then
\begin{equation}
 2 \eta \EE{t_0}[ \langle q_{v,t-1},  q_{w,t-1} \rangle ]
\geq  - \frac{ 2 \eta r  c'  }{( 1-\beta )  } (\Pi_{j=1}^{t-1}  (1+ \eta \theta_{j} \lambda ))^2 \epsilon.
\end{equation}
\end{lemma}

\end{mdframed}

\begin{proof}

Define $D_s:=  \Pi_{j=1}^{t-1} G_j  \Pi_{j=s+1}^{t-1} G_{j}$.
\begin{equation} \label{eq:vw:ta}
\begin{split}
& 2 \eta \EE{t_0}[ \langle q_{v,t-1},  q_{w,t-1} \rangle ]
=
2 \eta \EE{t_0}[  \langle  \big(  \Pi_{j=1}^{t-1} G_j \big) \big( r m_{\t} \big), 
 \sum_{s=1}^{t-1} \big( \Pi_{j=s+1}^{t-1} G_{j} \big) \sum_{k=1}^{s} \beta^{s-k} \nabla f(w_{\t}) \rangle ]
\\ & = 
2 \eta \EE{t_0}[  \langle  r m_{\t} , 
 \sum_{s=1}^{t-1} \big(  \Pi_{j=1}^{t-1} G_j  \Pi_{j=s+1}^{t-1} G_{j} \big) \sum_{k=1}^{s} \beta^{s-k} \nabla f(w_{\t}) \rangle ]
\\ &
= 2 \eta r \sum_{s=1}^{t-1}  \sum_{k=1}^{s} \beta^{s-k} \EE{t_0}[  \langle   m_{\t} , 
 D_s  \nabla f(w_{\t}) \rangle ]
\\ &
\overset{(a)}{ \geq} - 2 \eta^2 r  c'  
 \sum_{s=1}^{t-1}  \sum_{k=1}^{s} \beta^{s-k} \| D_s \|_2 \| \nabla f(w_{\t}) \|^2 
\\ &
\geq
- \frac{ 2 \eta^2 r  c' }{ 1 - \beta}  
 \sum_{s=1}^{t-1}   \| D_s \|_2 \| \nabla f(w_{\t}) \|^2,
\end{split}
\end{equation}
where $(a)$ is by the APCG property.
We also have that
\begin{equation}
\begin{split}
& \| D_s \|_2 = \| \Pi_{j=1}^{t-1} G_j  \Pi_{j=s+1}^{t-1} G_{j} \|_2
 \leq \|  \Pi_{j=1}^{t-1} G_j \|_2 \| \Pi_{j=s+1}^{t-1} G_{j} \|_2
\\ & \overset{(a)}{\leq} \|  \Pi_{j=1}^{t-1} G_j \|_2 \frac{\Pi_{j=1}^{t-1}  (1+ \eta \theta_{j} \lambda ) }{ (1+ \eta \epsilon)^s }
\overset{(b)}{ \leq } \frac{ \big( \Pi_{j=1}^{t-1}  (1+ \eta \theta_{j} \lambda ) \big)^2 }{ (1+ \eta \epsilon)^s }
\end{split}
\end{equation}
where (a) and (b) is by (\ref{qq:5}). Substituting the result back to (\ref{eq:vw:ta}), we get
\begin{equation}
\begin{split}
& 2 \eta \EE{t_0}[ \langle q_{v,t-1},  q_{w,t-1} \rangle ]
\geq 
- \frac{ 2 \eta^2 r  c' }{ 1 - \beta}  
 \sum_{s=1}^{t-1}   \| D_s \|_2 \| \nabla f(w_{\t}) \|^2
\\ &  \geq 
 - \frac{ 2 \eta^2 r  c' }{ 1 - \beta}  
 \sum_{s=1}^{t-1} \frac{ \big( \Pi_{j=1}^{t-1}  (1+ \eta \theta_{j} \lambda ) \big)^2 }{ (1+ \eta \epsilon)^s } \| \nabla f(w_{\t}) \|^2
\\ & \geq  - \frac{ 2 \eta^2 r  c' }{ ( 1 - \beta) \eta \epsilon}  
\big( \Pi_{j=1}^{t-1} ( 1+ \eta \theta_{j} \lambda ) \big)^2 \| \nabla f(w_{\t}) \|^2
\end{split}
\end{equation}
Using the fact that $\| \nabla f(w_{\t})\| \leq \epsilon$ completes the proof.

\end{proof}

\paragraph{Proof of Lemma~\ref{lem:escape}} 

Recall that the strategy is proving by contradiction.
Assume that the function value does not decrease at least $\FT$ in $\TT$ iterations on expectation.
Then, we can get an upper bound of the expected distance $\EE{t_0}[ \| w_{\t+ \TT} - w_{\t} \|^2] \leq C_{\text{upper}}$
but, by leveraging the negative curvature, we can also show a lower bound of the form $\EE{t_0}[ \| w_{\t+ \TT} - w_{\t} \|^2 ]\geq C_{\text{lower}}$.
The strategy is showing that the lower bound is larger than the upper bound, which leads 
to the contradiction and concludes that 
the function value must decrease at least $\FT$ in $\TT$ iterations on expectation.
To get the contradiction, according to Lemma~\ref{lem:3} and Lemma~\ref{lem:00},
we need to show that
\begin{equation}
\begin{aligned}
& \EE{t_0}[ \| q_{v,\TT-1} \|^2  ] + 2 \eta \EE{t_0}[ \langle 
q_{v,\TT-1} ,  q_{m,\TT-1} + q_{q,\TT-1} +  q_{w,\TT-1} + q_{\xi,\TT-1} \rangle]
\\ & > C_{upper}.
\end{aligned}
\end{equation}
Yet, by Lemma~\ref{lem:lb_qm} and Lemma~\ref{lem:lb:xi},
we have that $ \eta \EE{t_0}[ \langle q_{v,\TT-1} , q_{m,\TT-1} \rangle] \geq 0$
and  $ \eta \EE{t_0}[ \langle q_{v,\TT-1} , q_{\xi,\TT-1} \rangle] = 0$. So, it suffices to prove that
\begin{equation}
\begin{aligned}
 \EE{t_0}[ \| q_{v,\TT-1} \|^2  ] + 2 \eta \EE{t_0}[ \langle 
q_{v,\TT-1} , q_{q,\TT-1} +  q_{w,\TT-1} \rangle]
 > C_{upper},
\end{aligned}
\end{equation}
and it suffices to show that
\begin{itemize}
\item $\frac{1}{4} \EE{t_0}[ \| q_{v,\TT-1} \|^2 ] + 2 \eta \EE{t_0}[ \langle 
q_{v,\TT-1} , q_{q,\TT-1} \rangle ] \geq 0$.
\item $\frac{1}{4} \EE{t_0}[ \| q_{v,\TT-1} \|^2 ] + 2 \eta \EE{t_0}[ \langle 
q_{v,\TT-1} , q_{w,\TT-1} \rangle ] \geq 0$.
\item $\frac{1}{4} \EE{t_0}[ \| q_{v,\TT-1} \|^2 ] \geq C_{upper}$.
\end{itemize} 

\noindent
\textbf{Proving that $\frac{1}{4} \EE{t_0}[ \| q_{v,\TT-1} \|^2 ] + 2 \eta \EE{t_0}[ \langle 
q_{v,\TT-1} , q_{q,\TT-1} \rangle ] \geq 0$:}

By Lemma~\ref{lem:101} and Lemma~\ref{lem:lb_vq}, we have that
\begin{equation}
\begin{split}
& \frac{1}{4} \EE{t_0}[ \| q_{v,\TT-1} \|^2 ] + \EE{t_0}[ 2 \eta \langle q_{v,\TT-1} , q_{q,\TT-1} \rangle ] 
\\ & \geq 
\frac{1}{4}\big(  \Pi_{j=1}^{\TT-1} (1+\eta \theta_j \lambda) \big)^2  r^2 \gamma
 - 2 \eta \big( \Pi_{j=1}^{\TT-1}  (1+ \eta \theta_{j} \lambda ) \big)^2  r c_m
\\ &  \times \big[
\frac{ \beta L  c_m}{ \epsilon (1-\beta)^2 }
 + 
\frac{\rho}{ \eta \epsilon^2 }  
 \frac{ 8  \big( \FT + 2 r^2 c_h +  \frac{\rho}{3} r^3 c_m^3 \big)}{(1-\beta)^3}
+
\frac{\rho \big( 8  \frac{r^2 \sigma^2}{(1-\beta)^2} + 4 \eta^2 \big( \frac{\beta }{1- \beta}  \big)^2 c_m^2 + 2 r^2 c_m^2 \big)
 }{2 \eta \epsilon (1-\beta)}
].
\end{split}
\end{equation}
To show that the above is nonnegative, it suffices to show that
\begin{equation} \label{e:0}
\begin{aligned}
r^2 \gamma \geq  \frac{ 24 \eta r \beta L  c_m^2}{ \epsilon (1-\beta)^2 },
\end{aligned}
\end{equation}
and
\begin{equation} \label{e:1}
\begin{aligned}
r^2 \gamma \geq
\frac{24 \eta r c_m  \rho}{ (1-\beta) \eta \epsilon^2 }    \frac{8 \big(  \FT  +  
   2 r^2 c_h + \frac{\rho}{3} r^3 c_m^3
 \big)
 }{ (1-\beta)^2},
\end{aligned}
\end{equation}
and
\begin{equation} \label{e:2}
\begin{aligned}
r^2 \gamma \geq 
\frac{24 \eta r c_m }{1-\beta}
\frac{\rho \big( 8 \frac{r^2 \sigma^2}{(1-\beta)^2}  +  4 \eta^2 \big( \frac{\beta}{1- \beta} \big)^2 c_m^2
  + 2 r^2 c_m^2 \big)
 }{2 \eta \epsilon}.
 \end{aligned}
\end{equation}

Now w.l.o.g, we assume that $c_m$, $L$, $\sigma^2$, $c'$, and $\rho$ are not less than one and 
that $\epsilon \leq 1$.
By using the values of parameters on Table~\ref{table:parameters}, we have the following results;
a sufficient condition of (\ref{e:0}) is that 
\begin{equation} \label{eta:0}
\frac{c_r}{c_{\eta} } \geq \frac{ 24 L c_m^2 \epsilon^2 }{ (1-\beta)^2 }.
\end{equation}
A sufficient condition of (\ref{e:1}) is that
\begin{equation} \label{f:0}
\frac{c_r}{c_F} \geq \frac{576 c_m \rho }{ (1-\beta)^3},
\end{equation}
and
\begin{equation} \label{f:001}
1 \geq \frac{ 1152 c_m \rho c_h c_r }{ (1-\beta)^3  },
\end{equation}
and
\begin{equation} \label{r:1}
1 \geq \frac{ 192 c_m^4 \rho^2  c_r^2 }{ (1-\beta)^3  }.
\end{equation}
A sufficient condition of (\ref{e:2}) is that
\begin{equation} \label{eq:hi}
1 \geq  \frac{ 96 c_m \rho (\sigma^2 + 3 c_m^2) c_r \epsilon }{ (1-\beta)^3     },
\end{equation}
and a sufficient condition for the above (\ref{eq:hi}),
by the assumption that both $\sigma^2 \geq 1$ and $c_m \geq 1$, is
\begin{equation} \label{r:2}
1 \geq  \frac{ 576 c_m^3 \rho \sigma^2 c_r \epsilon }{  (1-\beta)^3  }.
\end{equation}

Now let us verify if (\ref{eta:0}), (\ref{f:0}), (\ref{f:001}), (\ref{r:1}), (\ref{r:2}) are satisfied.
For (\ref{eta:0}), using the constraint of $c_{\eta}$ on Table~{\ref{table:parameters}},
we have that $\frac{1}{c_{\eta}} \geq \frac{c_m^5 \rho L^2 \sigma^2 c' c_h }{c_1}$.
Using this inequality, it suffices to let $c_r \geq \frac{c_0 \epsilon^2}{c_m^3 \rho L \sigma^2 c' c_h (1-\beta)^2}$ for getting (\ref{eta:0}),
which holds by using the constraint that $c'(1-\beta)^2 > 1$ and $\epsilon \leq 1$.
For (\ref{f:0}), using the constraint of $c_{F}$ on Table~{\ref{table:parameters}},
we have that $\frac{1}{c_{F}} \geq \frac{c_m^4 \rho^2 L \sigma^4 c_h }{c_2}$.
Using this inequality, it suffices to let $c_r \geq \frac{c_0}{c_m^3 \rho L \sigma^4 (1-\beta)^3}$,
which holds by using the constraint that $\sigma^2 (1-\beta)^3 > 1$.
For (\ref{f:001}), it needs $\frac{(1-\beta)^3}{ 1152 c_m \rho c_h } \geq 
\frac{ c_0}{ c_m^3 \rho L \sigma^2 c_h } \geq c_r$,
which hold by using the constraint that $\sigma^2 (1-\beta)^3 > 1$.
For (\ref{r:1}), it suffices to let 
$\frac{(1-\beta)^2}{ 14 c_m^2 \rho} \geq \frac{c_0}{ c_m^3 \rho L \sigma^2 c_h } \geq c_r$ 
which holds by using the constraint that $\sigma^2 (1-\beta)^3 > 1$.
For (\ref{r:2}), it suffices to let 
$\frac{ (1-\beta) ^3}{ 576 c_m^3 \rho \sigma^2 \epsilon } \geq \frac{c_0}{c_m^3 \rho L \sigma^2 c_h} \geq c_r $,
which holds by using the constraint that $L (1-\beta)^3 > 1$ and $\epsilon \leq 1$.
Therefore, by choosing the parameter values as Table~{\ref{table:parameters}}, we can guarantee that
$\frac{1}{4} \EE{t_0}[ \| q_{v,\TT-1} \|^2 ] + 2 \eta \EE{t_0}[ \langle 
q_{v,\TT-1} , q_{q,\TT-1} \rangle ] \geq 0$.

\noindent
\textbf{Proving that $\frac{1}{4} \EE{t_0}[ \| q_{v,\TT-1} \|^2 ] + 2 \eta \EE{t_0}[ \langle 
q_{v,\TT-1} , q_{w,\TT-1} \rangle ] \geq 0$:}
By Lemma~\ref{lem:101} and Lemma~\ref{lem:lb_qvw}, we have that
\begin{equation} \label{e:3}
\begin{aligned}
& \frac{1}{4} \EE{t_0}[ \| q_{v,\TT-1} \|^2 ] + 2 \eta \EE{t_0}[ \langle 
q_{v,\TT-1} , q_{w,\TT-1} \rangle ] \\ & \geq 
\frac{1}{4}\big(  \Pi_{j=1}^{\TT-1} (1+\eta \theta_j \lambda) \big)^2  r^2 \gamma - 
    \frac{2 \eta r  c' }{( 1-\beta ) } (\Pi_{j=1}^{\TT-1}  (1+ \eta \theta_{j} \lambda ))^2 \epsilon.
\end{aligned}
\end{equation}

To show that the above is nonnegative, it suffices to show that
\begin{equation} 
r^2 \gamma \geq  \frac{8 \eta r c'  \epsilon }{( 1-\beta ) }.
\end{equation}
A sufficient condition is $\frac{c_r}{c_{\eta}} \geq \frac{8 \epsilon^4 c'}{1-\beta}$.
Using the constraint of $c_{\eta}$ on Table~{\ref{table:parameters}},
we have that $\frac{1}{c_{\eta}} \geq \frac{c_m^5 \rho L^2 \sigma^2 c' c_h}{c_1}$.
So, it suffices to let
$c_r \geq \frac{c_0 \epsilon^4 }{3 c_m^5 \rho L^2 \sigma^2 c_h (1-\beta)}$,
which holds by using the constraint that $L(1-\beta)^3 > 1$ (so that $L(1-\beta) > 1$) and $\epsilon\leq 1$.

\noindent
\textbf{Proving that $\frac{1}{4} \EE{t_0}[ \| q_{v,\TT-1} \|^2 ] \geq C_{upper}$:} \label{app:TT}

From Lemma~\ref{lem:101} and Lemma~\ref{lem:3}, we need to show that
\begin{equation} \label{grow:1}
\begin{split}
& \frac{1}{4} \big(  \Pi_{j=1}^{\TT-1} (1+\eta \theta_j \lambda) \big)^2  r^2 \gamma
\\ & \geq 
  \frac{ 8 \eta t \big( \FT + 2 r^2 c_h +  \frac{\rho}{3} r^3 c_m^3 \big)}{(1-\beta)^2} 
 +
8 \frac{r^2 \sigma^2}{(1-\beta)^2}  +  4 \eta^2 \big( \frac{\beta}{1- \beta} \big)^2 c_m^2
  + 2 r^2 c_m^2.
\end{split}  
\end{equation}
We know that $ \frac{1}{4} \big(  \Pi_{j=1}^{\TT-1} (1+\eta \theta_j \lambda) \big)^2  r^2 \gamma
\geq \frac{1}{4} \big(  \Pi_{j=1}^{\TT-1} (1+\eta \theta_j \epsilon) \big)^2  r^2 \gamma
$. It suffices to show that
\begin{equation} \label{grow:1}
\begin{split}
& \frac{1}{4} \big(  \Pi_{j=1}^{\TT-1} (1+\eta \theta_j \epsilon) \big)^2  r^2 \gamma
\\ & \geq 
  \frac{ 8 \eta t \big( \FT + 2 r^2 c_h +  \frac{\rho}{3} r^3 c_m^3 \big)}{(1-\beta)^2} 
 +
8 \frac{r^2 \sigma^2}{(1-\beta)^2}  +  4 \eta^2 \big( \frac{\beta}{1- \beta} \big)^2 c_m^2
  + 2 r^2 c_m^2.
\end{split}  
\end{equation}

Note that the left hand side is exponentially growing in $\TT$.
We can choose the number of iterations $\TT$ large enough to get the desired result.
Specifically, we claim that
$\TT \geq \frac{c (1-\beta)}{\eta \epsilon} \log ( \frac{L c_m \sigma^2 \rho c' c_h}{(1-\beta) \delta \gamma \epsilon }  ) $ for some constant $c>0$.
To see this, let us first apply $\log$ on both sides of (\ref{grow:1}),
\begin{equation} \label{grow:2}
2 \big( \sum_{j=1}^{\TT-1} \log ( 1 + \eta \theta_j  \epsilon) \big)
+ \log (r^2 \gamma )  \geq \log (8 a \TT + 8 b)
\end{equation}
where we denote
$a := \frac{ 4 \eta \big( \FT + 2 r^2 c_h +  \frac{\rho}{3} r^3 c_m^3 \big)}{(1-\beta)^2}$
    and $b:= 4 \frac{r^2 \sigma^2}{(1-\beta)^2}  +  2 \eta^2 \big( \frac{\beta}{1- \beta} \big)^2 c_m^2
  +  r^2 c_m^2$.
To proceed, we are going to use the inequality $\log (1+x) \geq \frac{x}{2} , \text{ for } x \in [0, \sim 2.51]$. We have that 
\begin{equation} \label{eta:another2}
 1 \geq \frac{\eta \epsilon}{ (1-\beta)}
\end{equation}
 as guaranteed by the constraint of $\eta$. So,
\begin{equation} \label{grow:3}
\begin{split}
& 2 \big( \sum_{j=1}^{\TT-1} \log ( 1 + \eta \theta_j  \epsilon) \big)
\overset{(a)}{\geq} \sum_{j=1}^{\TT-1} \eta \theta_j \epsilon
= \sum_{j=1}^{\TT-1} \sum_{k=0}^{j-1} \beta^k \eta \epsilon 
\\ & = \sum_{j=1}^{\TT-1} \frac{1 - \beta^j}{1-\beta} \eta \epsilon
\geq \frac{1}{1-\beta} ( \TT - 1  - \frac{\beta}{1-\beta}   ) \eta \epsilon.
\\ & \overset{(b)}{\geq} \frac{ \TT -1 }{ 2 ( 1 - \beta)} \eta \epsilon,
\end{split}
\end{equation}
where $(a)$ is by using the inequality
$\log (1+x) \geq \frac{x}{2} $ with $x= \eta \theta_j  \epsilon \leq 1$
 and $(b)$ is 
by making
$ \frac{\TT - 1}{ 2 (1-\beta)} \geq \frac{\beta}{ (1-\beta )^2} $,
which is equivalent to the condition that 
\begin{equation}\label{beta:1}
\TT \geq 1 + \frac{2 \beta}{1-\beta}
\end{equation}
Now let us substitute the result of (\ref{grow:3}) back to (\ref{grow:2}).
We have that
\begin{equation} \label{inq:TT}
\TT \geq 1+ \frac{2(1-\beta)}{\eta \epsilon} \log( \frac{ 8 a \TT + 8 b}{ \gamma r^2}    ),
\end{equation}
which is what we need to show.
By choosing $\TT$ large enough, 
\begin{equation} \label{TT}
\TT \geq 
\frac{c (1-\beta)}{\eta \epsilon} \log ( \frac{L c_m \sigma^2 \rho c' c_h}{(1-\beta) \delta \gamma \epsilon }  ) 
= O( (1-\beta) \log( \frac{1}{ (1-\beta) \epsilon } ) \epsilon^{-6}  )
\end{equation}
 for some constant $c>0$,
 we can guarantee that the above inequality (\ref{inq:TT}) holds.


\subsection{Proof of Lemma~\ref{lem:bbb}} \label{app:lem:bbb}

\textbf{Lemma~\ref{lem:bbb}}
(\cite{DKLH18})
 \textit{
Let us define the event
$\Upsilon_k := \{ \| \nabla f( w_{ k  \TT}) \| \geq \epsilon \text{ or } 
\lambda_{\min} (\nabla^2 f(w_{ k \TT})) \leq - \epsilon      \}.$
The complement is $\Upsilon_k^c := \{ \| \nabla f( w_{ k  \TT}) \| \leq \epsilon \text{ and } 
\lambda_{\min} (\nabla^2 f(w_{ k \TT})) \geq - \epsilon      \},$ 
which suggests that $w_{ k  \TT}$ is an $(\epsilon,\epsilon)$-second order stationary points.
Suppose that 
\begin{equation} \label{eq:lem:bbb}
\begin{aligned}
&  \E[ f(w_{ (k+1)  \TT} ) -  f(w_{ k \TT} )  | \Upsilon_k ] \leq - \Delta
\\ & 
\E[ f(w_{ (k+1)  \TT} ) -  f(w_{ k \TT} )  | \Upsilon_k^c ] \leq \delta \frac{\Delta}{2}.
\end{aligned}
\end{equation}
Set $T = 2 \TT \big( f(w_0) - \min_w f(w) \big) / (\delta  \Delta)$.
We return $w$ uniformly randomly from $w_{0},w_{\TT}, w_{2 \TT}, $  $\dots, w_{k \TT}, \dots,
w_{K \TT} $,
where $K := \lfloor T / \TT \rfloor$.
Then, with probability at least $1-\delta$,
we will have chosen a $w_k$ where $\Upsilon_k$ did not occur.
}

\begin{proof}
Let $P_k$ be the probability that $\Upsilon_k$ occurs.
\begin{equation}
\begin{aligned}
& \E[ f(w_{(k+1)\TT} ) - f(w_{k \TT}) ]
\\ & =   \E[ f(w_{(k+1)\TT} ) - f(w_{k \TT}) |  \Upsilon_k ]   P_k
+  \E[ f(w_{(k+1)\TT} ) - f(w_{k \TT}) |  \Upsilon_k^c ]   (1 - P_k )
\\ & \leq - \Delta P_k + \delta \Delta / 2 (1 - P_k)
\\ & =  \delta \Delta /2 - ( 1 + \delta/2) \Delta P_k
\\ & \leq \delta \Delta /2 - \Delta P_k.
\end{aligned}
\end{equation}
Summing over all $K$, we have
\begin{equation}
\begin{aligned}
& \frac{1}{K+1} \sum_{k=0}^K \E [ f(w_{ (k+1)  \TT} ) -  f(w_{ k \TT} )  ] 
\leq \Delta \frac{1}{K+1} \sum_{k=0}^K ( \delta/2 - P_k)
\\ & \Rightarrow \frac{1}{K+1} \sum_{k=0}^K P_k \leq \delta /2 + \frac{ f(w_0) - \min_w f(w)}{ (K+1) \Delta} \leq \delta
\\ &\Rightarrow \frac{1}{K+1} \sum_{k=0}^K (1 - P_k ) \geq 1 - \delta.
\end{aligned}
\end{equation}
\end{proof}


\subsection{Proof of Theorem~\ref{thm:main_escape}} \label{app:thm}

\textbf{Theorem~\ref{thm:main_escape}}
\textit{
Assume that the stochastic momentum satisfies CNC.
Set $r=O(\epsilon^2)$, $\eta = O(\epsilon^5)$, and $\TT =
\frac{c (1-\beta)}{\eta \epsilon} \log( \frac{L c_m \sigma^2 \rho c' c_h }{ (1-\beta)  \delta \gamma \epsilon }  ) = 
 O( (1-\beta) \log( \frac{L c_m \sigma^2 \rho c' c_h }{ (1-\beta)  \delta \gamma \epsilon }  ) \epsilon^{-6}  )$ for some constant $c > 0$.
If SGD with momentum (Algorithm~\ref{alg:0}) has
APAG property when gradient is large ($\| \nabla f(w) \| \geq \epsilon$), APCG$_{\TT}$
property when it enters a region of saddle points that exhibits a negative curvature ($\| \nabla f(w) \| \leq \epsilon$ and $\lambda_{\min}( \nabla^2 f(w) ) \leq - \epsilon$),
and GrACE property throughout the iterations,
then it
reaches an $(\epsilon, \epsilon)$ second order stationary point in $T = 2 \TT ( f(w_0) - \min_w f(w) ) / ( \delta \FT ) =  O( (1-\beta) \log( \frac{L c_m \sigma^2 \rho c' c_h }{ (1-\beta)  \delta \gamma \epsilon }  )\epsilon^{-10} )$ iterations with high probability $1-\delta$, where $\FT = O(\epsilon^4)$.
}

\paragraph{Proof sketch of Theorem~\ref{thm:main_escape}}
In this subsection, we provide a sketch of the proof of Theorem~\ref{thm:main_escape}.
The complete proof is available in Section~\ref{app:thm}. Our proof uses a lemma in (\cite{DKLH18}),
which is Lemma~\ref{lem:bbb} below. The lemma guarantees that uniformly
sampling a $w$ from $\{ w_{k\TT}\}$, $k=0,1,2, \dots, \lfloor T / \TT \rfloor$ gives
an $(\epsilon,\epsilon)$-second order stationary point with high probability.
We replicate the proof of Lemma~\ref{lem:bbb} in Section~\ref{app:lem:bbb}.

\begin{lemma} \label{lem:bbb}(\cite{DKLH18})
Let us define the event
$\Upsilon_k := \{ \| \nabla f( w_{ k  \TT}) \| \geq \epsilon \text{ or }
\lambda_{\min} (\nabla^2 f(w_{ k \TT})) \leq - \epsilon      \}.$
The complement is $\Upsilon_k^c := \{ \| \nabla f( w_{ k  \TT}) \| \leq \epsilon \text{ and }
\lambda_{\min} (\nabla^2 f(w_{ k \TT})) \geq - \epsilon      \},$
which suggests that $w_{ k  \TT}$ is an $(\epsilon,\epsilon)$-second order stationary points.
Suppose that
\begin{equation} \label{eq:lem:bbb}
\begin{aligned}
  \E[ f(w_{ (k+1)  \TT} ) -  f(w_{ k \TT} )  | \Upsilon_k ] \leq - \Delta
  \text{\quad\&\quad }   
\E[ f(w_{ (k+1)  \TT} ) -  f(w_{ k \TT} )  | \Upsilon_k^c ] \leq \delta \frac{\Delta}{2}.
\end{aligned}
\end{equation}
Set $T = 2 \TT \big( f(w_0) - \min_w f(w) \big) / (\delta  \Delta)$. \footnote{One can use any upper bound of $f(w_0) - \min_w f(w)$ as $f(w_0) - \min_w f(w)$ in the expression of $T$.}
We return $w$ uniformly randomly from $w_0, w_{\TT}, w_{2 \TT}, $  $\dots, w_{k \TT}, \dots,
w_{K \TT} $,
where $K := \lfloor T / \TT \rfloor$.
Then, with probability at least $1-\delta$,
we will have chosen a $w_k$ where $\Upsilon_k$ did not occur.
\end{lemma}

To use the result of Lemma~\ref{lem:bbb}, we need to let the conditions in (\ref{eq:lem:bbb}) be satisfied.
We can bound $\E [ f(w_{(k+1)\TT} ) - f(w_{k \TT}) |  \Upsilon_k ] \leq - \FT$,
based on the analysis of the large gradient norm regime (Lemma~\ref{lem:0b}) and the analysis for the scenario when the update is with small gradient norm but a large
negative curvature is available (Subsection~\ref{sub:hessian}).
For the other condition, $\E[ f(w_{ (k+1)  \TT} ) -  f(w_{ k \TT} )  | \Upsilon_k^c ] \leq \delta \frac{\FT}{2}$, it requires that the expected amortized
increase of function value due to taking the large step size $r$ is limited
(i.e. bounded by $\delta \frac{\FT}{2}$)
when $w_{ k \TT}$ is a second order stationary point.
By having the conditions satisfied, we can apply Lemma~\ref{lem:bbb} and finish the proof of the theorem.


\paragraph{Proof of Theorem~\ref{thm:main_escape}}

\begin{proof}
Our proof is based on Lemma~\ref{lem:bbb}. 
So, let us consider the events in Lemma~\ref{lem:bbb}, 
$\Upsilon_k := \{ \| \nabla f( w_{ k  \TT}) \| \geq \epsilon \text{ or } 
\lambda_{\min} (\nabla^2 f(w_{ k \TT})) \leq - \epsilon      \}.$
We first show that $\E [ f(w_{(k+1)\TT} ) - f(w_{k \TT}) |  \Upsilon_k ] \leq \FT$.
\\
\\ \noindent
\textbf{When $\| \nabla f(w_{k \TT}) \| \geq \epsilon$:}\\
Consider that
$\Upsilon_k$ is the case that $\| \nabla f(w_{k \TT}) \| \geq \epsilon$.
Denote $\t := k \TT$ in the following. We have that
\begin{equation}
\begin{aligned}
& \EE{\t}[ f(w_{\t + \TT}) - f(w_{\t})  ]
\\ & = \sum_{t=0}^{\TT-1} \EE{\t}[   \E[  f(w_{\t + t + 1}) - f(w_{\t + t}) | w_{0:\t+t} ] ]
\\ & = \EE{\t}[ f(w_{\t + 1}) - f(w_{\t }) ] 
+  \sum_{t=1}^{\TT-1} \EE{\t}[ \E[ f(w_{\t + t + 1}) - f(w_{\t + t}) | w_{0:\t+t} ] ], 
\end{aligned}
\end{equation}
which can be further bounded as
\begin{equation} \label{need:-1}
\begin{aligned}
\\ & \overset{(a)}{\leq} 
- \frac{r}{2} \| \nabla f(w_{\t}) \|^2   
  +  \frac{ L r^2 c_m^2 }{2}
+  \sum_{t=1}^{\TT-1} \EE{\t}[ \E[ f(w_{\t + t + 1}) - f(w_{\t + t}) | w_{0:\t+t} ] ] 
\\ & \overset{(b)}{\leq} 
- \frac{r}{2} \| \nabla f(w_{\t}) \|^2   
  +  \frac{ L r^2 c_m^2 }{2}
+  \sum_{t=1}^{\TT-1} \big(  \eta^2 c_h + \frac{\rho}{6} \eta^3 c_m^3 \big)
\\ & \overset{(c)}{\leq} 
- \frac{r}{2} \| \nabla f(w_{\t}) \|^2   
  +  \frac{ L r^2 c_m^2 }{2}
+ r^2 c_h + \frac{\rho}{6} r^3 c_m^3  
\\ & \overset{(d)}{\leq}  - \frac{r}{2} \| \nabla f(w_{\t}) \|^2    +  L r^2 c_m^2 + r^2 c_h
\\ & \overset{(e)}{\leq}  - \frac{r}{2} \epsilon^2    +  L r^2 c_m^2 + r^2 c_h
 \overset{(f)}{\leq}  - \frac{r}{4} \epsilon^2   
 \overset{(g)}{\leq}  - \FT,
\end{aligned}
\end{equation}
where $(a)$ is by using Lemma~\ref{lem:0a} with step size $r$,
$(b)$ is by using Lemma~\ref{lem:ch},
$(c)$ is due to the constraint that $\eta^2 \TT \leq r^2$,
$(d)$ is by the choice of $r$, $(e)$ is by $\| \nabla f(w_t) \| \geq \epsilon$,
$(f)$ is by the choice of $r$ so that $r \leq \frac{\epsilon^2}{4( L c_m^2 + c_h)}$,
and $(g)$ is by 
\begin{equation} \label{f:2}
 \frac{ r}{4} \epsilon^2 \geq \FT.
\end{equation}

\noindent
\textbf{When $\| \nabla f(w_{k \TT}) \| \leq \epsilon$
and $ \lambda_{\min} ( \nabla^2 f(w_{k \TT}) ) \leq - \epsilon $:}

The scenario that 
$\Upsilon_k$ is the case that
$\| \nabla f(w_{k \TT}) \| \leq \epsilon$
and $ \lambda_{\min} ( \nabla^2 f(w_{k \TT}) ) \leq - \epsilon $
has been analyzed in Section~\ref{app:lem:escape},
which guarantees that 
$\E[ f(w_{\t + \TT}) - f(w_{\t})  ] \leq  -\FT $ under the setting.

\noindent
\\\textbf{When $\| \nabla f(w_{k \TT}) \| \leq \epsilon$
and $ \lambda_{\min} ( \nabla^2 f(w_{k \TT}) ) \geq -  \epsilon $:}

Now let us switch to show that
$\E[ f(w_{ (k+1)  \TT} ) -  f(w_{ k \TT} )  | \Upsilon_k^c ] \leq \delta \frac{\FT}{2}.$
Recall that $\Upsilon_k^c$ means that
$\| \nabla f(w_{k \TT}) \| \leq \epsilon$ and $ \lambda_{\min} ( \nabla^2 f(w_{k \TT}) ) \geq - \epsilon $.
Denote $\t := k \TT$ in the following.
We have that
\begin{equation} \label{need:0}
\begin{split}
&\EE{\t}[ f(w_{\t + \TT}) - f(w_{\t})  ]
= \sum_{t=0}^{\TT-1} \EE{\t}[   \E[  f(w_{\t + t + 1}) - f(w_{\t + t}) | w_{0:\t+t} ] ]
\\ & = \EE{\t}[ f(w_{\t + 1}) - f(w_{\t })  ] 
+  \sum_{t=1}^{\TT-1} \EE{\t}[ \E[ f(w_{\t + t + 1}) - f(w_{\t + t}) | w_{0:\t+t} ] ] 
\\ & \overset{(a)}{\leq} 
 r^2 c_h + \frac{\rho}{6} r^3 c_m^3 + 
\sum_{t=1}^{\TT-1} \EE{\t}[ \E[ f(w_{\t + t + 1}) - f(w_{\t + t}) | w_{0:\t+t} ] ] 
\\ & \overset{(b)}{\leq} 
 r^2 c_h + \frac{\rho}{6} r^3 c_m^3 + 
\sum_{t=1}^{\TT-1}
\big( \eta^2 c_h + \frac{\rho}{6} \eta^3 c_m^3 \big) 
\\ & \overset{(c)}{\leq} 
2 r^2 c_h + \frac{\rho}{3} r^3 c_m^3 \leq 4 r^2 c_h
\overset{(d)}{\leq} \frac{\delta \FT}{2}.
\end{split}
\end{equation}
where $(a)$ is by using Lemma~\ref{lem:ch} with step size $r$,
$(b)$ is by using Lemma~\ref{lem:ch} with step step size $\eta$,
$(c)$ is by setting $\eta^2 \TT \leq r^2 $ and $\eta \leq r$,
$(d)$ is by the choice of $r$ so that $8 r^2 c_h \leq \delta \FT$.

Now we are ready to use Lemma~\ref{lem:bbb}, since both the conditions are satisfied.
According to the lemma and the choices of parameters value on Table~\ref{table:parameters}, we can set $T = 2 \TT \big( f(w_0) - \min_w f(w) \big) / (\delta  \FT)=
  O( (1-\beta) \log( \frac{L c_m \sigma^2  \rho c' c_h}{ (1-\beta)  \delta \gamma \epsilon }  )\epsilon^{-10} )$,
which will return a $w$ that is an $(\epsilon,\epsilon)$ second order stationary point. 
Thus, we have completed the proof.

\end{proof}

\section{Discussion: Over-parametrization}

In the previous sections, we show that Polyak's momentum helps fast saddle point escape. In this section, we consider a different technique  --- over-parametrization, which is recently very popular in modern machine learning.
Specifically, let us consider over-parametrizing the phase retrieval problem (\ref{obj:phase}) as follows,
\begin{equation} \label{obj:over}
\begin{split}
\textstyle
 \min_{W \in \reals^{d \times K}}\frac{1}{4n} \sum_{i=1}^n \big( (x_i^\top w^{(1)})^2 + (x_i^\top w^{(2)})^2 + ... + (x_i^\top w^{(K)})^2  - y_i    \big)^2,
\end{split}
\end{equation}
where $x_i \sim N(0,I_d)$, $y_i = (x_i^\top w_*)^2$, and  $w_* \in \reals^d$.
This is over-parametrization since the objective has more variables than necessary.

Let us try a simulation. We set the dimension $d=10$ and the number of training samples $n=200$.
We let $w_* = e_1$ with $e_1$ being the unit vector.
Each neuron $w^{(k)} \in \reals^d$ ($k\in [K]$) of the student network is initialized by sampling from an isotropic distribution 
and is close to the origin (i.e. $w_0^{(k)} \sim 0.01 \cdot N(0,I_d/d)$ ).
Figure~\ref{fig:over} show a very interesting result
of applying vanilla gradient descent 
to train different sizes of models. Each curve represents the progress of gradient descent for different $K$. 
It shows that for a larger $K$, gradient descent escapes ``the origin'' faster,
however, we remark that the origin is different for a different $K$, as each problem of $K$ has a different dimensional parameter space.

We also report a \emph{distance} measure on the same figure.
\begin{equation} \label{dis}
\textstyle \dist( W, w_*):= \min_{q \in \reals^K: \| q\|_2 \leq 1}  \|  W - w_*q^\top \|.
\end{equation}
This is due to our observation that for any $K$,
the global optimal solutions of (\ref{obj:over}) that achieve zero testing error are
$w_* q^\top \in \reals^{d \times K}$ for any $q \in \reals^K$ such that $\| q \|_2 = 1$.
To see this, substitute $W = w_* q^\top  \in \reals^{d \times K}$ into (\ref{obj:over}).
We have that for any $x_i \in \reals^d$ it holds that
$
(x_i^\top w^{(1)})^2 + (x_i^\top w^{(2)})^2 + ... + (x_i^\top w^{(K)})^2  - y_i =  \| x_i^\top W \|^2_F - (x_i^\top w_*)^2  = \text{tr}\big( (x_i^\top w_* q^\top)^\top (x_i^\top w_* q^\top) \big) - (x_i^\top w_*)^2   = 0 $.
Therefore, the metric $\dist( W, w_*)$ as be viewed as a surrogate of the testing error.
In particular,
$\dist( W_t, w_*)$ represents the distance of the current iterate $W_t$ and its closest global optimal solution to the over-parametrized objective (\ref{obj:over}) that achieves zero testing error. 
Note that the argmin of (\ref{dis}) is
$q_* := \frac{ W^\top w_*}{ \|W^\top w_*\|_2 } = \arg\min_{q \in \reals^K: \| q\|_2 \leq 1}  \|  W - w_*q^\top \|.$
Subfigure (b) of Figure~\ref{fig:over} plots the distance of the iterates generated by gradient descent and its closet global optimal solution for different sizes $K$ of models. We see that over-parametrization enables shrinking the distance $\dist(W^{\#K}_t, w_*)$ faster. 
\begin{figure}[t]
\centering
   \begin{subfigure}[t]{0.45\textwidth}
        \label{subfig-1:func1} \includegraphics[width=1.0\textwidth]{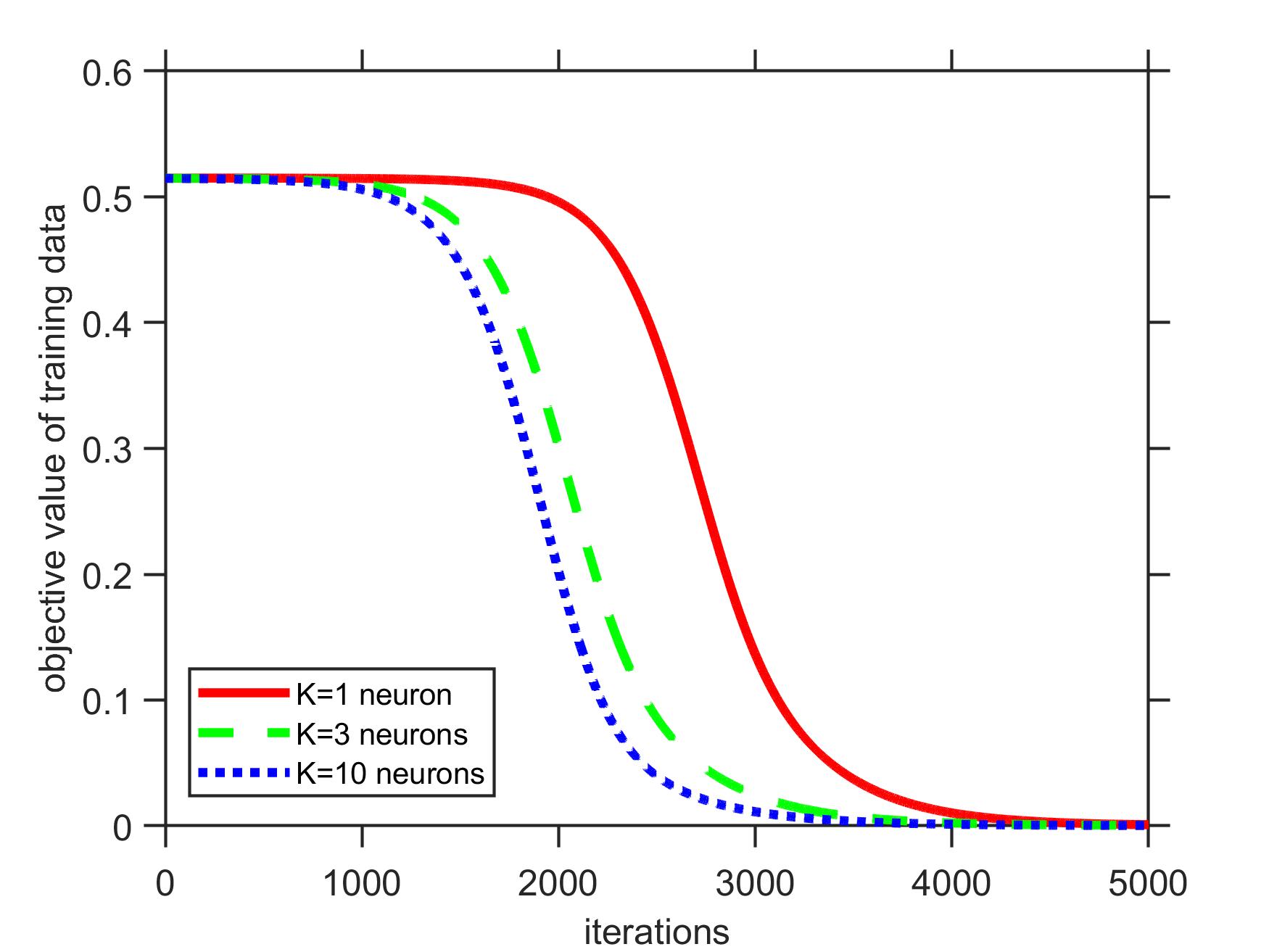}
        \caption{Objective value (\ref{obj:over}) vs. iteration $t$. }
   \end{subfigure} %
   \begin{subfigure}[t]{0.45\textwidth}
        \label{subfig-1:func1} \includegraphics[width=1.0\textwidth]{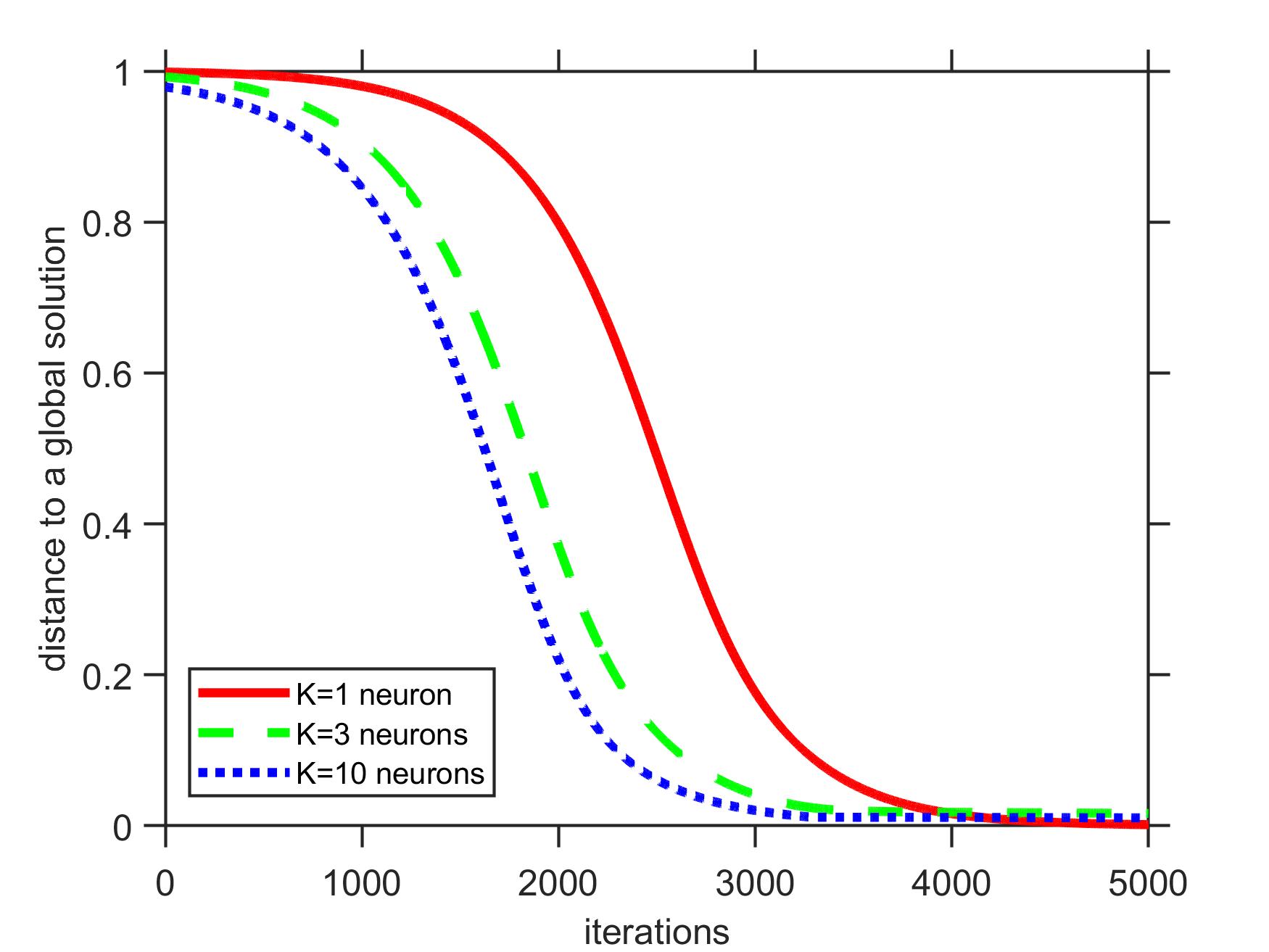}
        \caption{Distance (\ref{dis}) vs. iteration $t$.}
   \end{subfigure} %
   \caption{Vanilla gradient descent for training different over-parametrized models (\ref{obj:over}). We see that over-parametrization helps the iterate of gradient descent escapes the origin faster and hence converges faster.} \label{fig:over}
\end{figure}

\paragraph{Informal analysis}

Let us provide an informal analysis to explain why over-parametrization might help the fast escape.
Given the infinite number of samples $x_i \sim N(0,I_d)$, 
the population objective of (\ref{obj:over}) is
\begin{equation}
F(W):= \left(  \sum_{k=1}^K \| w^{(k)} \|^2 - \| w_* \|^2 \right)^2
+ 2 \| \sum_{k=1}^K w^{(k)} ( w^{(k)} )^\top - w_* w_*^\top \|^2_F.
\end{equation}
The Hessian at the origin $\nabla^2 F(0_{d \times K}) \in \reals^{d K \times d K}$ is in the following form,
\[
 \begin{bmatrix}  
 - 4 \| w_* \| I_d - 8 w_* w_*^\top &  & & \\
  & - 4 \| w_* \| I_d - 8 w_* w_*^\top  & & \\
  &   & \ddots & \\
    & & & - 4 \| w_* \| I_d - 8 w_* w_*^\top 
   \end{bmatrix}.
\]
Let $v^{(K)} \in \reals^{d \times K}$ be the bottom eigenvector of $\nabla^2 F(0_{d \times K})$
and $v^{(1)} \in \reals^d$ be the bottom eigenvector of $\nabla^2 F(0_{d \times 1})$.
We might be able to write
\begin{equation}
(v^{(K)})^\top \nabla^2 F(0_{d \times K}) v^{(K)} =  K \times  
(v^{(1)})^\top \nabla^2 F(0_{d \times 1}) v^{(1)}.
\end{equation}
That is, effectively a $K$-times-larger size of models results in a $k$-times-larger negative curvature at the origin.

A natural question is then ``\textit{Is the effect of over-parametrization equivalent to using a larger step size $\eta$?}''. To answer the question,
let us also consider an over-parametrized version of (\ref{obj-scale}),
\begin{equation} \label{obj-scale}
\begin{split}
\textstyle
 \min_{W \in \reals^{d \times K} }\hat{f}(W)  :=    \frac{1}{4n} \sum_{i=1}^n \big(C (x_i^\top w^{(1)})^2 + C (x_i^\top w^{(2)})^2 + ... + C (x_i^\top w^{(K)})^2  - y_i    \big)^2.
 \end{split}
\end{equation}
Applying vanilla gradient descent to (\ref{obj-scale}) of $K=1$ might be viewed as using a $C$-times larger step size as if it were optimizing (\ref{obj:over}), since $\nabla \hat{f}(W) = \frac{C}{n} \sum_{i=1}^n \big( C (x_i^\top w^{(1)})^2 - y_i \big) x_i$.
On Figure~\ref{fig:scale}, we report gradient descent with the same step size $\eta$ for solving (\ref{obj-scale}) under different $C$'s and $K$'s. 
It suggests that \emph{to some degree}, over-parametrization is kind of like using a larger step size. However, in order to converge to a good solution, an upper-bound of the step size should be required. A deeper investigation needs to be conducted. It is also interesting to check if the effect of over-parametrization also exists in other problems as well, not necessarily limited to phase retrieval.
\begin{figure}[t]
\centering
        \label{subfig-1:func1} \includegraphics[width=0.5\textwidth]{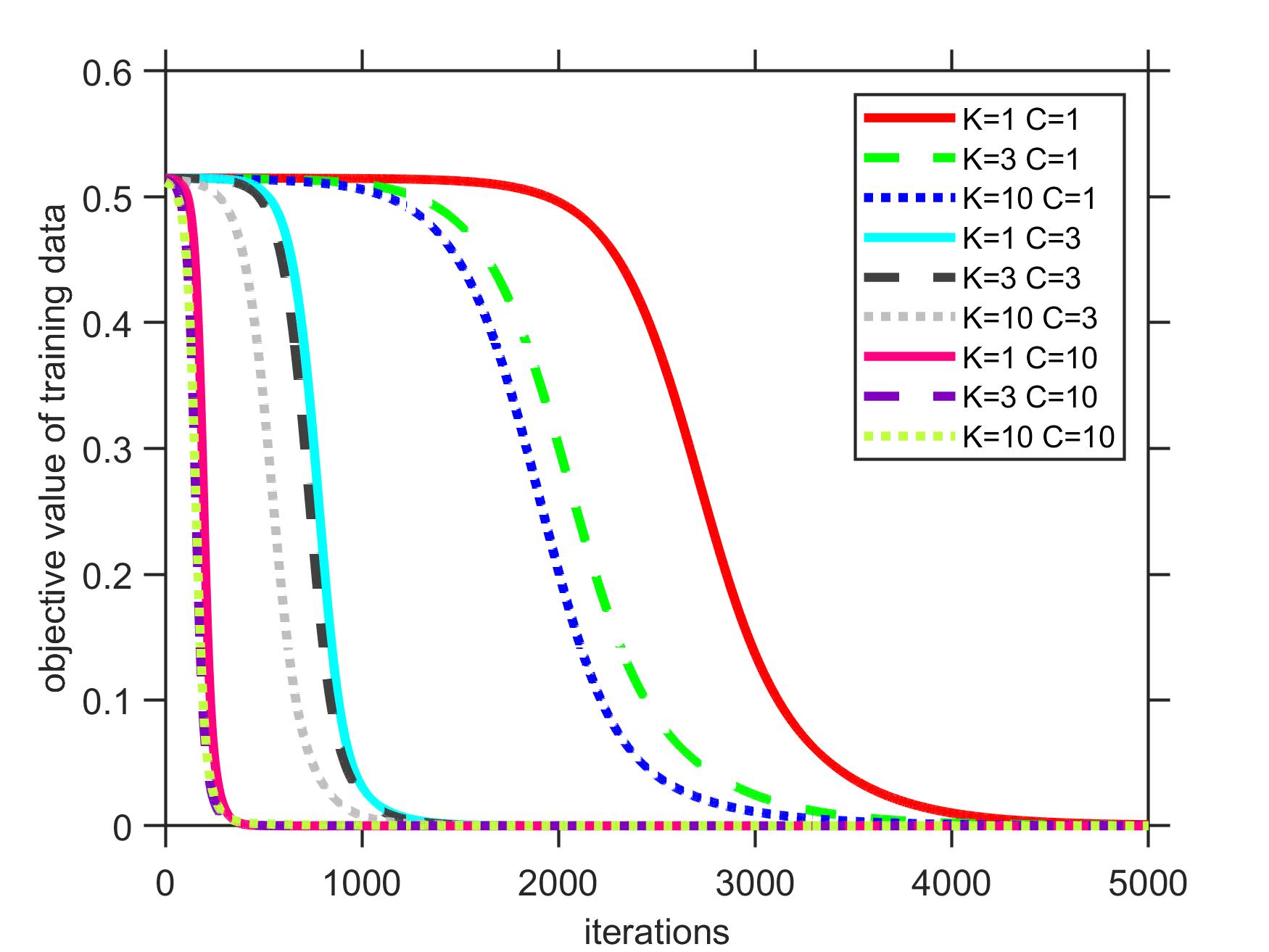}
   \caption{Gradient descent for (\ref{obj-scale}) under different sizes $K$ and scales of the outputs $C$} \label{fig:scale}
\end{figure}

\paragraph{Related works}

The observation that
a larger network can be trained to achieve a certain level of prediction performance with fewer iterations than that of a smaller net can be dated back as early as the work of Livni et al. (Section 5 of \cite{LSS14}), who try different levels of over-parametrization and report that SGD converges much faster and finds a better solution when it is used to train a larger network. However, the reason why over-parametrization can lead to an acceleration still remains a mystery, and very little theory has helped explain the observation, with perhaps the notable exception of Arora et al. \cite{ACH18}. 
Arora et al. \cite{ACH18} consider over-parametrizing a single-output linear regression with $l_p$ loss for $p>2$--the square loss corresponds to $p=2$--and they study the linear regression problem by replacing the model $w \in \reals^d$ by another model $w_1 \in \reals^d$ times a scalar $w_2 \in \reals$. They show that the dynamics of gradient descent on the new over-parametrized model are equivalent to the dynamics of gradient descent on the original objective function with an adaptive learning rate plus some momentum terms. However, in practice, people actually use the techniques of over-parametrization, adaptive learning rate, and momentum simultaneously in deep learning (see e.g. \cite{HHS17,KB15,KH1918,SMDH13}), as each technique appears to contribute to performance and they may, to some extent, be complementary. It has been suggested that over-parameterizing a model leads implicitly to an adaptive learning rate or momentum, but this does not appear to fully explain the performance improvement.

Finally, we also want to acknowledge some related works of understanding over-parametrization in different aspects (e.g. \cite{BG19,EGKZ20,MNSBHS20}).

\section{Conclusion}

In this work, we identify three properties that guarantee SGD with momentum in reaching a second-order stationary point faster by a higher momentum,
which justifies the practice of using a large value of momentum parameter $\beta$. 
We show that a greater momentum leads to escaping strict saddle points faster due to that SGD with momentum recursively enlarges the projection to an escape direction. 
However, how to make sure that SGD with momentum has the three properties is not very clear. 
It would be interesting to identify conditions that guarantee SGD with momentum to have the properties.
Perhaps a good starting point is understanding why the properties hold in phase retrieval.
We also discuss the effect of over-parametrization and report some interesting observations. We hope our results shed light on understanding the interaction between momentum and over-parametrization for exploiting negative curvatures.


    \makeBibliography
\end{thesisbody}

\end{document}